\documentclass[12pt]{book}

\usepackage{amssymb,amscd,amsmath,amsthm,color}
\newtheorem{theorem}{Theorem}[section]
\newtheorem{lemma}[theorem]{Lemma}
\newtheorem{cor}[theorem]{Corollary}
\newtheorem{prop}[theorem]{Proposition}
\newtheorem{conjecture}[theorem]{Conjecture}
\usepackage{enumerate,amssymb}
\theoremstyle{definition}
\newtheorem{definition}[theorem]{Definition}

\newtheorem{question}[theorem]{Question}
\theoremstyle{remark}
\newtheorem{remark}[theorem]{Remark}

\numberwithin{equation}{section}

\def\bS{\mathbb{S}}
\def\bJ{\mathbb{J}}
\def\bB{\mathbb{B}}
\def\bK{\mathbb{K}}
\def\bF{\mathbb{F}}
\def\eps{\varepsilon}
\def\bA{\mathbb{A}}
\def\bC{\mathbb{C}}
\def\bE{\mathbb{E}}
\def\bM{\mathbb{M}}
\def\bN{\mathbb{N}}
\def\bR{\mathbb{R}}
\def\diag{\mathrm{diag}}
\def\bI{\mathbb{I}}
\def\bB{\mathbb{B}}

\begin{document}

\setlength{\baselineskip}{0.2 in}

\setlength{\parindent}{5.0 true mm}

\setcounter{page}{1}

\title{

\ \vskip 150pt \centerline{\Huge \bf A journey into Matrix Analysis }

{\color{blue}
\vskip 30pt
{\Large 
$$
\left|\begin{bmatrix} 1&4&5 \\
2&6&7 \\
3&8&9 \\
\end{bmatrix}\right|^2
= U\left|\begin{bmatrix} 1 \\ 2\\ 3\end{bmatrix}\right|^2U^* + V\left|\begin{bmatrix} 4&5 \end{bmatrix}\right|^2V^*+
W\left|\begin{bmatrix} 6&7 \\ 8&9
\end{bmatrix}\right|^2W^*
$$
}
}

\ \vskip 100pt \centerline{\Large \bf Jean-Christophe Bourin }

}

\date{}

\maketitle

\chapter*{Outline of the thesis}

Matrix Analysis is essential in many areas of mathematics and sciences. This is the main topic of 
this thesis which presents a substantial part of my research.  There are 9 chapters.

\vskip 10pt\noindent Chapter 1 is an introductory chapter, some results from the period 1999-2010 are given. 

\vskip 10pt\noindent
 Chapters 2--8 are the central part of the thesis. Each  chapter presents
 an  article (with a blue title). This article is complemented with an additional section,  {\it Around this article}.

  We may divide these chapters into three groups. 
\begin{itemize}  
\item[] Chapters 2-4 deal with matrix inequalities, Chapter 2 is concerned with norm inequalities and logmajorization and Chapters 3-4 with functional calculus and a unitary orbit
technique that  I started to develop in 2003.

\item[]  Chapter 5 is a time-break in infinite dimensional Hilbert space operators

\item[]  Chapters 6-8 establish several decompositions for partitioned matrices, especially for positive block matrices. Some norm inequalities involving the numerical range are derived.
\end{itemize}  

\vskip 10pt\noindent Chapter 9 is for students; a proof of the Spectral Theorem for bounded operators is derived from the matrix case.

\vskip 10pt\noindent 
The thesis is divided into two parts of similar size.
  The first part establishes matrix inequalities involving symmetric norms, eigenvalues and unitary orbits. These results are also used in the second part,  dealing with operator diagonals of Hilbert space operators, partitioned matrices, and numerical ranges. 

\tableofcontents


\part{Matrix inequalities, norms and unitary orbits}

\ 
\ 

\vskip 100pt

{\Large
{\color{blue}
$$
f''(t)\le 0, \  f(0)\ge 0, \  A\ge 0, \ Z^*Z\ge I
$$
}
}

\vskip -20pt
{\Large
{\color{blue}
$$
\Longrightarrow
$$
}
}

\vskip -20pt

{\Large
{\color{blue}
$$
{\mathrm{Tr\,}} f(Z^*AZ)  \le {\mathrm{Tr\,}} Z^*f(A)Z 
$$
}
}



\chapter{Some results as an introduction}

This chapter gives some   theorems that are not discussed in the thesis. We will only mention their relation with the other chapters. This chapter covers a significant part of my work in the period 1999-2010.

\section{Rearrangement inequalities}

In an old paper \cite{B1999} (see \cite{B2006-mv} for a simplified proof)
I obtained the following inequality for the Hilbert-Schmidt norm $\|\cdot\|_2$.

Denote by $\bM_n$ the space of $n$-by-$n$ matrices and by $\bM_n^+$ the positive semidefinite cone. A pair $(A,B)$ in $\bM_n^+$ is called a monotone pair if
$A=f(C)$ and $B=g(C)$ for two non-decreasing functions $f(t),\,g(t)$ and some $C\in\bM_n^+$. If $f(t)$ is non-decreasing and
$g(t)$ is non-increasing, then we say that $(A,B)$ is antimonotone.

\begin{theorem}\label{th-monotone1} Let $Z\in\bM_n$ be a normal matrix. If $(A,B)$ is a monotone pair in $\bM_n^+$ then, 
$$
\| AZB \|_2 \le \| ZAB \|_2.
$$
If $(A,B)$ is antimonotone, then the inequality reverses.
\end{theorem}

The result still holds for Hilbert space operators whenever $A$ is Hilbert-Schmidt. Letting $Z$ be a rank one projection, we recapture
Chebyshev inequality for pairs of non-decreasing functions on $(0,1)$,
$$
\int_0^1 f(t)\, {\mathrm{d}}t\ \int_0^1  g(t) \,{\mathrm{d}}t \le \int_0^1  f(t)g(t) \, {\mathrm{d}}t.
$$
Letting $Z$ be a unitary matrix, we recapture von Neumann trace inequality (1937) for general matrices $X,Y\in\bM_n$,
$$
\left|{\mathrm{Tr}\,} XY \right| \le \sum_{j=1}^n \mu_j(X) \mu_j(Y)
$$
where $\mu_1(X)\ge \cdots \ge \mu_n(X)$ are the singular values of $X$.

Hence studying the product  $AZB$ is useful. Chapter 2 will be devoted to the study of the functional $(s,t)\mapsto A^sZB^t$ for a general matrix $Z$ and any pair of positive matrices $A,B$. This will provide a number of new inequalities extending some famous inequalities.

Two consequences of Theorem \ref{th-monotone1} for the operator norm $\|\cdot\|_{\infty}$ are given in \cite{B1999}.  If $(A,B)$ is a monotone pair in $\bM_n^+$, then
\begin{equation}\label{eqproj}
\| AEB\|_{\infty} \le \| EAB\|_{\infty}
\end{equation}
for all projections $E$, and
\begin{equation*}\label{eqsemi}
\| ASB\|_{\infty} \le \sqrt{2}\| SAB\|_{\infty}
\end{equation*}
for all semi-unitary matrices $S$. 
Here semi-unitary means that $S^*S=SS^*=F$ for some projection $F$. This suggests the following conjecture for  any symmetric norm $\|\cdot\|$ (i.e., unitarily invariant norm).

\vskip 5pt\noindent
{\bf Conjecture A.} {\it Let $Z\in\bM_n$ be a normal matrix. If $(A,B)$ is a monotone pair in $\bM_n^+$, then, for all symmetric norms,
$$
\| AZB \| \le \sqrt{2}\| ZAB \|.
$$
}

\vskip 7pt
Conjecture A is supported by the  remarkable result recently proved by Eric Ricard (private communication) : Conjecture A holds true if we take the constant 8. 

The norm inequality  \eqref{eqproj} can be considerably extended. In \cite{B2003} I give the singular value inequality:

\begin{theorem}\label{th-monotone2}  Let $(A,B)$ be a monotone pair in $\bM_n^+$ and  let $E\in\bM_n$ be a projection. Then for some unitary $V\in\bM_n$,
$$
\left| AEB \right| \le V\left| EAB \right| V^*
$$
\end{theorem}

\vskip 5pt
From this result follow several eigenvalue inequalities for compressions onto subspaces, and more generally for unital,  positive  linear map $\Phi$. For instance :

\begin{cor}\label{cor-monotone}  Let $(A,B)$ be a monotone pair in $\bM_n^+$ and  let $\Phi:\bM_n\to \bM_d$ be a positive, unital linear map. Then for some unitary $V\in\bM_d$,
$$
\Phi(ABA) \le V\Phi(A)\Phi(B)\Phi(A)  V^*.
$$
\end{cor}

\vskip 5pt
This was recorded in a joint paper with  Ricard \cite{BR}.
Concerning Rearrangement inequalities, one may state another conjecture \cite{bourin-IJM}.

\vskip 7pt\noindent
{\bf Conjecture B.} {\it Let $A,B \in \bM_n^+$ and $p,q>0$. Then, for all symmetric norms,
$$
\| A^pB^q+ B^pA^q \| \le \| A^{p+q} + B^{p+q} \|.
$$
}

Several authors, including Audenaert, Bhatia, Kittaneh, proved some very special cases of the conjecture, but the general case is still open.

\section{Subadditivity and superadditivity }

\vskip 10pt
In 1969, Rotfel'd stated a remarkable trace inequality for  all non-negative concave functions $f(t)$ defined  on $[0,\infty)$ and any pair $A,B\in\bM_n^+$, 
$$
{\mathrm{Tr\,}}f(A+B)\le {\mathrm{Tr\,}}f(A) + {\mathrm{Tr\,}}f(B).
$$

In a joint paper with  Aujla \cite{AB} we give the stronger statement,
\begin{equation}\label{eqAB}
f(A+B)\le Uf(A)U^* + Vf(B)V^*
\end{equation}
for some unitary matrices $U,V\in\bM_n$. This result and a number of related Jensen's type inequalities are discussed at length in Chapter 3.

The Rotfel'd trace inequality can also be extended as a norm inequality. In the paper written with Uchiyama \cite{BU}, we obtain the following theorem.

\vskip 5pt
\begin{theorem}\label{T-BU} If $f(t)$ is a nonnegative concave function on $[0,\infty)$ and
$A,B\in\bM_n^+$, then, for all symmetric norms,
$$
\| f(A+B)\| \le \| f(A) + f(B)\|.
$$
\end{theorem}

\vskip 5pt
This result was first proved for {\it operator} concave  functions  by Ando and Zhan (1999).
Our proof (2007) is completely different, and does
 not use the theory of operator monotone functions. A part of our proof comes from \cite{B3} where the following theorem was obtained.

\vskip 5pt
\begin{theorem}\label{T-B3} Let $f(t)$ be a nonnegative concave function on $[0,\infty)$, let $A\in\bM_n^+$, and let $Z\in\bM_n$ be expansive. Then, for all symmetric norms,
$$
\| f(Z^*AZ)\| \le \| Z^*f(A)Z\| .
$$
\end{theorem}

\vskip 5pt
Here $Z$ expansive means that $Z^*Z\ge I$. For the trace, the function $f(t)$ is not necessarily increasing, and we get the trace inequality \cite{B1} of Page 11. We may gather Theorems \ref{T-BU} and
\ref{T-B3} into a single statement \cite{BL-JOT} :

\vskip 5pt
\begin{theorem}\label{T-JOT}
In the space $\bM_n$, let $\{A_i\}_{i=1}^m$ be positive and  let  $\{Z_i\}_{i=1}^m$ be
expansive. Let
 $f(t)$ be a non-negative concave function on $[0,\infty)$. Then, for all symmetric norms,
$$
\left\|f\left(\sum_{i=1}^m Z_i^*A_iZ_i\right) \right\| \le \left\|  \sum_{i=1}^m
Z_i^*f(A_i)Z_i \right\|.
$$
\end{theorem}

\vskip 5pt
A basic principle for symmetric norms shows that Theorem \ref{T-JOT} entails a reversed inequality in case of convex functions.

\vskip 5pt
\begin{cor}
 In the space $\bM_n$, let $\{A_i\}_{i=1}^m$ be positive and  let  $\{Z_i\}_{i=1}^m$ be
expansive. Let $g(t)$ be a non-negative convex function  on $[0,\infty)$ with
  $g(0)=0$.  Then, for all symmetric norms,
\begin{equation*}
\left\| \sum_{i=1}^m  Z_i^*g(A_i)Z_i \right\| \le \left\|  g\left(\sum_{i=1}^m 
Z_i^*A_iZ_i \right) \right\|. \end{equation*}
\end{cor}

\vskip 5pt
Another extension of Theorem \ref{T-BU} is given in \cite{bourin-PAMS}.

\vskip 5pt
\begin{theorem}\label{T-normal} If $f(t)$ is a nonnegative concave function on $[0,\infty)$ and
$A,B\in\bM_n$ are normal matrices, then,  for all symmetric norms,
$$
\| f(|A+B|)\| \le \| f(|A|) + f(|B|)\|.
$$
\end{theorem}

\vskip 5pt
A number of corollaries follow.

\vskip 5pt
\begin{cor}\label{cor-normal1} If $f(t)$ is a nonnegative concave function on $[0,\infty)$ and
$Z\in\bM_n$ has the Cartesian decomposition $Z=A+iB$, then,  for all symmetric norms,
$$
\| f(|Z|)\| \le \| f(|A|) + f(|B|)\|.
$$
\end{cor}

\vskip 5pt
\begin{cor}\label{cor-normal2} If $f(t)$ is a nonnegative concave function on $[0,\infty)$ and
$Z\in\bM_n$, then,  for all symmetric norms,
$$
\| f(|Z+Z^*|)\| \le \| f(|Z|) + f(|Z^*|)\|.
$$
\end{cor}

\section{Diagonal blocks}

Chapters 6-8 deal with partitioned matrices and their diagonal blocks. Chapter 5 concerns
operator diagonals of Hilbert space operators.

My earlier result on operator diagonals of block matrices \cite{B-tot}  is :

\vskip 5pt
\begin{theorem}\label{T-tot} Let  $A\in \bM_{2n}$. Then, there exists some  $B\in\bM_n$ such that
$$
A\simeq \begin{bmatrix} B& \star \\ \star & B\end{bmatrix}.
$$
\end{theorem}

\vskip 5pt
Here, the stars hold for unspecified entries, and $\simeq$ means unitarily congruence. The simplest (and well-known) case is for $A\in \bM_2$, and it is the key for the proof of the Hausdorff-Toeplitz theorem (1918) ensuring that the numerical range is convex. The proof of
Theorem \ref{T-tot} cannot cover the case of Hilbert space operators. However we may propose a conjecture.

\vskip 5pt\noindent
{\bf Conjecture C.} {\it Let $A$ be an operator acting on a infinite dimensional separable Hilbert space ${\mathcal{H}}$. Then, for some subspace ${\mathcal{S}}\subset{\mathcal{H}}$, 
$$
A_{\mathcal{S}} \simeq A_{\mathcal{S}^{\perp}}.
$$
}

\vskip 5pt\noindent
{\bf Conjecture D.} {\it There exists  $A\in\bM_6$ such that, for any  $B\in\bM_2$,
$A$ is not unitarily equivalent to a matrix of the form
$$
\begin{bmatrix} B& \star & \star \\ \star & B& \star \\ \star & \star & B\end{bmatrix}.
$$
}

\vskip 5pt\noindent
Chapters 6-8  consider positive partitioned matrices; the matrix inequalities and the decomposition of Chapter 3  come into play.

\section{Missing topics}

To keep a reasonable lenght for the thesis, I do not give any results involving the matrix geometric mean and I do not develop the topic of symmetric antinorms. In a paper\footnote{Bourin-Hiai 2015} devoted to the study of antinorms, a version of \eqref{eqAB} for $\tau$-measurable operators is given (this is nontrivial when the algebra is not a factor). In a paper\footnote{Bourin-Shao 2020} devoted to the geodesics associated to the geometric mean, the following exotic H\"older inequality is obtained,
$$
  \left\| \sinh\left( \sum_{i=1}^m A_iB_i\right) \right\| \le  \left\| \sinh\left( \sum_{i=1}^m A_i^p\right) \right\|^{1/p} \left\| \sinh\left( \sum_{i=1}^m B_i^q\right) \right\|^{1/q},
$$
for all  positives matrices $A_i,  B_i$, such that $A_iB_i=B_iA_i$, $(i=1,\ldots, m)$, conjugate exponents $p,q>1$ and symmetric norms $\|\cdot\|$. A related result states that
$$
(t_1,\ldots,t_m) \mapsto {\mathrm{Tr\,}} \log \left(\sum_{i=1}^mX_i^*A_i^{t_i}X_i  \right)
$$
is jointly convex on $\bR^m$, where $A_i>0$  and $X_i$ is invertible  $(i=1,\ldots, m)$. Hence, the geometric mean provides results without the geometric mean ! A well-known phenomenon since the fundamental paper by Ando (1979).

\section{References of Chapter 1}

{\small
\begin{itemize}

\item[[18\!\!\!]]  J.-C.\ Bourin, Some inequalities for norms on matrices and operators, {\it Linear Algebra Appl}.\ 292 (1999), no.\ 1--3, 139--154.

\item[[19\!\!\!]] J.-C.\ Bourin,
Singular values of compressions, restrictions and dilations, {\it Linear Algebra Appl.}\ 360 (2003), 259--272.

\item[[20\!\!\!]]
 J.-C.\ Bourin,
Total dilations, {\it Linear Algebra Appl}.\ 368 (2003), 159--169.

\item[[22\!\!\!]] 
J.-C.\ Bourin, Convexity or concavity inequalities for Hermitian operators. \textit{Math. Inequal. Appl.} 7
(2004), no.\ 4, 607–620.

\item[[24\!\!\!]]
 J.-C.\ Bourin, A concavity inequality for symmetric norms, \textit{ Linear
Algebra Appl.}\  413 (2006), 212-217.

\item[[25\!\!\!]]
 J.-C.\ Bourin, Matrix versions of some classical inequalities. {\it Linear Algebra Appl.}\ 416 (2006), no.\ 2--3, 890--907.

\item[[26\!\!\!]]
 J.-C. Bourin, Matrix subadditivity inequalities and block-matrices, {\it Internat.\ J.\ Math.}\ 20 (2009), no.\
6, 679--691.

\item[[27\!\!\!]]
 J.-C. Bourin,
A matrix subadditivity inequality for symmetric norms, {\it Proc.\ Amer.\ Math.\ Soc.}\ 138 (2010), no.\ 2, 495--504.

\item[[30\!\!\!]] J.-C.\ Bourin and E.-Y.\ Lee, Concave functions of
positive operators, sums, and congruences, \textit{J.\ Operator Theory} \textbf{63}
(2010), 151--157.

\item[[48\!\!\!]]  J.-C.\ Bourin and E.\ Ricard,  An asymmetric Kadison's inequality, \textit{Linear Algebra Appl.}\
433 (2010) 499--510.

\item[[49\!\!\!]] J.-C.\ Bourin and M.\ Uchiyama, A matrix subadditivity inequality for $f(A+B)$ and $f(A)+f(B)$, \textit{Linear Algebra
Appl.}\
423 (2007), 512--518.

\end{itemize}
}


\chapter{Majorization and Perspective}

{\color{blue}{\Large {\bf Matrix Inequalities \\ from  a two variables functional} \large{\cite{BL-Persp}}}}

\vskip 10pt\noindent
{\small
{\bf Abstract.} We introduce a two variables norm functional and establish its joint log-convexity. This entails and improves many
remarkable matrix inequalities, most of them related to the log-majorization theorem of Araki. In particular: {\it if $A$ is a
positive semidefinite matrix and $N$ is a normal matrix, $p\ge 1$ and $\Phi$ is a sub-unital positive linear map, then
$|A\Phi(N)A|^p$ is weakly log-majorized by $A^p\Phi(|N|^p)A^p$}. This far extension of Araki's theorem (when $\Phi$ is the identity
and $N$ is positive) complements some recent results of Hiai and contains several special interesting cases such as a triangle
inequality for normal operators and some extensions of the Golden-Thompson trace inequality. Some applications to Schur products are
also obtained.
\vskip 5pt\noindent
{\it Keywords.} Matrix inequalities, Majorization, Positive linear maps, Schur products.
\vskip 5pt\noindent
{\it 2010 mathematics subject classification.} 47A30, 15A60.
}

\section{Log-majorization and log-convexity }

\vskip 5pt\noindent
  Matrices   are regarded as  non-commutative extensions of
scalars and functions. Since matrices do not commute in general, most scalars identities cannot be brought to the matrix setting,
however they sometimes have a matrix version, which is not longer an identity but an inequality. These kind of inequalities are of
fundamental importance in our understanding of the noncommutative world of matrices. A famous, fifty years old example of such an
inequality is the Golden-Thompson trace inequality: for Hermitian $n$-by-$n$ matrices $S$ and $T$,
$$
{\mathrm{Tr\,}} e^{S+T}  \le {\mathrm{Tr\,}} e^{S/2} e^T e^{S/2} .
$$
A decade after, Lieb and Thirring \cite{LT} obtained a stronger, remarkable trace inequality: for all positive semidefinite
$n$-by-$n$ matrices $A,B\in\bM_n^+$ and all integers $p\ge 1$,
\begin{equation}\label{LT}
{\mathrm{Tr\,}} (ABA)^p \le {\mathrm{Tr\,}} A^pB^pA^p.
\end{equation}
This was finally extended some fifteen years later by Araki \cite{Araki} as a very important theorem in matrix analysis and its
applications. Given $X,Y\in \bM_n^+ $, we write $X\prec_{w\!\log} Y$ when the series of $n$ inequalities holds,
$$
\prod_{j=1}^k \lambda_j(X) \le \prod_{j=1}^k \lambda_j(Y)
$$
for $k=1,\ldots n$, where $\lambda_j(\cdot)$ stands for the eigenvalues arranged in decreasing order. If further equality occurs for
$k=n$, we write $X\prec_{\log} Y$. Araki's theorem considerably strenghtens the Lieb-thirring trace inequality as the beautiful
log-majorization
\begin{equation}\label{araki}
 (ABA)^p \prec_{\log} A^pB^pA^p
\end{equation}
for all real numbers $p\ge 1$. In particular, this ensures \eqref{LT} for all $p\ge 1$.

Log- and weak log-majorization relations play a fundamental role in matrix analysis, a basic one for normal operators $X,Y\in\bM_n$
asserts that
\begin{equation}\label{folklore}
 |X+Y| \prec_{w\!\log} |X|+|Y|.
\end{equation}
This useful version of the triangle inequality belongs to the folklore and is a byproduct of Horn's inequalities, see the proof of
\cite[Corollary 1.4]{bourin-IJM}.

This article aims to provide new matrix inequalities containing \eqref{araki} and \eqref{folklore}. These inequalities are given in
Section 2. The first part dealing with positive operators is closely related to a recent paper of Hiai \cite{Hiai2}. The second part
of Section 2 considers normal operators and contains our main theorem (Theorem \ref{cornormal}), mentioned in the Abstract.

Our main idea, and technical tool, is Theorem \ref{th1} below. It establishes the log-convexity of a two variables functional. Fixing
one variable in this functional yields a generalization of $\eqref{araki}$ involving a third matrix $Z\in\bM_n $, of the form
\begin{equation*}
 (AZ^*BZA)^p \prec_{w\!\log} A^pZ^*B^pZA^p.
\end{equation*}
We will also derive the following weak log-majorization which contains both \eqref{araki} and \eqref{folklore} and thus unifies these
two inequalities.

\vskip 10pt
\begin{prop}\label{th0}  Let $A\in\bM_n^+$ and  let $X,Y\in\bM_n$  be normal. Then, for all $p\ge 1,$
\begin{equation*}
 |A(X+Y)A|^p \prec_{w\!\log} 2^{p-1}A^p(|X|^p+|Y|^p)A^p.
\end{equation*}
\end{prop}

\vskip 5pt
Letting $X=Y=B$ in Proposition \ref{th0} we have \eqref{araki}, more generally,
\begin{equation}\label{araki-normal}
|AXA|^p \prec_{\log} A^p|X|^pA^p
\end{equation}
for all $A\in\bM_n^+$ and normal matrices $X\in\bM_n$. When $X$ is Hermitian, this was noted by Audenaert \cite[Proposition
3]{Audenaert}.
If $A$ is the identity and $p=1$, Proposition \ref{th0} gives \eqref{folklore}. From \eqref{araki-normal} follows several nice
inequalities for the matrix exponential, due to Cohen and al.\ \cite{Cohen1}, \cite{Cohen2}, including the Golden-Thompson trace
inequality and the elegant relation
\begin{equation}\label{exponential}
|e^Z| \prec_{\log} e^{{\mathrm{Re}}\,Z}
\end{equation}
for all matrices $Z\in\bM_n$, where ${\mathrm{Re}}\,Z=(Z+Z^*)/2$, \cite[Theorem 2]{Cohen1}.

Fixing the other variable in Theorem \ref{th1} below entails a H\"older inequality due to Kosaki. Several 
matrix versions of an inequality of Littlewood related to H\"older's inequality will be also obtained. 

The two variables in Theorem \ref{th1} are essential and reflect a construction with the perspective of a convex function. Recall
that a norm on $\bM_n$ is symmetric whenever $\| UAV\|=\| A\|$ for all $A\in\bM_n$ and all unitary $U,V\in\bM_n$. For
$X,Y\in\bM_n^+$, the condition $X\prec_{w\!\log} Y$ implies $\| X\| \le
\| Y\|$ for all symmetric norms. We  state our log-convexity theorem.

\vskip 10pt
\begin{theorem}\label{th1} Let $A,B\in\bM_n^+$ and $Z\in\bM_n$. Then, for all symmetric norms and   $\alpha>0$, the map
$$
(p,t) \mapsto \left\| \left|A^{t/p}ZB^{t/p}\right|^{\alpha p} \right\| 
$$
is jointly log-convex on $(0,\infty)\times(-\infty,\infty)$.
\end{theorem}

\vskip 5pt
Here, if $A\in\bM_n^+$ is not invertible, we naturally define for $t\ge 0$, $A^{-t}:=(A+F)^{-t}E$ where $F$ is the projection onto
the nullspace of $A$ and $E$ is the range projection of $A$.

The next two sections present many hidden consequences of Theorem \ref{th1}, several of them extending \eqref{araki} and/or
\eqref{folklore}, for instance,
$$
\left|A\frac{T+T^*}{2}A\right|^p \prec_{w\!\log} A^p\frac{|T|^p+|T^*|^p}{2}A^p
$$
for all $A\in\bM_n^+$, $p\ge 1$, and any $T\in\bM_n$. The proof of Theorem \ref{th1} is in Section 4. The last section provides a
version of Theorem \ref{th1} for operators acting on an infinite dimensional Hilbert space.

\section{Araki type inequalities}

\subsection{With positive operators}

\vskip 5pt To obtain new Araki's type inequalities, we  fix $t=1$ in Theorem \ref{th1} and thus use  the following special case.

\vskip 5pt
\begin{cor}\label{cornew} Let $A,B\in\bM_n^+$ and $Z\in\bM_n$. Then, for all symmetric norms and   $\alpha>0$, the map
$$
p \mapsto \left\| \left|A^{1/p}ZB^{1/p}\right|^{\alpha p } \right\| 
$$
is  log-convex on $(0,\infty)$.
\end{cor}

\vskip 5pt
We may now state a series of corollaries extending Araki's theorem.

\vskip 5pt
\begin{cor}\label{cor1} Let $A,B\in\bM_n^+$ and $p\ge 1$. Let $Z\in\bM_n$ be a contraction. Then, for all symmetric norms and
$\alpha>0$,
$$
  \| (AZ^*BZA)^{\alpha p} \| \le \| (A^{p}Z^*B^{p}ZA^{p})^{\alpha}\|.
$$
\end{cor}

\vskip 5pt Let $I$ be the identity of $\bM_n$. A matrix $Z$ is contractive, or a contraction, if $Z^*Z\le I$, equivalently if its
operator norm satisfies $\| Z\|_{\infty}\le 1$.

\vskip 5pt
\begin{proof} The function $f(p)=\| | B^{1/p}ZA^{1/p} |^{2\alpha p} \|$ is log-convex, hence convex on $(0,\infty)$, and bounded
since $Z$ is contractive, $0\le f(p) \le \| B\|_{\infty}^{2\alpha}
\| A\|_{\infty}^{2\alpha} \| I\|$. Thus $f(p)$ is nonincreasing, so $f(1)\ge f(p)$ for all $p\ge1$. Replacing $B$ by $B^{p/2}$ and
$A$ by $A^{p}$ completes the proof.
\end{proof}

\vskip 5pt
Let $\|\cdot\|_{\{k\}}$, $k=1,\ldots, n$, denote the normalized Ky Fan $k$-norms on $\bM_n$,
$$
\| T\|_{\{k\}} =\frac{1}{k}\sum_{j=1}^k \lambda_j(|T|).
$$
Since, for all $A\in\bM_n^+$,
$$
\lim_{\alpha\to 0^+} \| A^{\alpha}\|_{\{k\}}^{1/\alpha} =\left\{\prod_{j=1}^k \lambda_j(A)\right\}^{1/k}
$$
we obtain from Corollary \ref{cor1} applied to the normalized Ky Fan $k$-norms, with $\alpha\to 0^+$, a striking
weak-log-majorization extending Araki's theorem.

\vskip 5pt
\begin{cor}\label{corstriking}  Let $A,B\in\bM_n^+$ and $p\ge 1$. Then,  for all contractions $Z\in \bM_n$,
$$
   (AZ^*BZA)^p \prec_{w\!\log}   A^{p}Z^*B^{p}ZA^{p}.
$$
\end{cor}

\vskip 5pt
If $Z=I$, we have the determinant equality and thus Araki's log-majorization \eqref{araki}.
Corollary \ref{cor1} and \ref{corstriking} are equivalent. Our proof of these extensions of Araki's theorem follows from the two
variables technic of Theorem \ref{th1}. It's worth mentioning that Fumio Hiai also obtained Corollary \ref{corstriking} in the
beautiful note \cite{Hiai2}. Hiai's approach is based on some subtle estimates for the operator geometric mean.

For $X,Y\in \bM_n^+ $, the notation $X\prec^{w\!\log} Y$ indicates that the series of $n$ inequalities holds, 
$$
\prod_{j=1}^k \nu_j(X) \ge \prod_{j=1}^k \nu_j(Y)
$$
for $k=1,\ldots n$, where $\nu_j(\cdot)$ stands for the eigenvalues arranged in increasing order. The following so-called super
weak-log-majorization is another extension of Araki's theorem.
A matrix  $Z$ is expansive when  $Z^*Z\ge I$.

\vskip 5pt
\begin{cor}\label{corsuper}  Let $A,B\in\bM_n^+$ and $p\ge 1$. Then,  for all expansive matrices $Z\in \bM_n$,
$$
   (AZ^*BZA)^p \prec^{w\!\log}   A^{p}Z^*B^{p}ZA^{p}.
$$
\end{cor}

\vskip 5pt
\begin{proof} By a limit argument, we may assume invertibility of $A$ and $B$. Taking inverses, and using that $Z^{-1}$ is
contractive, Corollary \ref{corsuper} is then equivalent to Corollary \ref{corstriking}.
\end{proof}

\vskip 5pt Corollaries \ref{corstriking}-\ref{corsuper} imply a host of trace inequalities. We say that a continuous function
$h:[0,\infty)\to (-\infty,\infty)$ is e-convex, (resp.\ e-concave), if $h(e^t)$ is convex, (resp.\ concave) on $(-\infty,\infty)$.
For instance, for all $\alpha >0$, $t\mapsto \log(1+t^{\alpha})$ is e-convex, while $t\mapsto \log(t^{\alpha}/(t+1))$ is e-concave.
The equivalence between Corollary \ref{corstriking} and Corollary \ref{cortrace} below is a basic property of majorization discussed
in any monograph on this topic such as \cite{Bh} and \cite{HP}.

\vskip 25pt
\begin{cor}\label{cortrace}  Let $A,B\in\bM_n^+$,  $Z\in\bM_n$, and $p\ge 1$. 
\begin{itemize}
\item[{\rm(a)}] If $Z$ is contractive and $f(t)$ is e-convex and nondecreasing, then
$$
  {\mathrm{Tr\,}} f((AZ^*BZA)^p) \le  {\mathrm{Tr\,}} f( A^{p}Z^*B^{p}ZA^{p}).
$$
\item[{\rm(b)}]  If $Z$ is expansive and $g(t)$ is e-concave and nondecreasing, then
 $$
  {\mathrm{Tr\,}} g((AZ^*BZA)^p) \ge  {\mathrm{Tr\,}} g( A^{p}Z^*B^{p}ZA^{p}).
$$
\end{itemize}
\end{cor}

\vskip 5pt 
We will propose in Section 4 a proof of Theorem \ref{th1} making use of antisymmetric tensor powers, likewise in the proof of Araki's
log-majorization. We will also indicate another, more elementary way, without antisymmetric tensors. The antisymmetric tensor technic
goes back to Hermann Weyl, cf.\ \cite{Bh}, \cite{HP}. We use it to derive our next corollary.

\vskip 5pt
\begin{cor}\label{corlim}  Let $A,B\in\bM_n^+$ and  $Z\in\bM_n$. For each $j=1,\ldots, n$, the function defined on $(0,\infty)$
$$
p\mapsto \lambda_j^{1/p}( A^{p}Z^*B^{p}ZA^{p})
$$ 
converges  as $p\to\infty$.
\end{cor}

\vskip 5pt
\begin{proof} We may assume that $Z$ is contractive. As in the proof of Corollary \ref{cor1} we then see that the function
$
g(p)=\lambda_1^p( A^{1/p}Z^*B^{1/p}ZA^{1/p})
$
is log-convex and bounded, hence nonincreasing on $(0,\infty)$. Therefore $g(p)$ converges as $p\to 0$ and so $g(1/p)$ converges as
$p\to \infty$. Thus
$
p\mapsto  \lambda_1^{1/p}( A^{p}Z^*B^{p}ZA^{p})
$
converges
 as $p\to\infty$. Considering $k$-th antisymmetric tensor products, $k=1,\ldots, n$, we infer the convergence of
$$
p\mapsto\prod_{j=1}^k \lambda_j^{1/p}( A^{p}Z^*B^{p}ZA^{p})=\lambda_1^{1/p}\left((\wedge^kA)^p\wedge^kZ^*
(\wedge^kB)^p\wedge^kZ(\wedge^kA)^p\right)
$$
 and so, the convergence of
$
p\mapsto \lambda_j^{1/p}( A^{p}Z^*B^{p}ZA^{p})
$
as $p\to\infty$, for each $j=1,2,\ldots.$
\end{proof}

\vskip 5pt
When $Z=I$, Audenaert and Hiai \cite{Aud-Hiai} recently gave a remarkable improvement of Corollary \ref{corlim} by showing that
$p\mapsto (A^pB^pA^p)^{1/p}$ converges in $\bM_n$ as $p\to \infty$. We do not know whether such a reciprocal Lie-Trotter limit still
holds with a third matrix $Z$ as in Corollary \ref{corlim}.

It is possible to state Corollary \ref{corstriking} in a stronger form involving a positive linear map $\Phi$. Such a map is called
sub-unital when $\Phi(I)\le I$.

\vskip 5pt
\begin{cor}\label{corstronger} Let $A,B\in\bM_n^+$ and $p\ge 1$. Then, for all positive linear, sub-unital map $\Phi: \bM_n\to\bM_n$,
$$
   (A\Phi(B)A)^p \prec_{w\!\log}   A^{p}\Phi(B^{p})A^{p}.
$$
\end{cor}

\vskip 5pt
\begin{proof} We may assume (the details are given, for a more general class of maps,  in the proof of Corollary 
\ref{poslin}) that
$$
\Phi(X) =\sum_{i=1}^m Z_i^* XZ_i
$$
where $m=n^2$ and $Z_i\in\bM_n$, $i=1,\ldots, m$, satisfy $\sum_{i=1}^m Z^*_iZ_i \le I$. Corollary \ref{corstronger} then follows
from Corollary \ref{corstriking} applied to the operators $\tilde{A},\tilde{B},\tilde{Z}\in \bM_{mn}$,
\begin{equation*}\label{eqblock}
\tilde{A}=\begin{pmatrix} A& 0_n&\cdots &0_n \\
0_n &0_n&\cdots &0_n \\
\vdots &\vdots &\ddots &\vdots \\
0_n& 0_n&\cdots &0_n \\
\end{pmatrix}, \
\tilde{B}=\begin{pmatrix} B& 0_n&\cdots &0_n \\
0_n &B&\cdots &0_n \\
\vdots &\vdots &\ddots &\vdots \\
0_n& 0_n&\cdots &B \\
\end{pmatrix}, \
\tilde{Z}=\begin{pmatrix} Z_1& 0_n&\cdots &0_n \\
Z_2 &0_n&\cdots &0_n \\
\vdots &\vdots &\ddots &\vdots \\
Z_m& 0_n&\cdots &0_n \\
\end{pmatrix}
\end{equation*}
where $0_n$ stands for the zero matrix in $\bM_n$.
\end{proof}

\vskip 5pt
  Corollary \ref{corstronger} can be applied for the Schur product $\circ$ (i.e., entrywise product) in $\bM_n$.

\vskip 5pt
\begin{cor}\label{corschur} Let $A,B,C\in\bM_n^+$ and $p\ge 1$. If $C$ has all its diagonal entries less than or equal to one, then \
$$
   (A(C\circ B)A)^p \prec_{w\!\log}   A^{p}(C\circ B^{p})A^{p}.
$$
\end{cor}

\vskip 5pt
\begin{proof} The map $X\mapsto C\circ X$ is a positive linear, sub-unital map on $\bM_n$.
\end{proof}

\vskip 5pt
Corollary \ref{corschur} with the matrix $C$ whose entries are all equal to one is Araki's log-majorization. With $C=I$, Corollary
\ref{corschur} is already an interesting extension of
Araki's theorem as we may assume that $B$ is diagonal in \eqref{araki}. We warn the reader that the super weak-log-majorization, for
$A,B\in\bM_n^+$ and $p\ge 1$,
$
   (A(I\circ B)A)^p \prec^{w\!\log}   A^{p}(I\circ B^{p})A^{p}
$
does not hold, in fact, in general, $\det^ 2 I\circ B< \det I\circ B^2 $. 

\subsection{With normal operators}

To obtain Proposition \ref{th0} we need the following generalization of Corollary \ref{corstronger}.

\vskip 5pt
\begin{theorem}\label{cornormal} Let $A\in \bM_n^+$ and let $N\in\bM_m$ be normal. Then, for all positive linear, sub-unital maps
$\Phi: \bM_m\to\bM_n$, and $p\ge 1$,
$$
   |A\Phi(N)A|^p \prec_{w\!\log}   A^{p}\Phi(|N|^{p})A^{p}.
$$
\end{theorem}

\vskip 5pt
\begin{proof} By completing, if necessary, our matrices $A$ and $N$ with some 0-entries, we may assume that $m=n$ and then, as in the
proof of Corollary \ref{corstronger}, that $\Phi$ is a congruence map with a contraction $\tilde{Z}$,
$\Phi(X)=\tilde{Z}X\tilde{Z}^*$. Now, we have with the polar decomposition $N=U|N|$,
\begin{align*}
|A\tilde{Z}N\tilde{Z}^*A| &= \left|A\tilde{Z}|N|^{1/2}U|N|^{1/2}\tilde{Z}^*A\right| \\
&\prec_{\log} A\tilde{Z}|N|\tilde{Z}^*A
\end{align*}
by using Horn's log-majorization $|XKX^*|\prec_{w\!\log} XX^*$ for all $X\in\bM_n$ and all contractions $K\in\bM_n$. Hence, from
Corollary \ref{corstriking}, for all $p\ge 1$,
$$
|A\tilde{Z}N\tilde{Z}^*A|^p \prec_{\log} \left|A\tilde{Z}|N|\tilde{Z}^*A\right|^p  \prec_{w\!\log} A^p\tilde{Z}|N|^p\tilde{Z}^*A^p
$$
which completes the proof.
\end{proof}

\vskip 10pt
We are in a position to prove Proposition \ref{th0}  whose $m$-variables version is given here.

\vskip 10pt
\begin{cor}\label{th0a}  Let $A\in\bM_n^+$ and  let $X_1,\cdots,X_m\in\bM_n$  be normal. Then, for all $p\ge 1,$
\begin{equation*}
 \left|A\left(\sum_{k=1}^m X_k\right)A\right|^p \prec_{w\!\log} m^{p-1}A^p\left(\sum_{k=1}^m |X_k|^p\right)A^p.
\end{equation*}
\end{cor}

\vskip 10pt\noindent
\begin{proof}  Applying Theorem \ref{cornormal} to
$N=X_1\oplus \cdots\oplus X_m$
and to the unital, positive linear map $\Phi:\bM_{mn}\to \bM_n$, 
$$
\begin{pmatrix} S_{1,1}&\cdots&S_{1,m} \\ \vdots& \ddots &\vdots \\ S_{m,1}&\cdots&S_{m,m}\end{pmatrix}\mapsto
\frac{1}{m}\sum_{k=1}^m S_{k,k}
$$
yields 
\begin{equation*}
 \left|A\frac{\sum_{k=1}^m X_k}{m}A\right|^p \prec_{w\!\log} A^p\frac{\sum_{k=1}^m |X_k|^p}{m}A^p
\end{equation*}
which is equivalent to the desired inequality.
\end{proof}

\vskip 5pt
A special case of Corollary \ref{th0a} deals with the Cartesian decomposition of an arbitrary matrix.

\vskip 5pt
\begin{cor}\label{corcartesian} Let $X,Y\in \bM_n$ be Hermitian. Then, for all $p\ge 1$, 
\begin{equation*}
 \left|A(X+iY)A\right|^p  \prec_{w\!\log} 2^{p-1}A^p(|X|^p+|Y|^p)A^p
\end{equation*}
where the constant $2^{p-1}$ is the best possible.
\end{cor}

\vskip 5pt
To check that $2^{p-1}$ is optimal, take $A=I\in\bM_{2n}$ and pick any two-nilpotent matrix,
$$
X+iY=\begin{pmatrix} 0&T \\ 0&0\end{pmatrix}.
$$

For a single normal operator, Corollary \ref{th0a} gives \eqref{araki-normal} as we have equality for the determinant. This entails
the following remarkable log-majorization for the matrix exponential.

\vskip 5pt
\begin{cor}\label{corthompson} Let $A,B\in \bM_n$. Then,  
\begin{equation*}
\left| e^{A+B}  \right|\prec_{\log} e^{{\mathrm{Re}\,} A/2}  e^{{\mathrm{Re}\,}B}
e^{{\mathrm{Re}\,} A/2}.
\end{equation*}
\end{cor}

\vskip 5pt
Corollary \ref{corthompson} contains \eqref{exponential} and shows that when $A$ and $B$ are Hermitian we have the famous Thompson
log-majorization, \cite[Lemma 6]{Thompson},
\begin{equation}\label{gtlog}
 e^{A+B}  \prec_{\log} e^{ A/2}  e^{B}
e^{ A/2}
\end{equation}
which entails
\begin{equation*}
 \| e^{A+B} \| \le \| e^{ A/2}  e^{B}
e^{ A/2}\|
\end{equation*}
for all symmetric norms. For the operator norm this is Segal's inequality while for the trace norm this is the Golden-Thompson
inequality. Taking the logarithms in \eqref{gtlog}, we have a classical majorization between $A+B$ and $\log e^{A/2}e^Be^{A/2}$.
Since $t\mapsto |t|$ is convex, we infer, replacing $B$ by $-B$ that
\begin{equation*}
 \| A-B \| \le \| \log (e^{ A/2}  e^{-B}
e^{ A/2})\|
\end{equation*}
for all symmetric norms.
For the Hilbert-Schmidt norm, this is the Exponential Metric Increasing inequality, reflecting the nonpositive curvature of the
positive definite cone with its Riemannian structure (\cite[Chapter 6]{Bhatia2}).

Corollary \ref{corthompson} follows from \eqref{araki-normal} combined with the Lie Product Formula \cite[p.\ 254]{Bh} as shown
in the next proof. Note that Corollary \ref{corthompson} also follows from Cohen's log-majorization \eqref{exponential} combined with
Thompson's log-majorization \eqref{gtlog}, thus we do not pretend to originality.

\vskip 5pt
\begin{proof} We have a Hermitian matrix $C$ such that, using the Lie Product Formula,
$$
e^{A+B}=e^{{\mathrm{Re}\,}A +{\mathrm{Re}\,} B +iC}=\lim_{n\to+\infty} \left(
e^{({\mathrm{Re}\,}A+{\mathrm{Re}\,}B)/2n}e^{iC/n}e^{({\mathrm{Re}\,}A
+{\mathrm{Re}\,}B)/2n}\right)^{n}.
$$
On the other hand, by \eqref{araki-normal}, for all $n\ge 1$,
$$
\left| \left( e^{({\mathrm{Re}\,}A+{\mathrm{Re}\,}B)/2n}e^{iC/n}e^{({\mathrm{Re}\,}A
+{\mathrm{Re}\,}B)/2n}\right)^{n} \right| \prec_{\log} e^{{\mathrm{Re}\,}A +{\mathrm{Re}\,} B}
$$
so that
$$
\left|e^{A+B}\right| \prec_{\log}  e^{{\mathrm{Re}\,}A +{\mathrm{Re}\,} B}.
$$
Using again the Lie Product Formula,
$$
e^{{\mathrm{Re}\,}A +{\mathrm{Re}\,} B} =\lim_{n\to+\infty} \left( e^{{\mathrm{Re}\,}A/2n}
e^{{\mathrm{Re}\,}B/n}e^{{\mathrm{Re}\,}A/2n}
\right)^{n},
$$
combined with \eqref{araki-normal} (or \eqref{araki}) completes the proof.
\end{proof}

\vskip 5pt
 Theorem 2.2.9 is the main result of Section 2 as all the other results in this section are special cases.
One more elegant extension of Araki's inequality  follows, involving an arbitrary matrix.

\vskip 5pt
\begin{cor}\label{corlast} Let $A\in \bM_n^+$ and $p\ge 1$. Then, for any $T\in\bM_n$,
\begin{equation*}
 \left|A\frac{T+T^*}{2}A\right|^p \prec_{w\!\log} A^p\frac{|T|^p+|T^*|^p}{2}A^p.
\end{equation*}
\end{cor}

\vskip 5pt
\begin{proof} It suffices to apply  Theorem \ref{cornormal} to 
$$
N=\begin{pmatrix} 0& T \\ T^*&0\end{pmatrix}
$$
and to the unital, positive linear map $\Phi:\bM_{2n}\to\bM_n$,
$$
\begin{pmatrix} B&C  \\ D&E \end{pmatrix}\mapsto \frac{B+C+D+E}{2}.
$$
\end{proof}

\vskip 5pt
We apply Theorem \ref{cornormal} to Schur products in the next two corollaries.

\vskip 5pt
\begin{cor}\label{corschurlast}  Let $A\in\bM_n^+$ and  let $X,Y\in\bM_n$  be normal. Then, for all $p\ge 1,$
\begin{equation*}
 |A(X\circ Y)A|^p \prec_{w\!\log} A^p(|X|^p\circ |Y|^p)A^p.
\end{equation*}
\end{cor}

\vskip 5pt
\begin{proof} We need to see the Schur product as a positive linear map,
$$
X\circ Y= \Phi(X\otimes Y)
$$
where $\Phi:\bM_n\otimes\bM_n\to \bM_n$ merely consists in extracting a principal submatrix. Setting $N=X\otimes Y$ in Theorem
\ref{cornormal} completes the proof.
\end{proof}

\vskip 5pt
We note that Corollary \ref{corschurlast} extends \eqref{araki-normal} (with $X$ in diagonal form and $Y=I$) and contains the
classical log-majorization for normal operators,
$$|X\circ Y|\prec_{w\\log} |X|\circ |Y|.$$
As a last illustration of the scope of Theorem \ref{cornormal} we have the following result.

\vskip 5pt
\begin{cor}\label{corverylast} Let $A\in \bM_n^+$ and $p\ge 1$. Then, for any $T\in\bM_n$,
\begin{equation*}
 \left|A(T\circ T^*)A\right|^p \prec_{w\!\log} A^p(|T|^p\circ|T^*|^p)A^p.
\end{equation*}
\end{cor}

\vskip 5pt
\begin{proof} We apply Corollary \ref{corschurlast} to the pair of Hermitian operators in $\bM_{2n}$,
$$
X=\begin{pmatrix} 0&T^* \\ T&0\end{pmatrix}, \quad Y=\begin{pmatrix} 0&T \\ T^*&0\end{pmatrix}
$$
with  $A\oplus A$ in $\bM_{2n}^+$. We then obtain
$$
\left| \begin{pmatrix} 0&A(T\circ T^*)A \\ A(T\circ T^*)A&0\end{pmatrix} \right|^p \prec_{w\!\log} 
\begin{pmatrix} A(|T|^p\circ |T^*|^p)A&0 \\ 0&A(|T|^p\circ |T^*|^p)A\end{pmatrix}
$$
which is equivalent to the statement of our corollary. 
\end{proof}

\section{H\"older type inequalities}

\vskip 5pt Now we turn to H\"older's type inequalities. Fixing $p=1$ in Theorem \ref{th1}, we have the following special case.

\vskip 5pt
\begin{cor}\label{corspe} Let $A,B\in\bM_n^+$ and $Z\in\bM_n$. Then, for all symmetric norms and   $\alpha>0$, the map
$$
t \mapsto \left\| \left|A^{t}ZB^{t}\right|^{\alpha } \right\| 
$$
is  log-convex on $(-\infty,\infty)$.
\end{cor}

\vskip 5pt
This implies a fundamental fact, the L\"owner-Heinz inequality stating the operator monotonicity of $t^p$, $p\in(0,1)$.

\vskip 5pt
\begin{cor}  Let $A,B\in\bM_n^+$. If $A\ge B$, then $A^p\ge B^p$ for all $p\in(0,1)$.
\end{cor}

\vskip 5pt
\begin{proof} Corollary \ref{corspe} for the operator norm, with $Z=I$, $\alpha=2$, and the pair $A^{-1/2},B^{1/2}$ in place of the
pair $A,B$ shows that
$
f(t)= \| A^{-t/2}B^{t}A^{-t/2}\|_{\infty}
$
is log-convex. Hence for $p\in(0,1)$, we have $f(p)\le f(1)^p f(0)^{1-p}$. Since $f(0)=1$ and by assumption $f(1)\le 1$, we obtain
$f(p)\le 1$ and so $A^p\ge B^p$.
\end{proof}

Corollary \ref{corspe} entails a H\"older inequality with a parameter. This inequality was first proved by Kosaki \cite[Theorem
3]{kosaki}. Here, we state it without the weight $Z$.

\vskip 5pt
\begin{cor}  Let $X,Y\in\bM_n$ and $p,q\ge 1$ such that $p^{-1}+q^{-1}=1$. Then, for all symmetric norms and   $\alpha>0$,  
$$
\left\| |XY|^{\alpha} \right\| \le  \left\|  |X|^{\alpha p}  \right\|^{1/p}  \left\|  |Y|^{q \alpha }  \right\|^{1/q}.
$$
\end{cor}

\vskip 5pt
\begin{proof} Let $A,B\in \bM_n^+$ with $B$ invertible. By replacing $B$ with $B^{-1}$
 and letting $Z=B$ in Corollary \ref{corspe} show that
$
t\mapsto  \| |A^tB^{1-t}|^{\alpha}\|
$
is log-convex on $(-\infty,\infty)$. Thus, for $t\in(0,1)$,
$
\| |A^tB^{1-t}|^{\alpha}\| \le \| A^{\alpha}\|^t\| B^{\alpha}\|^{1-t}
$.
Then, choose $A=|X|^{p}$, $B=|Y^*|^q$, $t=1/p$.
\end{proof}

\vskip 5pt More original H\"older's type inequalities are given in the next series of corollaries. 

\vskip 5pt
\begin{cor}\label{corcong} Let $A\in \bM_n^+$ and $Z\in\bM_{n,m}$. Then, for all symmetric norms and  $\alpha>0$, the map
$$
(p,t) \mapsto \left\| \left(Z^*A^{t/p}Z\right)^{\alpha p} \right\| 
$$
is jointly log-convex on $(0,\infty)\times(-\infty,\infty)$.
\end{cor}

\vskip 5pt
\begin{proof} By completing, if necessary, our matrices with some 0-entries, we may suppose $m=n$ and then apply Theorem \ref{th1}
with $B=I$. \end{proof}

\vskip 5pt
\begin{cor}\label{corlittlewood} Let $a=(a_1,\cdots,a_m)$ and $w=(w_1,\cdots,w_m)$ be two $m$-tuples in $\bR^+$ and define, for all
$p>0$,
$\| a\|_p:=(\sum_{i=1}^mw_ia_i^p)^{1/p}$. Then, for all $p,q>0$ and  $\theta\in(0,1)$, 
$$
\| a\|_{\frac{1}{\theta p +(1-\theta)q}} \le \| a\|_{\frac{1}{p}}^{\theta} \| a\|_{\frac{1}{q}}^{1-\theta}.
$$
\end{cor}

\vskip 5pt
\begin{proof} Fix $t=1$ and pick $A=\diag(a_1,\ldots a_m)$ and $Z^*=(w_1^{1/2},\ldots,w_m^{1/2})$ in the previous corollary.
\end{proof}

\vskip 5pt
Corollary \ref{corlittlewood} is the classical log-convexity of $p\to \|\cdot\|_{1/p}$, or Littlewood's version of H\"older's
inequality \cite[Theorem 5.5.1]{garling}. The next two corollaries, seemingly stronger but actually equivalent to Corollary
\ref{corcong}, are also generalizations of this inequality.

\vskip 5pt
\begin{cor} Let $A_i\in\bM_n^+$ and $Z_i\in\bM_{n,m}$, $i=1,\ldots,k$. Then, for all symmetric norms and  $\alpha>0$, the map
$$
(p,t) \mapsto \left\| \left\{\sum_{i=1}^k Z_i^*A_i^{t/p}Z_i\right\}^{\alpha p} \right\| 
$$
is jointly log-convex on $(0,\infty)\times(-\infty,\infty)$.
\end{cor}

\vskip 5pt The unweighted case, $Z_i=I$ for all $i=1,\ldots,k$, is especially interesting. With $t=\alpha=1$, it is a matrix version
of the unweighted Littlewood inequality.

\vskip 5pt
\begin{proof} Apply Corollary \ref{corcong} with $A=A_1\oplus\cdots\oplus A_k$ and $Z^*=(Z_1^*,\ldots,Z_k^*)$. \end{proof}

\vskip 5pt
\begin{cor}\label{poslin} Let $A\in \bM_m^+$ and let $\Phi:\bM_m^+\to\bM_n^+$ be a positive linear map. Then, for all symmetric norms
and $\alpha>0$, the map
$$
(p,t) \mapsto \left\| \left\{\Phi(A^{t/p})\right\}^{\alpha p} \right\| 
$$
is jointly log-convex on $(0,\infty)\times(-\infty,\infty)$.
\end{cor}

\vskip 5pt
\begin{proof} When restricted to the $*$-commutative subalgebra spanned by $A$, the map $\Phi$ has the form
\begin{equation}\label{decomp}\Phi(X)=\sum_{i=1}^m\sum_{j=1}^n Z_{i,j}^* XZ_{i,j}\end{equation}
for some rank 1 or 0 matrices $Z_{i,j}\in\bM_{m,n}$, $i=1,\ldots,m$, $j=1,\ldots,n$. So we are in the range of the previous
corollary. To check the decomposition \eqref{decomp}, write the spectral decomposition
$A=\sum_{i=1}^m \lambda_i(A) E_i$
with rank one projections $E_i=x_i x_i^*$  for some column vectors $x_i\in \bM_{m,1}$ and set 
$Z_{i,j}= x_i R_{i,j}$ where $R_{i,j}\in\bM_{1,n}$ is the $j$-th row of $\Phi(E_i)^{1/2}$. \end{proof}

\vskip 5pt
The above proof shows a classical fact, a positive linear map on a commutative domain is completely positive. Our proof seems shorter
than the ones in the literature.
We close this section with an application to Schur products.

\vskip 5pt
\begin{cor} Let $A,B\in \bM_n^+$. If $p\ge r\ge s\ge q$ and $p+q=r+s$, then, for all symmetric norms and   $\alpha>0$,
$$
\| \{A^r\circ B^s\}^{\alpha}\| \| \{A^s\circ B^r\}^{\alpha}\| \le \| \{A^p\circ B^q\}^{\alpha}\| \| \{A^q\circ B^p\}^{\alpha}\|
$$
and
$$
\| \{A^r\circ B^s\}^{\alpha}\| + \| \{A^s\circ B^r\}^{\alpha}\| \le \| \{A^p\circ B^q\}^{\alpha}\| +\| \{A^q\circ B^p\}^{\alpha}\|.
$$
\end{cor}

\vskip 5pt
\begin{proof} By a limit argument we may assume invertibility of $A$ and $B$. Let $w:=(p+q)/2$. We will show that the maps
\begin{equation}\label{eqshow}
t\mapsto \| \{A^{w+t} \circ B^{w-t}\}^{\alpha} \|, \quad t\mapsto \| \{A^{w-t} \circ B^{w+t}\}^{\alpha} \|
\end{equation}
are log-convex on $(-\infty,\infty)$. This implies that the functions
$$
f(t)=  \| \{A^{w+t} \circ B^{w-t}\}^{\alpha} \| \| \{A^{w-t} \circ B^{w+t}\}^{\alpha} \|
$$
and
$$
g(t)=  \| \{A^{w+t} \circ B^{w-t}\}^{\alpha} \| +\| \{A^{w-t} \circ B^{w+t}\}^{\alpha} \|
$$
are convex and even, hence nondecreasing on $[0,\infty)$. So we have
$$
f((r-s)/2)\le f((p-q)/2) \ {\text{and}} \ g((r-s)/2)\le g((p-q)/2)
$$
which prove the corollary.

To check the log-convexity of the maps \eqref{eqshow}, we  see the Schur product as a positive linear map acting on a tensor product,$A\circ B= \Psi(A\otimes B)$. 
 By Corollary \ref{poslin}, the map
$$
t\mapsto \| \{\Phi(Z^t)\}^{\alpha} \|
$$
is log-convex on $(-\infty,\infty)$ for any positive matrix
 $Z\in\bM_n\otimes\bM_n$ and any positive linear map $\Phi:\bM_n\otimes\bM_n\to \bM_n$. Taking
$
Z=A\otimes B^{-1}
$
and
$$
\Phi(X)= \Psi(A^{w/2}\otimes B^{w/2}\cdot  X \cdot A^{w/2}\otimes B^{w/2})
$$
we obtain the log-convexity of the first map
$
t\mapsto \| \{A^{w+t}\circ B^{w-t}\}^{\alpha} \|
$
in $\eqref{eqshow}$. The log-convexity of the second one is similar.
\end{proof}

\section{Proof of Theorem \ref{th1}}

In the proof of the theorem, we will denote the $k$-th antisymmetric power $\wedge^k T$ of an operator $T$ simply as $T_k$. The
symbol $\|\cdot\|_{\infty}$ stands for the usual operator norm while $\rho(\cdot)$ denotes the spectral radius. Given $A\in\bM_n^+$
we denote by $A^{\downarrow}$ the diagonal matrix with the eigenvalues of $A$ in decreasing order down to the diagonal,
$A^{\downarrow}={\diag(\lambda_j(A)})$.

\begin{proof} Recall that if $A\in\bM_n^+$ is not invertible and $t\ge 0$, we define $A^{-t}$ as the generalized inverse of $A^t$,
i.e., $A^{-t}:=(A+F)^{-t}E$ where $F$ is the projection onto the nullspace of $A$ and $E$ is the range projection of $A$. With this
convention, replacing if necessary $Z$ by $EZE'$ where $E$ is the range projection of $A$ and $E'$ that of $B$, we may and do assume
that $A$ and $B$ are invertible.

Let
$$
g_k(t):= \prod_{j=1}^k \lambda_j(|A^{t}ZB^{t}|)= \|  A_k^{t}Z_kB_k^{t}\|_{\infty}
$$
Then
\begin{align*}g_k((t+s)/2) &= \| A_k^{(t+s)/2}Z_kB_k^{t+s}Z^*_kA_k^{(t+s)/2} \|_{\infty}^{1/2} \\
&=\rho^{1/2}( A_k^{t}Z_kB_k^{t+s}Z^*_kA_k^{s}) \\
&\le  \| A_k^{t}Z_kB_k^{t+s}Z^*_kA_k^{s} \|_{\infty}^{1/2} \\
&\le  \| A_k^{t}Z_kB_k^{t}\|_{\infty}^{1/2} \| B_k^{s}Z^*_kA_k^{s} \|_{\infty}^{1/2} \\
&=\{g_k(t)g_k(s)\}^{1/2}.
\end{align*}
Thus $t\mapsto g_k(t)$ is log-convex on $(-\infty,\infty)$ and so $(p,t)\mapsto g_k^p(t/p)$ is jointly log-convex on
$(0,\infty)\times(-\infty,\infty)$. Indeed, its logarithm $p\log g_k(t/p)$
is the perspective of the convex function $\log g_k(t)$, and hence is jointly convex. Therefore
\begin{equation}\label{equa1}
g_k^{(p+q)/2}\left(\frac{(t+s)/2}{(p+q)/2}\right) \le \{ g_k^p(t/p)g_k^q(s/q)\}^{1/2}
\end{equation}
for $k=1,2,\ldots,n$, with equality for $k=n$ as it then involves the determinant. This is equivalent to the log-majorization
$$
\left|A^{\frac{t+s}{p+q}}ZB^{\frac{t+s}{p+q}}\right|^{\frac{p+q}{2}} \prec_{\log}
|A^{t/p}ZB^{t/p}|^{\frac{p}{2}\downarrow}\left|A^{s/q}ZB^{s/q}\right|^{\frac{q}{2}\downarrow}
$$
which is equivalent, for any  $\alpha>0$, to the log-majorization
$$
\left|A^{\frac{t+s}{p+q}}ZB^{\frac{t+s}{p+q}}\right|^{\alpha \frac{p+q}{2}} \prec_{\log}
\left|A^{t/p}ZB^{t/p}\right|^{\frac{\alpha p}{2}\downarrow}\left|A^{s/q}ZB^{s/q}\right|^{\frac{\alpha q}{2}\downarrow}
$$
ensuring that
\begin{equation}\label{equa2}
\left\|  \left|A^{\frac{t+s}{p+q}}ZB^{\frac{t+s}{p+q}}\right|^{\alpha \frac{p+q}{2}}\right\|
\le \left\| \left|A^{t/p}ZB^{t/p}\right|^{\frac{\alpha p}{2}\downarrow}\left|A^{s/q}ZB^{s/q}\right|^{\frac{\alpha q}{2}\downarrow}
\right\|
\end{equation}
for all symmetric norms. Thanks to the Cauchy-Schwarz inequality for symmetric norms, we then have
\begin{equation}\label{equa3}
\left\|  \left|A^{\frac{t+s}{p+q}}ZB^{\frac{t+s}{p+q}}\right|^{\alpha \frac{p+q}{2}}\right\|
\le \left\| \left|A^{t/p}ZB^{t/p}\right|^{\alpha p} \right\|^{1/2}\left\| \left|A^{t/q}ZB^{t/q}\right|^{\alpha q}\right\|^{1/2}
\end{equation}
which means that 
$$
(p,t) \mapsto \left\| \left|A^{t/p}ZB^{t/p}\right|^{\alpha p} \right\| 
$$
is jointly log-convex on $(0,\infty)\times(-\infty,\infty)$.
\end{proof}

\vskip 5pt 
Denote by $I_k$ the identity of $\bM_k$ and by $\det_k$
the determinant on $\bM_k$. If $n\ge k$, $\Theta(k,n)$ stands for the set of $n\times k$ isometry matrices $T$, i.e, $T^*T=I_k$. One
easily checks the variational formula, for $A\in\bM_n$ and $k=1,\ldots,n$,
$$
\prod_{j=1}^k \lambda_j(|A|)=\max_{V,W\in\Theta(k,n)}\left| {\mathrm{det}}_k\, V^*AW\right|.
$$
From this formula follow two facts, Horn's inequality,
$$
\prod_{j=1}^k \lambda_j(|AB|) \le \prod_{j=1}^k \lambda_j(|A|)\lambda_j(|B|)
$$
for all $A,B\in\bM_n$ and $k=1,\dots,n$, and, making use of Schur's triangularization,  the inequality
$$
\prod_{j=1}^k \lambda_j(|AB|) \le \prod_{j=1}^k \lambda_j(|BA|)
$$
whenever $AB$ is normal (indeed, by Schur's theorem we may assume that $BA$ is upper triangular with the eigenvalues of $BA$, hence
of $AB$, down to the diagonal and our variational formula then gives the above log-majorization). This shows that the proof of
Theorem \ref{th1} can be written without the machinery of antisymmetric tensors.

The novelty of this proof consists in using the perspective of a one variable convex function. One more perspective yields the
following variation of Corollary \ref{cornew}.

\vskip 5pt
\begin{cor} Let $A,B\in\bM_m^{+}$, let $Z\in\bM_m$. Then, for all symmetric norms and   $\alpha>0$, the map
$$
p \mapsto \left\| \left|A^{1/p}ZB^{1/p}\right|^{\alpha } \right\|^p 
$$
is  log-convex on $(0,\infty)$. 
\end{cor}

\vskip 5pt
\begin{proof} By Theorem \ref{th1} with fixed $p=1$, the map
$
t \mapsto \log\left( \| |A^{t}ZB^{t}|^{\alpha} \|\right)
$
is convex on $(0,\infty)$, thus its perpective
$$
(p,t) \mapsto p\log \left( \left\| |A^{t/p}ZB^{t/p}|^{\alpha} \right\|\right) = \log \left( \left\| |A^{t/p}ZB^{t/p}|^{\alpha}
\right\|^p\right)
$$
is jointly log-convex on $(0,\infty)\times(0,\infty)$. Now fixing $t=1$ completes the proof.
\end{proof}

\vskip 5pt
From this corollary we may derive the next one exactly as Corollary \ref{poslin}
follows from Theorem \ref{th1}. This result is another noncommutative version of Littlewood's inequality (\cite[Theorem
5.5.1]{garling}).

\vskip 5pt
\begin{cor}  Let $\Phi:\bM_m\to \bM_n$ be a positive linear map and let $A\in\bM_m^+$. Then, for all symmetric norms and  $\alpha>0$, the map
$$
p \mapsto \left\| \Phi^{\alpha }(A^{1/p}) \right\|^p
$$
is  log-convex on $(0,\infty)$. 
\end{cor}

\section{Hilbert space operators}

\vskip 5pt
In this section we give a version  
 of Theorem \ref{th1}  for the algebra $\bB$ of bounded linear operators on a separable, infinite dimensional 
Hilbert space ${\mathcal{H}}$. 
We first include a brief treatment of symmetric norms for operators in $\bB$. Our approach does not require to discuss any underlying
ideal, we refer the reader to \cite[Chapter 2]{Simon} for a much more complete discussion.

We may define symmetric norms on $\bB$ in a closely related way to the finite dimensional case as follows. Let $\bF$ be the set of
finite rank operators and $\bF^+$ its positive part.

\vskip 5pt\noindent
\begin{definition}\label{def}A symmetric norm $\|\cdot\|$ on $\bB$ is a functional taking value in $[0,\infty]$ such that:
\begin{itemize}
\item[(1)] $\|\cdot \|$ induces a norm on $\bF$.
\item[(2)] If $\{X_n\}$ is a  sequence  in $\bF^+$ strongly increasing to $X$, then $\|X\|=\lim_n \| X_n\|$.
\item[(3)] $\| KZL\| \le \| Z\|$ for all $Z\in \bB$ and all contractions $K,L\in\bB$.
\end{itemize}
\end{definition}

\vskip 5pt
The reader familiar to the theory of symmetrically normed ideals may note that our definition of a symmetric norm is equivalent to
the usual one. More precisely, restricting $\|\cdot\|$ to the set where it takes finite values, Definition \ref{def} yields the
classical notion of a symmetric norm defined on its maximal ideal.

Definition \ref{def} shows that a symmetric norm on $\bB$ induces a symmetric norm on $\bM_n$ for each $n$, say $\|\cdot\|_{\bM_n}$.
In fact $\|\cdot\|$ can be regarded as a limit of the norms $\|\cdot\|_{\bM_n}$, see Lemma \ref{lem1} for a precise statement, so
that basic properties of symmetric norms on $\bM_n$ can be extended to symmetric norms on $\bB$. For instance the Cauchy-Schwarz
inequality also holds for symmetric norms on $\bB$, (with possibly the $\infty$ value) as well as the Ky Fan principle for
$A,B\in\bB^+$:
If $A\prec_w B$, then $\| A\| \le \| B\|$ for all symmetric norms. In fact, even for a noncompact operator $A\in\bB^+$, the sequence
$\{\lambda_j(A)\}_{j=1}^{\infty}$ and the corresponding diagonal operator $A^{\downarrow}={\diag(\lambda_j(A)})$ are well defined,
via the minmax formulae (see \cite[Proposition 1.4]{Hiai})
$$
\lambda_j(A) = \inf_E\{ \| EAE \|_{\infty} \ : E  \ {\mathrm{ projection\ with\ }}{\mathrm{rank}}(I-E)=j-1 \}
$$
 The  Ky Fan principle then still holds for $A,B\in\bB^+$ by Lemma \ref{lem1} and the obvious property
$$
\lambda_j(A)=\lim_{n\to\infty} \lambda_j(E_nAE_n)
$$
for all sequences of finite rank projections $\{E_n\}_{n=1}^{\infty}$ strongly converging to the identity. 
Note also that we still have $\| \wedge^k A\|_{\infty} =\prod_{j=1}^k \lambda_j(A)$.

Thus we have the same tools as in the matrix case and we will be able to adapt the proof of Theorem \ref{th1} for $\bB$. The infinite
dimensional version of Theorem \ref{th1}
is the following statement.

\vskip 5pt
\begin{theorem}\label{th2} Let $A,B\in\bB^{+}$, let $Z\in\bB$. Then, for all symmetric norms and   $\alpha>0$, the map
$$
(p,t) \mapsto \left\| \left|A^{t/p}ZB^{t/p}\right|^{\alpha p} \right\| 
$$
is jointly log-convex on $(0,\infty)\times(0,\infty)$. This map takes its finite values in the  open quarter-plan 
$$\Omega(p_0,t_0)=\{ (p,t) \ | \ p>p_0, \ t>t_0\}$$
for some $p_0,t_0\in [0,\infty]$, or on its  closure $\overline{\Omega}(p_0,t_0)$.
\end{theorem}

\vskip 5pt
Note that, contrarily to Theorem \ref{th1}, we confine the variable $t$ to the positive half-line. Indeed, when dealing with a
symmetric norm, the operators $A$ and $B$ are often compact, so that, for domain reasons, we cannot consider two unbounded operators
such as $A^{-1}$ and $B^{-1}$.

\vskip 5pt
\begin{proof} Note that $AZB=0$ if and only if $A^qZB^q=0$, for any $q>0$. In this case, our map is the 0-map, and its logarithm with
constant value $-\infty$ can be regarded as convex. Excluding this trivial case, our map takes values in $(0,\infty]$ and it makes
sense to consider the log-convexity property.
   We may reproduce the  proof of Theorem \ref{th1} and obtain \eqref{equa1}
for all $k=1,2,\cdots$. This leads to weak-logmajorizations and so to a weak majorization equivalent (Ky Fan's principle in $\bB$) to
\eqref{equa2}, with possibly the $\infty$ value on the right side or both sides. The Cauchy-Schwarz inequality for symmetric norms in
$\bB$ yields \eqref{equa3} (possibly with the $\infty$ value). Therefore our map is jointly log-convex. To show that the domain where
it takes finite values is $\Omega(p_0,t_0)$ or $\overline{\Omega}(p_0,t_0)$, it suffices to show the following two implications:

\vskip 5pt
{\it Let $0<t<s$ and $0<p<q$. If $\left\| \left|A^{t/p}ZB^{t/p}\right|^{\alpha p} \right\| <\infty$, then
\begin{itemize}
\item[{\rm(i)}] $\left\| \left|A^{s/p}ZB^{s/p}\right|^{\alpha p} \right\| <\infty$, and
\item[{\rm (ii)}] $\left\| \left|A^{t/q}ZB^{t/q}\right|^{\alpha q} \right\| <\infty$.
\end{itemize} }

\vskip 5pt
Since $0<t<s$ ensures that, for some constant $c=c(s,t)>0$,
$$
\lambda_j(|A^{t/p}ZB^{t/p}|) \ge c\lambda_j(|A^{s/p}ZB^{s/p}|)
$$
for all $j=1,2,\ldots$, we obtain (i). To obtain (ii) we may assume that $Z$ is a contraction. Then arguing as in the proof of
Corollary \ref{cor1} we see that the finite value map
$$p\mapsto  \left\| \left|A^{t/p}ZB^{t/p}\right|^{\alpha p} \right\| $$
is nonincreasing for all Ky-Fan norms. Thus this map is also nonincreasing for all symmetric norms. This gives (ii).
\end{proof}

\vskip 5pt
Exactly as in the matrix case, we can derive the following two corollaries.

\vskip 5pt
\begin{cor}  Let $A,B\in\bB^+$ and $p\ge 1$. Then,  for all contractions $Z\in \bB$,
$$
   (AZ^*BZA)^p \prec_{w\!\log}   A^{p}Z^*B^{p}ZA^{p}.
$$
\end{cor}

\vskip 5pt
\begin{cor} Let $A,B\in\bB^+$ and let $Z\in\bB$ be a contraction. Assume that at least one of these three operators is compact. Then,
if $p\ge 1$ and
$f(t)$ is e-convex and nondecreasing,
$$
  {\mathrm{Tr\,}} f((AZ^*BZA)^p) \le  {\mathrm{Tr\,}} f( A^{p}Z^*B^{p}ZA^{p}).
$$
\end{cor}

\vskip 5pt
Here, we use the fact that for $X\in\bK^+$ and a nondecreasing continuous function $f:[0,\infty)\to (-\infty,\infty)$, we can define
${\mathrm{Tr\,}} f(X)$ as an element in $[-\infty,\infty]$ by
$$
  {\mathrm{Tr\,}} f(X)=\lim_{k\to \infty} \sum_{j=1}^k f(\lambda_j(X)).
$$

\vskip 5pt
Given a symmetric norm $\|\cdot\|$ on $\bB$, the set where $\|\cdot\|$ takes a finite value is an ideal. We call it the maximal ideal
of $\|\cdot\|$ or the domain of $\|\cdot\|$. From Theorem \ref{th2} we immediately infer our last corollary.

\vskip 5pt
\begin{cor} Let $A,B\in\bB^+$ and $Z\in\bB$. Suppose that $AZB\in\bJ$, the domain of a symmetric norm. Then, for all $q\in(0,1)$, we
also have $|A^qZB^q|^{1/q}\in\bJ$.
\end{cor}

\vskip 5pt
Following \cite[Chapter 2]{Simon}, we denote by $\bJ^{(0)}$ the $\|\cdot\|$-closure of the finite rank operators. In most cases
$\bJ=\bJ^{(0)}$, however the strict inclusion $\bJ^{(0)}\subset_{\neq} \bJ$ may happen. We do not know whether we can replace in the
last corollary $\bJ$ by $\bJ^{(0)}$.

We close our article with two simple lemmas and show how the Cauchy-Schwarz inequality for the infinite dimensional case follows from
the matrix case.

\vskip 5pt\noindent
\begin{lemma}\label{lem1} Let $\|\cdot \|$ be a symmetric norm on $\bB$ and let $\{E_n\}_{n=1}^{\infty}$ be an increasing sequence of
finite rank projections in $\bB$, strongly converging to $I$. Then, for all $X\in \bB$, $\| X\|=\lim_n \| E_nXE_n\|$.
\end{lemma}

\vskip 5pt
\begin{proof} We first show that $\| E_nX\| \to \| X\|$ as $n\to \infty$.  Since  
$\| E_n X\|=\| (X^*E_nX)^{1/2}\|$ and $(X^*E_nX)^{1/2} \nearrow |X|$ by operator monotonicity of $t^{1/2}$, we obtain $\lim_n\|
E_nX\|= \| X\|$ by Definition \ref{def}(2). Similarly,
$\lim_k\| E_nXE_k\|=\| E_nX\|$, and so $\lim_n\| E_nXE_{k(n)}\|=\| X\|$, and thus, by Definition \ref{def}(3), $\lim_p\| E_pXE_p\|=\|
X\|$.
\end{proof}

\vskip 5pt\noindent
\begin{lemma}\label{lem2} Let $\|\cdot \|$ be a symmetric norm on $\bB$ and let $\{E_n\}_{n=1}^{\infty}$ and $\{F_n\}_{n=1}^{\infty}$
be two increasing sequences of finite rank projections in $\bB$, strongly converging to $I$. Then, for all $X\in \bB$, $\|
X^*X\|=\lim_n \| E_nX^*F_nXE_n\|$.
\end{lemma}

\vskip 5pt
\begin{proof} By Definition \ref{def}(2)-(3), the map $n\mapsto \| E_nX^*F_nXE_n\|$ is nondecreasing. By Definition \ref{def}(2), for
any integer $p$, its limit is greater than or equal $\| E_pX^*XE_p\|$. By Lemma \ref{lem1}, the limit is precisely $\| X^*X\|$.
\end{proof}

\vskip 5pt
Let $X,Y\in\bB$, let $\|\cdot\|$ be a symmetric norm on $\bB$, and let $\{E_n\}_{n=1}^{\infty}$ be as in the above lemma. Let $F_n$
be the range projection of $YE_n$. We have by Lemma \ref{lem1}
$$
\| X^*Y\| =\lim_n \| E_n X^*Y E_n \| =\lim_n  \| E_n X^*F_nY E_n \|.
$$
Let ${\mathcal{H}}_n$ be the sum of the ranges of $E_n$ and $F_n$. This is a finite dimensional subspace, say
$\dim{\mathcal{H}}_n=d(n)$. Applying the Cauchy-Schwarz inequality for a symmetric norm on $\bM_{d(n)}$, we obtain, thanks to Lemma
\ref{lem2},
\begin{align*}
\| X^*Y\| &=\lim_n  \| E_n X^*F_nY E_n\|_{\bM_{d(n)}}  \\
&\le \lim_n  \| E_n X^*F_nX E_n\|_{\bM_{d(n)}}^{1/2}  \| E_n Y^*F_n YE_n\|_{\bM_{d(n)}}^{1/2}  \\
&= \|  X^*X \|^{1/2} \|  Y^*Y \|^{1/2}. 
\end{align*}
Thus the Cauchy-Schwarz inequality for a symmetric norm on $\bB$ follows from the Cauchy-Schwarz inequality for symmetric norms on
$\bM_{n}$. Of course, the two previous lemmas and this discussion are rather trivial, but we wanted to stress on the fact that
Theorem \ref{th2} is essentially of finite dimensional nature. However, it would be also desirable to extend these results in the
setting of a semifinite von Neumann algebra.

\section{Around this article}

Let $\Phi: \mathrm{M}_n\rightarrow \mathrm{M}_m$ be a positive linear map and let $N\in\bM_n$ be normal. Then
there exists a unitary $V\in\mathrm{M}_m$ such that
\begin{equation}\label{comp1}
|\Phi(N)|\leq \frac{\Phi(|N|)+V\Phi(|N|)V^*}{2}
\end{equation}
and
\begin{equation*}\label{comp2}
|\Phi(N)|\leq \Phi(|N|)+\frac{1}{4} V\Phi(|N|)V^*.
\end{equation*}
These two inequalities and
several consequences are proved in \cite{BL-IJM2}, \cite{BL-russodye}. As an application for the Schur product of two normal matrices
$A, B\in\mathrm{M}_n$, one may infer that
\begin{equation*}\label{inequality for Schur product}
|A\circ B|\leq|A|\circ|B|+\frac{1}{4}V(|A|\circ|B|)V^*
\end{equation*}
for some unitary $V\in\mathrm{M}_n$, where the constant $1/4$ is optimal.
Another interesting consequence of \eqref{comp1} is the following improvement of the Russo-Dye theorem stating that every positive linear map attains
its norm at the identity: if $Z\in\bM_n$ is a contraction, then
\begin{equation*}\label{RDimprov}
 |\Phi(Z)|  \le \frac{\Phi(I) +V\Phi(I)V^*}{2}
\end{equation*}
for some unitary $V\in\bM_n$. Applying this to the Schur product with $S\in\bM_n^+$ yields some exotic eigenvalue inequalities such as
\begin{equation*}\label{eqschur}
\lambda_3(|S\circ Z|) \le \delta_{2}(S)
\end{equation*}
where $\lambda_3(\cdot)$ stands for the third largest eigenvalue, and $\delta_2(\cdot)$ for the second largest diagonal entry.

\section{References of Chapter 2}

{\small
\begin{itemize}

\item[[7\!\!\!]] H.\ Araki, On an inequality of Lieb and Thirring, {\it Let.\ Math.\ Phys}.\ 19 (1990 )167-170.

\item[[8\!\!\!]] K.\ Audenaert,  On the Araki-Lieb-Thirring Inequality, {\it Int.\ J.\ Inf.\ Syst.\ Sci.}\ 4 (2008),  \ 78-83.

\item[[9\!\!\!]] K.\ Audenaert and F.\ Hiai,  Reciprocal Lie-Trotter formula, {\it Linear Mult.\ Algebra}, in \ press

\item[[13\!\!\!]]  R.\ Bhatia, Matrix Analysis, Gradutate Texts in Mathematics, Springer, New-York, 1996.

\item[[15\!\!\!]] R.\ Bhatia, Positive Definite Matrices, Princeton University press, Princeton 2007.

\item[[26\!\!\!]] J.-C. Bourin, Matrix subadditivity inequalities and block-matrices, {\it Internat.\ J.\ Math.}\ 20 (2009), no.\
6, 679--691.

\item[[36\!\!\!]]  J.-C.\ Bourin and E.-Y.\ Lee, Matrix inequalities from a two variables functional, {\it  Internat.\ J.\ Math.}\ 27 (2016), no.\ 9, 1650071, 19 pp.

\item[[38\!\!\!]]   J.-C.\ Bourin and E.-Y. Lee,  Positive linear maps on normal matrices, {\it Internat.\ J.\ Math.}\ 29 (2018), no.\ 12, 1850088, 11 pp.

\item[[39\!\!\!]]   J.-C.\ Bourin and E.-Y. Lee, On the Russo-Dye theorem for positive linear maps. {\it Linear Algebra Appl.}\ 571 (2019), 92--102. 

\item[[54\!\!\!]] J.E. Cohen, Spectral inequalities for matrix exponentials, {\it  Linear Algebra Appl.} 111  (1988)  25-28.

\item[[55\!\!\!]] J.E. Cohen, S. Friedland, T. Kate, and F. Kelly, Eigenvalue inequalities for
products of matrix exponentials, {\it  Linear Algebra Appl.} 45  (1982)  55-95.

\item[[59\!\!\!]] D.J.H.\ Garling, Inequalities  - A journey into linear analysis. Cambridge University Press, Cambridge, 2007.

\item[[66\!\!\!]] F. Hiai,
A generalization of Araki's log-majorization. {\it Linear Algebra Appl.}\ 501 (2016), 1--16.

\item[[67\!\!\!]]  F. Hiai, D. Petz, Introduction to Matrix Analysis and applications. Universitext, Springer, New Delhi, 2014.

\item[[74\!\!\!]] H.\ Kosaki, Arithmetic-geometric mean and related inequalities for operators. {\it J.\ Funct.\ Anal}.\  156  (1998),  no. 2, 429-451.

\item[[89\!\!\!]] B.\ Simon, Trace ideal and their applications, Cambridge University Press, Cambridge, 1979.

\end{itemize}
}


\chapter{Unitary Orbits and Functions}

{\color{blue}{\Large {\bf Unitary orbits of Hermitian operators  with convex or concave functions} \large{\cite{BL-London}}}}

\vskip 10pt\noindent
{\bf Abstract.}
This short but self-contained survey presents a number of elegant matrix/operator inequalities for general convex or concave
functions, obtained with a unitary orbit technique. Jensen, sub or super-additivity type inequalities are considered. Some of them
are substitutes to classical inequalities (Choi, Davis, Hansen-Pedersen) for operator convex or concave functions. Various trace,
norm and determinantal inequalities are derived. Combined with an interesting decomposition for positive semi-definite matrices,
several results for partitioned matrices are also obtained.

{\small\noindent
Keywords: Operator inequalities, positive linear map, trace, unitary orbit, convex function,
symmetric norm, anti-norm.

AMS subjects classification 2010: Primary 15A60, 47A30, 47A60}

\vskip 15pt

\section{Introduction}

\noindent
The functional analytic aspect of
Matrix Analysis is evident when matrices or operators are considered as non-commutative numbers, sequences or functions. In
particular, a significant part of this theory consists in establishing theorems for Hermitian matrices regarded
as generalized real numbers or functions. Two classical trace inequalities may illustrate quite well this assertion. Given two
Hermitian matrices $A$, $B$ and a concave function $f(t)$ defined on the real line,
\begin{equation}\label{VN1}
{\mathrm{Tr\,}} f\left(\frac{ A+B}{2}\right) \ge {\mathrm{Tr\,}}\frac{f(A) +f(B)}{2}
\end{equation}
and, if further $f(0)\ge 0$ and both $A$ and $B$ are positive semi-definite,
\begin{equation}\label{Rot1}
{\mathrm{Tr\,}} f(A+B) \le {\mathrm{Tr\,}} f(A)+ {\mathrm{Tr\,}}f(B).
\end{equation}
The first inequality goes back to von-Neumann in the 1920's, the second is more subtle and has been proved only in 1969 by Rotfel'd
\cite{Rot}. These trace inequalities are matrix versions of obvious scalar inequalities.

The aim of this short survey is to present in a unified and self-contained way two recent significant improvement of the trace
inequalities \eqref{VN1}-\eqref{Rot1} and some of their consequences. Our unitary orbit method is also used to prove some basic facts
such as the triangle inequality for Schatten $p$-norms or Minkowski's determinantal inequality.

By operator, we mean a linear operator on a finite dimensional Hilbert space. We use interchangeably the terms operator and matrix.
Especially, a positive operator means a positive (semi-definite) matrix. Consistently $\bM_n$ denotes the set of operators on a space
of dimension $n$ and $\bM_n^+$ stands for the positive part. As many operator inequalities, our results lie in the scope of matrix
techniques. Of course, there are versions for operators acting on infinite dimensional, separable Hilbert spaces (and operator
algebras); we will indicate the slight modifications which might then be necessary.

The rest of this introduction explains why inequalities with unitary orbits are relevant for inequalities involving functional
calculus of operators such as the concavity-subadditivity statements \eqref{VN1} and \eqref{Rot1}.

That inequalities with unitary orbits naturally occur can be seen from the following two elementary facts. Firstly, If $A,B\in
\bM_n^+$ are such that $A\ge B$ (that is $A-B$ is positive semi-definite) then, whenever $p>1$,
it does not follow in general that $A^p\ge B^p$. However, for any non-decreasing function $f(t)$, the eigenvalues (arranged in
decreasing order and counted with their multiplicities) of $f(A)$ are greater or equal to the corresponding ones of $f(B)$. By the
min-max characterization of eigenvalues, this is equivalent to
 \begin{equation}\label{fact}
f(A) \ge Uf(B)U^*
\end{equation}
for some unitary $U\in\bM_n$. Secondly, if $A\in\bM_n^+$ and $C\in\bM_n$ is a contraction, then we have $ C^*AC \le UAU^*$ for some
unitary $U\in\bM_n$, i.e., the eigenvalues of $ C^*AC$ are smaller or equal to those of $A$. Note also that
$C^*AC=VA^{1/2}CC^*A^{1/2}V^*\le A$ for some unitary $V$, since $TT^*$ and $T^*T$ are unitarily congruent for any operator $T$. The
reading of this paper does not require more knowledge about matrices, see \cite{Bh} for a good background.

The most well-known matrix inequality involving unitary orbits is undoubtedly the triangle inequality due to Thompson \cite{T}: If
$X$ and $Y$ are two operators in $\bM_n$, then
\begin{equation}\label{Thompson}
|X+Y| \le U|X|U^* + V|Y|V^*
\end{equation}
for some unitary $U, V\in \bM_n$. Here $|X|:=(X^*X)^{1/2}$ is the positive part of $X$ occurring in the polar decomposition $X=V|X|$
for some unitary $V$. By letting
$$
X=\begin{pmatrix} A^{1/2}& 0 \\ 0&0 \end{pmatrix}, \qquad Y=\begin{pmatrix} 0& 0\\ B^{1/2}&0 \end{pmatrix}
$$
where $A, B\in\bM_n^+$, the triangle inequality \eqref{Thompson} yields
$
\sqrt{A+B} \le K\sqrt{A}K^* + L\sqrt{B}L^*
$
for some contractions $K, L\in \bM_n$. Thus, for some unitaries
$U, V\in \bM_n$
\begin{equation}\label{sqrt}
\sqrt{A+B} \le U\sqrt{A}U^* + V\sqrt{B}V^*.
\end{equation}
This inequality for the  function $\sqrt{t}$ is a special case of the main theorem of Section 3.

If $f(t)$ is convex on $[0,\infty)$, then \eqref{Rot1} is obviously reversed. In case of $f(t)=t^p$ with exponents $p\in[1,2]$ a much
stronger inequality holds,
\begin{equation}\label{opconv}
\left(\frac{ A+B}{2}\right)^p \le \frac{A^p +B^p}{2},
\end{equation}
this says that $t^p$ is operator convex for $p\in[1,2]$, and this is no longer true if $p>2$. However by making use of \eqref{fact}
and \eqref{opconv} we get, for any $p>1$,
\begin{equation}\label{F-I.1}
\left(\frac{ A+B}{2}\right)^p \le U\frac{A^p +B^p}{2}U^*,
\end{equation}
for some unitary $U\in\bM_n$. In fact, if we assume that \eqref{F-I.1} holds for $p\in[2^n,2^{n+1}]$, $n$ a positive integer, then it
also holds for $2p\in[2^{n+1},2^{n+2}]$ since
$$
\left(\frac{ A+B}{2}\right)^{2p} \le U_0\left(\frac{A^2 +B^2}{2}\right)^{p} U^*_0\le  U_0U_1\frac{A^{2p} +B^{2p}}{2}U_1^*U^*_0 
$$
for some unitary $U_0,\, U_1$.
Inequality \eqref{F-I.1} may serve as a motivation for Section 2. It is worthwhile to notice that, in contrast with the theory of
operator convex functions, our methods are rather elementary.

\section{ A matrix Jensen type inequality }

\subsection{Jensen type inequalities via unitary orbits}

\noindent In this section we present some extension of \eqref{VN1}. The most general one involves a unital positive linear map. A linear
map $\Phi: \bM_n\to \bM_d$ is unital if $\Phi(I)=I$ where $I$ stands for the identity of any order, and $\Phi$ is positive if
$\Phi(A)\in\bM_d^+$ for all $A\in\bM_n^+$. The simplest case is given when $d=1$ by the map
\begin{equation}\label{F-2.1}
A \mapsto \langle h, Ah\rangle
\end{equation}
for some unit vector $h$ (our inner product is linear in the second variable). Restricting this map to the diagonal part (more
generally, to any commutative $*$-subalgebra) of $\bM_n$, we have
\begin{equation}\label{F-2.2}
A \mapsto \langle h, Ah\rangle =\sum_{i=1}^n w_i \lambda_i(A)
\end{equation}
where the $\lambda_i(A)$'s are the eigenvalues of the normal operator $A$ and the $w_i$'s form a probability weight. For this reason,
unital positive linear maps are regarded as non-commutative versions of expectations. If $A$ is Hermitian, and $f(t)$ is a convex
function defined on the real line, the Jensen's inequality may be written in term of the map $\Phi$ in \eqref{F-2.1}-\eqref{F-2.2} as
\begin{equation}\label{Jensen}
f( \langle h, Ah\rangle) \le  \langle h, f(A)h\rangle.
\end{equation}

The map \eqref{F-2.1} is a special case of a compression. Given an $n$-dimensional Hilbert space ${\mathcal{H}}$ and a $d$-dimensional
subspace ${\mathcal{S}} \subset\mathcal{H}$, we have a natural map from the algebra ${\mathrm{L}}(\mathcal{H})$ of operators on
$\mathcal{H}$ onto the algebra ${\mathrm{L}}({\mathcal{S}})$, the compression map onto ${\mathcal{S}}$,
$$
A\mapsto A_{\mathcal{S}} := EA_{|{\mathcal{S}}}, \qquad A\in {\mathrm{L}}(\mathcal{H}),
$$
where $E$ denotes the ortho-projection onto ${\mathcal{S}}$. Identifying ${\mathrm{L}}(\mathcal{H})$ with $\bM_n$ by picking an
orthonormal basis of ${\mathcal{H}}$ and ${\mathrm{L}}(\mathcal{S})$ with $\bM_d$ via an orthonormal basis ${\mathcal{S}}$, we may
consider compressions as unital positive linear maps acting from $\bM_n$ onto $\bM_d$, and they are then represented as
$$
A\mapsto J^*AJ, \qquad A\in\bM_n,
$$
where $J$ is any $n$-by-$d$ matrix such that $J^*J=I$, the identity of order $d$.

In view of \eqref{Jensen} it is quite natural to compare for a convex function $f(A_{\mathcal S})$ and $f(A)_{\mathcal S}$ when $A$
is a Hermitian on $\mathcal{H}$, i.e, a Hermitian in $\bM_n$. In this setting, the Jensen inequality \eqref{Jensen} is adapted by
using unitary orbits on $\mathcal{S}$. This is actually true for any unital positive linear maps, as stated in Theorem 2.1 below.
This is the main result of this section.
The notation  $\bM_n\{\Omega\}$  stands for the Hermitian part of $\bM_n$ with spectra in an interval $\Omega$ of the real line.
\vskip 10pt
\begin{theorem}\label{T-2.1} Let $\Phi : \bM_n\to \bM_d$ be a unital positive linear map, let $f(t)$ be a convex function on an
interval $\Omega$, and let $A,B\in \bM_n\{\Omega\}$. Then, for some unitary $U,\,V\in\bM_d$,
\begin{equation*}
f(\Phi(A))\le \frac{U\Phi( f(A))U^*+V\Phi( f(A))V^*}{2}.
\end{equation*}
If furthermore $f(t)$ is monotone, then we can take $U=V$. The inequality reverses for concave functions.
\end{theorem}

The next corollaries list some consequences of the theorem.
This statement for positive linear maps contains several Jensen type inequalities. The simplest one is obtained by taking $\Phi :
\bM_{2n}\to \bM_n$,
\begin{equation*}
\Phi\left(\begin{bmatrix}  A &X \\ Y &B\end{bmatrix}\right) :=\frac{A+B}{2}.
\end{equation*}
With $X=Y=0$, Theorem \ref{T-2.1} then says:

\vskip 5pt
\begin{cor} \label{C-2.2} If $A,B\in \bM_n\{\Omega\}$ and $f(t)$ is a convex function on an interval $\Omega$, then, for some
unitaries $U,\,V\in\bM_n$,
\begin{equation*}
f\left(\frac{A+B}{2}\right)\le \frac{1}{2}\left\{U\frac{f(A)+f(B)}{2}U^*+V\frac{f(A)+f(B)}{2}V^*\right\}.
\end{equation*}
If furthermore $f(t)$ is monotone, then we can take $U=V$.
\end{cor}

\vskip 5pt
From this corollary we can get a generalization of the famous  Minkowski  inequality,
\begin{equation}\label{Minkowski}
{\det}^{1/n} (A+B) \ge {\det}^{1/n} A + {\det}^{1/n} B, \qquad A,\,B\in \bM_n^+.
\end{equation}
A proof is given after the proof of Corollary \ref{C-2.10} below. Equivalently, \eqref{Minkowski} says that the Minkowski functional
$X\mapsto \det^{1/n}X$ is concave on the positive cone $\bM_n^+$.
Combined with the concave version of Corollary \ref{C-2.2}, this concavity aspect of \eqref{Minkowski} is  improved as:

\vskip 5pt
\begin{cor}\label{cor-minkowski}
If $f(t)$ is a non-negative concave function on an interval   $\Omega$ and if $A,B\in \bM_n\{\Omega\}$, then,
\begin{equation*}
{\det}^{1/n} f\left(\frac{A+B}{2}\right) \ge \frac{ {\det}^{1/n} f(A) + {\det}^{1/n} f(B)}{2}.
\end{equation*}
\end{cor}

\vskip 5pt
Corollary \ref{C-2.2} deals with the simplest convex combination, the arithmetic mean of two operators. Similar statements holds for
weighted means of several operators. In fact these means may even have operator weights (called $C^*$-convex combinations). An
$m$-tuple $\{Z_i\}_{i=1}^m$ in $\bM_n$ is an isometric column if $\sum_{i=1}^m Z_i^*Z_i =I$. We may then perform the $C^*$-convex
combination $\sum_{i=1}^m Z^*_i A_iZ_i $. If all the $A_i$'s
are Hermitian operators in $\bM_n\{\Omega\}$ for some interval $\Omega$, then so is $\sum_{i=1}^m Z^*_i A_iZ_i $. Hence, Corollary
\ref{C-2.2} is a very special case of the next one.

\vskip 5pt
\begin{cor}\label{C-2.3} Let $\{Z_i\}_{i=1}^m$ be an isometric column in $\bM_n$, let $\{A_i\}_{i=1}^m$ be in $\bM_n\{\Omega\}$ and
let $f(t)$ be a convex function on $\Omega$. Then, for some unitary $U,\,V\in\bM_n$,
\begin{equation*}
f\left( \sum_{i=1}^m Z^*_i A_iZ_i \right) \le \frac{1}{2} \left\{ U\left(\sum_{i=1}^m Z^*_i f(A_i)Z_i \right)U^*+V\left(\sum_{i=1}^m
Z^*_i f(A_i)Z_i\right) V^*\right\}.
\end{equation*}
If furthermore $f(t)$ is monotone, then we can take $U=V$. The inequality reverses for concave functions.
\end{cor}

\vskip 5pt
If all the $A_i$'s are zero except the first one, we obtain an inequality involving a congruence $Z_1^*A_1Z_1$ with a contraction
$Z_1$ (that is $Z_1^*Z_1 \le I$). We state the concave version in the next corollary. It is a matrix version of the basic inequality
$f(za) \ge zf(a)$ for a concave function with $f(0)\ge 0$ and real numbers $z,a$ with $z\in[0,1]$.

\vskip 5pt
\begin{cor}\label{C-2.5} Let $f(t)$ be a concave function on an interval $\Omega$ with $0\in\Omega$ and $f(0)\ge 0$, let $A\in
\bM_n\{\Omega\}$ and let $Z$ be a contraction in $\bM_n$. Then, for some unitaries $U,\,V\in\bM_n$,
\begin{equation*}
f\left(Z^*AZ\right)\ge \frac{ U\left( Z^*f(A)Z \right)U^*+V\left(Z^* f(A)Z\right) V^*}{2}.
\end{equation*}
If furthermore $f(t)$ is monotone, then we can take $U=V$.
\end{cor}

For a sub-unital positive linear map $\Phi$, i.e., $\Phi(I)\le I$, it is easy to see that Theorem \ref{T-2.1} can be extended in the convex
case when $f(0)\le 0$, and in the concave case, when $f(0)\ge 0$ (this sub-unital version is proved in the proof of Corollary \ref{C-2.7}
below). This also contains Corollary \ref{C-2.5}. The above results contains some inequalities for various norms and functionals, as
noted in some of the corollaries and remarks below. For instance we have the following Jensen trace inequalities.

\vskip 5pt
\begin{cor}\label{C-2.6} Let $f(t)$ be a convex function defined on an interval $\Omega$, let $\{A_i\}_{i=1}^m$ be in
$\bM_n\{\Omega\}$, and let $\{Z_i\}_{i=1}^m$ be an isometric column in $\bM_n$. Then,
\begin{equation}\label{HP}
{\mathrm{Tr\,}} f\left( \sum_{i=1}^m Z^*_i A_iZ_i \right)\le {\mathrm{Tr\,}} \sum_{i=1}^m Z^*_i f(A)_iZ_i.
\end{equation}
If further $0\in\Omega$ and $f(0)\le 0$, we also have
\begin{equation}\label{BK}
{\mathrm{Tr\,}} f\left( Z^*_1 A_1Z_1 \right)\le {\mathrm{Tr\,}} Z^*_1 f(A_1)Z_1.
\end{equation}
\end{cor}

\vskip 5pt
A typical example of positive linear map on $\bM_n$ is the Schur multiplication map $A\mapsto Z\circ A$ with an operator
$Z\in\bM_n^+$. Here $Z\circ A$ is the entrywise product of $A$ and $Z$. The fact that the Schur multiplication with $Z\in\bM_n^+$ is
a positive linear map can be easily checked by restricting the Schur product to positive rank ones operators.
Hence, Theorem \ref{T-2.1} contains results for the Schur product. In particular, the sub-unital version yields:

\vskip 5pt
\begin{cor}\label{C-2.7} Let $f(t)$ be a concave function on an interval $\Omega$ with $0\in\Omega$ and $f(0)\ge 0$, and let $A\in
\bM_n\{\Omega\}$. If $Z\in\bM_n^+$ has diagonal entries all less than or equal to 1, then, for some unitaries $U,\,V\in\bM_n$,
\begin{equation*}
f\left(Z\circ A\right)\ge \frac{ U\left( Z\circ f(A) \right)U^*+V\left(Z\circ f(A)\right) V^*}{2}.
\end{equation*}
If furthermore $f(t)$ is monotone, then we can take $U=V$.
\end{cor}

\vskip 5pt
\begin{proof} Let $\Psi:\bM_n\to M_d$ be a positive linear map and suppose that $\Psi$ is sub-unital, i.e., $\Psi(I)=C$ for some
contraction $C\in\bM_d^+$. Then the map $\Phi:\bM_{n+1}\to M_d$,
$$
\begin{bmatrix}
A & \vdots \\
\hdots & b
\end{bmatrix}
\mapsto
\Psi(A) + b(I-C)
$$
is  unital. Thus, by Theorem  \ref{T-2.1}, If $A\in\bM_n\{{\Omega}\}$ where $\Omega$ contains $0$ and if $f(t)$ is concave on $\Omega$,
$$
f\left(\Phi
(A \oplus 0)
\right) \ge
\frac{U\Phi(f(A\oplus 0))U^*+ V\Phi(f(A\oplus 0))V^*}{2}
$$
for some unitary $U,V\in\bM_d$,
equivalently,
$$
f(\Psi(A)) \ge
\frac{U\{\Psi(A)+f(0)(I-C)\}U^*+ V\{\Psi(A)+f(0)(I-C)\}V^*}{2}
$$
 hence, if further $f(0)\ge 0$, the sub-unital form of Theorem \ref{T-2.1} :
$$
f(\Psi(A)) \ge
\frac{U\Psi(f(A))U^*+ V\Psi(f(A))V^*}{2}.
$$
Applying this to the sub-unital map $\Psi : A\mapsto Z\circ A$ yields the corollary.  \end{proof}

\vskip 5pt Corollary \ref{C-2.7} obviously contains a trace inequality companion to \eqref{BK}. By making use of \eqref{Minkowski} we
also have the next determinantal inequality.

\vskip 5pt
\begin{cor}\label{C-2.8} Let $f(t)$ be a non-negative concave function on an interval $\Omega$, $0\in\Omega$, and let $A\in
\bM_n\{\Omega\}$. If $Z\in\bM_n^+$ has diagonal entries all less than or equal to 1, then,
\begin{equation}
\det f\left(Z\circ A\right)\ge \det Z\circ f(A).
\end{equation}
\end{cor}

\vskip 5pt
Some other consequences of Theorem \ref{T-2.1} are given below in Corollary \ref{C-2.10} and in Subsection 2.2, as well as references and
related results.

We turn to the proof of Theorem  \ref{T-2.1}. Thanks to the next lemma, we will see that it is enough to prove Theorem  \ref{T-2.1} for compressions.
By an abelian $*$-subalgebra $\mathcal{A}$ of $\bM_m$ we mean a subalgebra containing the identity of $\bM_m$ and closed under the
involution $A\mapsto A^*$. Any abelian $*$-subalgebra $ {\mathcal{A}}$ of $\bM_m$ is spanned by a total family of ortho-projections,
i.e., a family of mutually orthogonal projections adding up to the identity. A representation $\pi : \mathcal{A}\to \bM_n$ is a
unital linear map such $\pi(A^*B)=\pi^*(A)\pi(B)$.

\vskip 5pt
 \begin{lemma}\label{L-2.9}  Let $\Phi$ be a unital  positive  map from
 an abelian $*$-subalgebra ${\mathcal{A}}$ of $\bM_n$
to the algebra $\bM_m$ identified as ${\mathrm{L}}({\mathcal{S}})$. Then, there exists a space ${\mathcal{H}}\supset{\mathcal{S}}$,
$\dim {\mathcal{H}}\le nm$, and a
representation
$\pi$ from ${\mathcal{A}}$ to ${\mathrm{L}}({\mathcal{H}})$
such that
$$
\Phi (X)=(\pi(X))_{\mathcal{S}}.
$$
\end{lemma}

 \vskip 5pt
\begin{proof} ${\mathcal{A}}$ is generated by a total family of $k$
projections $E_i$, $i=1,\dots ,k$ (say $E_i$ are rank one, that is
$k=n$). Let $A_i=\Phi(E_i)$, $i=1,\dots ,n$. Since $\sum_{i=1}^n A_i$ is
the identity on ${\mathcal{S}}$, we can find operators $X_{i,j}$ such
that
$$
V=
\begin{pmatrix}
A_1^{1/2}&\dots &A_n^{1/2} \\
X_{1,1} &\dots &X_{n,1} \\
\vdots &\ddots &\vdots \\
X_{1,n-1} &\dots &X_{n,n-1}
\end{pmatrix}
$$
is a unitary operator on ${\mathcal{F}}=\oplus^n{\mathcal{S}}$. Let $R_i$ be the block matrix with the same $i$-th column than $V$
and with all other entries $0$. Then, setting $P_i=R_iR_i^*$, we obtain a total family of projections on ${\mathcal{F}}$ satifying
$A_i=(P_i)_{\mathcal{S}}$. We define $\pi$ by $\pi(E_i)=P_i$. \end{proof}

In the following proof of Theorem  \ref{T-2.1}, and in the rest of the paper,
the eigenvalues of a Hermitian $X$ on an $n$-dimensional space are denoted in non-increasing order as $\lambda_1(X)\ge
\cdots\ge\lambda_n(X)$.

\vskip 5pt\noindent
\begin{proof}
We consider the convex case. We first deal with a compression map. Hence $\bM_n$ is identified with ${\mathrm{L}}({\mathcal{H}})$ and
$
\Phi(A) = A_{\mathcal S}
$
where ${\mathcal{S}}$  is a subspace of ${\mathcal{H}}$.
We may find spectral subspaces ${\mathcal{S}}'$ and ${\mathcal{S}}''$ for $A_{\mathcal{S}}$ and a real $r$ such that

\begin{itemize}
\item[(a)]
 ${\mathcal{S}}={\mathcal{S}}'\oplus{\mathcal{S}}''$,

\item[(b)] the spectrum of $A_{\mathcal{S'}}$ lies on $(-\infty,r]$ and the spectrum of  $A_{\mathcal{S''}}$ lies on $[r,\infty)$,

 \item[(c)] $f$ is monotone both on $(-\infty,r]\cap\Omega$ and  $[r,\infty)\cap\Omega$.
\end{itemize}

Let $k$ be an integer, $1\le k\le \dim{\mathcal{S}}'$. There exists a spectral subspace ${\mathcal{F}}\subset{\mathcal{S}}'$ for
$A_{\mathcal{S'}}$ (hence for $f(A_{\mathcal{S'}})$), $\dim{\mathcal{F}}=k$, such that
\begin{align*} \lambda_k[f(A_{\mathcal{S'}})] &=\min_{h\in{\mathcal{F}};\ \Vert h\Vert=1} \langle h,f(A_{\mathcal{F}})h \rangle  \\
&= \min\{f(\lambda_1(A_{\mathcal{F}}))\,;\,f(\lambda_k(A_{\mathcal{F}}))\} \\
&= \min_{h\in{\mathcal{F}};\ \Vert h\Vert=1} f(\langle h,A_{\mathcal{F}}h \rangle)  \\
&= \min_{h\in{\mathcal{F}};\ \Vert h\Vert=1} f(\langle h,Ah \rangle)
\end{align*}
where at the second and third steps we use the monotony of $f$ on $(-\infty,r]$ and the fact that $A_{\mathcal{F}}$'s spectrum lies
on $(-\infty,r]$. The convexity of $f$ implies
$$
f(\langle h,Ah \rangle) \le \langle h,f(A)h \rangle
$$
for all normalized vectors $h$. Therefore, by the minmax principle,
\begin{align*}
\lambda_k[f(A_{\mathcal{S'}})] &\le \min_{h\in{\mathcal{F}};\ \Vert h\Vert=1} \langle h,f(A)h \rangle \\
&\le \lambda_k[f(A)_{\mathcal{S'}}].
\end{align*}
This statement is equivalent (by unitary congruence to diagonal matrices) to the existence of a unitary operator $U_0$ on
${\mathcal{S}}'$ such that
$$
f(A_{\mathcal{S'}})\le U_0 f(A)_{\mathcal{S'}}U_0^*.
$$
(Note that the monotone case is established.) Similarly we get a unitary $V_0$ on ${\mathcal{S}}''$ such that
$$
f(A_{\mathcal{S''}})\le V_0 f(A)_{\mathcal{S''}}V_0^*.
$$
Thus we have
$$
f(A_{\mathcal{S}})\le
\begin{pmatrix}
U_0 &0 \\ 0&V_0
\end{pmatrix}
 \begin{pmatrix}
f(A)_{\mathcal{S'}} &0 \\ 0& f(A)_{\mathcal{S''}}
\end{pmatrix}
 \begin{pmatrix}
U_0^* &0 \\ 0&V_0^*
\end{pmatrix}.
$$
Besides we note that, still in respect with the decomposition ${\mathcal{S}}={\mathcal{S}}'\oplus{\mathcal{S}}''$,
$$
\begin{pmatrix}
f(A)_{\mathcal{S'}} &0 \\ 0& f(A)_{\mathcal{S''}}
\end{pmatrix}
= \frac{1}{2}\left\{
\begin{pmatrix}
I &0 \\ 0& I
\end{pmatrix}
f(A)_{\mathcal{S}}
\begin{pmatrix}
I &0 \\ 0& I
\end{pmatrix}
+
\begin{pmatrix}
I &0 \\ 0& -I
\end{pmatrix}
f(A)_{\mathcal{S}}
\begin{pmatrix}
I &0 \\ 0& -I
\end{pmatrix}
\right\}.
$$
So, letting
$$
U=\begin{pmatrix}
U_0 &0 \\ 0& V_0
\end{pmatrix}
\quad{\rm and}\quad
V=
\begin{pmatrix}
U_0 &0 \\ 0& -V_0
\end{pmatrix}
$$
we get
\begin{equation}
f(A_{\mathcal{S}}) \le \frac{Uf(A)_{\mathcal{S}}U^* + Vf(A)_{\mathcal{S}}V^*}{2}
\end{equation}
for some unitary $U,V\in{\mathrm{L}}({\mathcal{S}})$, with $U=V$ if $f(t)$ is convex and monotone. This proves the case of
compression maps.

Next we turn to the case of a general unital linear map $\Phi:\bM_n\to \bM_m$.
Let $\mathcal{A}$ be the abelian $*$-subalgebra of $\bM_n$ spanned by $A$. By restricting $\Phi$ to $\mathcal{A}$ and by identifying
$\bM_m$ with ${\mathrm{L}}({\mathcal{S}})$, Lemma \ref{L-2.9} shows that $\Phi(X)=(\pi(X))_{\mathcal{S}}$ for all $X\in\mathcal{A}$.
Since $f$ and $\pi$ commutes, $f(\pi(A))=\pi(f(A))$, we have from the compression case some unitary $U, V\in
{\mathrm{L}}({\mathcal{S}})=\bM_m$ such that,
\begin{align*}
f(\Phi(A) ) &=f((\pi(A))_{\mathcal{S}})  \\ &\le \frac{U(f(\pi(A))_{\mathcal{S}}U^* + V(f(\pi(A))_{\mathcal{S}}V^*}{2} \\
&= \frac{U(\pi(f(A))_{\mathcal{S}}U^* + V(\pi (f(A))_{\mathcal{S}}V^*}{2} \\
&= \frac{U\Phi(f(A))U^* + V\Phi(f(A))V^*}{2},
\end{align*}
where we can take $U=V$ if the function is convex and monotone.
\end{proof}

\vskip 5pt
The following is an application of Theorem \ref{T-2.1} to norm inequalities.
A norm $\|\cdot\|$ on $\bM_n$ is a symmetric norm if $\|A\|=\|UAV\|$ for all $A\in\bM_n$ and all unitary $U,V\in\bM_n$. These norms
are also called unitarily invariant norms. They contain the Schatten $p$-norms $\|\cdot\|_p$, $1\le p<\infty$, defined as $\| A\|_p
=\{\mathrm{Tr\,}|A|^p\}^{1/p}.$
The polar decomposition shows that a symmetric norm $\|\cdot\|$ is well defined by its value on the positive cone $\bM_n^+$. The map
on $\bM_n^+$, $A\mapsto\|A\|$ is invariant under unitary congruence and is subadditive. There are also some interesting, related
superadditive functionals.
Fix $p<0$. The map $X\mapsto\|X^p\|^{1/p}$ is continuous on the invertible part of $\bM_n^+$. If $X\in\bM_n^+$ is not invertible,
setting $\| X^p\|^{1/p}:=0$, we obtain a continuous map on $\bM_n^+$.

\vskip 5pt
\begin{cor}\label{C-2.10} Let $A,B\in\bM_n^+$ and let $p<0$. Then, for all symmetric norms,
\begin{equation} \label{derived}
\|\, (A+B)^p\,\|^{1/p} \,\ge\, \| A^p\|^{1/p}+ \| B^p\|^{1/p}.
\end{equation}
\end{cor}
\vskip 5pt
\begin{proof}
 We  will apply Theorem \ref{T-2.1}  to the monotone convex function on  $(0,\infty)$, $t\mapsto t^p$.
First, assume that $A,B\in\bM_n^+$ are such that $\|A^p\|=\|B^p\|=1$ and let $s\in[0,1]$. Then, thanks to Theorem \ref{T-2.1} (or Corollary
\ref{C-2.3}),
$$
\|(sA+(1-s)B)^p\| \le \| sA^p + (1-s)B^p \|
\le  s\|A^p\| + (1-s)\|B^p\|
=1,
$$
hence
\begin{equation}\label{F-2.14}
\|(sA+(1-s)B)^p\|^{1/p} \ge 1.
\end{equation}
Now, for general invertible $A,B\in\bM_n^+$, insert $A/\|A^p\|^{1/p}$ and $B/\|B^p\|^{1/p}$ in place of $A, B$ in \eqref{F-2.14} and
take
$$
s=\frac{\|A^p\|^{1/p}}{\|A^p\|^{1/p} +\|B^p\|^{1/p}}.
$$
This yields \eqref{derived}.
\end{proof}

\vskip 5pt
Corollary \ref{C-2.10} implies Minkowski's determinantal inequality \eqref{Minkowski}. Indeed, in \eqref{derived} take the norm on $\bM_n^+$ defined by $\|
A\|:=\frac{1}{n}{\mathrm{Tr\,}}A$,
and note that $\det^{1/n} A=\lim_{p\nearrow 0} \| A^p\|^{1/p}$.
Hence,  the superadditivity of $A\mapsto \|A^p\|^{1/p}$ for $p<0$ entails the superadditivity of $A\mapsto \det^{1/n} A$.

If we apply Theorem \ref{T-2.1} (or Corollary \ref{C-2.2}) to the convex function on the real line $t\mapsto |t|$ we obtain: {\it If $A,B\in\bM_n$
are Hermitian, then
\begin{equation*}
|A+B| \le \frac{U(|A|+|B|)U^*+V(|A|+|B|)V^*}{2}
\end{equation*}
for some unitaries $U,V\in\bM_n$.} In fact, we can take $U=I$ and this remains true for normal operators $A,B$. This is shown in the
proof of the following proposition.

\vskip 5pt
\begin{prop} If $f(t)$ is a nondecreasing convex function on $[0,\infty)$ and if $Z\in\bM_n$ has a Cartesian decomposition $Z=A+iB$,
then, for some unitaries $U,V\in\bM_n$,
\begin{equation*}
f(|Z|) \le \frac{Uf(|A|+|B|)U^*+Vf(|A|+|B|)V^*}{2}.
\end{equation*}
\end{prop}

\vskip 5pt
\begin{proof} let $X$, $Y$ be two normal operators in $\bM_n$. Then, the following operators in $\bM_{2n}$ are positive
semi-definite,
$$
\begin{pmatrix}
|X| &X^* \\
X&|X|
\end{pmatrix} \ge 0,
\qquad
\begin{pmatrix}
|Y| &Y^* \\
Y&|Y|
\end{pmatrix} \ge 0,
$$
and consequently
$$
\begin{pmatrix}
|X| +|Y|&X^*+Y^* \\
X+Y&|X|+|Y|
\end{pmatrix}
\ge 0.
$$
Next, let $W$ be the unitary part in the polar decomposition $X+Y=W|X+Y|$. Then
$$
\begin{pmatrix}
I&-W^*
\end{pmatrix}
\begin{pmatrix}
|X| +|Y|&X^*+Y^* \\
X+Y&|X|+|Y|
\end{pmatrix}
\begin{pmatrix}
I \\ -W
\end{pmatrix}
\ge0,
$$
that is
$$
|X|+|Y|+W^*(|X|+|Y|)W- 2|X+Y|\ge 0.
$$
Equivalently,
\begin{equation}\label{F-2.11}
|X+Y| \le \frac{|X|+|Y|+W^*(|X|+|Y|)W}{2}.
\end{equation}
Letting $X=A$ and $Y=iB$, and applying $f(t)$ to both sides of \eqref{F-2.11},
  Corollary 2.2 completes the proof since $f(t)$ is nondecreasing and convex.
\end{proof}

\vskip 5pt
\begin{prop} If $f(t)$ is a nondecreasing convex function on $[0,\infty)$ and if $A,B\in\bM_n$ are Hermitian, then, for some
unitaries $U,V\in\bM_n$,
\begin{equation*}
f((A+B)_+) \le \frac{Uf(A_++B_+)U^*+Vf(A_++B_+)V^*}{2}.
\end{equation*}
\end{prop}

\vskip 5pt
\begin{proof} Here $A_+:=(A+|A|)/2$.
Note that $A+B\le A_++ B_+$. Let $E$ be the projection onto ${\mathrm{ran\,}}(A+B)_+$ and let $F$ be the projection onto
$\ker(A+B)_+$
Since
$(A+B)_+ =E(A+B)E$, we have
$$
(A+B)_+\le  E(A_++ B_+)E + F(A_++B_+)F,
$$
equivalently
\begin{equation}
(A+B)_+\le \frac{(A_++ B_+) +W(A_++ B_+)W^*}{2}
\end{equation}
where $W= E-F$ is a unitary. Applying Corollary 2.2 completes the proof.
 \end{proof}

\subsection{Comments and references}

\noindent
In this second part of Section 2, we collect few remarks which complete Theorem \ref{T-2.1} and the above corollaries. Good references for
positive maps and operator convex functions are the nice survey and book \cite{Hiai1} and \cite{Bh}.

\vskip 10pt
\begin{remark} Theorem \ref{T-2.1} appears in \cite{B2}. It is stated therein for compressions maps and for the case of or $*$-convex
combinations given in Corollaries \ref{C-2.3} and \ref{C-2.5} (the monotone case was earlier obtained in \cite{B1}). That the compression case
immediately entails the general case of an arbitrary unital positive map is mentioned in some subsequent papers, for instance in
\cite{AB} where some inequalities for Schur products are pointed out. From the Choi-Kraus representation of completely positive
linear maps, readers with a background on positive maps may also notice that Corollary \ref{C-2.5} and Theorem \ref{T-2.1} are equivalent.
For scalar convex combinations and with the assumption that $f(t)$ is non-decreasing, Theorem \ref{T-2.1} is first noted in Brown-Kosaki's
paper \cite{BK}; with these assumptions, it is also obtained in Aujla-Silva's paper \cite{AS}.
\end{remark}

\vskip 10pt
\begin{remark} Let $g(t)$ denote either the convex function $t\mapsto |t|$ or $t\mapsto t_+$. Let $A,B\in\bM_n$ be Hermitian. Then
(2.11) and (2.12) show that
$$
g\left(\frac{A+B}{2}\right) -\frac{g(A)+g(B)}{4} \le V\frac{g(A)+g(B)}{4}V^*
$$
for some unitary $V\in\bM_n$. It would be interesting to characterize convex functions for which such a relation holds.
\end{remark}

\vskip 10pt
\begin{remark} Theorem  \ref{T-2.1} holds for operators acting on infinite dimensional spaces, with an additional $rI$ term. We state here the
monotone version. ${\mathcal{H}}$ and ${\mathcal{S}}$ are two separable Hilbert spaces and $r>0$ is fixed. {\it
Let $\Phi : {\mathrm{L}}({\mathcal{H}})\to {\mathrm{L}}({\mathcal{H}})$ be a unital positive linear map, let $f(t)$ be a monotone
convex function on $(-\infty,\infty)$ and let
 $A,B\in  {\mathrm{L}}({\mathcal{H}}) $  be Hermitian. Then, for some unitary $U\in {\mathrm{L}}({\mathcal{S}})$,}
\begin{equation}\label{F-2.13}
f(\Phi(A))\le U\Phi( f(A))U^* +rI.
\end{equation}
The proof is given in the first author's thesis when $\Phi$ is a compression map, this entail the general case. For convenience, the
proof is given at the end of this section.
\end{remark}

\vskip 5pt
\begin{remark}
The trace inequality \eqref{HP} is due to Hansen-Pedersen \cite{HP2}, and the special case \eqref{BK} is due to Brown-Kosaki.
Corollary  \ref{C-2.3} considerably improves \eqref{HP}: In case of a monotony assumption on the convex function $f(t)$, we have eigenvalue
inequalities; and, in the general case we may still infer the majorization relation
\begin{equation}\label{F-2.9}
\sigma_k \left[f\left( \sum_{i=1}^m Z^*_i A_iZ_i \right)\right]\le \sigma_k\left[ \sum_{i=1}^m Z^*_i f(A)_iZ_i \right], \qquad
k=1,\ldots,n,
\end{equation}
where $\sigma_k[X]:=\sum_{j=1}^k\lambda_j[X]$ is the sum of the $k$ largest eigenvalues of a Hermitian $X$. In fact, the basic
relation $\sigma_k[X]=\max {\mathrm{Tr\,}} XE$, where the maximum runs over all rank $k$ projections $E$, shows that
$\sigma_k[\cdot]$ is convex, increasing on the Hermitian part of $\bM_n$ so that \eqref{F-2.9} is an immediate consequence of Theorem
\ref{T-2.1}. The theorem also entails (see \cite{B2} for details) a rather unexpected eigenvalue inequality:
$$
\lambda_{2k-1}\left[ f\left( \sum_{i=1}^m Z^*_i A_iZ_i \right)\right] \le \lambda_k\left| \sum_{i=1}^m Z^*_i f(A)_iZ_i \right] ,
\qquad 1\le k\le (n+1)/2.
$$
\end{remark}

\vskip 5pt
\begin{remark}  Choi's inequality \cite{Choi} claims: {\it for an operator convex function $f(t)$ on  $\Omega$,
\begin{equation}\label{Choi}
f(\Phi(A)) \le \Phi(f(A))
\end{equation}
for all $A\in\bM_n\{\Omega\}$ and all unital positive linear map.} Thus
Theorem \ref{T-2.1} is a substitute of Choi's inequality for a general convex function. In the special case of a compression map, then
\eqref{Choi} is Davis' inequality \cite{Davis}, a famous characterization of operator convexity. The most well-known case of Davis'
inequality is for the inverse map on positive definite matrices, it is then an old classical fact of Linear Algebra. Exactly as
Theorem  \ref{T-2.1} entails Corollary  \ref{C-2.3}, Choi's inequality contains
Hansen-Pedersen's inequality \cite{HP1}, \cite{HP2}: {\it If $f(t)$ is operator convex on $\Omega$, then
\begin{equation*}
f\left( \sum_{i=1}^m Z^*_i A_iZ_i \right) \le
\sum_{i=1}^m Z^*_i f(A)_iZ_i
\end{equation*}
for  all  unitary columns $\{Z_i\}_{i=1}^m$ in $\bM_n$ and $A_i\in\bM_n\{\Omega\}$, $i=1,\ldots,m$.}
For operator concave functions, the inequality reverses. A special case is Hansen's inequality \cite{Hansen}: {\it if $f(t)$ is
operator concave on $\Omega$, $0\in\Omega$ and $f(0)\ge 0$, then
\begin{equation}\label{Hansen}
f(Z^*AZ) \ge Z^*f(A)Z
\end{equation}
for all $A\in\bM_n\{\Omega\}$ and all contractions $Z\in\bM_n$.}
\end{remark}

\vskip 10pt
\begin{remark} Hansen's inequality \eqref{Hansen} may be formulated with an expansive operator $Z\in\bM_n$, i.e., $Z^*Z\ge I$; then
\eqref{Hansen} obviously reverses. We might expect that in a similar way, Corollary \ref{C-2.5} or the Brown-Kosaki trace inequality
reverses. But this does not hold. Corollary \ref{C-2.5} can not reverse when $Z$ is expansive, even under the monotony assumption on $f(t)$.
An unexpected positivity assumption is necessary, and we must confine to weaker inequalities, such as trace inequalities: {\it if
$f(t)$ is a concave function on the positive half-line with $f(0) \ge0$, then,
\begin{equation*}
{\mathrm{Tr\,}} f\left( Z^* AZ \right)\le {\mathrm{Tr\,}} Z^* f(A)Z
\end{equation*}
for all $A\in\bM_n^+$ and all expansive $Z\in\bM_n$.} For a proof, see \cite{B1} and also \cite{B3}, \cite{BL-JOT} where remarkable
extensions to norm inequalities are given.
\end{remark}

\vskip 10pt
\begin{remark} Lemma \ref{L-2.9} is a part of Stinespring' s theory of positive and completely positive linear maps in the influential 1955
paper \cite{S}. The proof given here is somewhat original
and is taken from \cite{AB}. Note that in the course of the proof, we prove Naimark's dilation theorem: {\it If $\{A_i\}_{i=1}^n$ are
positive operators on a space ${\mathcal{S}}$ such that $\sum_{i=1}^n A_i \le I$, then there exist some mutually orthogonal
projections $\{P_i\}_{i=1}^n$ on a larger space ${\mathcal{H}}\supset{\mathcal{S}}$ such that $(P_i)_{\mathcal{S}}=A_i$, ($1\le i\le
n$). }
\end{remark}

\vskip 10pt
\begin{remark}  Given a symmetric norm $\|\cdot\|$ and $p<0$, the  functionals defined on $\bM_n^+$,
$A\mapsto \|A^p\|^{1/p}$, are introduced in \cite{BH2} and called {\it derived anti-norms}. Corollary \ref{C-2.10} is given therein,
\cite[Proposition 4.6]{BH2}. The above proof is much simpler than the original one. For more details and many results on {\it
anti-norms} and {\it derived anti-norms}, often in connection with Theorem \ref{T-2.1}, see \cite{BH1} and \cite{BH2}. Several results in
these papers are generalizations of Corollary \ref{C-2.3}. By using \eqref{F-I.1} and arguing as in the proof of Corollary \ref{C-2.10}, we may derive the
triangle inequality for the Schatten $p$-norms on $\bM_n^+$, i.e.,
\begin{equation}\label{trpos}
\{ {\mathrm{Tr}\,} (A+B)^p \}^{1/p} \le \{ {\mathrm{Tr\,}} A^p \}^{1/p} + \{ {\mathrm{Tr\,}} B^p \}^{1/p}, \qquad A,B\in\bM_n^+, \
p>1.
\end{equation}
\end{remark}

\vskip 10pt
\begin{remark} The inequality \eqref{F-2.11} for normal operators can be extended to general $A,B\in\bM_n$, with a similar proof, as
\begin{equation}\label{BR}
|A+B|\le\frac{|A|+|B|+V(|A^*|+|B^*|)V^*}{2}
\end{equation}
for some unitary $V\in\bM_n$. This is pointed out in \cite{BR}. This is still true for operators $A,B$ in a von Neumann algebra
$\mathcal{M}$ with $V$ a partial isometry in $\mathcal{M}$. If $\mathcal{M}$ is endowed with a regular trace, this gives a short,
simple proof of the triangle inequality for the trace norm on $\mathcal{M}$.
Inequality \eqref{F-2.11} raises the question of comparison $|A+B|$ and $|A|+|B|$. The following result is given in \cite{Lee2}.
 {\it Let $A_1,\cdots, A_m$ be invertible operators  with
condition numbers dominated by  $\omega > 0 $. Then}
$$
|A_1+\cdots +A_m|\leq \frac{\omega+1}{2\sqrt{\omega}}(|A_1|+\cdots
+|A_m|).
$$
Here the condition number of an invertible operator $A$ on a Hilbert space is $\|A\|\|A^{-1}\|^{-1}$. Note that the bound is
independent of the number of operators. Though it is a rather low bound, it is not known whether it is sharp. Combining \eqref{BR}
and \eqref{trpos} we get the triangle inequality for the Schatten $p$-norms on the whole space $\bM_n$,
\begin{equation*}
\{ {\mathrm{Tr}\,} |A+B|^p \}^{1/p} \le \{ {\mathrm{Tr\,}} |A|^p \}^{1/p} + \{ {\mathrm{Tr\,}} |B|^p \}^{1/p}, \qquad A,B\in\bM_n, \
p>1.
\end{equation*}

\end{remark}

\vskip 5pt
It remains to give a proof of the infinite dimensional version \eqref{F-2.13} of the monotone case of Theorem 2.1, the non-monotone
case following in a similar way to the finite dimensional version. As for finite dimensional spaces, we may assume that $\Phi$ is a
compression map, thus we consider a subspace $\mathcal{S}\subset{\mathcal H}$ and the map $A\mapsto A_{\mathcal{S}}$. By replacing
$f(t)$ by $f(-t)$ and $A$ by $-A$, we may also assume that $f(t)$ is nondecreasing.

If $X$ is a Hermitian on $\mathcal{H}$, we define a sequence of numbers $\{\lambda_k(X)\}_{k=1}^{\infty}$,
$$
\lambda_k(X)=\sup_{\{{\cal F}\,:\,\dim{\cal F}=k\}}\,\inf_{\{h\in{\cal F}\,:\,\Vert h\Vert=1\}}\langle h, Xh\rangle
$$
where the supremum runs over $k$-dimensional subspaces. Note that $\{\lambda_k(X)\}_{k=1}^{\infty}$ is a non-increasing sequence
whose limit is the upper bound of the essential spectrum of $X$. We also define $\{\lambda_{-k}(X)\}_{k=1}^{\infty}$,
$$
\lambda_{-k}(X)=\sup_{\{{\cal F}\,:\,{\rm codim\,}{\cal F}=k-1\}}\,\inf_{\{h\in{\cal F}\,:\,\Vert h\Vert=1\}}\langle h, Xh\rangle.
$$
Then, $\{\lambda_{-k}(X)\}_{k=1}^{\infty}$ is a nondecreasing sequence whose limit is the lower bound of the essential spectrum of
$X$. The following fact (a) is obvious and fact (b) is easily checked.
\begin{itemize}
\item[(a)] If  $X\le Y$, then $\lambda_k(X)\le\lambda_k(Y)$ and $\lambda_{-k}(X)\le\lambda_{-k}(Y)$ for all $k=1,\cdots$.

\item[(b)] if $r>0$ and $X,\,Y$ are Hermitian, $\lambda_k(X)\le\lambda_k(Y)$ and $\lambda_{-k}(X)\le\lambda_{-k}(Y)$,for all
$k=1,\dots$, then $X\le UYU^* + rI$ for some unitary $U$.
\end{itemize}
These facts show that, given $r>0$, two Hermitians $X$, $Y$ with $X\le Y$, and a continuous nondecreasing function $\phi$, there
exists a unitary $U$ such that $\phi(X)\le U\phi(Y)U^*+rI$.

By fact (b) it suffices to show that
\begin{equation}\label{F-2.a}
\lambda_k(f(A_{\cal S}))\le\lambda_k(f(A)_{\cal S})
\end{equation}
and
\begin{equation}\label{F-2.b}
\lambda_{-k}(f(A_{\cal S}))\le\lambda_{-k}(f(A)_{\cal S})
\end{equation}
for all $k=1,\cdots$. Now, we prove \eqref{F-2.b}  and  distinguish two cases:
\vskip 5pt\noindent
1.\ $\lambda_{-k}(A_{\cal S})$ is an eigenvalue of $A_{\cal S}$. Then, for $1\le j\le k$, $\lambda_{-j}(f(A_{\cal S}))$ are
eigenvalues for $f(A_{\cal S})$. Consequently, there exists a subspace
${\cal F}\subset{\cal S}$, ${\rm codim}_{\cal S}\,{\cal F}=k-1$, such that
\begin{align*}
\lambda_{-k}(f(A_{\cal S})) =& \min_{\{h\in{\cal F}\,:\,\Vert h\Vert=1\}}\langle h, f(A_{\cal S})h\rangle \\
=& \min_{\{h\in{\cal F}\,:\,\Vert h\Vert=1\}}f(\langle h, A_{\cal S}h\rangle) \\
\le& \inf_{\{h\in{\cal F}\,:\,\Vert h\Vert=1\}}\langle h, f(A)h\rangle \le \lambda_{-k}(f(A)_{\cal S})
\end{align*}
where we have used that $f$ is non-decreasing and convex.

\noindent
2.\ $\lambda_{-k}(A_{\cal S})$ is not an eigenvalue of $A_{\cal S}$ (so, $\lambda_{-k}(A_{\cal S})$ is the lower bound of the
essential spectrum of $A_{\cal S}$).
 Fix $\varepsilon>0$ and choose  $\delta>0$
such that $|f(x)-f(y)|\le \varepsilon$ for all $x$, $y$ are in the convex hull of the spectrum of $A$ with $|x-y|\le\delta$. There
exists a subspace ${\cal F}\subset{\cal S}$, ${\rm codim}_{\cal S}\,{\cal F}=k-1$, such that
$$
\lambda_{-k}(A_{\cal S}) \le \inf_{\{h\in{\cal F}\,:\,\Vert h\Vert=1\}}\langle h, A_{\cal S}h\rangle +\delta
.
$$
Since $f$ is continuous nondecreasing we have $f(\lambda_{-k}(A_{\cal S}))=\lambda_{-k}(f(A_{\cal S}))$ so that, as $f$ is
nondecreasing,
$$
\lambda_{-k}(f(A_{\cal S})) \le f\left(\inf_{\{h\in{\cal F}\,:\,\Vert h\Vert=1\}}\langle h, A_{\cal S}h\rangle +\delta\right).
$$
Consequently,
$$
\lambda_{-k}(f(A_{\cal S})) \le \inf_{\{h\in{\cal F}\,:\,\Vert h\Vert=1\}}f(\langle h, A_{\cal S}h\rangle) + \varepsilon,
$$
 so, using the convexity of $f$ and the definition of $\lambda_{-k}(\cdot)$, we get
$$
\lambda_{-k}(f(A_{\cal S}))\le \lambda_{-k}(f(A)_{\cal S}) + \varepsilon.
$$
By letting $\varepsilon\longrightarrow 0$, the proof of \eqref{F-2.b} is complete.  The proof of \eqref{F-2.a} is similar.
Thus \eqref{F-2.13} is established.

\vskip 5pt

\section{A matrix subadditivity inequality}

\subsection{Sub/super-additivity inequalities via unitary orbits}

This section deals with some recent subadditive properties for concave functions, and similarly superadditive properties of convex
functions. The main result is:

\vskip 10pt
\begin{theorem}\label{T-3.1} Let $f(t)$ be a monotone concave function on $[0,\infty)$ with $f(0)\ge 0$ and let
$A,B\in\bM_n^+$. Then, for some unitaries $U,V\in\bM_n$,
$$
f(A+B)\le Uf(A)U^* + Vf(B)V^*.
$$
\end{theorem}

\vskip 5pt\noindent
Thus, the obvious scalar inequality $f(a+b)\le f(a)+f(b)$ can be extended to positive matrices $A$ and $B$ by considering element in
the unitary orbits of $f(A)$ and $f(B)$. This inequality via unitary orbits considerably improves the famous Rotfel'd trace
inequality \eqref{Rot1} for a non-negative concave function on the positive half-line,
and its symmetric norm version
\begin{equation}\label{subnorm}
\| f(A+B)\|  \le \|f(A)\| + \|f(B)\|
\end{equation}
for all $A,B\in \bM_n^+$ and all symmetric norms $\|\cdot\|$ on $\bM_n$.

Of course Theorem \ref{T-3.1} is equivalent to the next statement for  convex functions:

\vskip 10pt
\begin{cor}\label{cor-subadditivity} Let $g(t)$ be a monotone convex function on $[0,\infty)$ with $g(0)\le0$ and let
$A,B\in \bM_n^+$ . Then, for some unitaries $U,V\in\bM_n$,
\begin{equation}\label{supad}
g(A+B)\ge Ug(A)U^* + Vg(B)V^*.
\end{equation}
\end{cor}

\begin{proof} It suffices to prove the convex version, Corollary \ref{cor-subadditivity}. We may confine the proof to the case $g(0)=0$ as if
\eqref{supad} holds for a function $g(t)$ then it also holds for $g(t)-\alpha$ for any $\alpha>0$. This assumption combined with the
monotony of $g(t)$ entails that $g(t)$ has a constant sign $\varepsilon\in\{-1,1\}$, hence $g(t)=\varepsilon |g|(t)$.

We may also assume that $A+B$ is invertible. Then
 $$
 A=X(A+B)X^*
 \quad {\rm and} \quad
  B=Y(A+B)Y^*
 $$
where
 $X=A^{1/2}(A+B)^{-1/2}$ and $Y=B^{1/2}(A+B)^{-1/2}$ are
 contractions. For any $T\in \bM_n,$ $T^*T$ and $TT^*$ are unitarily
congruent. Hence, using Corollary \ref{C-2.5} we have  two unitary operators $U_0$ and $U$ such that
 \begin{align*}
 g(A) &= g(X(A+B)X^*)\\
&\leq U_0 Xg(A+B)X^* U_0^*\\
&= \varepsilon U^*(|g|(A+B))^{1/2}X^*X(|g|(A+B))^{1/2} U,
\end{align*}
so,
\begin{equation}
Ug(A)U^*\leq \varepsilon(|g|(A+B))^{1/2}X^*X(|g|(A+B))^{1/2}.
\end{equation}
Similarly there exists a unitary operator $V$ such that
\begin{equation}
Vg(B)V^*\leq\varepsilon (|g|(A+B))^{1/2}Y^*Y(|g|(A+B))^{1/2}.
\end{equation}
Adding (3.3) and (3.4) we get
$$
 Ug(A)U^*+Vg(B)V^*\leq g (A+B)
$$
since $X^*X+Y^*Y=I.$ \end{proof}

\vskip 10pt
The following corollary is matrix version of another obvious scalar inequality.
\vskip 10pt
 \begin{cor}\label{C-3.3}
 Let $f:[0,\infty)\to [0,\infty)$ be concave and let
$A,B\in\bM_n$ be Hermitian.
  Then, for some unitaries $U,V\in\bM_n$,
$$
 Uf(A)U^*-Vf(B)V^*\leq f (|A-B|).
$$
\end{cor}

\vskip 10pt
 \begin{proof} Note that
 $$
 A\leq |A-B|+B.
 $$
 Since $f(t)$ is non-decreasing  and concave there exists  unitaries $W,\,S,\,T$ such that
 $$
 Wf(A)W^*\leq f(|A-B|+B)\leq Sf(|A-B|)S^*+Tf(B)T^*.
 $$
Hence, we have
$$
 Uf(A)U^*-Vf(B)V^*\leq f (|A-B|)
$$
for some unitaries $U,\, V.$ \end{proof}

We can employ Theorem \ref{T-3.1} to get an elegant inequality for positive block-matrices,
 $$\begin{bmatrix} A &X \\
X^* &B\end{bmatrix}\in \bM_{n+m}^+, \qquad A\in\bM_n^+, \,  B\in\bM_m^+,
$$
which nicely extend \eqref{subnorm}. To this end we need an interesting
  decomposition lemma  for elements in $\bM_{n+m}^+$.

\begin{lemma}\label{L-3.4} For every matrix in  $\bM_{n+m}^+$ written in blocks, we have a decomposition
\begin{equation}\label{eqdecfund}
\begin{bmatrix} A &X \\
X^* &B\end{bmatrix} = U
\begin{bmatrix} A &0 \\
0 &0\end{bmatrix} U^* +
V\begin{bmatrix} 0 &0 \\
0 &B\end{bmatrix} V^*
\end{equation}
for some unitaries $U,\,V\in  \bM_{n+m}$.
\end{lemma}

\vskip 5pt
\begin{proof} To obtain this decomposition of the positive semi-definite block matrix, factorize it as a square of positive matrices,\begin{equation*}
\begin{bmatrix} A &X \\
X^* &B\end{bmatrix} =
\begin{bmatrix} C &Y \\
Y^* &D\end{bmatrix}
\begin{bmatrix} C &Y \\
Y^* &D\end{bmatrix}
\end{equation*}
and observe that it can be written as
\begin{equation*}
\begin{bmatrix} C &0 \\
Y^* &0\end{bmatrix}
\begin{bmatrix} C &Y \\
0 &0\end{bmatrix} +
\begin{bmatrix} 0 &Y \\
 0&D\end{bmatrix}
\begin{bmatrix} 0 &0 \\
Y^* &D\end{bmatrix} = T^*T + S^*S.
\end{equation*}
Then, use the fact that $T^*T$ and $S^*S$ are unitarily congruent to
$$
TT^*= \begin{bmatrix} A &0 \\
0 &0\end{bmatrix}
\quad \mathrm{and}
\quad
SS^*=\begin{bmatrix} 0 &0 \\
0 &B\end{bmatrix},
$$
completing the proof of the decomposition.
\end{proof}

Combined with Theorem \ref{T-3.1}, the lemma yields a norm inequality for block-matrices. A symmetric norm on $\bM_{n+m}$ induces a
symmetric norm on $\bM_{n}$, via $\|A\|=\|A \oplus 0\|$.

\vskip 10pt
\begin{cor}\label{corlee} Let $f(t)$ be a non-negative concave
function on $[0,\infty)$. Then, given an arbitrary partitioned
positive semi-definite matrix,
$$
\left\| \,f\left( \begin{bmatrix} A &X \\ X^* &B\end{bmatrix}\right)
 \right\|
\le \left\| f(A) \right\| +  \left\| f(B) \right\|
$$
for all symmetric norms.
\end{cor}

\vskip 10pt
\begin{proof}
From \eqref{eqdecfund} and Theorem \ref{T-3.1}, we have
\begin{equation*}
f\left(\begin{bmatrix} A &X \\
X^* &B\end{bmatrix}\right) = U
\begin{bmatrix} f(A) &0 \\
0 &f(0)I\end{bmatrix} U^* +
V\begin{bmatrix} f(0)I &0 \\
0 &f(B)\end{bmatrix} V^*
\end{equation*}
for some unitaries $U,\,V\in  \bM_{n+m}$. The result then follows from the simple fact
that symmetric norms are nondecreasing functions of the singular values.
\end{proof}

\vskip 10pt
 Applied to  $X=A^{1/2}B^{1/2}$, this result yields the  Rotfel'd type inequalities \eqref{Rot1}-\eqref{subnorm}, indeed,
$$
\begin{bmatrix} A &X \\
X^* &B\end{bmatrix}=\begin{bmatrix} A^{1/2} &0 \\
B^{1/2} &0\end{bmatrix}\begin{bmatrix} A^{1/2} &B^{1/2} \\
0 &0\end{bmatrix}
$$
is then  unitarily equivalent to $(A+B)\oplus 0$.
In case of the trace norm, the above result may be restated as a trace inequality without any non-negative assumption: {\it For all
concave functions $f(t)$ on the positive half-line and all positive block-matrices,}
$$
{\mathrm{Tr\,}} f\left( \begin{bmatrix} A &X \\ X^* &B\end{bmatrix}\right)
\le {\mathrm{Tr\,}} f(A) +  {\mathrm{Tr\,}} f(B).
$$
The case of $f(t)=\log t$ then gives Fisher's inequality,
$$
\det \begin{bmatrix} A &X \\ X^* &B\end{bmatrix} \le \det A \det B.
$$

Theorem \ref{T-3.1} may be used to extend another classical (superadditive and concavity) property of the determinant, the Minkowski
inequality \eqref{Minkowski}. We have the following extension:

\begin{cor}\label{C-3.6}  If $g:[0,\infty)\to [0,\infty)$ is a convex function, $g(0)=0$,
and  $A,\,B\in \bM_n^+$, then,
\begin{equation*}
{\det}^{1/n} g(A+B) \ge {\det}^{1/n} g(A) + {\det}^{1/n} g(B).
\end{equation*}
\end{cor}

\vskip 10pt
As another example of combination of Theorem \ref{T-3.1} and \eqref{eqdecfund}, we have:

\vskip 10pt
\begin{cor}\label{C-3.7} Let $f:[0,\infty)\to [0,\infty)$ be concave and let $A=(a_{i,j})$
be a positive semi-definite matrix in $\bM_n$. Then, for some
rank one ortho-projections $\{E_i\}_{i=1}^n$ in $\bM_n$,
$$
f(A)\le \sum_{i=1}^n f(a_{i,i})E_i.
$$
\end{cor}

\begin{proof} By a limit argument, we may assume that $A$ is invertible, and hence we may also assume that
$f(0)=0$,
indeed if the spectrum of $A$ lies in an interval $[r,s]$, $r>0$, we may replace $f(t)$
 by  any concave function on
$[0,\infty)$ such that $\tilde{f}(0)=0$ and $\tilde{f}(t)=f(t)$ for
$t\in[r,s]$.   By a repetition of \eqref{eqdecfund} we have
$$
A=\sum_{i=1}^n a_{i,i} F_i
$$
for some
rank one ortho-projections $\{F_i\}_{i=1}^n$ in $\bM_n$. An application of Theorem \ref{T-3.1} yields
$$
f(A) \le \sum_{i=1}^n U_i f(a_{i,i} F_i)U_i^*
$$
for some unitary operators $\{U_i\}_{i=1}^n$. Since $f(0)=0$, for each $i$, $U_i f(a_{i,i} F_i)U_i^*=f(a_{i,i})E_i$ for some rank one
projection $E_i$.
\end{proof}

\vskip 5pt
Corollary \ref{C-3.7} refines the standard majorization inequality relating a positive semi-definite $n$-by-$n$ matrix and its diagonal part,
$${\mathrm{Tr\,}} f(A) \le \sum_{i=1}^d f(a_{i,i}). $$

\subsection{Comments and references}

\vskip 5pt
\begin{remark} Theorem \ref{T-3.1}, Corollaries \ref{cor-subadditivity} and \ref{cor-subadditivity} are from \cite{AB}. In case of positive operators acting on an infinite
dimensional, separable Hilbert space, we have a version of Theorem \ref{T-3.1} with an additional $r$I term in the RHS, as in (2.13).
\end{remark}

\vskip 5pt
\begin{remark} The decomposition of a positive block-matrix in Lemma \ref{L-3.4} is due to the authors. It is used in \cite{Lee1} to obtain
the norm inequality stated in Corollary \ref{corlee}. The next two Corollaries \ref{C-3.6} and \ref{C-3.7} are new, though already announced in \cite{BH1}.
\end{remark}

\vskip 5pt
\begin{remark}
The concavity requirement on $f(t)$ in Rotfel'd inequality \eqref{Rot1} and hence in Theorem \ref{T-3.1} cannot be relaxed to a mere
superadditivity assumption; indeed take for $s,t>0$,
$$
A=\frac{1}{2}
\begin{bmatrix} s&\sqrt{st} \\
\sqrt{st}&t
\end{bmatrix},
\qquad B=\frac{1}{2}
\begin{bmatrix} s&-\sqrt{st} \\
-\sqrt{st}&t
\end{bmatrix},
$$
and observe that the trace inequality $
\mathrm{Tr\,}f(A+B) \le \mathrm{Tr\,}f(A)+f(B)
$
 combined with $f(0)=0$ means that $f(t)$ is concave. 
\end{remark}

\vskip 5pt
\begin{remark} There exists a norm version of Rotfel'd inequality which considerably improves \eqref{subnorm}. {\it If
$f:[0,\infty)\to[0,\infty)$ is concave and $A, B\in\bM_n^+$, then
$$
\| f(A+B) \| \le  \|f(A) +f(B) \|
$$
for all symmetric norms.} The case of operator concave functions is given in \cite{AZ} and the general case is established in
\cite{BU}, see also \cite{BL-JOT} for further results. Concerning differences, the following inequality holds
$$
\| f(A)-f(B) \| \le \| f(|A-B|) \|
$$
for all symmetric norms, $A,B\in\bM_n^+$, and operator monotone functions $f:[0,\infty)\to[0,\infty)$. This is a famous result of
Ando \cite{Ando}. Here the operator monotonicity assumption is essential, see \cite{AuAu} for some counterexamples. A very interesting
paper by Mathias \cite{M} gives a direct proof, without using the integral representation of operator monotone functions.
\end{remark}

\vskip 5pt
\begin{remark} There exists also some subaditivity results involving  convex functions \cite{BH1}. For instance:
{\it
Let $g(t)=\sum_{k=0}^m a_kt^k$ be a polynomial of degree $m$ with all non-negative
coefficients. Then, for all positive operators $A,\,B$ and all symmetric norms,}
\begin{equation*}
\| g(A+B) \|^{1/m} \le \|g(A) \|^{1/m} + \| g(B) \|^{1/m}.
\end{equation*}
\end{remark}

\vskip 5pt
\begin{remark} It is not known wether the monotonicity assumption in Theorem 3.1 can be deleted, i.e.,  wether the theorem holds for
all concave functions $f(t)$ on $[0,\infty)$ with $f(0) \ge 0$.
\end{remark}

\section{Around this article }

We will see several applications of   this article in the next chapters. Here we recall some results from \cite{BL-Clarkson} improving the Clarkson-McCarthy inequalities
 for $p\ge 2$. These inequalities  show that the unit ball for the Schatten $p$-norm is uniformly convex, 
\begin{equation*}\label{Clark}
 \left\| \frac{ A+B}{2}\right\|_p^p +   \left\| \frac{ A-B}{2}\right\|_p^p
\le 
\frac{\| A\|_p^p +\| B\|_p^p}{2}, \quad p\ge 2,
\end{equation*}
i.e.,
\begin{equation}\label{Clark'}
{\mathrm{Tr\,}} \left| \frac{ A+B}{2}\right|^p +   {\mathrm{Tr\,}} \left| \frac{ A-B}{2}\right|^p
\le 
\frac{{\mathrm{Tr\,}} | A|^p +{\mathrm{Tr\,}} | B|^p}{2}, \quad p\ge 2.
\end{equation}
Thus, if $\|A\|_p=\|B\|_p=1$ and $\| A-B\|_p = \varepsilon$, 
$$
\left\| \frac{ A+B}{2}\right\|_p \le \left(1-(\varepsilon/2)^p\right)^{1/p}, \quad  p\ge 2,
$$
which estimates the uniform convexity modulus of the Schatten $p$-classes for $p\ge 2$. 

Some nice extensions of these inequalities have been given by Bhatia-Holbrook \cite{Bh-H} 
 and Hirzallah-Kittaneh \cite{Hir-Kit}. The next series of corollaries provide several eigenvalue inequalities completing those of  Bhatia-Holbrook and Hirzallah-Kittaneh. First, we state the following improvement of \eqref{Clark'}.

\vskip 5pt
\begin{theorem}\label{thMcCarthy} Let $A,B\in\bM_n$ and $p>2$. Then there exists two unitarie $U,V\in\bM_n$ such that
$$
U\left| \frac{A+B}{2}\right|^pU^* +V\left| \frac{A-B}{2}\right|^pV^*\le \frac{|A|^p+|B|^p}{2}.
$$
\end{theorem}

\vskip 5pt
\begin{proof} Note that
$$
\left| \frac{A+B}{2}\right|^p= \left(\frac{|A|^2+|B|^2 +A^*B +B^*A}{4}\right)^{p/2}
$$
and 
$$
\left| \frac{A-B}{2}\right|^p= \left(\frac{|A|^2+|B|^2 -(A^*B +B^*A)}{4}\right)^{p/2}.
$$
Now, recall 
Corollary \ref{cor-subadditivity} : Given two positive matrices $X,Y$ and a monotone convex function $g(t$) defined on $[0,\infty)$
such that $g(0)\le 0$, we have
\begin{equation}\label{eq-key1}
g(X+Y) \ge U_0g(X)U_0^* + V_0g(Y)V_0^*
\end{equation}
for some pair of unitary matrices $U_0$ and $V_0$. Applying this to $g(t)=t^{p/2}$,
$$
X=\frac{|A|^2+|B|^2 +A^*B +B^*A}{4}
$$
and
$$
Y=\frac{|A|^2+|B|^2 -(A^*B +B^*A)}{4},
$$
 we obtain
\begin{equation}\label{e1}
\left(\frac{|A|^2+|B|^2}{2}\right)^{p/2} \ge U_0\left| \frac{A+B}{2}\right|^pU_0^* +V_0\left| \frac{A-B}{2}\right|^pV_0^*.
\end{equation}
Next, recall Corollary \ref{C-2.2} : Given two positive matrices $X,Y$ and a monotone convex function $g(t$) defined on
$[0,\infty)$ , we have
\begin{equation}\label{eq-key2}
\frac{g(X)+g(Y)}{2} \ge Wg\left(\frac{X+Y}{2}\right)W^*
\end{equation}
for some unitary matrix $W$. Applying this to $g(t)=t^{p/2}$, $X=|A|^2$ and $Y=|B|^2$, we get
\begin{equation}\label{e2}
 \frac{|A|^p+|B|^p}{2}\ge W\left(\frac{|A|^2+|B|^2}{2}\right)^{p/2}W^*.
\end{equation}
Combining \eqref{e1} and \eqref{e2} completes the proof with $U=WU_0$ and $V=WV_0$.
\end{proof}

\vskip 5pt
\begin{cor}\label{cor1}
 Let $A,B\in\bM_n$ and $p>2$. Then, for all $k=1,2,\ldots,n$,
$$
\sum_{j=1}^k \lambda_j^{\uparrow}\left(\frac{|A|^p+|B|^p}{2}\right) \ge 
\sum_{j=1}^k \lambda_j^{\uparrow}\left(\left|\frac{A+B}{2}\right|^p\right)
+
\sum_{j=1}^k \lambda_j^{\uparrow}\left(\left|\frac{A-B}{2}\right|^p\right).
$$
\end{cor}

\vskip 5pt
\begin{cor}\label{cor2}
 Let $A,B\in\bM_n$ and $p>2$. Then, for all $k=1,2,\ldots,n$,
$$
\left\{\prod_{j=1}^k \lambda_j^{\uparrow}\left(\frac{|A|^p+|B|^p}{2}\right)\right\}^{1/k} \ge 
\left\{\prod_{j=1}^k \lambda_j^{\uparrow}\left(\left|\frac{A+B}{2}\right|^p\right)\right\}^{1/k}
+
\left\{\prod_{j=1}^k \lambda_j^{\uparrow}\left(\left|\frac{A-B}{2}\right|^p\right)\right\}^{1/k}.
$$
\end{cor}

\vskip 5pt
Here $\lambda_{1}^{\uparrow}(X )\le \lambda_{2}^{\uparrow}(X )\le \cdots\le \lambda_{n}^{\uparrow}(X ) $ stand for the eigenvalues of
$X\in\bM_n^+$ arranged in the nondecresaing order.
These two corollaries follow from Theorem \ref{thMcCarthy} and the fact that the functionals on $\bM_n^+$
$$X\mapsto \sum_{j=1}^k \lambda_j^{\uparrow}(X)$$
and
$$X\mapsto\left\{\prod_{j=1}^k  \lambda_j^{\uparrow}(X)\right\}^{1/k}$$
are two basic examples of symmetric anti-norms, see \cite{BH1}, \cite{BH2}.

The next corollary follows from Theorem \ref{thMcCarthy} combined with a classical inequality of Weyl for the eigenvalues of the sum of two
Hermitian matrices.

\vskip 5pt
\begin{cor}\label{cor3}
 Let $A,B\in\bM_n$ and $p>2$. Then, for all $j,k\in\{0,\ldots,n-1\}$ such that $j+k+1\le n$,
$$
 \lambda_{j+1}^{\downarrow}\left(\frac{|A|^p+|B|^p}{2}\right) \ge 
 \lambda_{j+k+1}^{\downarrow}\left(\left|\frac{A+B}{2}\right|^p\right)
+
 \lambda_{k+1}^{\uparrow}\left(\left|\frac{A-B}{2}\right|^p\right).
$$
\end{cor}

\vskip 5pt
For a monotone concave function $g(t$) defined on $[0,\infty)$ such that $g(0)\ge 0$, the inequalities \eqref{eq-key1} and \eqref{eq-key2}
are reversed.
Applying this to $g(t)=t^{q/2}$, $2>q>0$, the same proof than that of Theorem \ref{thMcCarthy} gives the following statement.

\vskip 5pt
\begin{theorem}\label{th2} Let $A,B\in\bM_n$ and $2>q>0$. Then, for some unitaries $U,V\in\bM_n$,
$$
U\left| \frac{A+B}{2}\right|^qU^* +V\left| \frac{A-B}{2}\right|^qV^*\ge \frac{|A|^q+|B|^q}{2}.
$$
\end{theorem}

\vskip 5pt
By using Weyl's inequality, this theorem yields an interesting eigenvalue estimate.

\vskip 5pt
\begin{cor}\label{cor4}
 Let $A,B\in\bM_n$ and $2>q>0$. Then, for all $j,k\in\{0,\ldots,n-1\}$ such that $j+k+1\le n$,
$$
 \lambda_{j+k+1}^{\downarrow}\left(\frac{|A|^q+|B|^q}{2}\right) \le 
 \lambda_{j+1}^{\downarrow}\left(\left|\frac{A+B}{2}\right|^q\right)
+
 \lambda_{k+1}^{\downarrow}\left(\left|\frac{A-B}{2}\right|^q\right).
$$
\end{cor}

\section{References of Chapter 3}

{\small
\begin{itemize}

\item[[4\!\!\!]]  T.\ Ando, Comparison of norms $ \vert\vert\vert f(A)-f(B) \vert\vert\vert $ and $\| f(\vert A-B\vert )\vert\vert\vert
$. {\it Math. Z.}\ {\bf197} (1988), no. 3, 403-409.

\item[[10\!\!\!]] K.\ Audenaert and  J.\ S.\ Aujla, On Ando's inequalities for convex and concave functions,  arXiv:0704.0099v1.

\item[[11\!\!\!]]  J.\ S.\ Aujla
and F.\ C.\ Silva, Weak majorization inequalities and convex
functions, \textit{Linear Algebra Appl.} \textbf{369} (2003), 217-233.

\item[[12\!\!\!]]   J.S.\ Aujla and J.-C.\ Bourin, Eigenvalue inequalities for convex and log-convex functions,
\textit{Linear Algebra Appl.} 424 (2007), 25--35.

 \item[[13\!\!\!]]  R.\ Bhatia, Matrix Analysis, Gradutate Texts in Mathematics, Springer, New-York, 1996.

\item[[16\!\!\!]]  R.\ Bhatia and J.\ Holbrook, On the Clarkson-McCarthy inequalities, {\it Math. Ann.}\ 281 (1988), no.\ 1, 7--12.

\item[[22\!\!\!]]  J.-C.\ Bourin, Convexity or concavity inequalities for Hermitian operators. \textit{Math. Inequal. Appl.} 7
(2004), no.\ 4, 607–620.

\item[[23\!\!\!]]  J.-C. Bourin, Hermitian operators and convex functions,
\textit{J. Inequal. Pure Appl. Math.}\ 6 (2005), Article 139, 6 pp.

\item[[24\!\!\!]]  J.-C.\ Bourin, A concavity inequality for symmetric norms, \textit{ Linear
Algebra Appl.}\  413 (2006), 212-217.

\item[[28\!\!\!]] J.-C. Bourin and F. Hiai,
Norm and anti-norm inequalities for positive semi-definite matrices,
{\it Internat. J. Math.} \textbf{63} (2011), 1121-1138.

\item[[29\!\!\!]]  J.-C. Bourin and F. Hiai, Jensen and Minkowski inequalities for operator means and anti-norms. Linear Algebra Appl. 456 (2014), 22–53.

\item[[30\!\!\!]]  J.-C.\ Bourin and E.-Y.\ Lee, Concave functions of
positive operators, sums, and congruences, \textit{J.\ Operator Theory} \textbf{63}
(2010), 151--157.

\item[[31\!\!\!]]  J.-C.\ Bourin and E.-Y.\ Lee, Unitary orbits of Hermitian operators with convex or concave functions, {\it Bull.\ Lond.\
Math.\ Soc.}\ 44 (2012), no.\ 6, 1085--1102.

\item[[40\!\!\!]] J.-C.\ Bourin and E.-Y. Lee,
Clarkson-McCarthy inequalities with unitary and isometry orbits, {\it  Linear Algebra Appl.}\ 601 (2020), 170--179.

\item[[48\!\!\!]] J.-C.\ Bourin and E.\ Ricard,  An asymmetric Kadison's inequality, \textit{Linear Algebra Appl.}\
433 (2010) 499--510.

\item[[49\!\!\!]] J.-C.\ Bourin and M.\ Uchiyama, A matrix subadditivity inequality for $f(A+B)$ and $f(A)+f(B)$, \textit{Linear Algebra
Appl.}\
423 (2007), 512--518.

\item[[50\!\!\!]] L.\ G.\ Brown and H.\ Kosaki, Jensen's inequality in semi-finite
von Neuman algebras, \textit{J.\ Operator theory } 23
(1990), 3--19.

\item[[52\!\!\!]]  M.-D. Choi, A Schwarz inequality for positive linear maps on
$C^*$-algebras, \textit{Illinois J. Math.} \textbf{18} (1974), 565--574.

\item[[56\!\!\!]] C.\ Davis, \textit{A Schwarz inequality for convex operator functions}, Proc. Amer. Math. Soc. {\bf 8} (1957), 42-44.

\item[[61\!\!\!]]  F.\ Hansen,  An operator inequality, {\it Math. Ann}.\ 246 (1979/80), no. 3, 249–250.

\item[[62\!\!\!]]  F. Hansen and G. K. Pedersen, Jensen's inequality for operators and
L\"owner's theorem, \textit{Math. Ann.} 258 (1982), 229--241.

\item[[63\!\!\!]] F. Hansen and G. K. Pedersen, Jensen's operator inequality, {\it Bull.\ London Math.\ Soc.} 35 (2003), no. 4,
553--564.

\item[[64\!\!\!]]  F. Hiai, Matrix Analysis: Matrix Monotone Functions, Matrix Means, and
Majorization (GSIS selected lectures),
\textit{Interdisciplinary Information Sciences} 16 (2010), 139--248.

\item[[69\!\!\!]] O.\ Hirzallah and F.\ Kittaneh,  Non-commutative Clarkson inequalities for unitarily invariant norms, {\it Pacific J.\ Math.}\ 202 (2002), no.\ 2, 363--369.

\item[[75\!\!\!]] E.-Y. Lee, Extension of Rotfel’d Theorem,  {\it Linear Algebra Appl.}\
435 (2010), 735--741.

\item[[76\!\!\!]] E.-Y. Lee, How to compare the absolute values of operator sums and the sums of absolute values ?, to appear in {\it
Operator and Matrices}.

\item[[79\!\!\!]] R.\ Mathias, Concavity of monotone matrix functions of finite order, {\it Linear and Multilinear Algebra} {\bf 27}
(1990), no. 2, 129-138.

\item[[88\!\!\!]] S. Ju. Rotfel'd, The singular values of a
sum of completely continuous operators,  \textit{ Topics in
Mathematical Physics, Consultants Bureau}, Vol. \textbf{3} (1969)
73-78.

\item[[90\!\!\!]] W.\ F.\ Stinespring, Positive functions on $C^*$-algebras, {\it Proc.\ Amer.\ Math.\ Soc.}\ {\bf 6}, (1955). 211-216.

\item[[92\!\!\!]] R.-C.\ Thompson, Convex and concave functions of singular values of matrix sums, {\it Pacific J.\ Math.}\ 66 (1976),
285--290.

\end{itemize}


\chapter{Around Hermite-Hadamard}

{\color{blue}{\Large {\bf Matrix inequalities and majorizations around Hermite-Hadamard's inequality} \large{\cite{BL-HH}}}}

\vskip 10pt\noindent
{\bf Abstract.} We study the classical Hermite-Hadamard inequality in the matrix setting.
This leads to a number of interesting matrix inequalities such as the Schatten $p$-norm estimates
$$
\left(\|A^q\|_p^p + \|B^q\|_p^p\right)^{1/p} \le \|(xA+(1-x)B))^q\|_p+ \|(1-x)A+xB)^q\|_p
$$
for all  positive (semidefinite) $n\times n$ matrices $A,B$ and $0<q,x<1$. A related   decomposition, with the assumption $X^*X+Y^*Y=XX^*+YY^*=I$, is
$$
(X^*AX+Y^*BY)\oplus (Y^*AY+X^*BX) =\frac{1}{2n}\sum_{k=1}^{2n} U_k (A\oplus B)U_k^*
$$
for some family   of $2n\times 2n$ unitary matrices $U_k$.  This is a majorization which is obtained by using the Hansen-Pedersen trace inequality.

{\small 
\vskip 5pt\noindent
{\it Keywords.}     Positive definite matrices, block matrices, convex functions, matrix inequalities. 
\vskip 5pt\noindent
{\it 2010 mathematics subject classification.} 15A18, 15A60,  47A30.

}

\section{Elementary scalar inequalities }

Extending basic scalar inequalities, for instance $|a+b|\le |a|+|b|$, to matrices lies at the very heart of matrix analysis. Here, we are interested in the elementary inequality which supports the Hermite-Hadamard inequality. 
This classical  theorem can be stated  as follows:

\vskip 5pt
\begin{prop}\label{propHH1}  Let $f(t)$ be a convex function defined on the interval $[a,b]$. Then,
$$
f\left(\frac{a+b}{2}\right)\ \le \int_0^1 f((1-x)a+xb)\, {\mathrm{d}}x\le \frac{f(a)+f(b)}{2}.
$$
\end{prop}

In spite of its simplicity, the Hermite-Hadamard inequality is a powerful tool for deriving a number of important inequalities; see the nice paper \cite{Ni} and references therein.

The first inequality immediately follows from the convexity assumption
\begin{equation}\label{eqbasic}
f\left(\frac{a+b}{2}\right) \le \frac{f((1-x)a+xb) +f(xa +(1-x)b)}{2} 
\end{equation}
The second inequality is slightly more subtle; it follows from the extremal property
\begin{equation}\label{eqext1}
 f((1-x)a+xb) +f(xa +(1-x)b) \le f(a) + f(b)
\end{equation}
which requires   the convexity assumption twice.
This  is the key for Proposition \ref{propHH1} and it has a clear geometric interpretation;  \eqref{eqext1} is equivalent to the increasingness of
$$
\varphi(t) := f(m+t) +f(m-t)
$$
with $m=(a+b)/2$ and $t\in[0,b-m]$. In fact,  if we  assume that $f(t)$ is $C^2$ and observe that for $t\in[0,b-m]$,
$$
\varphi'(t)=f'(m+t)-f'(m-t) =\int_{m-t}^{m+t} f''(s)  \,{\mathrm{d}}s,
$$
we can estimate  $\varphi'(t)$ with $f''(s)\ge 0$.

The extremal property \eqref{eqext1} of  $f(t)$  says that for four points in $[a,b]$,
\begin{equation}\label{eqext2}
p\le s\le t\le q, \ p+q=s+t \Rightarrow f(s)+f(t) \le f(p)+f(q).
\end{equation}

Let us see now what can be said for matrices.
Important matrix versions of  \eqref{eqbasic} are well-known. Let $\bM_n$ denote the space of $n\times n$ matrices and   $\bM_n^{s.a}$ its self-adjoint (Hermitian) part with the usual order $\le $ induced by the positive semidefinite cone $\bM_n^+$. We recall \cite[Corollary 2.2]{BL-London}.

\vskip 5pt
\begin{theorem} \label{thB2005} Let  $A,B\in \bM_n^{s.a}$ with spectra in $[a,b]$ and let   $f(t)$ be a convex function on $[a,b]$.  Then, for some unitaries $U,\,V\in\bM_n$,
\begin{equation*}
f\left(\frac{A+B}{2}\right)\le \frac{1}{2}\left\{U\frac{f(A)+f(B)}{2}U^*+V\frac{f(A)+f(B)}{2}V^*\right\}.
\end{equation*}
If furthermore $f(t)$ is monotone, then we can take $U=V$.
\end{theorem}

\vskip 5pt
Theorem \ref{thB2005} is a major improvement of the classical trace inequality of von Neumann (around 1920),
$$
{\mathrm{Tr\,}}f\left(\frac{A+B}{2}\right) \le {\mathrm{Tr\,}}\frac{f(A)+f(B)}{2}
$$
which entails the following trivial extension of Proposition \ref{propHH1}.

\vskip 5pt
\begin{prop}\label{propHH2}  Let $f(t)$ be a convex function defined on the interval $[a,b]$ and let $A,B\in\bM_n^{s.a}$ with spectra in $[a,b]$. Then,
$$
 {\mathrm{Tr\,}}f\left(\frac{A+B}{2}\right)\ \le  {\mathrm{Tr\,}}\int_0^1 f((1-x)A+xB)\, {\mathrm{d}}x\le  {\mathrm{Tr\,}}\frac{f(A)+f(B)}{2}.
$$
\end{prop}

What about matrix versions of the equivalent scalar inequalities \eqref{eqext1} and \eqref{eqext2}? There is no hope for \eqref{eqext2} : in general the trace inequality
$$
{\mathrm{Tr}\,} f(P) + {\mathrm{Tr}\,} f(Q)  \le {\mathrm{Tr}\,} f(S) + {\mathrm{Tr\,}} f(T) 
$$
does not hold for all Hermitian matrices $P,Q,S,T$ with spectra in $[a,b]$ and such that
$$
P\le S\le T \le Q, \quad P+Q=S+T.
$$

In the matrix setting \eqref{eqext1} and \eqref{eqext2} are not equivalent. This paper aims to establish two  matrix versions of the extremal inequality \eqref{eqext1}.  Doing so, we will obtain several new matrix inequalities. 

 For an operator convex functions $h(t)$ on $[a,b]$, and $A,B\in\bM_n^{s.a}$ with spectra in this interval, the matrix version of \eqref{eqext1} (as well as Proposition \ref{propHH1}) obvioulsy holds,
\begin{equation}\label{eqext3}
 h((1-x)A+xB) +h(xA +(1-x)B) \le h(A) + h(B)
\end{equation}
for any $0<x<1$. Our results hold for much more general convex/concave functions and have  applications to eigenvalue inequalities that cannot be derived from
\eqref{eqext3}, even in the case of the simplest operator convex/concave function $h(t)=t$. 

 We often use a crucial assumption: our functions are defined on the positive half-line and we deal with positive semidefinite matrices. In the matrix setting, the interval of definition of a function  may be quite important; for instance the class of operator monotone functions on the whole real line reduces to affine functions.

\section{The extremal property for matrices}

 For  concave functions, the inequality \eqref{eqext1} is reversed. Here is the matrix version.  An isometry 
$U\in\bM_{2n,n}$ means a $2n\times n$ matrix such that $U^*U=I$, the identity of $\bM_n$.

\vskip 5pt
\begin{theorem}\label{th-HH1}  Let $A, B\in\bM_{n}^+$, let $0<x<1$, and let $f(t)$ be a monotone concave function  on $[0,\infty)$ with $f(0)\ge 0$. Then, for some isometry matrices $U,V\in\bM_{2n,n}$,
$$
f(A)\oplus f(B) \le Uf((1-x)A+xB)U^* + Vf(xA+(1-x)B)V^*.
$$
\end{theorem}

\vskip 5pt
\begin{proof} Let $x=\sin^2 \theta$, $1-x=\cos^2\theta$, consider the unitary Hermitian matrix
$$
R:=\begin{bmatrix}
\sqrt{1-x} I& \sqrt{x}I\\  \sqrt{x}I& -\sqrt{1-x} I
\end{bmatrix}
$$
and note the unitary congruence
\begin{equation}\label{eq-cong}
R
\begin{bmatrix}
A & 0 \\ 0 & B
\end{bmatrix}R=
\begin{bmatrix} (1-x)A+xB
 & \star \\ \star &  xA+(1-x)B
\end{bmatrix}
\end{equation}
where the stars hold for unspecified entries. 

Now, recall the decomposition \cite[Lemma 3.4]{BL-London} : Given any positive semidefinite matrix $\begin{bmatrix}
C & X \\ X^* & D
\end{bmatrix}$ partitioned in four blocks in $\bM_n$, we have
$$
\begin{bmatrix}
C & X \\ X^* & D
\end{bmatrix}
=
U_0\begin{bmatrix}
C & 0 \\ 0 & 0
\end{bmatrix}
U_0^*+
V_0
\begin{bmatrix}
0 & 0 \\ 0 & D
\end{bmatrix}
V_0^*
$$
for some unitary matrices $U_0,V_0\in\bM_{2n}$. Applying this to \eqref{eq-cong}, we obtain
\begin{equation}\label{eq-dec} 
\begin{bmatrix}
A & 0 \\ 0 & B\end{bmatrix} = U_1\begin{bmatrix}
(1-x)A+xB & 0 \\ 0 & 0
\end{bmatrix}
U_1^*+
V_1
\begin{bmatrix}
0 & 0 \\ 0 & xA+(1-x)B
\end{bmatrix}
V_1^*
\end{equation}
for two unitary matrices $U_1,V_1\in\bM_{2n}$.

Next, recall the subadditivity inequality \cite[Theorem 3.4]{BL-London} : Given any pair of positive semidefinite matrices $S,T\in\bM_d$, we have
$$
f(S+T) \le U_2f(S)U_2^* + V_2f(T)V_2^*
$$
for two unitary matrices $U_2,V_2\in\bM_{d}$. Applying this to \eqref{eq-dec} yields
$$
\begin{bmatrix}
f(A) & 0 \\ 0 & f(B)\end{bmatrix} \le U\begin{bmatrix}
f((1-x)A+xB) & 0 \\ 0 & f(0)I
\end{bmatrix}
U^*+
V
\begin{bmatrix}
f(0)I & 0 \\ 0 & f(xA+(1-x)B)
\end{bmatrix}
V^*
$$
for two unitary matrices $U,V\in\bM_{2n}$. This proves the theorem when $f(0)=0$.

 To derive the general case,  we may assume that $f(t)$ is continuous (a concave function on $[0,\infty)$ might be discontinuous at $0$). Indeed, it suffices to consider the values of $f(t)$  on a finite set, the union of the spectra of the four matrices $A$, $B$, $(1-x)A+xB$ and $xA+(1-x)B$. Hence we may replace $f(t)$ by a piecewise affine monotone concave function $h(t)$ with $h(0)\ge 0$. Now, since $h(t)$ is continuous, a limit argument allows us to suppose that $A$ and $B$ are invertible. Therefore, letting $\lambda_n^{\downarrow}(Z)$ denote the smallest eigenvalue of $Z\in\bM_n^{s.a}$,
$$
r:=\min\{\lambda_n^{\downarrow}(A), \lambda_n^{\downarrow}(B)\}>0.
$$
We may then replace $h(t)$ by $h_r(t)$ defined as $h_r(t):=h(t)$ for $t\ge r$, $h_r(0):=0$ and $h_r(s):=h(r)\frac{s}{r}$ for $0\le s\le r$. The function $h_r(t)$ is monotone concave on $[0,\infty)$ and vanishes at $0$, thus the case $f(0)=0$ entails the general case.
\end{proof}

\vskip 5pt
Let  $\lambda_j^{\downarrow}(Z)$, $j=1,2,\ldots n$, denote the  eigenvalues of $Z\in\bM_n^{s.a}$ arranged in the nonincreasing order.

\vskip 5pt
\begin{cor}\label{coreigen1}  Let $A, B\in\bM_{n}^+$, let $0<x<1$, and let $f(t)$ be a nonnegative concave function  on $[0,\infty)$. Then, for  $j=0,1,\ldots, n-1$,
$$
\lambda_{1+2j}^{\downarrow}\left(f(A\oplus B)\right) \le \lambda_{1+j}^{\downarrow}\left(f(xA+(1-x)B) \right) + \lambda_{1+j}^{\downarrow}\left(f((1-x)A+xB)\right)
$$
and
$$
\lambda_{1+j}^{\downarrow}\left\{\left(f(xA+(1-x)B) \right) + \left(f((1-x)A+xB)\right)\right\}
 \le 2\lambda_{1+j}^{\downarrow}\left(f\left(\frac{A+B}{2}\right) \right).
$$
\end{cor}

\vskip 5pt
\begin{proof} The first inequality is a straighforward consequence of Theorem \ref{th-HH1} combined with the  inequalities of Weyl \cite[p.\ 62]{Bh} :
For all $S,T\in\bM_d^{s.a}$ and $j,k\in\{0,\ldots,d-1\}$ such that $j+k+1\le d$,
$$
\lambda_{1+j+k}^{\downarrow}(S+T) \le 
\lambda_{1+j}^{\downarrow}(S) + 
\lambda_{1+k}^{\downarrow}(T).
$$
The second inequality is not new; it follows from Theorem \ref{th-convex}.
\end{proof}

\vskip 5pt
\begin{cor}\label{corSchatt1}  Let $A, B\in\bM_{n}^+$, let $0<x<1$, and let $f(t)$ be nonnegative concave function  on $[0,\infty)$. Then, for all $p\ge 1$,
$$
\left(\|f(A)\|_p^p + \|f(B)\|_p^p\right)^{1/p} \le \|f(xA+(1-x)B))\|_p+ \|f((1-x)A+xB)\|_p.
$$
\end{cor}

\vskip 5pt
\begin{proof} From Theorem \ref{th-HH1} we have 
$$
\|f(A)\oplus f(B)\|_p \le \|Uf((1-x)A+xB)U^* + Vf(xA+(1-x)B)V^*\|_p.
$$
The triangle inequality for $\|\cdot\|_p$ completes the proof.
\end{proof}

\vskip 5pt
Corollary \ref{corSchatt1} with $f(t)=t^q$ reads as the following trace inequality.

\vskip 5pt
\begin{cor}\label{cortrace1}  Let $A, B\in\bM_{n}^+$ and $0<x<1$. Then, for all $p\ge 1\ge q\ge 0$,
$$
\left\{{\mathrm{Tr\,}}A^{pq} + \mathrm{Tr\,}B^{pq}\right\}^{1/p} \le \left\{{\mathrm{Tr\,}}(xA+(1-x)B)^{pq}\right\}^{1/p}+ \left\{{\mathrm{Tr\,}}((1-x)A+xB))^{pq}\right\}^{1/p}.
$$
\end{cor}

\vskip 5pt
Choosing in Corollary \ref{cortrace1} $q=1$ and $x=1/2$ yields McCarthy's inequality,
$$
{\mathrm{Tr\,}}A^{p} +{\mathrm{Tr\,}}B^{p} \le {\mathrm{Tr\,}}(A+B)^{p}. $$
This shows that Theorem \ref{th-HH1} is already significant with $f(t)=t$.
Our next corollary, for convex functions, is equivalent to Theorem \ref{th-HH1}.

\vskip 5pt
\begin{cor}\label{cor-HH1}  Let $A, B\in\bM_{n}^+$, let $0<x<1$, and let $g(t)$ be a monotone convex function  on $[0,\infty)$ with $g(0)\le 0$. Then, for some isometry matrices $U,V\in\bM_{2n,n}$,
$$
g(A)\oplus g(B) \ge Ug((1-x)A+xB)U^* + Vg(xA+(1-x)B)V^*.
$$
\end{cor}

\vskip 5pt
Since the Schatten $q$-quasinorms $\|\cdot\|_q$, $0<q<1$, are superadditive functionals on $\bM_n^+$, Corollary \ref{cor-HH1} yields the next one.

\vskip 5pt
\begin{cor}\label{corSchatt2} Let $A, B\in\bM_{n}^+$, let $0<x<1$, and let $g(t)$ be a nonnegative convex function  on $[0,\infty)$ with $g(0)\le 0$. Then, for all $0<q<1$,  
$$
\left(\|g(A)\|_q^q+ \|g(B)\|_q^q\right)^{1/q} \ge \|g(xA+(1-x)B))\|_q+ \|g((1-x)A+xB)\|_q.
$$
\end{cor}

\vskip 5pt
From the first inequality of Corollary \ref{coreigen1} we also get the following statement.

\vskip 5pt
\begin{cor}\label{coreigen2}  Let $A, B\in\bM_{n}^+$ and let $f(t)$ be a nonnegative concave function  on $[0,\infty)$. Then, for $j=0,1,\ldots, n-1$,
$$
\lambda_{1+2j}^{\downarrow}\left(f(A\oplus B)\right) \le 2\int_{0}^1
\lambda_{1+j}^{\downarrow}\left(f(xA+(1-x)B) \right) \,{\mathrm{d}}x.
$$
\end{cor}

\vskip 5pt
Up to now we have dealt with  convex combinations $(1-x)A +xB$ with scalar weights. It is natural to search for extensions with matricial weights ($C^*$-convex combinations). We may
generalize Theorem \ref{th-HH1} with commuting normal weights.

\vskip 5pt
\begin{theorem}\label{th-HH2}  Let $A, B\in\bM_{n}^+$ and let $f(t)$ be a monotone concave function  on $[0,\infty)$ with $f(0)\ge 0$. If $X,Y\in\bM_n$ are normal and satisfy  $XY=YX$ and $X^*X+Y^*Y=I$, then, for some isometry  matrices $U,V\in\bM_{2n,n}$,
$$
f(A)\oplus f(B) \le Uf(X^*AX+Y^*BY)U^* + Vf(Y^*AY+X^*BX)V^*.
$$
\end{theorem}

\vskip 5pt
\begin{proof} The proof is quite similar to that  of Theorem \ref{th-HH1} except that we first observe that the $2n\times 2n$ matrix
$$
H:=\begin{bmatrix}
X & Y \\ Y & -X
\end{bmatrix}
$$
is unitary. Indeed for two normal operators, $XY=YX$ ensures $X^*Y=YX^*$ and a direct computation shows that $H^*H$ is the identity in $\bM_{2n}$.
We then use the unitary congruence
\begin{equation}\label{eq-cong2}
H^*
\begin{bmatrix}
A & 0 \\ 0 & B
\end{bmatrix}H=
\begin{bmatrix} X^*AX+Y^*BY
 & \star \\ \star &  Y^*AY+X^*BX
\end{bmatrix}
\end{equation}
where the stars hold for unspecified entries. 
\end{proof}

\vskip 5pt
Hence, in the first inequality of Corolloray \ref{coreigen1} and in the series of Corollaries \ref{corSchatt1}-\ref{corSchatt2}, we can replace the scalar convex combinations $(1-x)A+B$ and $xA+(1-x)B$ by $C^*$-convex combinations $X^*AX+Y^*BY$ and $Y^*AY + X^*BX$ with commuting normal  weights. Here we explicitly state the generalization of Corollary \ref{cor-HH1}.

\vskip 5pt
\begin{cor}\label{cor-HH2}  Let $A, B\in\bM_{n}^+$ and let $g(t)$ be a monotone convex function  on $[0,\infty)$ with $g(0)\le 0$.  If $X,Y\in\bM_n$ are normal and satisfy  $XY=YX$ and $X^*X+Y^*Y=I$, then, for some isometry  matrices $U,V\in\bM_{2n,n}$,
$$
g(A)\oplus g(B) \ge Ug(X^*AX+Y^*BY)U^* + Vg(Y^*AY+X^*BX)V^*.
$$
\end{cor}

\section{Majorization}

The  results of Section 2 require two essential assumptions : to deal with positive matrices and with subadditive (concave) or superadditive (convex) functions. Thanks to these assumptions, we have obtained operator inequalities for the usual order in the positive cone.

The results of this section will consider Hermitian matrices and general convex or concave functions. We will  obtain majorization relations. We also consider $C^*$-convex combinations more general than those with commuting normal weights.  

We  recall the notion of majorization. Let $A,B\in\bM_n^{s.a}$.
We say that $A$ is weakly majorized by $B$ and we write $A\prec_{w}B$, if 
$$
\sum_{j=1}^k\lambda_j^{\downarrow}(A) \le  \sum_{j=1}^k\lambda_j^{\downarrow}(B)
$$
for all $k=1,2,\ldots n$. If
 furthemore  the equality holds for $k=n$, that  is $A$ and $B$ have the same trace, then we say that $A$  is majorized by $B$, written $A\prec B$. See  \cite[Chapter 2]{Bh} and \cite{Hiai1} for a background on majorization. One easily checks that $A\prec_wB$ is equivalent to $A+C\prec B$ for some $C\in\bM_n^+$. We  need two fundamental  principles:
\begin{itemize}
\item[(1)]
  $A\prec B \Rightarrow g(A) \prec_w g(B)$ for all convex functions $g(t)$. 

\item[(2)] 
$A\prec_w B \iff {\mathrm{Tr\,}} f(A) \le {\mathrm{Tr\,}} f(B) $ for all nondecreasing convex functions $f(t)$. Equivalently $A\prec B \iff {\mathrm{Tr\,}} g(A) \le {\mathrm{Tr\,}} g(B) $ for all convex functions $g(t)$. 
\end{itemize}

\vskip 5pt
\begin{lemma}\label{lemma-trace}  Let $A, B\in\bM_{n}^{s.a}$ and let $g(t)$ be a convex function defined on an interval containing the spectra of $A$ and $B$. If $X,Y\in\bM_n$ satisfy $X^*X+Y^*Y=XX^*+YY^*=I$, then,
$$
{\mathrm{Tr\,}}\left\{g(X^*AX+Y^*BY)+g(Y^*AY+X^*BX)\right\}\le {\mathrm{Tr\,}} \left\{g(A)+g(B)\right\}.
$$
\end{lemma}

\vskip 5pt
\begin{proof}
By the famous Hansen-Pedersen trace inequality \cite{HP}, see also \cite[Corollary 2.4]{BL-London} for a generalization,
\begin{align*}
{\mathrm{Tr\,}}&\left\{g(X^*AX+Y^*BY)+g(Y^*AY+X^*BX)\right\} \\
&\le {\mathrm{Tr\,}}\left\{X^*g(A)X+Y^*g(B)Y)+Y^*g(A)Y+X^*g(B)X\right\} \\
&= {\mathrm{Tr\,}}\left\{(g(A) + g(B))(XX^* +YY^*)\right\}= {\mathrm{Tr\,}} \left\{g(A)+g(B)\right\}
\end{align*}
where the first equality follows from the cyclicity of the trace.
\end{proof}

\vskip 5pt
\begin{theorem}\label{th-orbit}  Let $A, B\in\bM_{n}^{s.a}$ and $X,Y\in\bM_n$. If  $X^*X+Y^*Y=XX^*+YY^*=I$, then, for some  unitary matrices $\{U_k\}_{k=1}^{2n}$ in $\bM_{2n}$,
$$
(X^*AX+Y^*BY)\oplus (Y^*AY+X^*BX) =\frac{1}{2n}\sum_{k=1}^{2n} U_k (A\oplus B)U_k^*.
$$
\end{theorem}

\vskip 5pt
\begin{proof} From Lemma \ref{lemma-trace}, we have the trace inequality
$$
{\mathrm{Tr\,}}g\left((X^*AX+Y^*BY)\oplus (Y^*AY+X^*BX) \right)\le  {\mathrm{Tr\,}}g(A\oplus B) 
$$
for all convex functions defined on $(-\infty,\infty)$. By a basic principle of majorization, this is equivalent to
\begin{equation}\label{eqmaj}
(X^*AX+Y^*BY)\oplus (Y^*AY+X^*BX) \prec A\oplus B.
\end{equation}
By \cite[Proposition 2.6]{BL-PAMS}, the majorization in $\bM_d^{s.a}$, $S\prec T$, ensures that (and thus is equivalent to)
$$
S\le \frac{1}{d} \sum_{j=1}^d V_jTV_j^*
$$
for $d$ unitary matrices $V_j\in\bM_d$. Applying this to \eqref{eqmaj} completes the proof.
 \end{proof}

\vskip 5pt
\begin{remark} We can prove the majorization \eqref{eqmaj} in a different way by oberving that our assumption on $X$ and $Y$ ensures that the map $\Phi$, defined on $\bM_{d}$, (here $d=2n$),
$$
\Phi \left(\begin{bmatrix} A&C \\ D&B\end{bmatrix}\right):=\begin{bmatrix}X^*AX+Y^*BY&0 \\ 0&Y^*AY+X^*BX)\end{bmatrix},
$$
is a  positive linear map unital and trace preserving. Such maps are also called doubly stochastic. It is a classical result (see Ando's survey \cite[Section 7]{Ando2} and references therein) that we have
\begin{equation}\label{eqchan}
\Phi(Z) \prec Z
\end{equation}
for every doubly stochastic map on $\bM_d$ and Hermitian $Z$. Here, our map is even completely positive (it is a so called  quantum channel). It is well known in the litterature (\cite[Theorem 7.1]{Ando2}) that we have then 
$$
\Phi(Z)=\sum_{j=1}^{m} t_j U_j^* ZU_j
$$
for some convex combination $0<t_j\le 1$, $\sum_{j=1}^{m} t_j=1$, and some unitary matrices $U_j$ (these scalars $t_i$ and matrices $U_i$ depend on $Z$). In 2003, Zhan \cite{Zhan}  noted that we can take $m=d$. That we can actually take an average,
$$
\Phi(Z)=\frac{1}{d}\sum_{j=1}^{d}  U_j^* ZU_j,
$$
follows from the quite  recent observation \cite[Proposition 2.6]{BL-PAMS} mentioned in the proof of Theorem \ref{th-orbit}. For quantum channels, one has the Choi-Kraus decomposition (see \cite[Chapter 3]{Bhatia2})
$$
\Phi(A) =\sum_{i=1}^{d^2} K_iAK_i^*
$$
for some weights $K_i\in\bM_d$ such that $\sum_{i=1}^{d^2} K_iK_i^*=\sum_{i=1}^{d^2} K_i^*K_i=I$. So, the proof of
Lemma 3.1  shows that, for quantum channels, one may derive the fundamental majorization \eqref{eqchan} from two results for convex functions: the basic principle of majorisation and Hansen-Pedersen's trace inequality. 
\end{remark}

\vskip 5pt
Theorem \ref{th-orbit} combined with Theorem \ref{thB2005} provide a number of interesting operator inequalities. The next corollary can be regarded as another matrix version of the scalar inequality \eqref{eqext1}. We will give a proof independent of Theorem \ref{thB2005}.

\vskip 5pt
\begin{cor}\label{cor-ext2}  Let $A, B\in\bM_n^{s.a}$ and let $f(t)$ be a convex function defined on an interval containing the spectra of $A$ and $B$.
 If $X,Y\in\bM_n$ satisfy $X^*X+Y^*Y=XX^*+YY^*=I$, then, for some  unitary matrices $\{U_k\}_{k=1}^{2n}$ in $\bM_{2n}$,
$$
f\left(X^*AX+Y^*BY\right)\oplus f\left(Y^*AY+X^*BX \right) \le\frac{1}{2n}\sum_{k=1}^{2n} U_k f\left(A\oplus B\right)U_k^*.
$$
\end{cor}

\vskip 5pt
In particular,  for the absolute value, we note that
$$
\left|X^*AX+Y^*BY\right|\oplus \left|Y^*AY+X^*BX \right| \le\frac{1}{2n}\sum_{k=1}^{4n} U_k \left|A\oplus B\right|U_k^*.
$$

\vskip 5pt
\begin{proof} The majorization \eqref{eqmaj} and the basic principle of majorizations show
that
$$
f\left(X^*AX+Y^*BY\right)\oplus f\left(Y^*AY+X^*BX \right) \prec_w f(A\oplus B)
$$
for all convex functions defined on an interval containing the spectra of $A$ and $B$.
Thus
$$
f\left(X^*AX+Y^*BY\right)\oplus f\left(Y^*AY+X^*BX \right) +C \prec f(A\oplus B)
$$
for some positive semidefinte matrice $C\in\bM_{2n}^+$. Hence, as in the previous proof,
$$f\left(X^*AX+Y^*BY\right)\oplus f\left(Y^*AY+X^*BX \right) +C \le \frac{1}{2n}\sum_{k=1}^{2n} U_k f\left(A\oplus B\right)U_k^*
$$
for some family of unitary matrices $U_k\in\bM_{2n}$.
\end{proof}

\vskip 5pt
\begin{cor}\label{cordet}  Let $A, B\in\bM_n^+$.  If $X,Y\in\bM_n$ satisfy $X^*X+Y^*Y=XX^*+YY^*=I$, then,
$$
\det(X^*AX+Y^*BY)\det (Y^*AY+X^*BX) \ge \det A\det B.
$$
\end{cor}

\vskip 5pt
\begin{proof} Since the classical Minkowski functional $Z\mapsto \det^{1/2n} Z$ is concave on $\bM_{2n}^+$, the result is an immediate consequence of  Theorem \ref{th-orbit}.
 \end{proof}

The case $A=B$ and $X,Y$ are two orthogonal projections reads as the classical Fisher's inequality.

\vskip 5pt
Corollary \ref{cor-ext2} yields inequalities for symmetric norms and antinorms on $\bM_{2n}^+$.
Symmetric norms, $\|\cdot\|$, also called unitarily invariant norms, are classical objects in matrix analysis. We refer to \cite{Bh}, \cite{Hiai1}, \cite[Chapter 6]{HP} and, in the setting of compact operators, \cite{Simon}. The most famous examples are the Schatten $p$-norms, $1\le p\le \infty$, and the Ky Fan $k$-norms.

Symmetric anti-norms $\|\cdot\|_!$ are the concave counterpart of symmetric norms. Famous examples are the Schatten $q$-quasi norms, $0<q<1$ and the Minkowski functional considered in the proof of Corollary \ref{cordet}. We refer to \cite{BH1} and \cite[Section 4]{BH2} for much more examples.

\vskip 5pt
\begin{cor}\label{corfinal}  Let $f(t)$ and $g(t)$ be  two nonnegative functions defined on  $[a,b]$ and let $A,B\in\bM_n$ be Hermitian with spectra in $[a,b]$.  If $X,Y\in\bM_n$ satisfy $X^*X+Y^*Y=XX^*+YY^*=I$, then :
\begin{itemize}
\item[(i)] If $f(t)$ is concave, then, for all symmetric antinorms,
$$
 \left\| f\left(X^*AX+Y^*BY\right)\oplus f\left(Y^*AY+X^*BX \right) \right\|_!\ge \left\|  f\left(A\oplus B\right)\right\|_!.
$$
\item[(ii)] If $g(t)$ is convex, then, for all symmetric norms,
$$
\left\| g\left(X^*AX+Y^*BY\right)\oplus g\left(Y^*AY+X^*BX \right) \right\| 
\le\left\| g\left(A\oplus B\right) \right\| .
$$
\end{itemize} 
\end{cor}

\begin{proof} Since symmetric antinorms are unitarily invariant and superadditive, the first assertion follows from the version of Corollary \ref{cor-ext2} for concave version. The second assertion is an immediate consequence of  Corollary \ref{cor-ext2}. \end{proof}

We close this section by mentioning Moslehian's weak majorization which  provides a matrix version of the first inequality of Proposition \ref{propHH1}. We may restate \cite[Corollary 3.4]{Mo}
as inequalities
  for symmetric and antisymmetric norms.

\vskip 5pt
\begin{prop}\label{prop-HH-mox}  Let $f(t)$ and $g(t)$ be  two nonnegative functions defined on  $[a,b]$ and let $A,B\in\bM_n$ be Hermitian with spectra in $[a,b]$. 
\begin{itemize}
\item[(i)] If $f(t)$ is concave, then, for all symmetric antinorms,
$$
  \left\|f\left(\frac{A+B}{2}\right)\right\|_!
\ge
 \left\|\int_0^1 f((1-x)A+xB)\, {\mathrm{d}}x\right\|_!.
$$
\item[(ii)] If $g(t)$ is convex, then, for all symmetric norms,
$$
\left\|g\left(\frac{A+B}{2}\right)\right\|
\le\left\|\int_0^1 g((1-x)A+xB)\, {\mathrm{d}}x\right\| .
$$
\end{itemize}
\end{prop}

\section{References of chapter 4}

{\small
\begin{itemize}

\item[[5\!\!\!]] T.\ Ando, Majorization, doubly stochastic matrices, and comparison of eigenvalues,
{\it Linear Algebra Appl.\ } 118 (1989), 163-248.

\item[[13\!\!\!]]  R.\ Bhatia, Matrix Analysis, Gradutate Texts in Mathematics, Springer, New-York, 1996.

\item[[15\!\!\!]] R.\ Bhatia, Positive Definite Matrices, Princeton University press, Princeton 2007.

\item[[28\!\!\!]]  J.-C. Bourin and F. Hiai,
Norm and anti-norm inequalities for positive semi-definite matrices,
{\it Internat. J. Math.} \textbf{63} (2011), 1121-1138.

\item[[29\!\!\!]]  J.-C. Bourin and F. Hiai, Jensen and Minkowski inequalities for operator means and anti-norms. Linear Algebra Appl. 456 (2014), 22–53.

\item[[31\!\!\!]] J.-C.\ Bourin and E.-Y.\ Lee, Unitary orbits of Hermitian operators with convex or concave functions, {\it Bull.\ Lond.\
Math.\ Soc.}\ 44 (2012), no.\ 6, 1085--1102.

\item[[42\!\!\!]] J.-C.\ Bourin and E.-Y.\ Lee,  A Pythagorean Theorem for partitioned matrices, Proc. Amer. Math. Soc., in press.

\item[[43\!\!\!]]  J.-C.\ Bourin and E.-Y.\ Lee, Matrix inequalities and majorizations around Hermite-Hadamard's inequality, preprint

\item[[64\!\!\!]] F. Hiai, Matrix Analysis: Matrix Monotone Functions, Matrix Means, and
Majorization (GSIS selected lectures),
\textit{Interdisciplinary Information Sciences} 16 (2010), 139--248.

\item[[67\!\!\!]] F. Hiai, D. Petz, Introduction to Matrix Analysis and applications. Universitext, Springer, New Delhi, 2014.

\item[[81\!\!\!]] M.\ S.\ Moslehian, 
Matrix Hermite-Hadamard type inequalities, {\it Houston J.\ Math.}\ 39 (2013), no.\ 1, 177--189.

\item[[83\!\!\!]] C.\ P.\ Niculescu; L.-E.\ Persson,  Old and new on the Hermite-Hadamard inequality, {\it Real Anal.\ Exchange}\ 29 (2003/04), no.\ 2, 663--685.

\item[[89\!\!\!]] B.\ Simon, Trace ideal and their applications, Cambridge University Press, Cambridge, 1979.

\item[[94\!\!\!]] X.\ Zhan, The sharp Rado theorem for majorizations, {\it Amer.\ Math.\ Monthly} 110
(2003) 152--153.

\end{itemize}

}


\part{Operator diagonals and partitionned matrices}

\ 
\ 

\vskip 130pt

{\color{blue}
{\Large 
$$
\begin{bmatrix} A&X \\ X&B \end{bmatrix} =\frac{U(A+B)U^* + V(A+B)V^*}{2}
$$
}
}



\chapter{Pinchings and Masas}

For a background on the essential numerical and the proof of the pinching theorem used in this article, see Section 5.6 ``Around this article".

\vskip 10pt\noindent
{\color{blue}{\Large {\bf Pinchings and positive linear maps} \large{\cite{BL-JFA}}}}

\vskip 10pt\noindent
{\bf Abstract.} We employ the  pinching theorem, ensuring that some operators $A$ admit {\it any} sequence of
contractions as an operator diagonal of $A$, to deduce/improve two recent theorems of Kennedy-Skoufranis and Loreaux-Weiss for
conditional expectations onto a masa in the algebra of operators on a Hilbert space. Similarly, we obtain a proof of a theorem of
Akeman and Anderson showing that positive contractions in a continuous masa can be lifted to a projection.
 We also  discuss a few corollaries  for sums  of two operators in the same unitary orbit.

\vskip 3pt
{\small\noindent
Keywords: Pinching, essential numerical range, positive linear maps, conditional expectation onto a masa, unitary orbit.
}
\vskip 3pt \noindent
{\small\noindent
2010 Mathematics Subject Classification.  46L10, 47A20, 47A12.}

\section{The pinching theorem}

We recall two theorems which are fundamental in the next sections to obtain several results about positive linear maps, in particular
conditional expectations, and unitary orbits.
These theorems were established in \cite{B-JOT}, we also
refer to this article for various definitions and properties of the essential numerical range $W_e(A)$ of an operator $A$ in the
algebra ${\mathrm{L}}({\mathcal{H}})$ of all
 (bounded linear) operators on an infinite dimensional, separable (real or complex) Hilbert space ${\mathcal{H}}$. 

We denote by ${\mathcal{D}}$ the unit disc of $\bC$. We write $A\simeq B$ to mean that the operators $A$ and $B$ are unitarily
equivalent. This relation is extended to operators possibly acting on different Hilbert spaces, typically, $A$ acts on
${\mathcal{H}}$ and $B$ acts on an infinite dimensional subspace ${\mathcal{S}}$ of ${\mathcal{H}}$, or on the spaces
${\mathcal{H}}\oplus {\mathcal{H}}$ or $\oplus^{\infty}{\mathcal{H}}$.

\vskip 5pt
\begin{theorem}\label{pinching} Let $A\in{\mathrm{L}}({\mathcal{H}}) $ with $W_e(A)\supset{\mathcal{D}}$ and $\{X_i\}_{i=1}^{\infty}$
a sequence in ${\mathrm{L}}({\mathcal{H}})$ such that
$\sup_i\|X_i\|<1$. Then,  a  decomposition  ${\mathcal{H}}=\bigoplus_{i=1}^{\infty}{\mathcal{H}}_i$ holds with 
 $A_{{\mathcal{H}}_i}\simeq X_i$ for all $i$.
\end{theorem}

\vskip 5pt
Of course, the direct sum refers to an orthogonal decomposition, and $A_{{\mathcal{H}}_i}$ stands for the compression of $A$ onto the
subspace ${\mathcal{H}}_i$.

Theorem \ref{pinching} tells us that we have a unitary congruence between an operator in ${\mathrm{L}}(\oplus^{\infty}{\mathcal{H}})$
and a ``pinching" of $A$,
$$
\bigoplus_{i=1}^{\infty} X_i \simeq \sum_{i=1}^{\infty} E_iAE_i
$$
for some sequence of mutually orthogonal infinite dimensional projections $\{E_i\}_{i=1}^{\infty}$ in ${\mathrm{L}}({\mathcal{H}})$
summing up to the identity $I$. Thus $\{X_i\}_{i=1}^{\infty}$ can be regarded as an operator diagonal of $A$. In particular, if $X$
is an operator on ${\mathcal{H}}$ with $\| X\| <1$, then, $A$ is unitarily congruent to an operator on ${\mathcal{H}}\oplus
{\mathcal{H}}$ of the form,
\begin{equation}\label{compression}
A\simeq \begin{pmatrix} X&\ast \\ \ast& \ast\end{pmatrix}.
\end{equation}

For a sequence of normal operators, Theorem \ref{pinching} admits a variation. Given ${\mathcal{A}},{\mathcal{B}}\subset \bC$, the
notation ${\mathcal{A}}\subset_{st}{\mathcal{B}}$ means that ${\mathcal{A}}+r{\mathcal{D}}\subset {\mathcal{B}}$ for some $r>0$.

\vskip 5pt
\begin{theorem}\label{pinchingnormal} Let $A\in{\mathrm{L}}({\mathcal{H}}) $ and let $\{X_i\}_{i=1}^{\infty}$ be a sequence of normal
operators in ${\mathrm{L}}({\mathcal{H}})$ such that
$\cup_{i=1}^{\infty}W(X_i)\subset_{st} W_e(A)$. Then, a decomposition ${\mathcal{H}}=\bigoplus_{i=1}^{\infty}{\mathcal{H}}_i$ holds
with
 $A_{{\mathcal{H}}_i}\simeq X_i$ for all $i$. 
\end{theorem}

\vskip 5pt
If all the operators are self-adjoint, this is true for the strict inclusion in $\bR$ (if ${\mathcal{A}},{\mathcal{B}}\subset \bR$,
the notation ${\mathcal{A}}\subset_{st}{\mathcal{B}}$ then means that ${\mathcal{A}}+r{\mathcal{I}}\subset {\mathcal{B}}$ for some
$r>0$, where ${\mathcal{I}}=[-1,1]$). This is actually an easy consequence of Theorem \ref{pinching} or Theorem \ref{pinchingnormal}.

\vskip 5pt
\begin{cor}\label{pinchingself} 
Let $A\in{\mathrm{L}}({\mathcal{H}}) $ be self-adjoint and let $\{X_i\}_{i=1}^{\infty}$ be a sequence of self-adjoint operators in
${\mathrm{L}}({\mathcal{H}})$ such that
$\cup_{i=1}^{\infty}W(X_i)\subset_{st} W_e(A)$. Then, a decomposition ${\mathcal{H}}=\bigoplus_{i=1}^{\infty}{\mathcal{H}}_i$ holds
with
 $A_{{\mathcal{H}}_i}\simeq X_i$ for all $i$. 
\end{cor}

\vskip 5pt
To get Corollary \ref{pinchingself} from Theorem \ref{pinchingnormal}, let $W_e(A)=[a,b]$ and write $A=A_1+A_2+A_3+A_4$ where
$A_iA_j=0$ if $i\neq j$, with $a\in \sigma_e(A_1)\cap\sigma_e(A_2)$ and $b\in \sigma_e(A_3)\cap\sigma_e(A_4)$. Apply Theorem
\ref{pinchingnormal} to the normal operator $\tilde{A}=(1+i) A_1 +(1-i)A_2 + (1+i)A_3 +(1-i)A_4$ as $\cup_{i=1}^{\infty}W(X_i)
\subset_{st} {\mathrm{conv}}\{(1\pm i)a;(1\pm i)b\} \subset W_e(\tilde{A})$. We get a decomposition
${\mathcal{H}}=\bigoplus_{i=1}^{\infty}{\mathcal{H}}_i$ with $\tilde{A}_{{\mathcal{H}}_i}\simeq X_i$ for all $i$. Therefore, taking
real parts we also have
 $A_{{\mathcal{H}}_i}\simeq X_i$.

In Section 3, our concern is the study of generalized diagonals, i.e., conditional expectations onto a masa in
${\mathrm{L}}({\mathcal{H}}) $, of the unitary orbit of an operator. The pinching theorems are the good tools for this study; we
easily obtain and considerably improve two recent theorems, of Kennedy and Skoufranis
for normal operators, and Loreaux and Weiss for idempotent operators. For self-adjoint idempotents, i.e., projections, and continuous
masas, we obtain a theorem due to Akemann and Anderson.
Section 4 deals with an application to the class of unital, positive linear maps which are trace preserving. 
Section 5 collects a few questions on possible extension of Theorems \ref{pinching} and
\ref{pinchingnormal} in the setting of von Neumann algebras.

The next section gives applications which only require \eqref{compression}. These results mainly focus on sums of two operators in a
unitary orbit.

\section{Sums in a unitary orbit}

\vskip 5pt\noindent
We recall  a straightforward consequence of \eqref{compression} for the weak convergence, \cite[Corollary 2.4]{B-JOT}.

\vskip 5pt
\begin{cor}\label{cor1seq} Let $A,X\in{\mathrm{L}}({\mathcal{H}})$ with $W_e(A)\supset{\mathcal{D}}$ and $\| X\| \le 1$. Then there
exists a sequence of unitaries $\{U_n\}_{n=1}^{\infty}$ in ${\mathrm{L}}({\mathcal{H}})$ such that
$${\mathrm{wot}}\!\!\! \lim_{n\to+\infty}U_nAU_n^*= X. $$
\end{cor} 

\vskip 5pt
Of course, we cannot replace the weak convergence by the strong convergence; for instance if $A$ is invertible and $\| Xh\| <\|
A^{-1}\|^{-1}$ for some unit vector $h$, then $X$ cannot be a strong limit from the unitary orbit of $A$. However, the next best
thing does happen. Moreover, this is even true for the $\ast$-strong operator topology.

\vskip 5pt
\begin{cor}\label{cor2seq} Let $A,X\in{\mathrm{L}}({\mathcal{H}})$ with $W_e(A)\supset{\mathcal{D}}$ and $\| X\| \le 1$. Then there
exist two sequences of unitaries $\{U_n\}_{n=1}^{\infty}$ and $\{V_n\}_{n=1}^{\infty}$ in ${\mathrm{L}}({\mathcal{H}})$ such that
$$\ast\,{\mathrm{sot}}\!\!\! \lim_{n\to+\infty}\frac{U_nAU_n^*+V_nAV_n^* }{2}= X. $$
\end{cor} 

\vskip 5pt
\begin{proof} From \eqref{compression} we  have
\begin{equation*}
A\simeq \begin{pmatrix} X&-R \\ -S&T\end{pmatrix}.
\end{equation*}
Hence there exist two unitaries $U,V:{\mathcal{H}}\to{\mathcal{H}}\oplus{\mathcal{H}}$
such that 
\begin{equation}\label{eq1}
\frac{UAU^*+VAV^*}{2}= \begin{pmatrix} X&0 \\ 0&T\end{pmatrix}.
\end{equation}
Now let $\{e_n\}_{n=1}^{\infty}$ be a basis of ${\mathcal{H}}$ and choose any unitary $W_n:{\mathcal{H}}\oplus{\mathcal{H}}\to
{\mathcal{H}} $ such that $W_n(e_j\oplus 0)=e_j$ for all $ j\le n$. Then
$$
X_n:=W_n  \begin{pmatrix} X&0 \\ 0&T\end{pmatrix} W_n^*
$$
strongly converges to $X$. Indeed, $\{X_n\}$ is bounded in norm and, for all $j$, $X_ne_j\to Xe_j$. Taking adjoints,
$$
X_n^*=W_n  \begin{pmatrix} X^*&0 \\ 0&T^*\end{pmatrix} W_n^*,
$$
we also have $X_n^*\to X$ strongly. Setting $U_n=W_nU$ and $V_n=W_nV$ and using \eqref{eq1} completes the proof. \end{proof}

\vskip 5pt 
\begin{remark}\label{remmean} Corollary \ref{cor2seq} does not hold for the convergence in norm. We give an example. Consider the
permutation matrix
$$
T=\begin{pmatrix} 0&0&1 \\ 1&0&0 \\ 0&1&0 \end{pmatrix}
$$
and set $A=2(\oplus^{\infty} T)$ regarded as an operator in ${\mathrm{L}}({\mathcal{H}})$. Then $W_e(A)\supset{\mathcal{D}}$, however
$0$ cannot be a norm limit of means of two operators in the unitary orbit of $A$. Indeed $0$ cannot be a norm limit of means of two
operators in the unitary orbit of $(A+A^*)/2$ as $(A+A^*)/2= 2I-3P$ for some projection $P$.
\end{remark}

\vskip 5pt 
\begin{remark} The converse of Corollary \ref{cor2seq} holds: if $A\in {\mathrm{L}}({\mathcal{H}})$ has the property that any
contraction is a strong limit of a mean of two operators in its unitary orbit, then necessarily $W_e(A)\supset {\mathcal{D}}$. This
is checked by arguing as in the proof of Corollary \ref{cordoubly}.
\end{remark}

\vskip 5pt
We reserve the word ``projection" for self-adjoint idempotent. A strong limit of uniformly bounded idempotent operators is still an
idempotent; thus, the next corollary is rather surprising.

\vskip 5pt
\begin{cor}\label{idem2seq} Fix $\alpha>0$. There exists an idempotent $Q\in{\mathrm{L}}({\mathcal{H}})$ such that for every
$X\in{\mathrm{L}}({\mathcal{H}})$ with $\| X\| \le \alpha$ we have
 two sequences of unitaries $\{U_n\}_{n=1}^{\infty}$ and $\{V_n\}_{n=1}^{\infty}$ in ${\mathrm{L}}({\mathcal{H}})$ for which
$$\ast\,{\mathrm{sot}}\!\!\! \lim_{n\to+\infty} U_nQU_n^*+V_nQV_n^* = X. $$
\end{cor}

\vskip 5pt
\begin{proof}
Let $a>0$, define a two-by-two idempotent matrix
\begin{equation}\label{eqidempotent}
M_a=\begin{pmatrix} 1&0 \\ a&0 \end{pmatrix}
\end{equation}
and set $Q=\oplus^{\infty}M_a$  regarded as an operator in ${\mathrm{L}}({\mathcal{H}})$. Since the numerical range $W(\cdot)$ of
$$
\begin{pmatrix} 0&0 \\ 2&0 \end{pmatrix}
$$
is ${\mathcal{D}}$, we infer that $W(2\alpha^{-1}M_a)=W_e(2\alpha^{-1}Q)\supset{\mathcal{D}}$ for a large enough $a$. The result then
follows from Corollary \ref{cor2seq} with $A=2\alpha^{-1}Q$ and the contraction $\alpha^{-1}X$.
 \end{proof}

\vskip 5pt 
 Corollary \ref{idem2seq} does not hold for the convergence in norm.

\vskip 5pt
\begin{prop}\label{propsingle} Let $X\in{\mathrm{L}}({\mathcal{H}})$ be of the form $\lambda I +K$ for a compact operator $K$
and a scalar $\lambda\notin \{0,1,2\}$. Then $X$ is not norm limit of $U_nQU_n^*+V_nQV_n^*$ for any sequences of unitaries
$\{U_n\}_{n=1}^{\infty}$ and $\{V_n\}_{n=1}^{\infty}$ and any idempotent $Q$ in ${\mathrm{L}}({\mathcal{H}})$.
\end{prop}

\vskip 5pt
\begin{proof} First observe that if $\{A_n\}_{n=1}^{\infty}$ and $\{B_n\}_{n=1}^{\infty}$ are two bounded sequences in
${\mathrm{L}}({\mathcal{H}})$ such that $A_n-B_n\to 0$ in norm, then we also have $A_n^2-B_n^2\to 0$ in norm; indeed
$$
A_n^2-B_n^2= A_n(A_n-B_n) + (A_n-B_n)B_n.
$$
Now, suppose that $\lambda\neq 1$ and that we have  the (norm) convergence,
$$
U_nQU_n^*+V_nQV_n^*\to \lambda I +K.
$$
Then we also have 
\begin{equation}\label{eqL1}
W_nQW_n^*-\left(-Q+\lambda I +U_n^*KU_n\right)\to 0
\end{equation}
where $W_n:=U_n^*V_n$. Hence,
 by the previous observation,
$$
(W_nQW_n^*)^2-\left(-Q+\lambda I +U_n^*KU_n\right)^2\to 0,
$$
that is
\begin{equation}\label{eqL2}
W_nQW_n^*-\left(-Q+\lambda I +U_n^*KU_n\right)^2\to 0.
\end{equation}
Combining \eqref{eqL1} and \eqref{eqL2} we get
$$
\left(-Q+\lambda I +U_n^*KU_n\right)-\left(-Q+\lambda I +U_n^*KU_n\right)^2\to 0
$$
hence
$$
(-2+2\lambda)Q +(\lambda -\lambda^2 )I + K_n\to 0
$$
for some bounded sequence of compact operators $K_n$. Since $\lambda\neq 1$, we have
$$
Q=\frac{\lambda}{2} I+L
$$
for some compact operator $L$. Since $Q$ is idempotent, either $\lambda=2$ or $\lambda=0$.
\end{proof}

\vskip 5pt The operator $X$ in Proposition \ref{propsingle} has the special property that $W_e(X)$ is reduced to a single point. However
Proposition \ref{propsingle} may also hold when $W_e(X)$ has positive measure.

\vskip 5pt
\begin{cor} Let $Q$ be an idempotent in ${\mathrm{L}}({\mathcal{H}})$ and $z\in\bC\setminus\{0,1,2\}$. Then, there exists $\alpha >0$
such that the following property holds:
\begin{itemize}
\item[]
If $X\in{\mathrm{L}}({\mathcal{H}})$ satisfies $\| X-zI\| \le \alpha$, then $X$ is not norm limit of $U_nQU_n^*+V_nQV_n^*$ for any
  sequences of unitaries $\{U_n\}_{n=1}^{\infty}$ and $\{V_n\}_{n=1}^{\infty}$ in ${\mathrm{L}}({\mathcal{H}})$.
\end{itemize}
\end{cor}

\vskip 5pt
\begin{proof} By the contrary, $zI$ would be a norm limit of $U_nQU_n^*+V_nQV_n^*$ for some unitaries $U_n,V_n$, contradicting
Proposition \ref{propsingle}.
\end{proof}

\vskip 5pt More operators with large numerical and essential numerical ranges are given in the next proposition. An operator $X$ is
stable when its real part $(X+X^*)/2$ is negative definite (invertible).

\vskip 5pt
\begin{prop}\label{prop2} If $X\in{\mathrm{L}}({\mathcal{H}})$ is stable, then $X$ is not norm limit of $U_nQU_n^*+V_nQV_n^*$ for any
sequences of unitaries $\{U_n\}_{n=1}^{\infty}$ and $\{V_n\}_{n=1}^{\infty}$ and any idempotent $Q$ in ${\mathrm{L}}({\mathcal{H}})$.
\end{prop}

\vskip 5pt
\begin{proof} We have a decomposition ${\mathcal{H}}={\mathcal{H}}_s\oplus {\mathcal{H}}_{ns}$ in two invariant subspaces of $Q$ such
that $Q$ acts on ${\mathcal{H}}_s$ as a selfadjoint projection $P$, and $Q$ acts on ${\mathcal{H}}_{ns}$
as a purely nonselfadjoint idempotent, that is $A_{{\mathcal{H}}_{ns}}$ is unitarily equivalent to an operator on
${\mathcal{F}}\oplus{\mathcal{F}}$ of the form
\begin{equation}\label{purelyns}
Q_{{\mathcal{H}}_{ns}}\simeq\begin{pmatrix} I& 0 \\ R&0\end{pmatrix}
\end{equation}
where $R$ is a nonsingular (i.e., a zero kernel) positive operator  on a Hilbert space ${\mathcal{F}}$, so 
\begin{equation}\label{purelyns2}
Q\simeq P\oplus \begin{pmatrix} I& 0 \\ R&0\end{pmatrix}.
\end{equation}

Let $Y$ be  a norm limit of the sum of two sequences in the unitary orbit of $Q$. 
If the purely non-selfadjoint part ${\mathcal{H}}_{ns}$ is vacuous, then $Y$ is positive, hence $Y\neq X$. If ${\mathcal{H}}_{ns}$ is
not vacuous, \eqref{purelyns2} shows that
\begin{align*}
Q+Q^*&\simeq  2P\oplus \begin{pmatrix} 2I& R \\ R&0\end{pmatrix} \\
&\simeq  2P\oplus\left\{ \begin{pmatrix} I& I \\ I&I\end{pmatrix}
+  \begin{pmatrix} R& 0 \\ 0&-R\end{pmatrix}\right\}.
\end{align*} 
This implies that $\| (Q+Q^*)_+\| \ge \| (Q+Q^*)_-\|$, therefore $Y+Y^*$ cannot be negative definite, hence $X\neq Y$.
\end{proof} 
 
\vskip 5pt 
 It is known \cite{PT} that any operator is the sum of five idempotents.
We close this section by asking whether Corollorary \ref{idem2seq} admits a substitute for Banach space operators.

\vskip 5pt
\begin{question} Let ${\mathcal{ X}}$ be a separable Banach space and $T\in{\mathrm{L}}({\mathcal{ X}})$, the linear operators on
${\mathcal{ X}}$. Do there exist two sequences $\{P_n\}_{n=1}^{\infty}$ and $\{Q_n\}_{n=1}^{\infty}$ of idempotents in
${\mathrm{L}}({\mathcal{ X}})$ such that $T={\mathrm{sot}} \lim_{n\to+\infty} (P_n +Q_n)$ ?
\end{question}

\section{Conditional expectation onto a masa}

\subsection{Conditional expectation of general operators}

Kennedy and Skoufranis have studied the following problem: Let ${\mathfrak{X}}$ be a maximal abelian $\ast$-subalgebra (masa) of a
von Neumann algebra ${\mathfrak{M}}$, with corresponding expectation $\bE_{\mathfrak{X}}: {\mathfrak{M}}\to {\mathfrak{X}}$ (i.e., a
unital positive linear map such that $\bE_{\mathfrak{X}}(XM)= X\bE_{\mathfrak{X}}(M)$ for all $X\in{\mathfrak{X}}$ and
$M\in{\mathfrak{M}}$). Given a normal operator $A\in{\frak{M}}$, determine the image by $\bE_{\frak{X}}$ of the unitary orbit of $A$,
$$
\Delta_{\mathfrak{X}}(A)=\{\, \bE_{\mathfrak{X}}(UAU^*)\  : \ U \, {\mathrm{a\ unitary\ in}}\ {\mathfrak{M}}\, \}.
$$
In several cases, they determined the norm closure of $\Delta_{\frak{X}}(A)$. In particular, \cite[Theorem 1.2]{KS} can be stated in
the following two propositions.

\vskip 5pt
\begin{prop}\label{propKS} Let ${\mathfrak{X}}$ be a masa in ${\mathrm{L}}({\mathcal{H}})$, $X\in {\mathfrak{X}}$, and $A$ a normal
operator in $ {\mathrm{L}}({\mathcal{H}})$. If $\sigma(X)\subset {\mathrm{conv}}\sigma_e(A)$, then $X$ lies in the norm closure of
$\Delta_{\frak{X}}(A)$.
\end{prop}

\vskip 5pt
\begin{prop}\label{propKS2} Let ${\mathfrak{X}}$ be a continuous masa in ${\mathrm{L}}({\mathcal{H}})$, $X\in {\mathfrak{X}}$, and
$A$ a normal operator in $ {\mathrm{L}}({\mathcal{H}})$. If $X$ lies in the norm closure of $\Delta_{\frak{X}}(A)$, then
$\sigma(X)\subset {\mathrm{conv}}\sigma_e(A)$.
\end{prop}

\vskip 5pt
Since we deal with normal operators, $\sigma(X)\subset {\mathrm{conv}}\sigma_e(A)$ means $W(X)\subset W_e(A)$. 
Proposition \ref{propKS2} needs the continuous assumption. It is a rather simple fact; we generalize it in Lemma \ref{lemred}: {\it
Conditional expectations reduce essential numerical ranges, $W_e(\bE_{\frak{X}}(T))\subset W_e (T)$ for all $T\in
{\mathrm{L}}({\mathcal{H}})$}. Thus, the main point of \cite[Theorem 1.2]{KS} is Proposition \ref{propKS} which says that if
$W(X)\subset W_e(A)$ then $X$ can be approximated by operators of the form $\bE_{\frak{X}}(UAU^*)$ with unitaries $U$. With the
slightly stronger assumption $W(X)\subset_{st} W_e(A)$, Theorem \ref{pinchingnormal} guarantees, via the following corollary, that
$X$ is exactly of this form. Furthermore the normality assumption on $A$ is not necessary.

\vskip 10pt
\begin{cor}\label{corKS} Let ${\mathfrak{X}}$ be a masa in ${\mathrm{L}}({\mathcal{H}})$, $X\in {\mathfrak{X}}$ and $A\in
{\mathrm{L}}({\mathcal{H}})$.
If $W(X)\subset_{st} W_e(A)$, then $X=\bE_{\mathfrak{X}}(UAU^*)$ for some unitary operator $U\in{\mathrm{L}}({\mathcal{H}})$.
\end{cor} 

\vskip 5pt
\begin{proof} First, we note a simple fact: Let $\{P_i\}_{i=1}^{\infty}$ be a sequence of orthogonal projections in ${\mathfrak{X}}$
such that $\sum_{i=1}^{\infty}P_i=I$, and let $Z\in {\mathrm{L}}({\mathcal{H}})$ such that $P_iZP_i\in {\mathfrak{X}}$ for all $i$.
Then, we have a strong sum
$$\bE_{\mathfrak{X}}(Z)=\sum_{i=1}^{\infty} P_iZP_i.$$

Now, denote by ${\mathcal{H}}_i$ the range of $P_i$ and assume $\dim{\mathcal{H}}_i=\infty$ for all $i$. We have
$W(X_{{\mathcal{H}}_i})\subset W(X)$, hence $$\cup_{i=1}^{\infty}W(X_{{\mathcal{H}}_i})\subset_{st} W_e(A).$$
We may then apply Theorem \ref{pinchingnormal} and get a unitary $U$ on ${\mathcal{H}}=\bigoplus_{i=1}^{\infty}{\mathcal{H}}_i$ such
that
$$
A\simeq UAU^* =
\begin{pmatrix} 
X_{{\mathcal{H}}_1} &\ast & \cdots &\cdots  \\ 

\ast &X_{{\mathcal{H}}_2} &\ast &\cdots  \\ 

 \vdots &\ast  & \ddots &\ddots \\ 

\vdots  &\vdots  &\ddots  &\ddots \\ 
\end{pmatrix}.
$$
Since $0\oplus\cdots\oplus X_{{\mathcal{H}}_i}\oplus0\cdots \in {\mathfrak{X}}$ for all $i$, the previous simple fact shows that
$$
\bE_{\mathfrak{X}}(UAU^*)=\bigoplus_{i=1}^{\infty}X_{{\mathcal{H}}_i}= X. 
$$
\end{proof}

\vskip 5pt
We remark that
Corollary \ref{corKS} also covers the assumption $W(X)\subset W_e(A)$ of Proposition \ref{propKS}. Indeed, $W_e(A)\subset_{st}
W_e(A+D)$ for some normal operator $D$ with arbitrarily small norm, and we may apply Corollary \ref{corKS} to $X$ and $A+D$.

Kadison's article \cite{Kad} completely describes the diagonals of a projection, thus, providing in his terminology a carpenter
theorem for discrete masas ${\mathrm{L}}({\mathcal{H}})$. Our method can be used to obtain a similar statement for continuous masas
in ${\mathrm{L}}({\mathcal{H}})$ given in the next corollary. This result is due to Akemann and Anderson, see \cite[Corollary
6.19]{AA}. The case of a masa in a type-${\mathrm{II}}_{1}$ factor is solved in Ravichandran's paper \cite{Ra}.

\vskip 5pt
\begin{cor}\label{corCT} Let ${\mathfrak{X}}$ be a continuous masa in ${\mathrm{L}}({\mathcal{H}})$ and let $X\in {\mathfrak{X}}$ be
a positive contraction. Then there exists a projection $P$ in $ {\mathrm{L}}({\mathcal{H}})$ such that $X=\bE_{\mathfrak{X}}(P)$.
\end{cor}

\vskip 5pt
\begin{proof} We may suppose that $0<\| Xh\| <1$ for all unit vectors $h\in{\mathcal{H}}$; otherwise decompose $X$ as $Q+X_0$ for
some projection $Q\in{\mathfrak{X}}$ and consider $X_0$ in place of $X$. Then we have a decomposition
${\mathcal{H}}=\bigoplus_{i=1}^{\infty}{\mathcal{H}}_i$, where each summand ${\mathcal{H}}_i$ is the range of some projection in
${\mathfrak{X}}$, such that the decomposition
$
X=\bigoplus_{i=1}^{\infty} X_{{\mathcal{H}}_i}
$
satisfies $W(X_{{\mathcal{H}}_i})\subset_{st} [0,1]$ for all $i$.
We may further decompose each Hilbert space ${\mathcal{H}}_i$ as 
${\mathcal{H}}_i=\bigoplus_{j=1}^{\infty}{\mathcal{H}}_{i,j}$ where
${\mathcal{H}}_{i,j}$ is the range of some nonzero projection in ${\mathfrak{X}}$. 
We have a corresponding decomposition 
$
X=\bigoplus_{i=1}^{\infty}\left( \bigoplus_{j=1}^{\infty} X_{{\mathcal{H}}_{i,j}}\right)
$
Since $\cup_{j=1}^{\infty}W( X_{{\mathcal{H}}_{i,j}})\subset W(X_{{\mathcal{H}}_i})\subset_{st} [0,1] $, Corollary \ref{pinchingself}
yields
 some projection $P_i\in {\mathrm{L}}({\mathcal{H}}_i)$, $W_e(P)=[0,1]$, such that
$$
P_i=
\begin{pmatrix} 
X_{{\mathcal{H}}_{i,1}} &\ast & \cdots &\cdots  \\ 

\ast &X_{{\mathcal{H}}_{i,2}} &\ast &\cdots  \\ 

 \vdots &\ast  & \ddots &\ddots \\ 

\vdots  &\vdots  &\ddots  &\ddots \\ 
\end{pmatrix}.
$$
Letting $P=\bigoplus_{i=1}^{\infty}P_i$, we have  $X=\bE_{\mathfrak{X}}(P)$.
 \end{proof}

\subsection{A reduction lemma}

The following result extends Proposition \ref{propKS2}, the ``easy" part of Kennedy-Skoufranis' theorem \cite[Theorem 1.2]{KS}.

\vskip 5pt
\begin{lemma}\label{lemred} If ${\mathfrak{X}}$ is a masa in ${\mathrm{L}}({\mathcal{H}})$ and $Z\in{\mathrm{L}}({\mathcal{H}})$,
then $W(\bE_{\mathfrak{X}}(Z))\subset
\overline{W}(Z)$ and $W_e(\bE_{\mathfrak{X}}(Z))\subset
W_e(Z)$.
\end{lemma}

\vskip 5pt
\begin{proof} (1) Assume $Z$ is normal. We may identify the unital $C^*$-algebra ${\mathfrak{A}}$ spanned by $Z$ with
$C^0(\sigma(Z))$ via a $\ast$-isomorphism $\varphi: C^0(\sigma(Z))\to {\mathfrak{A}}$ with $\varphi(z\mapsto z)=Z$. Let
$h\in{\mathcal{H}}$ be a unit vector.
For $f\in C^0(\sigma(Z))$, set
$$
\psi(f)=\langle h, \bE_{\mathfrak{X}}(\varphi(f))h\rangle.
$$
Then $\psi$ is a positive linear functional on $C^0(\sigma(Z))$ and $\psi(1)=1$. Thus $\psi$ is a Radon measure induced by a
probabilty measure $\mu$,
$$
\psi(f)=\int_{\sigma(Z)} f(z)\, {\mathrm{d}}\mu(z).
$$
We then have $
\langle h, \bE_{\mathfrak{X}}(Z)h\rangle =\psi(z) \in {\mathrm{conv}}(\sigma(Z)).
$
Since ${\mathrm{conv}}(\sigma(Z))=\overline{W}(Z)$, we obtain $W(\bE_{\mathfrak{X}}(Z))\subset
\overline{W}(Z)$.

(2) Let $Z$ be a general operator in ${\mathrm{L}}({\mathcal{H}})$ and define a conditional expectation
$$\bE_2:{\mathrm{L}}({\mathcal{H}}\oplus{\mathcal{H}})\to {\mathfrak{X}}\oplus{\mathfrak{X}}$$ by
$$
\bE_2\left( 
\begin{pmatrix} A&C \\ D& B\end{pmatrix}
\right)=
\begin{pmatrix} \bE_{\mathfrak{X}}(A)&0 \\ 0& \bE_{\mathfrak{X}}(B)\end{pmatrix}.
$$
From the first part of the proof, we infer
$$
W(\bE_{\mathfrak{X}}(Z))\subset W\left( 
\begin{pmatrix} \bE_{\mathfrak{X}}( Z)&0 \\ 0& \bE_{\mathfrak{X}}(B)\end{pmatrix}
\right)
\subset\overline{W}\left( 
\begin{pmatrix} Z&C \\ D& B\end{pmatrix}
\right)
$$
whenever 
$
\begin{pmatrix} Z&C \\ D& B\end{pmatrix}
$
is normal. Since  we have, by a simple classical fact \cite{Hal},
$$
\overline{W}(Z)=\bigcap \overline{W}\left(\begin{pmatrix} Z&C \\ D& B\end{pmatrix}\right)
$$
where the intersection runs over all $B,C,D$ such that
$
\begin{pmatrix} Z&C \\ D& B\end{pmatrix}
$
is normal, we obtain $W(\bE_{\mathfrak{X}}(Z))\subset
\overline{W}(Z)$. 

(3) We deal with the essential numerical range inclusion. We can split ${\mathfrak{X}}$ into its discrete part ${\mathfrak{D}}$ and
continuous part ${\mathfrak{C}}$ with the corresponding decomposition of the Hilbert space,
$$
{\mathfrak{X}}={\mathfrak{D}}\oplus{\mathfrak{C}}, \qquad {\mathcal{H}}={\mathcal{H}}_d\oplus{\mathcal{H}}_c.
$$
We then have
\begin{equation}\label{red1}
W_e(\bE_{\mathfrak{X}}(Z))={\mathrm{conv}}\left\{
W_e(\bE_{\mathfrak{D}}(Z_{{\mathcal{H}}_d}))
;W_e(\bE_{\mathfrak{C}}(Z_{{\mathcal{H}}_c}))
 \right\}.
\end{equation}
We have an obvious inclusion
\begin{equation}\label{red2}
W_e(\bE_{\mathfrak{D}}(Z_{{\mathcal{H}}_d}))\subset W_e(Z_{{\mathcal{H}}_d}).
\end{equation}
On the other hand, for all compact operators $K\in{\mathrm{L}}({\mathcal{H}})$,
$$
W_e(\bE_{\mathfrak{C}}(Z_{{\mathcal{H}}_c}))=W_e(\bE_{\mathfrak{C}}(Z_{{\mathcal{H}}_c})+K_{{\mathcal{H}}_c})=W_e(\bE_{\mathfrak{C}}(Z_{{\mathcal{H}}_c}+K_{{\mathcal{H}}_c}))\subset\overline{W}(Z_{{\mathcal{H}}_c}+K_{{\mathcal{H}}_c})
$$
by the simple folklore fact that a conditional expectation onto a continous masa vanishes on compact operators and part (2) of the
proof. Thus, when $K$ runs over all compact operators, we obtain
\begin{equation}\label{red3}
W_e(\bE_{\mathfrak{D}}(Z_{{\mathcal{H}}_c}))\subset W_e(Z_{{\mathcal{H}}_c}).
\end{equation}
Combining \eqref{red1}, \eqref{red2} and  \eqref{red3} completes the proof.
\end{proof}

\subsection{Conditional expectation of idempotent operators}

\vskip 5pt
For discrete masas, unlike continuous masas \cite{KadSing}, there is a unique conditional expectation, which merely consists in
extracting the diagonal with respect to an orthonormal basis. In a recent article, Loreaux and Weiss give a detailed study of
diagonals of idempotents in ${\mathrm{L}}({\mathcal{H}})$. They established that a nonzero idempotent $Q$ has a zero diagonal with
respect to some orthonormal basis if and only if $Q$ is not a Hilbert-Schmidt perturbation of a projection (i.e., a self-adjoint
idempotent). They also showed that any sequence $\{a_n\}\in l^{\infty}$ such that $|a_n|\le \alpha$ for all $n$ and, for some
$a_{n_0}$, $a_k=a_{n_0}$ for infinitely many $k$, one has an idempotent $Q$ such that $\| Q\|\le 18\alpha +4$ and $Q$ admits
$\{a_n\}$ as a diagonal with respect to some orthonormal basis \cite[Proposition 3.4]{LW}. Using this, they proved that any sequence
in $l^{\infty}$ is the diagonal of some idempotent operator \cite[Theorem 3.6]{LW}, answering a question of Jasper. This statement is
in the range of Theorem \ref{pinching}. Further, it is not necessary to confine to diagonals, i.e., discrete masas, and the constant
$18\alpha +4$ can be improved; in the next corollary we have $3$ as the best constant when $\alpha=1$.

\vskip 5pt
\begin{cor}\label{corLW} Let ${\mathfrak{X}}$ be a masa in ${\mathrm{L}}({\mathcal{H}})$ and $\alpha>0$. There exists an idempotent
$Q\in {\mathrm{L}}({\mathcal{H}})$, such that for all $X\in {\mathfrak{X}}$ with $\| X\|<\alpha$, we have $X=\bE_{\frak{X}}(UQU^*)$
for some unitary operator $U\in{\mathrm{L}}({\mathcal{H}})$. If $\alpha=1$, $\| Q\|=3$ is the smallest possible norm.
\end{cor} 

\vskip 5pt
\begin{proof} As in the proof of Corollary \ref{idem2seq} we have an idempotent $Q$ such that $W_e(Q)\supset \alpha{\mathcal{D}}$,
hence the first and main part of Corollary \ref{corLW} follows from
Corollary \ref{corKS}. The remaining parts require a few computations.

To obtain the bound $3$ when $\alpha=1$ we get a closer look at
 $\oplus^{\infty} M_a$ with $M_a$ given by \eqref{eqidempotent} where $a$ is a positive scalar. We have
\begin{align*}
W(M_a)&=\left\{  \langle h, M_ah\rangle \ : \ h\in\bC^2, \|h\|=1\right\} \\
&=\left\{  |h_1|^2 + a\overline{h_2}h_1 \ : |h_1|^2 +|h_2|^2=1\right\}, 
\end{align*}
hence, with $h_1=re^{i\theta}$, $h_2=\sqrt{1-r^2}e^{i\alpha}$,
$$
W(M_a)=\bigcup_{0\le r\le 1}\left\{  r^2 + ar\sqrt{1-r^2}e^{i(\theta-\alpha)} : \ \theta, \alpha \in [0,2\pi] \right\}. 
$$
Therefore $W(M_a)$ is a union of circles $\Gamma_r$ with centers $r^2$ and radii $ar\sqrt{1-r^2}$. To have ${\mathcal{D}}\subset
W(M_a)$ it is necessary and sufficient that $-1\in \Gamma_r$ for some $r\in [0,1]$, hence
\begin{equation}\label{min}
a=\frac{1+r^2}{r\sqrt{1-r^2}}.
\end{equation}
Now we minimize $a=a(r)$ given by \eqref{min} when $r\in(0,1)$ and thus obtain the matrix $M_{a_{\ast}}$ with smallest norm such that
$W(M_{a_{\ast}})\supset{\mathcal{D}}$. Observe that $a(r)\to +\infty$ as $r\to 0$ and as $r\to 1$, and
$$
r^2(1-r^2)^{3/2}a'(r)= 3r^2-1.
$$
Thus $a(r)$  takes its minimal value $a_{\ast}$ when $r=1/\sqrt{3}$.
We have
$a_{\ast}=2\sqrt{2}$, hence 
$$\| M_{a_{\ast}}\|=3.$$ 
Now, letting $Q=\oplus^{\infty} M_{a_{\ast}}$, we have $W_e(Q)=W(M_{a_{\ast}})$, so that $Q$ is an idempotent in
${\mathrm{L}}({\mathcal{H}})$
such that $W_e(Q)\supset {\mathcal{D}}$, and thus by Corollary \ref{corKS} any operator $X$ such that $\| X\|<1$ satifies
$\bE_{\mathfrak{X}}(UQU^*)=X$ for some unitary $U$.

It remains to check that if $Q$ is an idempotent such that Corollary \ref{corLW} holds for any operator $X$ such that $\| X\|<1$,
then $\| Q\| \ge 3$. To this end, we consider the purely nonselfadjoint part $Q_{{\mathcal{H}}_{ns}}$ of $Q$ in \eqref{purelyns},
\begin{equation*}
Q_{{\mathcal{H}}_{ns}}\simeq\begin{pmatrix} I& 0 \\ R&0\end{pmatrix}.
\end{equation*}
We have $W_e(Q)\supset{\mathcal{D}}$ if and only if $W_e(Q_{{\mathcal{H}}_{ns}})\supset{\mathcal{D}}$. By Lemma \ref{lemred} this is
necessary. We may approximate $W_e(Q_{{\mathcal{H}}_{ns}})$ with sligthly larger essential numerical ranges, by using a positive
diagonalizable operator $R_{\varepsilon}$ such that $R_{\varepsilon} \ge R \ge R_{\varepsilon}-\varepsilon I$, for which
$$
W_e\left( \begin{pmatrix} I& 0 \\ R_{\varepsilon}&0\end{pmatrix} \right)=
W_e\left( \bigoplus_{n=1}^{\infty}\begin{pmatrix} 1& 0 \\ a_n&0\end{pmatrix}\right)
$$
where $\{a_n\}_{n=1}^{\infty}$ is a sequence of positive scalars, the eigenvalues of $R_{\varepsilon}$. By the previous step of the
proof, this essential numerical range contains $\mathcal{D}$ if and only if $\overline{\lim}\, a_n \ge a_{\ast}$. If this holds for
all $\varepsilon >0$, then $\| Q\|\ge 3$. \end{proof}

\section{Unital, trace preserving positive linear maps}

\vskip 5pt\noindent
Unital positive linear maps $\Phi:\bM_n\to\bM_n$, the matrix algebra, which preserve the trace play an important role in matrix
analysis and its applications. These maps are sometimes called doubly stochastic \cite{Ando2}.

We say that $\Phi: {\mathrm{L}}({\mathcal{H}})\mapsto {\mathrm{L}}({\mathcal{H}})$ is trace preserving if it preserves the trace
ideal ${\mathcal{T}}$ and ${\mathrm{Tr}}\,\Phi(Z)= {\mathrm{Tr}}\, Z$ for all $Z\in{\mathcal{T}}$.

\vskip 5pt
\begin{cor}\label{cordoubly} Let $A\in {\mathrm{L}}({\mathcal{H}})$. The following two conditions are equivalent:
\begin{itemize}
\item[(i)] $W_e(A)\supset {\mathcal{D}}$.
\item[(ii)] For all $X\in {\mathrm{L}}({\mathcal{H}})$ with $\| X\|<1$, there exists a unital, trace preserving, positive linear map
$\Phi: {\mathrm{L}}({\mathcal{H}})\to {\mathrm{L}}({\mathcal{H}})$ such that $\Phi(A)=X$.
\end{itemize}
We may further require in (ii) that $\Phi$ is completely positive and {\rm{sot}}- and {\rm{wot}}-sequentially continuous.
\end{cor}

\vskip 5pt
\begin{proof} Assume (i). By Theorem \ref{pinching} we have a unitary $U: {\mathcal{H}}\to\oplus^{\infty}{\mathcal{H}}$ such that
$$
A\simeq UAU^* =
\begin{pmatrix} 
X &\ast & \cdots &\cdots  \\ 

\ast &X &\ast &\cdots  \\ 

 \vdots &\ast  & X &\ddots \\ 

\vdots  &\vdots  &\ddots  &\ddots \\ 
\end{pmatrix}.
$$
Now consider the map $\Psi: {\mathrm{L}}(\oplus^{\infty}{\mathcal{H}})\to {\mathrm{L}}({\mathcal{H}})$,
$$
\begin{pmatrix} 
Z_{1,1} & Z_{1,2} & \cdots  \\ 

Z_{2,1} &Z_{2,2} &\cdots   \\ 

 \vdots &\vdots &\ddots \\  
\end{pmatrix}
\mapsto
\sum_{i=1}^{\infty} 2^{-i} Z_{i,i}
$$
and define $\Phi: {\mathrm{L}}({\mathcal{H}})\to {\mathrm{L}}({\mathcal{H}})$ as $\Phi(T)=\Psi(UTU^*)$. Since both $\Psi$ and the
unitary congruence with $U$ are sot- and wot-sequentially continuous, and trace preseverving, completely positive and unital, so is
$\Phi$. Further $\Phi(A)=X$.

Assume (ii) and suppose that $z\notin W_e(A) $ and $|z|<1$ in order to reach a contradiction. If $z=|z| e^{i\theta}$, replacing $A$
by $ e^{-i\theta}A$, we may assume $1>z\ge 0$. Hence,
$$
W_e((A+A^*)/2)\subset (-\infty, z]
$$
and there exists a selfadjoint compact operator $L$ such that
$$
\frac{A+A^*}{2}  \le zI +L.
$$
This implies that $X:=\frac{1+z}{2}I$ cannot be in the range of $\Phi$ for any unital, trace preserving positive linear map. Indeed,
we would have
$$
\frac{1+z}{2}I= \frac{X+X^*}{2}=\Phi\left( \frac{A+A^*}{2}\right)\le zI +\Phi(L)
$$
which is not possible as $\Phi(L)$ is compact.
\end{proof}

\vskip 5pt
In the finite dimensional setting, two Hermitian matrices $A$ and $X$ satisfy the relation $X=\Phi(A)$ for some positive, unital,
trace preserving linear map if and only if $X$ is in the convex hull of the unitary orbit of $A$. In the infinite dimensional
setting, if two Hermitian $A,X\in{\mathrm{L}}({\mathcal{H}})$ satisfy $W_e(A)\supset [-1,1]$ and $\| X\| \le 1$, then $X$ is in the
norm closure of the unitary orbit of $A$. This is easily checked by approximating the operators with diagonal operators. Such an
equivalence might not be brought out to the setting of Corollary \ref{cordoubly}.

\vskip 5pt
\begin{question} Do there exist $A,X\in {\mathrm{L}}({\mathcal{H}})$ such that $W_e(A)\supset {\mathcal{D}}$, $\|X\| <1$, and $X$
does not belong to the norm closure of the convex hull of the unitary orbit of $A$ ?
\end{question}

Here we mention a result of Wu \cite[Theorem 6.11]{Wu}: {\it If $A\in{\mathrm{L}}({\mathcal{H}})$ is not of the form scalar plus
compact,
then every $X\in{\mathrm{L}}({\mathcal{H}})$ is a linear combination of operators in the unitary orbit of $A$.}

 If one deletes the positivity assumption, the most regular class of linear maps on
${\mathrm{L}}({\mathcal{H}})$ might be given in the following definition.

\vskip 5pt
\begin{definition}\label{ultra} A linear map $\Psi: {\mathrm{L}}({\mathcal{H}})\to {\mathrm{L}}({\mathcal{H}})$ is said ultra-regular
if it fulfills two conditions:
\begin{itemize}
\item[(u1)] $\Psi(I)=I$ and $\Psi$ is trace preserving.
\item[(u2)] Whenever a sequence $A_n\to A$ for either the norm-, strong-, or weak-topology, then we also have $\Psi(A_n)\to \Psi(A)$
for the same type of convergence.
\end{itemize}
\end{definition}

\vskip 5pt
Any ultra-regular linear map preserves the set of essentially scalar operators (of the form $\lambda I + K$ with $\lambda\in\bC$ and
a compact operator $K$). For its complement, we state our last corollary.

\vskip 5pt
\begin{cor}\label{corultra} Let $A\in {\mathrm{L}}({\mathcal{H}})$ be essentially nonscalar. Then, for all $X\in
{\mathrm{L}}({\mathcal{H}})$ there exists a ultra-regular linear map $\Psi: {\mathrm{L}}({\mathcal{H}})\to
{\mathrm{L}}({\mathcal{H}})$ such that $\Psi(A)=X$.
\end{cor}

\vskip 5pt
\begin{proof} An operator is essentially nonscalar precisely when its essential numerical range is not reduced to a single point. So,
let $a,b\in W_e(A)$, $a\neq b$.
By a lemma of Anderson and Stampfli \cite{AndS}, $A$ is unitarily equivalent to an operator on ${\mathcal{H}}\oplus{\mathcal{H}}$ of
the form
$$
B=\begin{pmatrix} D & \ast \\
\ast & \ast
\end{pmatrix}
$$
where $D=\oplus_{n=1}^{\infty} D_n$, with two by two matrices $D_n$,
$$
D_n=
\begin{pmatrix} a_n & 0\\
0 & b_n\end{pmatrix}
$$
such that $a_n\to a$ and $b_n\to b$ as $n\to\infty$. We may assume that, for some $\alpha,\beta >0$, we have $\alpha >|a_n|+|b_n|$
and $|a_n-b_n|>\beta$. Hence there exist $\gamma >0$ and two by two intertible matrices $T_n$ such that, for all $n$,
$W(T_nD_nT_n^{-1})\supset {\mathcal{D}} $ and $\| T_n\| +\| T_n^{-1}\|\le \gamma$. So, letting $T=\left(\oplus_{n=1}^{\infty}
T_n\right)\oplus I$, we obtain an invertible operator $T$ on ${\mathcal{H}}\oplus{\mathcal{H}}$ such that $W_e(TBT^{-1})\supset
{\mathcal{D}}$.

Hence we have an invertible operator $S$ on ${\mathcal{H}}$ such that $W_e(SAS^{-1})\supset {\mathcal{D}}$. Therefore we may apply
Corollary \ref{cordoubly} and obtain a wot- and sot-sequentially continuous, unital, trace preserving map $\Phi$ such that
$\Phi(SAS^{-1})=X$. Letting $\Psi(\cdot)=\Phi(S\cdot S^{-1})$ completes the proof.
\end{proof}

\vskip 5pt
We cannot find an alternative proof, not based on the pinching theorem, for Corollaries
\ref{cordoubly} and \ref{corultra}. 

If we trust in Zorn, there exists a linear map $\Psi :{\mathrm{L}}({\mathcal{H}})\to {\mathrm{L}}({\mathcal{H}})$ which satifies the
condition (u1) but not the condition (u2). Indeed, let $\{a_p\}_{p\in\Omega}$ be a basis in the Calkin algebra
${\frak{C}}={\mathrm{L}}({\mathcal{H}})/{\mathrm{K}}({\mathcal{H}})$, indexed on an ordered set $\Omega$, whose first element
$a_{p_0}$ is the image of $I$ by the canonical projection $\pi: {\mathrm{L}}({\mathcal{H}})\to {\frak{C}}$. Thus, for each operator
$X$, we have a unique decomposition
$
\pi(X)= \sum_{p\in\Omega} (\pi(X))_p a_p
$
with only finitely many nonzero terms. Further $(\pi(X))_{p_0} = 0$ if $X$ is compact, and $(\pi(I))_{p_0} = 1$.
We then define a map  $\psi:{\mathrm{L}}({\mathcal{H}})\to  {\mathrm{L}}({\mathcal{H}}\oplus {\mathcal{H}})$ by
$$
\psi(X)= \begin{pmatrix} X&0 \\ 0& (\pi(X))_{p_0}I\end{pmatrix}.
$$
Letting $\Psi(X)=V\psi(X)V^*$ where $V:{\mathcal{H}}\oplus {\mathcal{H}}\to {\mathcal{H}}$ is unitary, we obtain a linear map $\Psi
:{\mathrm{L}}({\mathcal{H}})\to {\mathrm{L}}({\mathcal{H}})$ which satifies (u1) but not (u2): it is not norm continuous.

Let $\omega$ be a Banach limit on $l^{\infty}$ and define a map $\phi:l^{\infty}\to l^{\infty}$, $\{a_n\}\mapsto \{b_n\}$, where
$b_1=\omega(\{a_n\})$ and $b_n=a_{n-1}$, $n\ge 2$. Letting $\Psi(X)=\phi(diag(X))$, where $diag(X)$ is the diagonal of $X\in
{\mathcal{H}}$ in an orthonormal basis, we obtain a linear map $\Psi$ which is norm continuous, satisfies (u1) but not (u2): it is
not strongly sequentially continuous.

However, it seems not possible to define explicitly a linear map $\Psi :{\mathrm{L}}({\mathcal{H}})\to {\mathrm{L}}({\mathcal{H}})$
satisfying (u1) but not (u2).

\section{Pinchings in factors ?}

We discuss possible extensions to our results to a von Neumann algebra ${\mathfrak{R}}$ acting on a separable Hilbert space
${\mathcal{H}}$. First, we need to define an essential numerical range $W_e^{\mathfrak{R}}$ for ${\mathfrak{R}}$. Let
$A\in{\mathfrak{R}}$. If ${\mathfrak{R}}$ is type-${\mathrm{III}}$, then $W_e^{\mathfrak{R}}(A):=W_e(A)$. If ${\mathfrak{R}}$ is
type-${\mathrm{II}}_{\infty}$, then
$$
W_e^{\mathfrak{R}}(A):=\bigcap_{K\in {\mathcal{T}}}  \overline{W}(A+K)
$$
where ${\mathcal{T}}$ is the trace ideal in ${\mathfrak{R}}$ (we may also use its norm closure ${\mathcal{K}}$, the ``compact"
operators in ${\mathfrak{R}}$, or any dense sequence in ${\mathcal{K}}$)

\vskip 5pt
\begin{question} In Corollaries \ref{cor1seq}, \ref{cor2seq} and \ref{idem2seq}, can we replace ${\mathrm{L}}({\mathcal{H}})$ by a
type-${\mathrm{II}}_{\infty}$ or -${\mathrm{III}}$ factor ${\mathfrak{R}}$ with $W_e^{\mathfrak{R}}$ ?
\end{question}

\vskip 5pt
\begin{question} In Corollaries \ref{corKS} and \ref{corLW}, can we replace ${\mathrm{L}}({\mathcal{H}})$ by a
type-${\mathrm{II}}_{\infty}$ or -${\mathrm{III}}$ factor ${\mathfrak{R}}$ with $W_e^{\mathfrak{R}}$ ?
\end{question}

\vskip 5pt
\begin{question} In Corollaries \ref{cordoubly} and \ref{corultra}, can we replace ${\mathrm{L}}({\mathcal{H}})$ by a
type-${\mathrm{II}}_{\infty}$ factor ${\mathfrak{R}}$ with $W_e^{\mathfrak{R}}$ ?
\end{question}

\vskip 5pt
Recently, Dragan and Kaftal \cite{DK} obtained some decompositions for positive operators in von Neumann factors, which, in the case
of ${\mathrm{L}}({\mathcal{H}})$ were first investigated in \cite{BL-CR1}-\cite{BL-CR2} by using Theorem \ref{pinching}. This suggests that
our questions dealing with a possible extension to type-${\mathrm{II}}_{\infty}$ and -${\mathrm{III}}$ factors also have an
affirmative answer. In fact, it seems pausible that Theorem \ref{pinching} and Theorem \ref{pinchingnormal} admit a version for such
factors and this would affirmatively answer these questions.

Let ${\mathfrak{R}}$ be a type-${\mathrm{II}}_{\infty}$ or  -${\mathrm{III}}$ factor. 

\vskip 5pt
\begin{definition} A sequence $\{V_i\}_{i=1}^{\infty}$ of isometries in ${\mathfrak{R}}$ such that $\sum_{i=1}^{\infty} V_iV_i^*=I$
is called an isometric decomposition of ${\mathfrak{R}}$.
\end{definition}

\vskip 5pt
\begin{conjecture} Let $A\in{\mathfrak{R}} $ with $W_e^{\frak{R}}(A)\supset{\mathcal{D}}$ and $\{X_i\}_{i=1}^{\infty}$ a sequence in
${\mathfrak{R}}$ such that
$\sup_i\|X_i\|<1$. Then, there exists an isometric decomposition  $\{V_i\}_{i=1}^{\infty}$ of ${\mathfrak{R}}$ such that
 $V_i^*AV_i=X_i$ for all $i$.
\end{conjecture}

\section{Around this article}

\subsection{Essential numerical range}

Let us give three equivalent definitions of the essential numerical range $W_e(A)$ of an operator $A$ acting on the Hilbert space ${\mathcal{H}}$.

\begin{itemize}
\item[(1)]\ $W_e(A)=\cap\overline{W}(A+K)$, the intersection running over the compact operators $K$ 

\item[(2)]\ Let $\{E_n\}$ be any sequence of finite rank projections converging strongly to the 
identity and denote by $B_n$ the compression of $A$ to the subspace $E_n^{\perp}$. Then
$W_e(A)=\cap_{n\ge1}\overline{W}(B_n)$ 

\item[(3)]\ $W_e(A)=\{\lambda\ |\ {\rm there\ is\ an\ orthonormal\ system\ \{ e_n \}_{n=1}^\infty \
with\ \lim\langle e_n, Ae_n\rangle =\lambda} \}.$
\end{itemize}

It follows that $W_e(A)$ is a compact convex set containing the essential spectrum of $A$,
$Sp_e(A)$. The equivalence between these definitions has been known since the early seventies
if not soone. The very first definition of $W_e(A)=$ is (1); however (3)
is also a natural notion and easily entails convexity and compactness of the essential numerical
range.

\subsection{Proof of the pinching theorem}

Recall that an operator  mean an element of the algebra  ${\rm L}({\cal H})$ of all bounded linear
operators acting on the usual (i.e.\ complex, separable, infinite dimensional) Hilbert space 
${\cal H}$. We will denote by the same letter a projection and the corresponding subspace. Thus,
if $F$ is a projection and $A$ is an operator, we denote by $A_F$ the compression of $A$ by $F$,
that is the restriction of $FAF$ to the subspace $F$. Given a
total sequence of nonzero mutually orthogonal projections $\{E_n\}$, we consider the pinching
$$
{\cal P}(A)=\sum_{n=1}^{\infty}E_nAE_n=\bigoplus_{n=1}^{\infty}A_{E_n}.
$$
If $\{A_n\}$ is a sequence of operators acting on separable Hilbert spaces with $A_n$ 
unitarily equivalent to $A_{E_n}$ for all $n$, we also naturally write 
${\cal P}(A)\simeq\bigoplus_{n=1}^{\infty}A_n$. Our main result  \cite{B-JOT} for operator diagonals
can then be stated as:

\vskip 10pt\noindent
\begin{theorem}\label{th2003} Let $A$ be an operator with $W_e(A)\supset{\cal D}$
and let $\{A_n\}_{n=1}^{\infty}$ be a sequence of operators such that
$\sup_n\Vert A_n\Vert_{\infty}<1$. Then, we have a pinching
$$
{\cal P}(A)\simeq \bigoplus_{n=1}^{\infty}A_n.
$$
\end{theorem}

\vskip 5pt
We need two lemmas. The first one is Theorem \ref{th2003} for a single strict contraction:

\vskip 5pt\noindent
\begin{lemma}\label{lemA}Let $A$ be an operator with $W_e(A)\supset{\cal D}$ and let $X$ be a strict contraction. Then there exists a
projection $E$ such that $A_E=X$. \end{lemma}

\vskip 5pt\noindent
The second Lemma  is a refined version of the first one:

\vskip 5pt\noindent
\begin{lemma}\label{lemB}  Let $a\ge 1$ and $1>\rho>0$ be two constants. Let $h$ be a norm one vector, let $X$ be a strict contraction
with $\Vert X\Vert_{\infty}<\rho$ and let $B$ be an operator with $\Vert B\Vert_{\infty}\le a$ and $W_e(B)\supset{\cal D}$. Then,
there exist a number $\varepsilon>0$, only depending on $\rho$ and $a$, and a projection $E$ such that:
\vskip 5pt
\ {\rm (i)} $\dim E=\infty$ and $B_E=X$,
\vskip 5pt
{\rm (ii)} $\dim E^{\perp}=\infty$, $W_e(B_{E^{\perp}})\supset{\cal D}$ and $\Vert Eh\Vert\ge
\varepsilon$. 
\end{lemma}

\vskip 10pt\noindent
{\bf Proof of Theorem \ref{th2003}.} The proof is organized in five steps:

\noindent
{\it Step 1.} Some preliminaries are given.

\noindent
{\it Step 2.} Proof of Lemma \ref{lemA} in the special case when $X$ is  normal, diagonalizable.

\noindent
{\it Step 3.} Proof of Lemma \ref{lemA} in the general case.

\noindent
{\it Step 4.} Proof of Lemma \ref{lemB}.

\noindent
{\it Step 5.} Conclusion.

\vskip 10pt\noindent
\centerline{ {\it 1. Preliminaries} }
\vskip 7pt\noindent
We shall use a sequence $\{V_k\}_{k\ge1}$ of orthogonal matrices acting on spaces
of dimensions $2^k$. This sequence is built up by induction:
$$
V_1=\frac{1}{\sqrt2}\begin{pmatrix}1&1\\ -1&1\end{pmatrix}\quad{\rm then}\quad
V_k=\frac{1}{\sqrt2}\begin{pmatrix}V_{k-1}&V_{k-1}\\ -V_{k-1}&V_{k-1}\end{pmatrix}\quad{\rm for}\ k\ge2.
$$
Given a Hilbert space ${\cal G}$ and a decomposition
$$
{\cal G}=\bigoplus_{j=1}^{2^k}{\cal H}_j\quad {\rm with}\ {\cal H}_1=\dots={\cal H}_{2^k}
={\cal H},
$$
we may consider the unitary
(orthogonal) operator on ${\cal G}$ : $W_k=V_k\bigotimes I$,
where $I$ denotes the identity on ${\cal H}$, 

Now, let $B:{\cal G}\rightarrow{\cal G}$ be an operator which, with respect to the above
decomposition of ${\cal G}$, has a block diagonal matrix
$$
B=\begin{pmatrix}B_1&\ &\ \\ \ &\ddots&\ \\ \ &\ &B_{2^k}\end{pmatrix}.
$$
We observe that the block matrix representation of $W_kBW_k^*$ has its diagonal
entries all equal to
$$
\frac{1}{2^k}\left(B_1+\dots B_{2^k}\right).
$$
So, the orthogonal operators $W_k$ allow us to pass from a block diagonal matrix representation
to a block matrix representation in which
the diagonal entries are all equal.

\vskip 10pt
\centerline{ {\it 2. Proof of Lemma \ref{lemA} when $X$ is normal, diagonalizable.}  }
\vskip 7pt\noindent
Let $\{ \lambda_n(X)\}_{n\ge1}$ be the eigenvalues of $X$ repeated
according to their multiplicities. Since $|\lambda_n(X)|<1$ for all $n$ and  $W_e(A)
\supset{\cal D}$, we may find a norm one vector $e_1$ such that $\langle e_1,Ae_1\rangle=
\lambda_1(T)$. Let $F_1=[{\rm span}\{e_1,Ae_1,A^*e_1\}]^{\perp}$. As $F_1$ is of finite 
codimension, $W_e(A_{F_1})\supset{\cal D}$. So, there exists a norm one vector $e_2\in F_1$
such that $\langle e_2,Ae_2\rangle=\lambda_2(T)$. Next, we set
$F_2=[{\rm span}\{e_1,Ae_1,A^*e_1,e_2,Ae_2,A^*e_2\}]^{\perp}$, $\dots$. If we go on like this, 
we exhibit an orthonormal system $\{e_n\}_{n\ge1}$ such that, setting 
$E={\rm span}\{e_n\}_{n\ge1}$, we have $A_E=X$. 

\vskip 10pt\noindent
\centerline{ {\it 3. Proof of Lemma \ref{lemA} in the general case.}  }
\vskip 7pt\noindent
The contraction $Y=(1/\Vert X\Vert_{\infty})X$ can be dilated in a unitary 
$$
U=\begin{pmatrix}Y&-(I-YY^*)^{1/2}\\ (I-Y^*Y)^{1/2}&Y^*\end{pmatrix}
$$
thus $X$ can be dilated in a normal operator $N=\Vert X\Vert_{\infty} U$ with
$\Vert N\Vert_{\infty}<\rho$.
This permits to restrict  to the case when $X$ is a normal strict contraction. So, let
 $X$ be a  normal operator with $\Vert X\Vert_{\infty}<\rho<1$. 
 We remark with the Berg-Weyl-von Neumann theorem, that $X$  can be written as
$$
X=D+K \eqno (1)
$$
where $D$ is normal diagonalizable, $\Vert D\Vert_{\infty}=\Vert X\Vert_{\infty}<\rho$, and $K$ is compact with an arbitrarily small
norm. Let $K={\rm Re}K+{\rm iIm}K$ be the Cartesian decomposition of $K$.
We can find an integer $l$, a real $\alpha$
and a real $\beta$ such that  decomposition (1) satisfies:

\noindent
a) the operators $\alpha D$, $\beta{\rm Re}K$, $\beta{\rm Im}K$ are dominated in norm by $\rho$,

\noindent
b) there are positive integers $m$, $n$ with $2^l=m+2n$ and
$$
X=\frac{1}{2^l}(m\alpha D+n\beta{\rm Re}K+n\beta{\rm iIm}K). \eqno (2)
$$
More precisely we can take any $l$ such that $[2^l/(2^l-2)].\Vert X\Vert_{\infty}<\rho$. Next, assuming
$\Vert K\Vert_{\infty}<\rho/2^l$, we can take $m=2^l-2$, $n=1$, $\alpha=2^l/(2^l-2)$ and $\beta=2^l$.

Let then $T$ be the diagonal normal operator acting on the space
$$
{\cal G}=\bigoplus_{j=1}^{2^l}{\cal H}_j\quad {\rm with}\ {\cal H}_1=\dots={\cal H}_{2^l}
={\cal H},
$$
and defined by
$$
T=\left( \bigoplus_{j=1}^m D_j\right)\bigoplus
\left( \bigoplus_{j=m+1}^{m+n} R_j\right)\bigoplus
\left( \bigoplus_{j=m+n+1}^{2^l} S_j\right)
$$
where $D_j=\alpha D$, $S_j=\beta{\rm Re}K$ and $S_j=\beta{\rm iIm}K$.
\vskip 5pt\noindent
We note that $\Vert T\Vert_{\infty}<\rho<1$ and that the operator $W_lTW_l^*$, represented in the
preceding decomposition of ${\cal G}$, has its diagonal entries all equal to
$X$ by (2). Hence, applying the preceding step  to $T$ yields Lemma \ref{lemA}. 

\vskip 10pt\noindent
\centerline{ {\it 4. Proof of Lemma \ref{lemB}.}  }
\vskip 7pt\noindent
Let $a\ge 1$ and let $1>\rho>0$ be two constants. We take an arbitrary norm one vector $h$ and any operator $B$ satisfying to the
assumptions of Lemma \ref{lemB}. We can show, using the same reasoning as that applied in
the above  Step 2, that we have an orthonormal system $\{f_n\}_{n\ge0}$,
with $f_0=h$, such that:

\vskip 5pt\noindent
a) $\langle f_{2j}, Bf_{2j}\rangle=0$ for all $j\ge1$.

\noindent
b) $\{ \langle f_{2j+1}, Bf_{2j+1}\rangle \}_{j\ge0}$ is a dense sequence in ${\cal D}$. 

\noindent
c) If $F={\rm span}\{f_j\}_{j\ge0}$, then $B_F$ is the normal operator
$$
\sum_{j\ge0}\langle f_j, Bf_j\rangle f_j\otimes f_j.
$$
Setting $F_0={\rm span}\{f_{2j}\}_{j\ge0}$ and $F'_0={\rm span}\{f_{2j+1}\}_{j\ge0}$, we then have:

\vskip 5pt\noindent 
a) With respect to the decomposition $F=F_0\bigoplus F'_0$, $B_F$ can be written
$$
B_F=\begin{pmatrix}B_{F_0}&0\\ 0&B_{F'_0}\end{pmatrix}.
$$

\noindent 
b) $W_e(B_{F'_0})\supset{\cal D}$ and $h\in F_0$.

\vskip 5pt\noindent
We can then write a decomposition of  $F'_0$, $F'_0=\bigoplus_{j=1}^{\infty}F_j$
where for each index $j$, $F_j$ commutes with $B_F$ and $W_e(B_{F_j})\supset{\cal D}$;
so that the decomposition $F=\bigoplus_{j=0}^{\infty}F_j$ yields a representation
of $B_F$ as a block diagonal matrix, 
$$
B_F=\bigoplus_{j=0}^{\infty}B_{F_j}.
$$
Since $W_e(B_{F_j})\supset{\cal D}$ when $j\ge1$, the same reasoning as  in Step 3 
 entails that for any sequence  
$\{ X_j \}_{j\ge0}$ of strict contractions we have decompositions
$(\dagger)$ $F_j=G_j\bigoplus G'_j$ allowing us to write, for $j\ge1 $,
$$  
B_{F_j}=\begin{pmatrix}X_j&*\\ \ast&*\end{pmatrix}.
$$
Since $\Vert X\Vert_{\infty}<\rho<1$ and $\Vert B\Vert_{\infty}\le a$, we can find an integer $l$ only depending on $\rho$
and $a$, as well as strict contractions $X_1,\dots,X_{2^l}$, such that
$$
X=\frac{1}{2^l}\left( B_{F_0}+\sum_{j=1}^{2^l-1}X_j \right). \eqno (3)
$$
Considering decompositions $(\dagger)$ adapted to these $X_j$,  we set
$$
G=F_0\bigoplus \left( \bigoplus_{j=1}^{2^l-1}G_j \right). 
$$
With respect to this decomposition,
$$
B_G=\begin{pmatrix}B_{F_0}&\ &\ &\ \\
\ &X_1&\ &\ \\
\ &\ &\ddots&\ \\
\ &\ &\ &X_{2^l-1}\end{pmatrix}.
$$ 
Then we deduce from (3) that the block matrix $W_lB_GW_l^*$ has its diagonal entries 
all equal to $X$.
\vskip 5pt
Summary:
$h\in G$ and there exists a decomposition $G=\bigoplus_{j=1}^{2^l}E_j$, in which $l$ depends only on $\rho$ and $a$, such that
$B_{E_j}=X$ for each $j$. Thus we have an integer $j_0$ such that, setting
$E_{j_0}=E$, we have
$$
B_E=X \quad {\rm and} \quad \Vert Eh\Vert\ge \frac{1}{\sqrt{2^l}}.                                        
$$ 
Taking $\varepsilon=1/{\sqrt{2^l}}$ ends the proof of Lemma \ref{lemB}.

\vskip 10pt\noindent
\centerline{ {\it 4. Conclusion.}  }
\vskip 7pt\noindent
Fix a dense sequence $\{ h_n\}$ in the unit sphere of ${\cal H}$ and set $a=\Vert A\Vert_{\infty}$. We claim that the statement (i)
and (ii) of Lemma  \ref{lemB} ensure that
there exists a sequence of mutually orthogonal projections $\{E_j\}$ such that, setting
$F_n=\sum_{j\le n}E_j$, we have for all integers $n$:
\vskip 5pt\noindent
\quad \ ($*$) $A_n=A_{E_n}$ and $W_e(A_{F_n^{\perp}})\supset{\cal D}$ \
(so $\dim F_n^{\perp}=\infty$),
\vskip 5pt\noindent
\quad ($**$) $\Vert F_n h_n\Vert \ge \varepsilon$.
\vskip 5pt\noindent
In Lemma \ref{lemB}, set $a=\Vert A\Vert_{\infty}$.
Replacing $B$ by $A$, Lemma  \ref{lemB} proves ($*$) and ($**$) for $n=1$. Suppose this holds for an $N\ge1$. Let
$\nu(N)\ge N+1$ be the first integer for which $F_N h_{\nu(N)} \neq 0$. Note that
$\Vert A_{F_N^{\perp}}\Vert_{\infty} \le \Vert A\Vert_{\infty}$. We apply Lemma  \ref{lemB} to $B=A_{F_N^{\perp}}$,
$X=A_{N+1}$ and $h=F_N h_{\nu(N)} / \Vert F_N h_{\nu(N)}\Vert$. 
We then deduce that $(*)$ and $(**)$ are still valid for $N+1$. Therefore
$(*)$ and $(**)$ hold for all $n$. Denseness of $\{ h_n\}$ and $(**)$ show that $F_n$ strongly
increases to the identity $I$ so that $\sum_{j=1}^{\infty}E_j=I$ as required. \quad $\Box$

\subsection{M\"uller-Tomilov's theorem}

Define the diagonal set $\Delta(A)$ of an operator $A$ as the scalars $\lambda\in\bC$ such that
$$
\lambda=\langle e_n, Ae_n\rangle, \ n=1,2,\ldots
$$
for some orthonormal basis $\{e_n\}_{n=1}^{\infty}$. One has
$$
 {\mathrm{int}}\, W_e(A) \subset \Delta (A) \subset W_e(A)
$$
The set $\Delta(A)$ is an analytic set, I do not know examples where $\Delta(A)$ is not Borel.
The boundary of $\Delta(A)$ might be quite complicated.
In 2003 \cite{B-JOT}, I asked wether $\Delta(A)$ is always convex. M\"uller and Tomilov give a positive answer in the recent paper \cite{MT}.

\section{References of Chapter 5}

{\small
\begin{itemize}

\item[[1\!\!\!]] C.A.\ Akemann and J.\ Anderson. Lyapunov theorems for operator algebras. {\it Mem. Amer.
Math. Soc.,} 94 no 458, 1991.

\item[[3\!\!\!]]  J.H.\ Anderson and J.G.\ Stampfli, Commutators and compressions,  {\it Israel J.\ Math.\ }
10 (1971), 433--441.

\item[[5\!\!\!]]  T.\ Ando, Majorization, doubly stochastic matrices, and comparison of eigenvalues,
{\it Linear Algebra Appl.\ } 118 (1989), 163-248.

\item[[21\!\!\!]]  J.-C. Bourin, Compressions and pinchings, {\it J.\ Operator Theory} 50 (2003), no.\ 2, 211-220.

\item[[33\!\!\!]] J.-C. Bourin and E.-Y. Lee, Sums of Murray-von Neumann equivalent positive operators,
{\it C.\ R.\  Math.\ Acad.\ Sci.\ Paris} 351 (2013), no.\ 19-20, 761-764.

\item[[34\!\!\!]] 
J.-C. Bourin and E.-Y. Lee, Sums of unitarily equivalent positive operators, {\it C.\ R.\ Math.\ Acad.\ Sci.\ Paris} 352 (2014), no.\
5, 435--439.

\item[[37\!\!\!]] J.-C.\ Bourin and E.-Y. Lee, Pinchings and positive linear maps, {\it J.\ Funct.\ Anal.}\ 270 (2016), no.\ 1, 359--374.

\item[[58\!\!\!]] C.\ Dragan and V.\ Kaftal, Sums of equivalent sequences of positive operators in von Neumann factors, preprint,
arXiv:1504.03193.

\item[[60\!\!\!]] P.R.\ Halmos, \ Numerical ranges and normal dilations, {\it Acta Sci.\ Math.\ (Szeged)} 25 (1964) 1--5. 

\item[[71\!\!\!]]  R.\ Kadison, 
The Pythagorean theorem. II. The infinite discrete case, {\it Proc.\ Natl.\
Acad.\ Sci.\ } USA 99 (2002), no.\ 8, 5217-5222.

\item[[72\!\!\!]]  R.\ Kadison and I.\ Singer, Extensions of pure states, {\it Amer.\ J.\ Math.}\ (1959),
383-400.

\item[[73\!\!\!]]  M.\ Kennedy and P.\ Skoufranis, The Schur-Horn Problem for Normal Operators, {\it Proc.\ London Math.\ Soc.,} in press,
arXiv:1501.06457.

\item[[78\!\!\!]] J.\ Loreaux and G.\ Weiss, Diagonality and idempotents with applications to problems in operator theory and frame
theory, {\it J.\ Operator Theory}, in press, arXiv:1410.7441.

\item[[82\!\!\!]] V. M\"uller, Y.\ Tomilov, In search of convexity: diagonals and numerical ranges, {\it Bull.\ London Math soc.}, 53 (2021), no.\ 4, 1016--1029.

\item[[85\!\!\!]] C.\ Pearcy and D.\ Topping, Sums of small numbers of idempotents, {\it  Michigan J.\ Math.\ }
14 (1967) 453--465.

\item[[87\!\!\!]] M.\ Ravichandran, The Schur-Horn theorem in von Neumann algebras, preprint, arXiv:1209.0909.

\item[[93\!\!\!]] P.Y.\ Wu, Additive combination of special operators, {\it Banach Center Publ.} 30  (1994), 337-361.

\end{itemize}


\chapter{Partial trace}

{\color{blue}{\Large {\bf Decomposition and partial trace of  positive   matrices with Hermitian blocks
} \large{\cite{BL-IJM1}}}}

\vskip 10pt\noindent
{\bf Abstract.}
\noindent Let $H=[A_{s,t}]$ be a positive definite matrix written in $\beta\times\beta$ Hermitian blocks and let
$\Delta=A_{1,1}+\cdots+A_{\beta,\beta}$ be its partial trace. Assume that $\beta=2^p$ for some $p\in\bN$. Then, up to a direct sum
operation, $ H$ is the average of $\beta$ matrices isometrically congruent to $ \Delta$.
A few corollaries are given, related to important inequalities in quantum information theory such as the Nielsen-Kempe separability
criterion.

{\small\noindent
Keywords: Positive definite matrices,  norm inequalities, partial trace, separable state.

\noindent
AMS subjects classification 2010:  15A60, 47A30,  15A42.}

\section{Introduction and a key lemma}

Positive semi-definite matrices partitioned in two by two blocks occur as an efficient tool in matrix analysis, sometimes a magic
tool ! $-$ according to Bhatia's famous book \cite{Bh}. These partitions allow to derive a lot of important inequalities and
those with Hermitian blocks shed much light on the geometric and harmonic matrix means. Partitions into a larger number of blocks are
naturally involved with tensor products, in the theory of positive linear maps and in their application in quantum physics.

This article deals with positive matrices partitioned in Hermitian blocks. By using unitary or isometry congruences, we will improve
some nice majorisations, or norm estimates, first obtained in the field of quantum information theory.

For  partitioned positive matrices,
the diagonal blocks play a quite special role. This is apparent in a rather striking decomposition due to the authors \cite{BL-London}.

\begin{lemma} \label{BL-lemma} For every matrix in  $\bM_{n+m}^+$ written in blocks, we have a decomposition
\begin{equation*}
\begin{bmatrix} A &X \\
X^* &B\end{bmatrix} = U
\begin{bmatrix} A &0 \\
0 &0\end{bmatrix} U^* +
V\begin{bmatrix} 0 &0 \\
0 &B\end{bmatrix} V^*
\end{equation*}
for some unitaries $U,\,V\in  \bM_{n+m}$.
\end{lemma}

This lemma leads to study partitions via unitary congruences. It is the key of the subsequent results. A proof and several
consequences can be found in \cite{BL-London} and \cite{BH1}.
Of course, $\bM_n$ is the algebra of $n\times n$ matrices with real or complex entries, and $\bM_n^+$ is the positive part. That is,
$\bM_n$ may stand either for $\bM_n(\bR)$, the matrices with real entries, or for $\bM_n(\bC)$, those with complex entries. The
situation is different
in the next statement, where complex entries seem unavoidable. 

\vskip 10pt
\begin{theorem}\label{thm-four} Given any matrix in $\bM_{2n}^+(\bC)$ written in blocks in $\bM_n(\bC)$ with Hermitian off-diagonal
blocks, we have
\begin{equation*}
\begin{bmatrix} A &X \\
X &B\end{bmatrix}= \frac{1}{2}\left\{ U(A+B)U^* +V(A+B)V^*\right\}
\end{equation*} for some isometries $U,V\in\bM_{2n,n}(\bC)$.
\end{theorem}

\vskip 10pt
Here $\bM_{p,q}(\bC)$ denote the space of $p$ rows and $q$ columns matrices with complex entries, and $V\in\bM_{p,q}(\bC)$ is an
isometry if $p\ge q$ and $V^*V=I_q$. Even for a matrix in $\bM_{2n}^+(\bR)$, it seems essential to use isometries with complex
entries ! The result, due to Lin and the authors, is based on Lemma \ref{BL-lemma}, a proof is in \cite{BLL2} and implicitly in
\cite{BLL1}.

There is no evidence whether a positive block-matrix $H$ in $\bM^+_{3n}$, 
$$
H=\begin{bmatrix}
A&X&Y \\ X&B&Z \\Y&Z&C
\end{bmatrix}
$$
with Hermitian off-diagonal blocks $X,Y,Z$, could be decomposed as
$$
H=\frac{1}{3}\left\{ U\Delta U^*+ V\Delta V^* +W\Delta W^*\right\}
$$
where $\Delta=A+B+C$ and $U,V,W$ are isometries. In fact, this would be surprising. However, a quite nice decomposition is possible
by considering direct sum copies: this provides a substitute to Theorem \ref{thm-four} for partitions into an arbitrary number of
blocks; it is the main result of this article.

These decompositions entail some nice inequalities. Lemma \ref{BL-lemma} yields a simple estimate for all symmetric (or unitarily
invariant) norms,
\begin{equation}\label{fact-0}
\left\|\begin{bmatrix} A &X \\
X^* &B\end{bmatrix} \right\|\le \| A\|+\| B \|. 
\end{equation}
Recall that a symmetric norm on $\bM_m$ satisfies $\|A\|=\|UA\|=\|AU\|$ for all $A\in\bM_m$ and all unitaries $U\in\bM_m$. This
obviously induces a symmetric norm on $\bM_n$, $1\le n\le m$. The most familiar symmetric norms are the Schatten $p$-norms, $1\le p<
\infty$,
\begin{equation}\label{Schatt}
\| A\|_p = \{{\mathrm{Tr\,}} (A^*A)^{p/2}\}^{1/p},
\end{equation}
and, with $p\to\infty$, the operator norm. In general, the sum of the norms $\|A\| +\| B\|$ can not be replaced in \eqref{fact-0} by
the norm of the sum $\| A+B \|$. However, Theorem \ref{thm-four} implies the following remarkable corollary.

\begin{cor} Given any matrix in $\bM_{2n}^+$ written in blocks in $\bM_n$ with Hermitian off-diagonal blocks, we have
\begin{equation*}
\left\|\begin{bmatrix} A &X \\
X &B\end{bmatrix} \right\|\le \| A+B \|
\end{equation*} for all symmetric norms.
\end{cor}

This is the simplest case of Hiroshima's theorem, discussed in the next section.
There are some positive  matrices in $\bM_6$ partitioned in blocks in $\bM_3$, with  normal off-diagonal blocks $X$, $X^*$, such that\begin{equation*}
\left\| \begin{pmatrix} A& X\\ X^*&B \end{pmatrix} \right\|_{\infty} > \| A+B\|_{\infty}.
\end{equation*}
Hence the assumptions  are rather optimal. 

In Section 2, we state our decomposition and derive several inequalities, most of them related to Hiroshima's theorem. The proof of
the decomposition is given in Section 3. A discussion of previous results for small partitions and some remarks related to quantum
information are given in the last section.

\section{Direct sum and partial trace}

A typical example of positive matrices written in blocks are formed by tensor products. Indeed, 
the tensor product $A\otimes B$ of $A\in\bM_{\beta}$ with $B\in\bM_n$ can be identified with an element of
$\bM_{\beta}(\bM_n)=\bM_{\beta n}$. Starting with positive matrices in $\bM_{\beta}^+$ and $\bM_n^+$ we then get a matrix in
$\bM_{\beta n}^+$ partitioned in blocks in $\bM_n$. In quantum physics, sums of tensor products of positive semi-definite (with trace
one) occur as so-called separable states. In this setting of tensor products, the sum of the diagonal block is called the partial
trace (with respect to $\bM_{\beta}$). We will use this terminology.

\vskip 10pt
\begin{theorem}\label{thm-direct} Let $H=[A_{s,t}]\in \bM_{\beta n}^+$ be written in $\beta\times \beta$ Hermitian blocks in $\bM_n$
and let $\Delta=\sum_{s=1}^{\beta}{A_{s,s}}$ be its partial trace. If $\beta$ is dyadic, then, with $m=2^{\beta}$, we have
\begin{equation*}
\oplus^{m} H =\frac{1}{\beta} \sum_{k=1}^{\beta} V_k\left(\oplus^{m}\Delta\right) V_k^*
\end{equation*} 
where   $\{V_k\}_{k=1}^{\beta}$  is a family of isometries in $\bM_{m\beta n,mn}$.
\end{theorem}

\vskip 10pt
Here the spaces $\bM_n$ and $\bM_{p,q}$ denote either the real or complex spaces of matrices. By a dyadic number $\beta$, we mean
$\beta=2^p$ for some $p\in\bN$.

This theorem has strong links with quantum information theory (QIT) as detailed in Section 4. Researchers in QIT may like to restate
the theorem by replacing direct sums with tensors products, $\oplus^{m} H \rightarrow I_m\otimes H $ and $\oplus^{m}\Delta\rightarrow
I_m\otimes \Delta$. Tensoring in identity means that an operator on a Hilbert space ${\mathcal H}$ is lifted to an operator acting on
${\mathcal F}\otimes {\mathcal H}$ where ${\mathcal F}$ is an auxiliary Hilbert space, an {\it ancilla} space in the QIT terminology.

\subsection{Around Hiroshima's theorem}

A straightforward application of Theorem \ref{thm-direct} is the following beautiful result first proved by Hiroshima in 2003 (see
Section 4 for the complete form of Hiroshima's theorem and is relevance in quantum physics).

\vskip 10pt
\begin{cor}\label{Hiroshima} Let $H=[A_{s,t}]\in \bM_{\alpha n}^+$ be written in $\alpha\times \alpha$ Hermitian blocks in $\bM_n$
and let $\Delta=\sum_{s=1}^{\alpha}{A_{s,s}}$ be its partial trace. Then, we have
\begin{equation*}
\left\|H \right\|\le \left\| \Delta\right\|
\end{equation*} for all symmetric norms.
\end{cor}

\vskip 10pt
\begin{proof} By completing $H$ with some zero rows and columns, we may assume that $\alpha=\beta$ is dyadic. Theorem
\ref{thm-direct} then implies, with $m=2^{\beta}$,
$$
\| \oplus^m H \| \le \| \oplus^m \Delta \| 
$$
for all symmetric norms, which is equivalent to the claim of the corollary.
\end{proof}

\vskip 10pt
By the Ky Fan principle, Corollary \ref{Hiroshima} is equivalent to the majorisation relation
$$
\sum_{i=1}^j\lambda_{i}(H) \le \sum_{i=1}^j\lambda_{i}(\Delta)
$$
for all $j=1,\ldots,\alpha n$. (we set $\lambda_j(A)=0$ when $A\in\bM_d^+$ and $j>d$).

Theorem \ref{thm-direct} says much more than this majorisation. For instance, we may completes these eigenvalue relations with the
following ones.

\vskip 10pt
\begin{cor}\label{cor-eigenvalue1} Let $H=[A_{s,t}]\in \bM_{\alpha n}^+$ be written in $\alpha\times \alpha$ Hermitian blocks in
$\bM_n$ and let $\Delta=\sum_{s=1}^{\alpha}{A_{s,s}}$ be its partial trace. Then, we have
\begin{equation*}
\lambda_{1+\beta k}(H) \le \lambda_{1+k}( \Delta)
\end{equation*} 
 for all $k=0,\ldots,n-1$,    where $\beta$ is the smallest dyadic number such that $\alpha \le\beta$.
\end{cor}

\vskip 10pt
\begin{proof} By completing $H\in\bM_{\alpha n}^+$ with some zero blocks, we may assume that $H\in\bM_{\beta n}^+$. Theorem
\ref{thm-direct} then yields the decomposition
\begin{equation*}
\oplus^{m} H =\frac{1}{\beta} \sum_{k=1}^{\beta} V_k\left(\oplus^{m}\Delta\right) V_k^*
\end{equation*} 
where   $\{V_k\}_{k=1}^{\beta}$  is a family of isometries in $\bM_{m\beta n,mn}$ and $m=2^{\beta}$.
We recall a simple fact, Weyl's theorem: if $Y, Z \in\bM_d$ are Hermitian,  then
\[\lambda_{r+s+1}(Y+Z)\le \lambda_{r+1}(Y) + \lambda_{s+1}(Z)\]
for all nonnegative integers $r,s$ such that $r + s\le d-1$. When $Y,Z$ are positive, this still holds for all nonnegative integers
$r,s$ with our convention ($\lambda_j(A)=0$ when $A\in\bM_d^+$ and $j>d$). From the previous decomposition we thus infer
\begin{equation}\label{F-1}
\lambda_{1+\beta k}\left(\oplus^m H\right) \le \lambda_{1+ k}\left(\oplus^m \Delta\right)
\end{equation}
for all $k=0,1,\ldots$. Then,  observe that for all $A\in\bM_d^+$ and all $j=0,1,\ldots$,
\begin{equation}\label{F-2}
\lambda_{1+j}\left(\oplus^m A\right) = \lambda_{\langle(1+j)/m\rangle}(A)
\end{equation}
where $\langle u\rangle$ stands for the smallest integer greater than or equal to $u$. Combining \eqref{F-1} and \eqref{F-2} 
 we get
\begin{equation*}
\lambda_{\langle (1+\beta j)/m\rangle}\left(H\right) \le \lambda_{\langle (1+ j)/m\rangle}\left( \Delta\right)
\end{equation*}
for all $j=0,1,\ldots$. Taking $j=km$, $k=0,1\ldots$ completes the proof.
\end{proof}

\vskip 10pt
The above proof actually shows more eigenvalue inequalities.

\vskip 10pt
\begin{cor}\label{cor-eigenvalueC} Let $H=[A_{s,t}]\in \bM_{\alpha n}^+$ be written in $\alpha\times \alpha$ Hermitian blocks in
$\bM_n$ and let $\Delta=\sum_{s=1}^{\alpha}{A_{s,s}}$ be its partial trace. Then, we have
$$
\lambda_{1+\beta k}( S) \le \frac{1}{\beta}\left\{ \lambda_{1+k_1}\left( \Delta \right)
+\cdots+ \lambda_{1+k_{\beta}}\left( \Delta \right)\right\}
$$
where $k_1+\cdots+k_{\beta}=\beta k$ and $\beta$ is the smallest dyadic number such that $\alpha \le\beta$.
\end{cor}

\vskip 10pt
Corollary \ref{Hiroshima} implies the following rearrangement inequality.

\vskip 10pt
\begin{cor}\label{cor-rearr}
 Let $\{S_i\}_{i=1}^{\alpha}$ be a commuting family of Hermitian operators in $\bM_n$ and let $T\in \bM_n^+$. Then,
\begin{equation*}
\left\| \sum_{i=1}^{\alpha} S_iT^2S_i \right\| \le \left\|  \sum_{i=1}^{\alpha} TS^2_iT \right\|
\end{equation*}
for all symmetric norms.
\end{cor}

\vskip 10pt
\begin{proof}
Define a matrix $Z\in\bM_{\alpha n}$ by
$$ Z= XX^*=
\begin{bmatrix} TS_1\\ \vdots \\  TS_{\alpha}
\end{bmatrix}
\begin{bmatrix} S_1T & \cdots  &S_{\alpha}T
\end{bmatrix}.
$$
Hence $Z=[TS_iS_jT]$ is positive and partitioned in Hermitian blocks in $\bM_n$, with diagonal blocks $TS_i^2T$, $1\le i\le \alpha$.
Thus, for all symmetric norms,
$$
\| Z\| \le \left\|\sum_{i=1}^{\alpha} TS_i^2T \right\|
$$
Since $XX^*$ and $X^*X$ have  same symmetric norms for any rectangular matrix $X$, we infer
\begin{equation*}
\left\|\sum_{i=1}^{\alpha}  S_iT^2S_i \right\| \le \left\|\sum_{i=1}^{\alpha} TS_i^2T \right\|
\end{equation*}
as claimed.
\end{proof}

\vskip 10pt
From Corollary \ref{cor-eigenvalue1} we similarly get the next one.

\vskip 10pt
\begin{cor}\label{cor-rearr2}
 Let $\{S_i\}_{i=1}^{\alpha}$ be a commuting family of Hermitian operators in $\bM_n$ and let $T\in \bM_n^+$. Then,
\begin{equation*}
\lambda_{1+\beta k}\left( \sum_{i=1}^{\alpha} S_iT^2S_i \right) \le \lambda_{1+k} \left(  \sum_{i=1}^{\alpha} TS^2_iT \right)
\end{equation*}
 for all $k=0,\ldots,n-1$,    where $\beta$ is the smallest dyadic number such that $\alpha \le\beta$.
\end{cor}

\subsection{Around Rotfel'd inequality}

\noindent
 Given two Hermitian matrices $A$, $B$ in $\bM_n$ and a concave function $f(t)$ defined on the real line,
\begin{equation}\label{VN}
{\mathrm{Tr\,}} f\left(\frac{ A+B}{2}\right) \ge {\mathrm{Tr\,}}\frac{f(A) +f(B)}{2}
\end{equation}
and, if further $f(0)\ge 0$ and both $A$ and $B$ are positive semi-definite,
\begin{equation}\label{Rot}
{\mathrm{Tr\,}} f(A+B) \le {\mathrm{Tr\,}} f(A)+ {\mathrm{Tr\,}}f(B).
\end{equation}
The first inequality goes back to von-Neumann in the 1920's, the second is more subtle and has been proved only in 1969 by Rotfel'd
\cite{Rot}. These trace inequalities are matrix versions of obvious scalar inequalities. Theorem \ref{thm-direct} yields a refinement
of the Rotfel'd inequality for families of positive operators $\{A_i\}_{i=1}^{\alpha}$ by considering these operators as the diagonal
blocks of a partitioned matrix $H$ as follows.

\vskip 10pt
\begin{cor}\label{cor-rot} Let $H=[A_{s,t}]\in \bM_{\alpha n}^+$ be written in $\alpha\times \alpha$ Hermitian blocks in $\bM_n$.
Then, we have
$$
{\mathrm{Tr}}\, f\left(\sum_{s=1}^{\alpha} A_{s,s}\right)
  \le  {\mathrm{Tr}}\,f(H)
 \le  \sum_{s=1}^{\alpha}{\mathrm{Tr}}\,f(A_{s,s})
$$
for all concave functions $f(t)$ on $\bR^+$ such that $f(0)\ge 0$.
\end{cor}

\vskip 10pt
\begin{proof}
Note that if these inequalities hold for a non-negative concave function $f(t)$ with $f(0)=0$, then they also hold for the function
$f(t)+c$ for any constant $c>0$. Therefore it suffice to consider concave functions vanishing at the origine. This assumption entails
that
\begin{equation}\label{FF-1}
f(VAV^*)=Vf(A)V^*
\end{equation}
for all $A\in\bM_n^+$ and all isometries $V\in\bM_{m,n}$. By Lemma \ref{BL-lemma} we have a decomposition
$$
H=\sum_{s=1}^{\alpha} V_sA_{s,s}V_s^*
$$
for some isometries $V_s\in\bM_{\alpha n,n}$. Inequality \eqref{Rot} then yields
$$
 {\mathrm{Tr}}\,f(H) \le \sum_{s=1}^{\alpha} {\mathrm{Tr}}\,f(V_sA_{s,s}V_s^*)
$$
and using \eqref{FF-1} establishes the second inequality. To prove the first inequality, we use Theorem \ref{thm-direct}. By
completing $H$ with some zero blocks (we still suppose $f(0)=0$) we may assume that $\alpha=\beta$ is dyadic. Thus we have a
decomposition
\begin{equation*}
\oplus^{m} H =\frac{1}{\beta} \sum_{k=1}^{\beta} V_k\left(\oplus^{m}\Delta\right) V_k^*
\end{equation*} 
where   $\{V_k\}_{k=1}^{\beta}$  is a family of isometries in $\bM_{m\beta n,mn}$ and $m=2^{\beta}$. Inequality \eqref{VN} then gives$$ {\mathrm{Tr}}\, f\left(\oplus^{m} H \right) \ge 
\frac{1}{\beta}\sum_{k=1}^{\beta}{\mathrm{Tr}}\,f\left( V_k\left(\oplus^{m}\Delta\right) V_k^*\right)
$$
and using \eqref{FF-1} we obtain
$$ {\mathrm{Tr}}\ f\left(\oplus^{m} H \right) \ge {\mathrm{Tr}}\ f\left(\oplus^{m} \Delta \right).
$$
The proof is completed by dividing both sides by $m$.
\end{proof}

\vskip 10pt
\begin{remark} The first inequality of Corollary \ref{cor-rot} is actually equivalent to Corollary \ref{Hiroshima}
by a well-known majorisation principle for convex/concave functions. The above proof does not require this principle. The simplest
case of Corollary \ref{cor-rot} is the double inequality
$$
 f(a_1+\cdots+a_n) \le {\mathrm{Tr}}\, f(A)
\le f(a_1)+\cdots +f(a_n)
$$
for all $A\in\bM_n^+$ with diagonal entries $a_1,\ldots,a_n$.

\end{remark}

\vskip 10pt
\begin{remark} A special case of Corollary \ref{cor-rot} refines a well-known determinantal inequality. Taking as a concave function
on $\bR^+$, $f(t)=\log(1+t)$, we obtain:
Let $A,B\in \bM_n^+$. Then, for any Hermitian $X\in\bM_n$  such that 
$$H=\begin{bmatrix} A&X \\X&B\end{bmatrix}$$
is positive semi-definite, we have 
\begin{equation*}
\det(I+A+B)
\le \det(I+H) \le \det(I+A)\det(I+B).
\end{equation*}
This was noted in \cite{BLL2}.
\end{remark}

\section{Proof of Theorem 2.1}

\noindent
A Clifford algebra ${\mathcal{C}}_{\beta}$ is the associative real algebra generated by $\beta$ elements $q_1,\ldots,q_{\beta}$
satisfying the canonical anticommutation relations $q_i^2=1$ and
$$
q_iq_j +q_jq_i =0 
$$
for $i\neq j$. This structure was introduced by Clifford in \cite{Cli}. It turned out to be of great importance in quantum theory and
operator algebras, for instance see the survey \cite{Der}. From the relation
$$
\begin{pmatrix}
0&1 \\
1&0
\end{pmatrix}
\begin{pmatrix}
1&0 \\
0&-1
\end{pmatrix}
+
\begin{pmatrix}
1&0 \\
0&-1
\end{pmatrix}
\begin{pmatrix}
0&1 \\
1&0
\end{pmatrix}=0
$$
we infer a representation of ${\mathcal{C}}_{\beta}$ as a a real subalgebra of $M_{2^{\beta}}=\otimes^{\beta} \bM_2$ by mapping the
generators $q_j\mapsto Q_j$, $1\le j\le \beta$, where
\begin{equation}\label{cli-gen}
Q_j=\left\{\otimes^{j-1}\begin{pmatrix}
1&0 \\
0&-1
\end{pmatrix}\right\}\otimes
\begin{pmatrix}
0&1 \\
1&0
\end{pmatrix}
\otimes
\left\{\otimes^{\beta-j}\begin{pmatrix}
1&0 \\
0&1
\end{pmatrix}
\right\}.
\end{equation}
We use these matrices in the following proof of Theorem \ref{thm-direct}.

\vskip 10pt
\begin{proof} 
First, replace the positive block matrix $H=[A_{s,t}]$ where $1\le s,t,\le \beta$
and all blocks are Hermitian by a bigger one in which each block in counted $2^{\beta}$ times :
$$G= [G_{s,t}]:= \left[I_{2^{\beta}}\otimes A_{s,t}\right]= \left[\oplus^{2^{\beta}} A_{s,t}\right]$$
where $I_r$ stands for the identity of $\bM_r$.
Thus $G\in\bM_{\beta 2^{\beta}n}$ is written in $\beta$-by-$\beta$ blocks in $\bM_{2^{\beta}n}$. Then perform a unitary congruence
with the unitary $W\in\bM_{\beta2^{\beta}n}$ defined as
\begin{equation}\label{unitaryW}
W=\bigoplus_{j=1}^{\beta} \left\{Q_j \otimes I_n\right\}
\end{equation}
where $Q_j$ is given by \eqref{cli-gen}, $1\le j\le \beta$.
Thanks to the anticommutation relation for each pair of summands in
 \eqref{unitaryW},
$$
\left\{Q_j \otimes I_n\right\}\left\{Q_l \otimes I_n\right\}+\left\{Q_l \otimes I_n\right\}\left\{Q_j \otimes I_n\right\}=0, \quad
j\neq l,
$$
 the block matrix (with $W=W^*$)
\begin{equation}\label{omega1}
\Omega:=WGW^*=[ \Omega_{s,t}]
\end{equation}
satisfies the following  :
For $1\le s<t\le \beta$,
\begin{equation}\label{omega2}
\Omega_{s,t}=-\Omega_{t,s}.
\end{equation}

Next, consider the reflexion matrix
$$
J_1=\frac{1}{\sqrt{2}} \begin{pmatrix} 1&1 \\1&-1 \end{pmatrix}
$$
and define inductively for all integers $p>1$, a reflexion
$$
J_p=\frac{1}{\sqrt{2}} \begin{pmatrix} J_{p-1} & J_{p-1}  \\ J_{p-1} &- J_{p-1}  \end{pmatrix},
$$
that is $J_p=\otimes^p J_1$. Observe, that given any matrix $S\in\bM_{2^p}$, $S=[s_{i,j}]$, such that $s_{i,j}=-s_{i,j}$ for all
$i\neq j$, the matrix
$$
T=J_pSJ_p^*
$$
has its diagonal entries $t_{j,j}$ all equal to the normalized trace
$2^{-p}{\mathrm{Tr\,}} S$. Indeed, letting $J_p=[z_{i,j}]$, 
\begin{align*}
t_{j,j} &=\sum_{k} z_{j,k} \left(\sum_{l}s_{k,l}z_{l,j}\right) \\
&=\sum_{k,l} z_{j,k} s_{k,l} z_{l,j} \\
&=\sum_{k} z_{j,k} s_{k,k} z_{k,j} +\sum_{k\neq l}  z_{j,k} s_{k,l} z_{l,j}\\
&=2^{-p}{\mathrm{Tr\,}} S +\sum_{k<l}\left( s_{k,l} z_{j,k}z_{l,j}+s_{l,k} z_{l,j}z_{k,j}\right)\\
&=2^{-p}{\mathrm{Tr\,}} S.
\end{align*}

Now, since we assume that $\beta=2^p$ for some integer $p$, we may perform a unitary congruence to the matrix $\Omega$ in
\eqref{omega1} with the unitary matrix
$$
R_p=J_p\otimes I_{2^{\beta}}\otimes I_n
$$
and, making use of \eqref{omega2} and the above property of $J_p$, we note that 
$
R_p\Omega R_p^*
$
has  its $\beta$ diagonal blocks $(R_p\Omega R_p^*)_{j,j}$, $1\le j\le \beta$, all equal to the matrix $D\in\bM_{2^{\beta}n}$,
$$
D =\frac{1}{\beta}\sum_{s=1}^{\beta} \left\{ \oplus^{2^{\beta}}A_{s,s}\right\}.
$$
Thanks to the decomposition of Lemma \ref{BL-lemma} and its obvious extension for $\beta\times\beta$ partitions, there exist some
isometries $U_k\in\bM_{\beta2^{\beta}n, 2^{\beta}n}$, $1\le k\le \beta$, such that
$$
\Omega=\sum_ {k=1}^{\beta} U_k D U_k^*.
$$
Since $\Omega$ is unitarily equivalent to $\oplus^{2^{\beta}} H$, that is $\Omega=V^*(\oplus^{2^{\beta}} H)V$ for some unitary $V\in
\bM_{\beta 2^{\beta}n, 2^{\beta}n}$, we get
$$
\oplus^{2^{\beta}} H = \sum_ {k=1}^{\beta} VU_k D U_k^*V^*
$$
wich is the claim of Theorem \ref{thm-direct} by setting $VU_k=:V_k$, $1\le k\le \beta$, as $2^{\beta}=m$, and
$D=\frac{1}{\beta}\oplus^m\Delta$.
\end{proof}

\section{Comments}

\subsection{Complex matrices and small partitions}

If one uses isometries with complex entries, then, in case of partitions into a small number of $\beta\times\beta$ blocks, the number
$m$ of copies in the direct sum $\oplus^m H$ and $\oplus^m \Delta$ can be reduced. For $\beta=2$, Theorem \ref{thm-four} shows that
it suffices to take $m=1$. For $\beta=3$ or $\beta=4$ the following result holds \cite{BLL2}.

\vskip 10pt
\begin{theorem}\label{thm-quaternion} Let $H=[A_{s,t}]\in \bM_{\beta n}^+(\bC)$ be written in Hermitian blocks in $\bM_n(\bC)$ with
$\beta\in\{3,4\}$ and let $\Delta=\sum_{s=1}^{\beta}A_{s,s}$ be its partial trace. Then,
\begin{equation*}
H\oplus H =\frac{1}{4}\sum_{k=1}^4 V_k\left(\Delta\oplus\Delta\right)V_k^*
\end{equation*} for some isometries $V_k\in\bM_{2\beta n,2n}(\bC)$, $k=1,2,3,4$.
\end{theorem}

\vskip 10pt
Likewise for Theorem \ref{thm-four}, we must consider isometries with complex entries, even for a full matrix $H$ with real entries.
The proof makes use of quaternions and thus confines to $\beta\le 4$.

\subsection{Separability criterion}

\vskip10pt
Let ${\mathcal{H}}$ and ${\mathcal{F}}$ be two finite dimensional Hilbert spaces that may be either real spaces, identified to
$\bR^n$ and $\bR^m$, or complex spaces, identified to $\bC^n$ and $\bC^m$. The space of operators on ${\mathcal{H}}$, denoted by
${\mathrm{B}}({\mathcal{H}})$, is identified with the matrix algebra $\bM_n$ (with real or complex entries according the nature of
${\mathcal{H}}$). A positive (semi-definite) operator $Z$ on the tensor product space ${\mathcal{H}}\otimes{\mathcal{F}}$ is said to
be separable if it can be decomposed as a sum of tensor products of positive operators,
\begin{equation}\label{sum-tens2}
Z=\sum_{j=1}^k A_j\otimes B_j
\end{equation}
where $A_j$'s are positive operators on ${\mathcal{H}}$ and so $B_j$'s are on ${\mathcal{F}}$ (the positivity assumption on the
$A_j$'s and $B_j$'s is essential, otherwise \ref{sum-tens2} is always possible for any $Z$).
It is  difficult in general to determine if a given
positive operator in the  matrix algebra $\bM_n\otimes\bM_m$
is separable or not, though some theoretical criteria do exist \cite{Hor}, \cite{Chen-Wu}. The partial trace of $Z$ with respect to
${\mathcal{H}}$ is the operator acting on ${\mathcal{F}}$,
$$
{\mathrm{Tr}}_{\mathcal{H}} Z= \sum_{j=1}^k ({\mathrm{Tr}}A_j)B_j.
$$
These notions have their own mathematical interest and moreover play a fundamental role in the description of bipartite systems in
quantum theory, see \cite[Chapter 10]{Petz}, where the positive operators act on complex spaces and are usually normalized with trace
one and called states. Thus a separable state is an operator of the type \eqref{sum-tens2} with ${\mathrm{Tr\,}} Z=1$. The richness
of the mathematical theory of separable operators/states and their application in quantum physics is apparent in many places in the
literature, for instance in \cite{Hor} and \cite{AlSh}. Nielsen and Kempe in 2001 proved a majorisation separability criterion
\cite{NK}. It can be stated as the following norm comparison.

\vskip 10pt
\begin{theorem}\label{part-trace} Let $Z$ be a separable state on the tensor product of two finite dimensional Hilbert spaces
${\mathcal{H}}$ and ${\mathcal{F}}$. Then, for all symmetric norms,
$$
\| Z\| \le \left\| {\mathrm{Tr}}_{\mathcal{H}} Z \right\|.
$$
\end{theorem}

\vskip 10pt
Regarding ${\mathrm{B}}({\mathcal{H}}\otimes{\mathcal{F}})$ as $\bM_n(\bM_m)$, an operator
$Z\in{\mathrm{B}}({\mathcal{H}}\otimes{\mathcal{F}})$ is written as a block-matrix $Z=[Z_{i,j}]$ with $Z_{i,j}\in\bM_m$, $1\le i,j\le
n$. The partial trace of $Z$ with respect to ${\mathcal{H}}$ is then the sum of the diagonal blocks,
$$
{\mathrm{Tr}}_{\mathcal{H}} Z=\sum_{j=1}^n Z_{j,j}.
$$
This observation makes obvious that Theorem \ref{part-trace} is a straightforward consequence of Corollary \ref{Hiroshima} whenever
the factor ${\mathcal H}$ is a {\bf real} Hilbert space. Indeed, we then have
$$
Z=\sum_{j=1}^k A_j\otimes B_j
$$
where, for each index $j$, $A_j\in\bM_n(\bR)$ and $B_j$ is Hermitian in $\bM_m$, so that $A_j\otimes B_j$ can be regarded as an
element of $\bM_n( \bM_m)$ formed of Hermitian blocks.

From Corollary \ref{cor-eigenvalue1}, we may complete the majorisation of Theorem \ref{part-trace}, when a factor is a real space
with a few more eigenvalue estimate as stated in the next corollary.

\vskip 10pt
\begin{cor} Let $Z$ be a separable positive operator on the tensor product of two finite dimensional Hilbert space
${\mathcal{H}}\otimes{\mathcal{F}}$ with a {\bf real} factor ${\mathcal{H}}$. Then,
$$
\lambda_{1+\beta k}( Z) \le \lambda_{1+k}\left( {\mathrm{Tr}}_{\mathcal{H}} Z \right)
$$
for all $k=0,\ldots,\dim{\mathcal{F}}-1$,    where $\beta$ is the smallest dyadic number such that $\dim{\mathcal{H}} \le\beta$.
\end{cor}

\vskip 10pt
Similarly, from Corollary \ref{cor-eigenvalueC}, we actually have a larger set of eigenvalue inequalities.

\vskip 10pt
\begin{cor} Let $Z$ be a separable positive operator on the tensor product of two finite dimensional Hilbert space
${\mathcal{H}}\otimes{\mathcal{F}}$ with a {\bf real} factor ${\mathcal{H}}$. Then,
$$
\lambda_{1+\beta k}( Z) \le \frac{1}{\beta}\left\{ \lambda_{1+k_1}\left( {\mathrm{Tr}}_{\mathcal{H}} Z  \right)
+\cdots+ \lambda_{1+k_{\beta}}\left( {\mathrm{Tr}}_{\mathcal{H}} Z  \right)\right\}
$$
for all $k=0,\ldots,\dim{\mathcal{F}}-1$,
where $k_1+\cdots+k_{\beta}=\beta k$ and $\beta$ is the smallest dyadic number such that $\dim{\mathcal{H}} \le\beta$.
\end{cor}

\vskip 10pt
Of course, confining to real spaces is a severe restriction. It would be desirable to obtain similar estimates for usual complex
spaces. A related problem would be to obtain a decomposition like in Theorem \ref{thm-direct} for the class of partitioned matrices
considered in the full form of Hiroshima's theorem:
\vskip 10pt\noindent
{\it If $A=[A_{s,t}]$ is a positive matrix partitioned in $\alpha\times \alpha$ blocks such that $B=[B_{s,t}]:=[A_{s,t}^T]$ is
positive too, then the majorisation of Corollary \ref{Hiroshima} holds.}
\vskip 10pt\noindent
 Here $X^T$ means the transposed matrix. This statement extends  Corollary \ref{Hiroshima} and implies Theorem \ref{part-trace}.

\section{Around this article}

A companion result of Lemma \ref{BL-lemma} is given in \cite{BLL1} :

\begin{lemma} \label{BLL-lemma} Let $\begin{bmatrix} A &X \\
X^* &B\end{bmatrix}$ be a positive matrix partitioned into four blocks in $\bM_n$. Then, for some unitary $U,V\in\bM_{2n}$,
\begin{equation*}
\begin{bmatrix} A &X \\
X^* &B\end{bmatrix} = U
\begin{bmatrix} \frac{A+B}{2} + {\mathrm{Re}}\, X&0 \\
0 &0\end{bmatrix} U^* +
V\begin{bmatrix} 0 &0 \\
0 &\frac{A+B}{2}- {\mathrm{Re}}\, X\end{bmatrix} V^*
\end{equation*}.
\end{lemma}

\vskip 5pt
This lemma is the main tool to obtain original estimates between the full matrix 
and its partial trace $A+B$. These estimates involve the geometry of the numerical range $W(X)$. Here we state two theorems, the next chapter will develop this
topic in more details. 
 The main theorem of \cite{BM} reads as follows.

\vskip 5pt
\begin{theorem}\label{thBM} Let $\begin{bmatrix} A &X \\
X^* &B\end{bmatrix} $ be a positive matrix  partitioned into four blocks
in $\bM_n$. Suppose that $W(X)$ has the width $\omega$. Then, for all symmetric norms,
$$
\left\| \begin{bmatrix} A &X \\
X^* &B\end{bmatrix}\right\| \le \| A+B +\omega I \|.
$$
\end{theorem}

\vskip 5pt
Here the width  of $W(X)$ is the smallest distance between two parallel straight lines such
that the strip between these two lines contains $W(X)$. Hence the partial trace $A+B$ may be used to give an upper bound for the
norms of the full block-matrix.  

A lower bound is given in  \cite{BL-BAMS} in which the distance from $0$ to  $W(X)$ contributes. We state the main result of \cite{BL-BAMS}.

\vskip 10pt
\begin{theorem}\label{th-dist} Let $\begin{bmatrix} A &X \\
X^* &B\end{bmatrix} $ be a positive matrix  partitioned into four blocks
in $\bM_n$ and let $d={\mathrm{dist}}(0,W(X))$. Then, for all symmetric norms,
$$
 \left\| \begin{bmatrix} A &X \\
X^* &B\end{bmatrix}\right\| \ge
\left\| \left(\frac{A+B}{2} +dI\right) \oplus  \left(\frac{A+B}{2} -dI\right) \right\|.
$$
\end{theorem}

\vskip 5pt
Several consequences follow. We mention two corollaries

\vskip 10pt
\begin{cor}\label{cor-diam} For every  positive matrix  partitioned into four blocks of same size,
$$ 
{\mathrm{diam}}\, W\left(\begin{bmatrix} A &X \\
X^* &B\end{bmatrix}\right)
- {\mathrm{diam}}\, W\left(\frac{A+B}{2}\right) \ge 2d,
$$
where $d$ is the distance from $0$ to $W(X)$.
\end{cor}

\vskip 10pt
\begin{cor}\label{cor-det} Let $\begin{bmatrix} A &X \\
X^* &B\end{bmatrix} $ be a positive matrix  partitioned into four blocks
in $\bM_n$ and let $d={\mathrm{dist}}(0,W(X))$. Then, 
$$
\det \left\{\left(\frac{A+B}{2}\right)^2 -d^2I\right\} \ge
 \det \left(\begin{bmatrix} A &X \\
X^* &B\end{bmatrix}\right).$$
\end{cor}

\vskip 5pt
Letting $X=0$, we recapture a basic property: the determinant is a log-concave map on the positive cone of $\bM_n$. Hence Corollary
\ref{cor-det} refines this property.

We will see in the next, short chapter several eigenvalue inequalities associated to Theorem \ref{thBM}.

\section{References of Chapter 6}

{\small
\begin{itemize}

\item[[13\!\!\!]] R.\ Bhatia, Matrix Analysis, Gradutate Texts in Mathematics, Springer, New-York, 1996.

\item[[28\!\!\!]] J.-C. Bourin and F. Hiai,
Norm and anti-norm inequalities for positive semi-definite matrices,
{\it Internat. J. Math.} \textbf{63} (2011), 1121-1138.

\item[[31\!\!\!]] J.-C.\ Bourin and E.-Y.\ Lee, Unitary orbits of Hermitian operators with convex or concave functions, {\it Bull.\ Lond.\
Math.\ Soc.}\ 44 (2012), no.\ 6, 1085--1102.

\item[[32\!\!\!]] J.-C.\ Bourin and E.-Y.\ Lee, Decomposition and partial trace of positive matrices with Hermitian blocks, {\it Internat.\ J.\ Math.}\ 24 (2013), no.\ 1, 1350010, 13 pp.

\item[[41\!\!\!]] J.-C.\ Bourin and E.-Y.\ Lee, 
Numerical range and positive block matrices, {\it Bull.\ Aust.\ Math.\ Soc.}\ 103 (2021), no.\ 1, 69--77.

\item[[45\!\!\!]] J.-C. Bourin, E.-Y. Lee and M.\ Lin, On a decomposition lemma for positive semi-definite block-matrices, {\it Linear Algebra Appl.}\
437 (2012), 1906--1912.

\item[[46\!\!\!]] J.-C. Bourin, E.-Y. Lee and  M.\ Lin, Positive matrices partitioned into a small number of Hermitian blocks, {\it Linear Algebra Appl.}\ 438 (2013), no.\ 5, 2591--2598. 

\item[[47\!\!\!]] J.-C. Bourin, A.\ Mhanna,
Positive block matrices and numerical ranges, {\it C.\ R.\ Math.\ Acad.\ Sci.\ Paris} 355 (2017), no.\ 10, 1077--1081.

\item[[51\!\!\!]] K.\ Chen and L.-A.\ Wu, A matrix realignment method for recognizing entanglement, {\it Quantum Inf.\ Comput.}\ \textbf{3} (2003),
193-202

\item[[53\!\!\!]] Clifford, Applications of Grassmann' extensive algebra, {\it Amer.\ Journ.\ Math.}\ \textbf{1} (1878), 350-358.

\item[[57\!\!\!]] J.\ Derezi\'nski, {\it Introduction to representations of the canonical commutation and anticommutation relations},
Lect.\ Note Phys.\ {\bf 695}, 65-145 (2006), Springer.

\item[[70\!\!\!]] M.\ Horodecki, P.\ Horodecki, R.\ Horodecki, Separability of mixed states: necessary and sufficient conditions,
{\it Phys.\ Lett.\ A} \textbf{223} (1996) l-8.

\item[[84\!\!\!]] M.\ A.\ Nielsen and J.\ Kempe, Separable states are more disordered globally than locally, {\it Phys.\ Rev.\ Lett.}\
\textbf{86} (2001) 5184-5187.

\item[[86\!\!\!]] D.\ Petz, Matrix Analysis with some applications, $<$http://www.math.hu/petz$>$.

\item[[88\!\!\!]] S. Ju. Rotfel'd, The singular values of a
sum of completely continuous operators,  \textit{ Topics in
Mathematical Physics, Consultants Bureau}, Vol. \textbf{3} (1969)
73-78.

\end{itemize}


\chapter{Block matrices and numerical range}

For sake of completeness the proof of Theorem \ref{thBM} is given in Section 7.4 ``Around this article".

\vskip 10pt\noindent
{\color{blue}\Large {\bf Eigenvalue inequalities for  positive block matrices with the inradius of the numerical range} \large{\cite{BL-inradius}}}

\vskip 10pt\noindent
{\small 
{\bf Abstract.} We prove the operator norm inequality, for a positive matrix partitioned into four blocks in $\bM_n$, 
$$
\left\| \begin{bmatrix} A &X \\
X^* &B\end{bmatrix}\right\|_{\infty} \le \| A+B  \|_{\infty} +\delta(X),
$$
where $\delta(X)$ is the diameter of the largest possible disc in the  numerical range of $X$. This shows that
the inradius   $\varepsilon(X):=\delta(X)/2$ satisfies
$
\varepsilon(X) \ge \| X\|_{\infty} - \| (|X^*|+ |X|)/2\|_{\infty}.
 $ 
Several  eigenvalue inequalities are  derived. In particular, if $X$ is a normal matrix whose spectrum lies in a disc of radius $r$, the third eigenvalue of the full matrix is bounded by the second eigenvalue of the sum of the diagonal block,
$$
\lambda_{3}\left(\begin{bmatrix} A &X \\
X^* &B\end{bmatrix}\right)  \le \lambda_{2}( A+B ) + r.
$$
 We think that $r$ is optimal and we propose a conjecture related to a norm inequality of Hayashi. 
\vskip 5pt\noindent
{\it Keywords.}  Numerical range,   Partitioned matrices, eigenvalue inequalities. 
\vskip 5pt\noindent
{\it 2010 mathematics subject classification.} 15A60, 47A12, 15A42, 47A30.
}

\section{Introduction}

Positive matrices partitioned into four blocks play a central role in Matrix Analysis, and in applications, for instance quantum information theory. A lot of 
important theorems deal with these matrices.  Some of these results give comparison between the full matrix and its diagonal blocks, in particular the sum of the diagonal blocks (the partial trace in the quantum terminology). This note focuses on a recent result  of Bourin and Mhana \cite{BM}, involving the numerical range  of the offdiagonal block.   Recall that a symmetric norm $\|\cdot\|$ on $\bM_{2n}$ means a unitarily invariant norm.  It induces a symmetric norm on $\bM_n$ in an obvious way. The Schatten $p$-norms $\|\cdot\|_p$, $1\le p\le \infty$, and the operator norm $(p=\infty)$ are classical examples of symmetric norms. The main result of \cite{BM} reads as follows.

\vskip 5pt
\begin{theorem}\label{thBM} Let $\begin{bmatrix} A &X \\
X^* &B\end{bmatrix} $ be a positive matrix  partitioned into four blocks
in $\bM_n$. Suppose that $W(X)$ has the width $\omega$. Then, for all symmetric norms,
$$
\left\| \begin{bmatrix} A &X \\
X^* &B\end{bmatrix}\right\| \le \| A+B +\omega I \|.
$$
\end{theorem}

\vskip 5pt
Here $I$ stands for the identity matrix, $W(X)$ denotes the numerical range of $X$,  and the width of $W(X)$ is the smallest distance between two parallel straight lines such that the strip between these two lines contains $W(X)$. If $\omega=0$, that is $W(X)$ is a line segment, Theorem 1.1 was first proved by Mhanna \cite{Mhanna}. Recently \cite{BL-BAMS}, Theorem \ref{thBM} has been completed with the reversed inequality
$$
\left\| \begin{bmatrix} A &X \\
X^* &B\end{bmatrix}\right\| \ge \left\|  \begin{bmatrix} \frac{A+B}{2} +d I& 0 \\ 0&\frac{A+B}{2} -d I  \end{bmatrix}\right\|.
$$
where $d:=\min\{ |z|\, :\, z\in W(X)\}$ is the distance from $0$ to $W(X)$. Several applications
were derived.

Some equality cases in Theorem \ref{thBM} occur for the operator norm $\|\cdot\|_{\infty}$ with the following block matrices, where $a,b$ are two arbitrary nonnegative real numbers.
\begin{equation*}\label{ex1}
\begin{bmatrix} \begin{pmatrix} a&0 \\ 0&b \end{pmatrix} &   \begin{pmatrix} 0&a \\ b&0 \end{pmatrix}\\
 \begin{pmatrix} 0&b \\ a&0 \end{pmatrix} &  \begin{pmatrix} b&0 \\ 0&a \end{pmatrix}\end{bmatrix}.
\end{equation*}
This follows from the fact that 
$
W\left(\begin{pmatrix} 0&b \\ a&0 \end{pmatrix} \right)
$
has the width $2\left| |a|-|b|\right|$.

Though Theorem \ref{thBM} is sharp for the operator norm, a subtle improvement is possible. This is our concern in the next section.
Once again, a geometric feature of $W(X)$ will contribute: its inradius.  Our approach leads to a remarkable list of eigenvalue that cannot be derived from the norm inequalities of Theorem \ref{thBM}.
 The last section is devoted  to some related operator norm inequalities, in particular we will discuss a property due to Hayashi  (2019) and propose a conjecture.

\section{Eigenvalue inequalities}

We define the indiameter $\delta(\Lambda)$ of a compact convex set $\Lambda\subset\bC$  as the diameter of the largest possible  disc in $\Lambda$. For matrices $X\in\bM_n$, we shorten
$\delta(W(X))=:\delta(X)$. Recall that the numerical range of a two-by-two matrix is an elliptical disc (or a line segment, or a single point).

A matrix  $X\in\bM_n$ is identified as an operator on $\bC^n$. If ${\mathcal{S}}$ is a subspace of $\bC^n$, we denote by $X_ {\mathcal{S}}$ the compression of $X$ onto ${\mathcal{S}}$.
We then define the {\it elliptical width} of  $X$ as
$$
\delta_2(X):= \sup_{\dim {\mathcal{S}}=2} \delta(X_ {\mathcal{S}}).
$$
Of course $\delta_2(X)\le \delta(X)\le \omega$ where $\omega$ still denotes the width of $W(X)$. If $X$ is a contraction, then $\delta_2(X)\le 1$, while $\delta(X)$ may be arbitrarily close to 2 (letting $n$ be large enough). We state our main result.

\vskip 5pt
\begin{theorem}\label{thinner} Let $\begin{bmatrix} A &X \\
X^* &B\end{bmatrix} $ be a positive matrix  partitioned into four blocks
in $\bM_n$.  Then, for all $j\in\{0,1,\ldots, n-1\}$,
$$
\lambda_{1+2j}\left(\begin{bmatrix} A &X \\
X^* &B\end{bmatrix}\right)  \le \lambda_{1+j}( A+B ) + \delta_2(X).
$$
\end{theorem}

\vskip 5pt
Here $\lambda_1(S)\ge\cdots\ge \lambda_d(S)$ stand for the eigenvalue of any Hermitian matrix $S\in\bM_d$. If we denote by $\lambda_1^{\uparrow}(S)\le \cdots\le \lambda^{\uparrow}_d(S)$ these eigenvalues arranged in the increasing order, then Theorem \ref{thinner} reads as
$$
\lambda_{2k}^{\uparrow}\left(\begin{bmatrix} A &X \\
X^* &B\end{bmatrix}\right)  \le \lambda_{k}^{\uparrow}( A+B ) + \delta_2(X).
$$
for all $k\in\{1,2,\ldots, n\}$.

The case $j=0$ in Theorem \ref{thinner} improves Theorem \ref{thBM} for the operator norm. We may consider that Theorem \ref{thinner} is trivial for $j=n-1$. Indeed, using the  decomposition \cite[Lemma 3.4]{BL-London},
\begin{equation*}
\begin{bmatrix} A &X \\
X^* &B\end{bmatrix} = U
\begin{bmatrix} A &0 \\
0 &0\end{bmatrix} U^* +
V\begin{bmatrix} 0 &0 \\
0 &B\end{bmatrix} V^*,
\end{equation*}
for some unitary matrices $U,\,V\in  \bM_{2n}$, we obtain from Weyl's inequality 
\cite[p.\ 62]{Bh},
\begin{align*}
\lambda_{2n-1}\left(\begin{bmatrix} A &X \\
X^* &B\end{bmatrix}\right)  &\le \lambda_{n} \left(\begin{bmatrix} A &0 \\
0 &0\end{bmatrix}\right)+ \lambda_{n} \left(\begin{bmatrix} 0 &0 \\
0 &B\end{bmatrix}\right) \\
&=  \lambda_{n}(A) + \lambda_n(B) \\ 
&\le \lambda_n(A+B).
\end{align*}

\vskip 5pt
We turn to the proof of the theorem. 

\vskip 5pt
\begin{proof} We first consider the case $j=0$. We may assume that the norm of the block matrix is strictly greater than
the norms of its two diagonal blocks $A$ and $B$, otherwise the statement is trivial. Hence we have two nonzero (column) vectors $h_1, h_2\in\bC^n$ such that $\| h_1\|^2+\| h_2\|^2=1$ and
$$
\lambda_1\left( \begin{bmatrix} A &X \\
X^* &B\end{bmatrix}\right)=
\left\| \begin{bmatrix} A &X \\
X^* &B\end{bmatrix}\right\|_{\infty}= 
\begin{pmatrix}
h^*_1 & h_2^*
\end{pmatrix}
\begin{bmatrix} A &X \\
X^* &B\end{bmatrix}
\begin{pmatrix} h_1 \\ h_2
\end{pmatrix}.
$$
Therefore, denoting by $E_1$ and $E_2$ the rank one projections corresponding to the one dimensional subspaces spanned by $h_1$ and by $h_2$, we have
$$
\left\| \begin{bmatrix} A &X \\
X^* &B\end{bmatrix}\right\|_{\infty}= 
\left\|
\begin{bmatrix} E_1 &0 \\
0 &E_2\end{bmatrix}
\begin{bmatrix} A &X \\
X^* &B\end{bmatrix}
\begin{bmatrix} E_1 &0 \\
0 &E_2\end{bmatrix}
\right\|_{\infty}
$$
Hence, denoting by $F$ a rank two projection such that $E_1\le F$ and $E_2\le F$, we have
\begin{align*}
\left\| \begin{bmatrix} A &X \\
X^* &B\end{bmatrix}\right\|_{\infty}&= 
\left\|
\begin{bmatrix} F &0 \\
0 &F\end{bmatrix}
\begin{bmatrix} A &X \\
X^* &B\end{bmatrix}
\begin{bmatrix} F &0 \\
0 &F\end{bmatrix}
\right\|_{\infty} \\
&=\left\| \begin{bmatrix} FAF &FXF \\
FX^*F &FBF\end{bmatrix}\right\|_{\infty}.
\end{align*}
So, letting ${\mathcal{S}}$  denote the range of $F$, we have
$$
\left\| \begin{bmatrix} A &X \\
X^* &B\end{bmatrix}\right\|_{\infty}
=\left\| \begin{bmatrix} A_{\mathcal{S}} &X_{\mathcal{S}} \\
X^*_{\mathcal{S}} &B_{\mathcal{S}}\end{bmatrix}\right\|_{\infty}.
$$ 
Hence applying Theorem \ref{thBM} for the operator norm, we obtain 
$$
\left\| \begin{bmatrix} A &X \\
X^* &B\end{bmatrix}\right\|_{\infty}\le
\| A_{\mathcal{S}} +B_{\mathcal{S}}  \|_{\infty} + \eps
$$ 
where $\eps$ is the width of $W(X_{\mathcal{S}})$. Since $W(X_{\mathcal{S}})$ is an elliptical disc (as $X_{\mathcal{S}}$ acts on a two-dimensional space), its width equals to its indiameter, hence $\eps\le \delta_2(X)$, and since
$$
\| A_{\mathcal{S}} +B_{\mathcal{S}}  \|_{\infty} =\| (A+B)_{\mathcal{S}} \|_{\infty} \le \| A+B \|_{\infty}=\lambda_1(A+B),
$$
the proof for $j=0$ is complete.

We turn to the general case, $j=1,\ldots, n-1$. By the min-max principle,
\begin{align*}
\lambda_{1+2j}\left(\begin{bmatrix} A &X \\
X^* &B\end{bmatrix}\right)  &\le \inf_{\dim{\mathcal{S}}=n-j}\lambda_{1}\left(\begin{bmatrix} A &X \\
X^* &B\end{bmatrix}_{{\mathcal{S}}\oplus{\mathcal{S}}}\right)  \\
&= \inf_{\dim{\mathcal{S}}=n-j}\lambda_{1}\left(\begin{bmatrix} A_{\mathcal{S}} &X_{\mathcal{S}} \\
X^*_{\mathcal{S}} &B_{\mathcal{S}}\end{bmatrix}\right),
\end{align*}
hence, from the first part of the proof,
\begin{align*}
\lambda_{1+2j}\left(\begin{bmatrix} A &X \\
X^* &B\end{bmatrix}\right)  &\le 
\inf_{\dim{\mathcal{S}}=n-j}\lambda_{1}\left( A_{\mathcal{S}} +
B_{\mathcal{S}}\right) + \delta_2(X) \\
&=\lambda_{1+j}(A+B) + \delta_2(X)
\end{align*}
which is the desired claim. \end{proof}

\vskip 5pt
If $X\in\bM_n$, we denote by ${\mathrm{dist}}(X,\bC I)$ the $\|\cdot\|_{\infty}$-distance from $X$ to  $\bC I$.
Thus, for a  scalar perturbation of a contraction,   $X=\lambda I + C$ for some contraction $C\in\bM_n$ and some $\lambda\in\bC$, we have  ${\mathrm{dist}}(X,\bC I)\le 1$.

\vskip 5pt
\begin{cor}\label{pertu1} Let $\begin{bmatrix} A &X \\
X^* &B\end{bmatrix} $ be a positive matrix  partitioned into four blocks
in $\bM_n$. Then, for all $j\in\{0,1,\ldots, n-1\}$, 
$$
\lambda_{1+2j}\left( \begin{bmatrix} A &X \\
X^* &B\end{bmatrix}\right)  \le \lambda_{1+j} (A+B) + {\mathrm{dist}}(X,\bC I).
$$
\end{cor}

\vskip 5pt
\begin{proof} For any subspace ${\mathcal{S}}\subset C^n$, we have 
$$
{\mathrm{dist}}(X,\bC I)\ge {\mathrm{dist}}(X_{\mathcal{S}},\bC I_{\mathcal{S}}).
$$
If ${\mathcal{S}}$ has dimension 2, then 
$$
{\mathrm{dist}}(X_{\mathcal{S}},\bC I_{\mathcal{S}} )\ge \delta\left(W(X_{\mathcal{S}})\right).
$$ 
Therefore ${\mathrm{dist}}(X,\bC I)\ge \delta_2(X)$ and Theorem \ref{thinner} completes the proof.
\end{proof}

\vskip 5pt
\begin{cor}\label{cor0}  Let $A,B\in\bM_n$. Then, for every $j\ge 0$ such that $1+2j\le n$,
$$
\lambda_{1+2j}\left( A^*A + B^*B \right) \le \lambda_{1+j}\left(AA^*+ BB^*  \right)+\delta_2(AB^*)
$$
\end{cor}

\vskip 5pt
\begin{proof} 
Note that 
$$
\lambda_{1+2j}\left(A^*A + B^*B \right) = \lambda_{1+2j}\left( T^*T\right)  =  \lambda_{1+2j}\left( TT^*\right)  
$$
with $T= \begin{bmatrix} A \\ B \end{bmatrix}$ and
$
TT^*= \begin{bmatrix} AA^* &AB^* \\
BA^* &BB^*\end{bmatrix}
$
so that Theorem \ref{thinner} yields the  desired claim.
\end{proof}

\vskip 5pt
\begin{cor}\label{cornormal1} Let $\begin{bmatrix} A &N \\
N^* &B\end{bmatrix} $ be a positive matrix  partitioned into four blocks
in $\bM_n$. If $N$ is  normal and its spectrum is contained in a disc of radius $r$, then, 
$$
\lambda_{1+2j}\left(\begin{bmatrix} A &N \\
N^* &B\end{bmatrix}\right)  \le \lambda_{1+j}( A+B ) + r.
$$
for all $j=0,1,\ldots,n-1$.
 \end{cor}

\vskip 5pt
\begin{proof}
Corollary \ref{cornormal1} is a special case of corollary \ref{pertu1}, as $N=\lambda I +R$,
where $\lambda$ is the center of the disc of radius $r$ containing the spectrum of $N$, and $\| N\|_{\infty}\le r$.
\end{proof}

\vskip 5pt
\begin{question}\label{quest1} Fix $r>0$ and $\varepsilon>0$. Can we find (with $n$ large enough) a normal matrix $N$  with spectrum in a disc of radius $r$ and  a positive block matrix $\begin{bmatrix} A &N \\
N^* &B\end{bmatrix} $ 
such that
$$
\lambda_{1+2j}\left(\begin{bmatrix} A &N \\
N^* &B\end{bmatrix}\right)  \ge \lambda_{1+j}( A+B ) + r-\varepsilon
$$
for some $j\in\{0,\ldots, n-1\}$ ? Is it true for for $j=0$ ?
\end{question}

\section{Norm inequalities }

Corollary \ref{cornormal1} with $j=0$ reads as follows.

\vskip 5pt
\begin{cor}\label{cornormal} Let $\begin{bmatrix} A &N \\
N^* &B\end{bmatrix} $ be a positive matrix  partitioned into four blocks
in $\bM_n$. If $N$ is  normal and its spectrum is contained in a disc of radius $r$, then, 
$$
\left\| \begin{bmatrix} A &N \\
N^* &B\end{bmatrix}\right\|_{\infty}  \le \| A+B \|_{\infty}  + r.
$$
 \end{cor}

\vskip 5pt
 We do not know wether the constant $r$ is sharp or not (Question \ref{quest1}). If $n=2$, we can replace $r$ by $0$ as the numerical range of $N$ is then a line segment. If $n=3$ there are some simple examples with $N=U$ unitary such that
$$
\left\| \begin{bmatrix} A &U \\
U^* &B\end{bmatrix}\right\|_{\infty}  > \| A+B \|_{\infty}.
$$
See Hayashi's example  in the discussion of [H, Problem 3] and  the interesting study and examples in [G] where we further have $A+B=kI$ for some scalars $k$. The next result is due to Hayashi [H, Theorem 2.5].

\vskip 5pt
\begin{theorem}\label{thHa} Suppose that $X\in\bM_n$ is invertible with $n$ distinct singular values. If
the inequality 
$$
\left\| \begin{bmatrix} A &X \\
X^* &B\end{bmatrix}\right\|_{\infty}  \le \| A+B \|_{\infty}.
$$
holds for all positive block-matrix with $X$ as off-diagonal block, then $X$ is normal.
\end{theorem}

\vskip 5pt
Theorem \ref{thHa} and Theorem \ref{thBM} suggest a natural conjecture. If $W(T)$ is line segment, then $T$ is a so-called  essentially Hermitian matrix.

\vskip 5pt
\begin{conjecture}\label{quest2} Let $X\in\bM_n$. If
the inequality 
$$
\left\| \begin{bmatrix} A &X \\
X^* &B\end{bmatrix}\right\|_{\infty}  \le \| A+B \|_{\infty}
$$
holds for all positive block-matrix with $X$ as off-diagonal block, then $X$ is essentially Hermitian.
\end{conjecture}

\vskip 5pt
 If we replace the operator norm by the Frobenius (or Hilbert-Schmidt) norm $\|\cdot\|_2$ then the following characterization holds.

\vskip 5pt
\begin{prop}\label{prop1} Let $X\in \bM_n$. Then, 
the inequality 
$$
\left\| \begin{bmatrix} A &X \\
X^* &B\end{bmatrix}\right\|_{2}  \le \| A+B \|_{2}
$$
holds for all positive block-matrix with $X$ as off-diagonal block if and only if $X$ is normal.
 \end{prop}

\vskip 5pt
\begin{proof} Suppose that $X$ is normal. To prove the inequality, squaring both side,  it suffices to establish the trace inequality
\begin{equation}\label{tr}
{\mathrm{Tr\,}} X^*X \le {\mathrm{Tr\,}} AB.
\end{equation}
Note that $X=A^{1/2}KB^{1/2}$, for some contraction $K$. Recall that, for all symmetric norms on $\bM_n,$ and any normal matrix $N\bM_n$, decomposed as $N=ST$, we have
$\| N\| \le \| TS \|$. Therefore
$$
\| X\| = \| A^{1/2}KB^{1/2}\| \le\| KB^{1/2}A^{1/2}\|.
$$
Squaring this inequality with the Frobenius norm yields the desired  inequality \eqref{tr}.

Suppose that $X$ is nonnormal, and note
that
$
 \begin{bmatrix} |X^*|&X \\
X^* &|X|\end{bmatrix}
$ is positive semidefinite and satisfies 
$$
\left\| \begin{bmatrix} |X^*|&X \\
X^* &|X|\end{bmatrix}
\right\|_2^2= 4 \| |X| \|_2^2
$$
while
$$
\| |X^*| + |X| \|_2^2= 2 \| |X| \|_2^2 + 2 {\mathrm{Tr\,}} |X| |X^*|
$$
In the Hilbert space $(\bM_n, \|\cdot\|_2)$, the assumption
$$
\| |X| \|_2=\| |X^*| \|_2,\quad |X| \neq |X^*|
$$
ensures strict inequality in the Cauchy-Schwarz inequality
$$
{\mathrm{Tr\,}} |X| |X^*| < \| X\|_2^2.
$$
Therefore 
$$
\| |X^*| + |X| \|_2^2 < \left\| \begin{bmatrix} |X^*|&X \\
X^* &|X|\end{bmatrix}
\right\|_2^2
$$
and this completes the proof.
\end{proof}

\vskip 5pt
Proposition \ref{prop1} suggests a question: for which $p\in[1,\infty]$, the schatten $p$-norm inequality
$$
\left\| \begin{bmatrix} A &N \\
N^* &B\end{bmatrix}\right\|_p  \le \| A+B \|_p
$$
holds for any positive partitioned matrices with a normal off-diagonal block $N$ ?

\vskip 5pt
\begin{cor}\label{cor01}  Let $H,K,X\in\bM_n$ be Hermitian. If $X$ is invertible and  $HK$ is a scalar perturbation of a contraction, then,
$$
\left\| XH^2X+ X^{-1}K^2X^{-1} \right\|_{\infty}  \le \left\| HX^2H+ KX^{-2}K  \right\|_{\infty}  + 1.
$$
\end{cor}

\vskip 5pt
\begin{proof} We apply Corollary \ref{cor0} with $j=0$ and $A=HX$, $B=KX^{-1}$, to get
$$
\left\| XH^2X+ X^{-1}K^2X^{-1} \right\|_{\infty}  \le \left\| HX^2H+ KX^{-2}K  \right\|_{\infty}  + \delta_2(HK)
$$
Since $(HK)_{\mathcal{S}}$ is a scalar perturbation of a contraction acting on a space of dimension 2, necessarily $\delta_2(HK)\le 1$.
\end{proof}

\vskip 5pt
For a normal operator,  the numerical range is the convex hull of the spectrum. For a non normal operator $X$,  several lower bounds for the indiameter
of $W(X)$ can be obtained from the left and right modulus $|X^*|$ and $|X|$.

\vskip 5pt
\begin{cor}\label{corcomp} Let $X \in \bM_n$ and let $f(t)$ and $g(t)$ are two nonnegative functions defined on $[0,\infty)$ such that $f(t)g(t)=t^2$. Then,
$$
\delta_2(X) \ge \left\|  f(|X|)+g(|X|) \right\|_{\infty} - \left\|   f(|X^*|) + g(|X|) \right\|_{\infty}.
$$
\end{cor}

\vskip 5pt
\begin{proof} First, observe that we have a function $h(t)$ defined on $[0,\infty)$ such that
\begin{equation}\label{h1}
f(t)=th^2(t^{1/2}), \quad g(t)=th^{-2}(t^{1/2}),
\end{equation}
and $h(t)>0$ for all $t\ge0$ (we may, for instance, set $h(0)=1$). Hence $h(T)$ is invertible for any positive $T$, and from the polar decomposition
$$X=|X^{*}|^{1/2} U |X|^{1/2}$$
with a unitary factor $U$, we infer the factorization
$$
X=|X^{*}|^{1/2}h(|X^{*}|^{1/2}) U |X|^{1/2}h^{-1}(|X|^{1/2}).
$$
Thus $X=AB^*$ where
$
A=|X^{*}|^{1/2}h(|X^{*}|^{1/2}) $  and $B^*=U |X|^{1/2}h^{-1}(|X|^{1/2})$.
Therefore Corollary \ref{cor0} yields
$$
\left\| |X^{*}|h^2(|X^{*}|^{1/2}) +U |X|h^{-2}(|X|^{1/2})U^* \right\|_{\infty} \le
\left\| |X|h^2(|X|^{1/2}) +|X|h^{-2}(|X|^{1/2})U^* \right\|_{\infty} 
+\delta_2(X)
$$
Using \eqref{h1} and the fact that $\varphi(|X^*|)=U\varphi(|X|)U^*$ for any function $\varphi(t)$ defined on $[0,\infty)$, the proof is complete.
\end{proof}

\vskip 5pt
The following special case shows that Corollary \ref{corcomp} is rather optimal.

\vskip 5pt
\begin{cor}\label{corsimple} If $X \in \bM_n$ has a numerical range of inradius  $\varepsilon(X)$,  then, for all $a\in\bC$,
$$
\varepsilon(X) \ge \left\|  X-aI\right\|_{\infty} - \left\| \frac{ |X-aI| +|X^*-\overline{a}I|}{2} \right\|_{\infty}.
$$
If $X\in\bM_2$ and $a=\tau$ is the normalized trace of $X$, then this inequality is an equality.
\end{cor}

\vskip 5pt
\begin{proof} Applying Corollary \ref{corcomp} with $X-aI$ and $f(t)=g(t)=t$ yields the inequality. If $X\in\bM_2$, then $X$ is unitarily equivalent to
$$
\begin{pmatrix}
\tau & y \\
x & \tau 
\end{pmatrix}.
$$
So
\begin{align*}
\| X -\tau I\|_{\infty}-\left\| \frac{ |X-\tau I| +|X^*-\overline{\tau}I|}{2} \right\|_{\infty}
&=
\left\|
\begin{pmatrix}
|x| & 0 \\
0 & |y| 
\end{pmatrix}
\right\|_{\infty}
-
\frac{1}{2}\left\|
\begin{pmatrix}
|x| +|y| & 0 \\
0 & |x|+|y| 
\end{pmatrix}
\right\|_{\infty} \\
&= \left| |x|-|y| \right| \\
&=\varepsilon(X)
\end{align*}
establishing the desired equality.
\end{proof}

\vskip 5pt
The special case $f(t)=g(t)=t$ in Corollary \ref{corcomp} seems important, we record it as a proposition:

\vskip 5pt
\begin{prop}\label{thcomp} The elliptical width of the numerical range  of  $X \in \bM_n$ satisfies
$$\delta_2(X)\ge 2\| X\|_{\infty}-\||X|+|X^*|\|_{\infty}. $$
\end{prop}

In particular, the inradius $\varepsilon(X)$ of the numerical range of $X$ satisfies $$\varepsilon(X)\ge \| X\|_{\infty}-\|(|X|+|X^*|)/2\|_{\infty}.$$

\vskip 5pt
\begin{remark}
Our results  still hold for operators on infinite dimensional separable Hilbert space (assuming in Corollary \ref{corcomp} that $f(t)$ and
$g(t)$ are Borel functions).
\end{remark}

\section{Around this article}

We give the proof of Theorem \ref{thBM}.

\vskip 5pt
\begin{proof} (Theorem \ref{thBM}) By using the unitary congruence implemented by $$ \begin{bmatrix} e^{i\theta}&0 \\
0 &I\end{bmatrix}$$
we see that our block matrix is unitarily equivalent to
$$
 \begin{bmatrix} A &e^{i\theta}X \\
e^{-i\theta}X^* &B\end{bmatrix}.
$$
As $W(e^{i\theta} X)=e^{i\theta}W(X)$, by choosing the adequate $\theta$ and replacing $X$ by $e^{i\theta}X$, we may and do assume that $W(X)$ lies in a strip ${\mathcal{S}}$
of width $\omega$ and parallel to the imaginary axis,
$$
{\mathcal{S}} =\left\{\, x+iy\ : \quad y\in\bR, \quad r \le x \le r + \omega\, \right\}.
$$
The projection property for the real part ${\mathrm{Re}\,} W(X) =W({\mathrm{Re}\,} X)$, then ensures that
\begin{equation}\label{proof1}
rI \le {\mathrm{Re}\,} X \le (r+\omega) I.
\end{equation}

Now we use the decomposition \cite[Corollary 2.1]{BLL1} derived from \eqref{key},
\begin{equation}\label{proof2}
\begin{bmatrix} A &X \\
X^* &B\end{bmatrix} = U
\begin{bmatrix} \frac{A+B}{2} +{\mathrm{Re}\,} X&0 \\
0 &0\end{bmatrix} U^* +
V\begin{bmatrix} 0 &0 \\
0 &\frac{A+B}{2} -{\mathrm{Re}\,} X\end{bmatrix} V^*
\end{equation}
for some unitaries $U,\,V\in  \bM_{2n}$. Note that the two matrices in the right hand side of \eqref{proof2} are positive since so are
$$
\begin{bmatrix} I &I \end{bmatrix} 
\begin{bmatrix} A &X \\
X^* &B\end{bmatrix} \begin{bmatrix} I \\I \end{bmatrix} \quad{\mathrm{and}}  
\begin{bmatrix} I &-I \end{bmatrix} 
\begin{bmatrix} A &X \\
X^* &B\end{bmatrix} \begin{bmatrix} I \\-I \end{bmatrix}. 
$$
Combining \eqref{proof1} and \eqref{proof2} yields
\begin{equation*}
\begin{bmatrix} A &X \\
X^* &B\end{bmatrix} \le U
\begin{bmatrix} \frac{A+B}{2} + (r+\omega)I&0 \\
0 &0\end{bmatrix} U^* +
V\begin{bmatrix} 0 &0 \\
0 &\frac{A+B}{2} -rI \end{bmatrix} V^*
\end{equation*}
where the two matrices of the right hand side are positive. From each Ky Fan $k$-norm, $k=1,2,\ldots, 2n$, we then have 
\begin{align*}
\left\|\begin{bmatrix} A &X \\
X^* &B\end{bmatrix}\right\|_{(k)} &\le \left\| \frac{A+B}{2} + (r+\omega)I\right\|_{(k)} + \left\| \frac{A+B}{2} - rI\right\|_{(k)}    \\
& =\left\| A+B + \omega I\right\|_{(k)}.
\end{align*}
The Ky Fan principle then guarantees that this inequality hold for all symmetric norms.
\end{proof}

\section{References of Chapter 7}

{\small
\begin{itemize}

\item[[G\!\!\!]] M.\ Gumus, J.\ Liu, S.\ Raouafi, T-Y.\ Tam,  Positive semi-definite $2\times 2$ block matrices and norm inequalities, {\it Linear Algebra Appl.}\ 551 (2018), 83--91.

\item[[H\!\!\!]] T.\ Hayashi, On a norm inequality for a positive block-matrix, {\it Linear Algebra Appl.}\ 566 (2019), 86--97.

\item[[41\!\!\!]] J.-C.\ Bourin and E.-Y.\ Lee, 
Numerical range and positive block matrices, {\it Bull.\ Aust.\ Math.\ Soc.}\ 103 (2021), no.\ 1, 69--77.

\item[[44\!\!\!]] J.-C.\ Bourin and E.-Y.\ Lee,
Eigenvalue inequalities for positive block matrices with the inradius of the numerical range,
preprint

\item[[45\!\!\!]]
J.-C. Bourin, E.-Y. Lee and M.\ Lin, On a decomposition lemma for positive semi-definite block-matrices, {\it Linear Algebra Appl.}\
437 (2012), 1906--1912.

\item[[47\!\!\!]] J.-C. Bourin, A.\ Mhanna,
Positive block matrices and numerical ranges, {\it C.\ R.\ Math.\ Acad.\ Sci.\ Paris} 355 (2017), no.\ 10, 1077--1081. 

\item[[80\!\!\!]] A.\ Mhanna, On symmetric norm inequalities and positive definite block-matrices,  {\it Math.\ Inequal.\ Appl.}\ 21 (2018), no.\ 1, 133--138.

\end{itemize}


\chapter{Block matrices and Pythagoras}

{\color{blue}\Large {\bf A Pythagorean theorem for partitioned matrices} \large{\cite{BL-PAMS}}}

\vskip 10pt\noindent
{\small 
{\bf Abstract.} We establish a Pythagorean theorem for the absolute values of the blocks of a partitioned matrix. This leads to a
series of remarkable operator inequalities. For instance, if the matrix $\bA$ is partitioned into three blocks $A,B,C$, then
$$
|\bA|^3 \ge U|A|^3U^* + V|B|^3V^*+ W|C|^3W^*,
$$
$$
\sqrt{3} |\bA| \ge  U|A|U^* + V|B|V^*+ W|C|W^*,
$$
for some isometries $U,V,W$, and
$$
\mu_4^2(\bA) \le \mu_3^2(A) +\mu_2^2(B) + \mu_1^2(C)
$$
where $\mu_j$ stands for the $j$-th singular value. Our theorem may be used to extend a result by Bhatia and Kittaneh for the
Schatten $p$-noms and to give a singular value version of Cauchy's Interlacing Theorem.

\vskip 5pt\noindent
{\it Keywords.}     Partitioned matrices, functional calculus, matrix inequalities. 
\vskip 5pt\noindent
{\it 2010 mathematics subject classification.} 15A18, 15A60,  47A30.

}

\section{Introduction}

Let $\bM_{d}$ denote the space of $d$-by-$d$ matrices. If $\bA\in\bM_{d}$, the polar decomposition holds, 
\begin{equation}\label{polar}
\bA = U |\bA|
\end{equation}
where $|\bA|\in\bM_{d}$ is positive semi-definite and $U\in\bM_{d}$ is a unitary matrix. The matrix $|\bA|$ is called the absolute
value of $\bA$, and its
eigenvalues are the singular values of $\bA$. The absolute value can be defined for $d\times d'$ matrices $\bA\in\bM_{d,d'}$ as a
positive matrix $|\bA|\in\bM_{d'}$, and the factor $U$ in \eqref{polar} is an isometry ($d\ge d'$) or a coisometry ($d<d'$).

If $\bA$ is partitioned in some number of rectangular blocks, say four blocks $A,B,C,D$, it is of interest to have a relation between
the absolute value $|\bA|$ and the absolute values of the blocks. By using the standard inner product of $\bM_{d,d'}$, we immediately
have the trace relation
$$
{\mathrm{Tr\,}} |\bA|^2 = {\mathrm{Tr\,}} |A|^2 + {\mathrm{Tr\,}} |B|^2+ {\mathrm{Tr\,}} |C|^2 +{\mathrm{Tr\,}} |D|^2.
$$
This note aims to point out a much stronger Pythagorean theorem, Theorem \ref{th-pyth}, and several consequences. This result holds
for many partitionings of $\bA$, not only when $A$ is a block matrix partitioned into $p\times q$ blocks. For instance, one may
consider the matrix
$$
\bA=\begin{pmatrix} a_1&a_2 &b_1 &b_2 &b_3 \\ a_3&a_4 &b_4 &b_5 &b_6 \\ a_5&a_6 &c_1 &c_2 &d_1 \\ a_7&a_8 &c_3 &c_4 &d_2 \\
a_9&a_{10} &c_5 &c_6 &d_3
\end{pmatrix}
$$
partitioned into four obvious blocks $A,B,C,D$.

If $\bA$ is partitioned into $r$ blocks $A_k\in\bM_{n_k,m_k}$, we write
\begin{equation}\label{part}
\bA=\bigcup_{k=1}^r A_k =A_1 \cup \cdots \cup A_r
\end{equation}
where we can use the $=$ sign if one considers  $A_k$ not only as an element of $\bM_{n_k,m_k}$ but also as a submatrix of
$\bA$ with its position in $\bA$. 

We say that the partitioning \eqref{part} is column compatible, or that $\bA$ is partitioned into column compatible blocks if for all
pairs of indexes $k,l$, either $A_k$ and $A_l$ lie on the same set of columns of $\bA$, or $A_k$ and $A_l$ lie on two disjoint sets
of columns of $\bA$. Similarly, \eqref{part} is row compatible, if for all pairs of indexes $k,l$, either $A_k$ and $A_l$ lie on the
same set of rows of $\bA$, or$A_k$ and $A_l$ lie on two disjoint sets of rows of $\bA$.

Our Pythagorean Theorem \ref{th-pyth} will be stated for row or column compatible blocks. An application is a Theorem of Bhatia and
Kittaneh for the Schatten $p$-norms (Corollary \ref{corBK}). Another application is an inequality for the singular values of
compression onto hyperplanes. A matrix $A\in\bM_d$ is an operator on $\bC^d$. Given a hyperplane ${\mathcal{S}}$ of $\bC^d$, we have
a unit vector $h$ such that ${\mathcal{S}}
=h^{\perp}$, that is $x\in{\mathcal{S}}\iff \langle h, x\rangle =h^*x=0$. The compression $ A_{\mathcal{S}}$ of $A$ onto
${\mathcal{S}}$ is the operator acting on ${\mathcal{S}}$ defined as the restriction of $EA$ to ${\mathcal{S}}$ where $E$ stands for
the (orthogonal) projection onto ${\mathcal{S}}$. Theorem \ref{th-pyth} entails a bound for the singular values of $A_{\mathcal{S}}$
in terms of those of $A$. These results are given in Section 3; we state a special case in the following corollary. Let $\mu_j$
denote the $j$-th singular value arranged in nonincreasing order.

\vskip 5pt
\begin{cor}\label{normal} 
Let $A\in\bM_d$ be a normal matrix and let ${\mathcal{S}}$ be a hyperplane of $\bC^d$ orthogonal to the unit vector $h$. Set $\beta=
\| Ah\|^2 -|\langle h,Ah\rangle|^2$. Then,
for $j=1,\ldots, d-1$,
$$
 \mu_j^2(A)\ge \mu_{j}^2(A_{\mathcal{S}}) \ge \mu^2_{j+1}(A) -\beta.
$$
\end{cor}

\vskip 5pt
We discuss the case of four and five blocks in Section 4. For four blocks, our Pythagorean theorem entails an interesting inequality
stated in the next corollary.

\vskip 5pt
\begin{cor}\label{revtriangle} Let $\bA\in\bM_{d,d'}$ be partitioned into four blocks $A, B, C, D$. Then, there exist some isometries
$U,V,W,X$ of suitable sizes such that
$$
2|\bA| \ge U|A|U^* + V|B|V^* + W|C|W^* + X|D|X^*.
$$
\end{cor}

The last section is devoted to   several other operator inequalities such as the first inequality in the abstract.

\section{A Pythagorean theorem}

\vskip 5pt
\begin{theorem}\label{th-pyth} Let $\bA\in\bM_{d,d'}$ be partitioned into $r$ row or column compatible blocks $A_k\in\bM_{n_k, m_k}$.
Then, there exist some isometries $U_k\in\bM_{d',m_k}$ such that
$$
|\bA|^2= \sum_{k=1}^{r} U_k  |A_k|^2 U_k^*.
$$
\end{theorem}

\vskip 5pt Recall that $U\in\bM_{d',m}$, $m\le d'$, is an isometry if $U^*U=\bold{1}_{m}$, the identity on $\bC^m$. If
$\bA\in\bM_{d,1}$, then the theorem reads as Pythagoras' Theorem.

\begin{proof} Consider a positive matrix in $\bM_{n+m}$ partitioned as
$$
\begin{bmatrix} A& X \\ X^* &B
\end{bmatrix}
$$
with diagonal blocks $A\in\bM_n$ and $B\in\bM_m$. By  \cite[Lemma 3.4]{BL-London} we have two unitary matrices $U,V\in\bM_{n+m}$ such that
\begin{equation}\label{key}
\begin{bmatrix} A& X \\ X^* &B
\end{bmatrix}
=U\begin{bmatrix} A& 0 \\ 0 &0
\end{bmatrix}U^* + V\begin{bmatrix} 0& 0 \\ 0 &B
\end{bmatrix}V^*,
\end{equation}
equivalently,
\begin{equation*}
\begin{bmatrix} A& X \\ X^* &B
\end{bmatrix}
=U_1AU_1^* +
U_2BU_2^*
\end{equation*}
for two isometry matrices $U_1\in\bM_{n+m,n}$ and $U_2\in\bM_{n+m,m}$. An obvious iteration of \eqref{key} shows that, given a
positive block matrix in $\bM_m$ partitioned into $p\times p$ blocks,
$$
\bB=\left(B_{i,j}\right)_{1\le i,j\le p},
$$
with square diagonal blocks $B_{i,i}\in \bM_{n_i}$ and $n_1+\cdots+n_p=m$, we have the decomposition
\begin{equation}\label{key2}
\bB=\sum_{i=1}^p U_i B_{i,i} U_i^*
\end{equation}
for some isometries $U_i\in\bM_{m, n_i}$.

We use \eqref{key2} to prove the theorem. Consider first the column compatible case. Thus we have a partitioning into $p$ block
columns,
\begin{equation}\label{col}
\bA={\mathbf{C}}_1\cup\cdots\cup {\mathbf{C}}_p,
\end{equation}
and each block $A_k$ belongs to one block column ${\mathbf{C}}_q$. By relabelling the $A_k$'s if necessary, we may assume that we
have $p$ integers $1=\alpha_1<\alpha_2<\cdots<\alpha_p<r$ such that
$$
{\mathbf{C}}_q= A_{\alpha_q} \cup \cdots \cup A_{\alpha_{q+1}-1}, \quad 1\le q< p, \quad{\text{and}} \quad {\mathbf{C}}_p=
A_{\alpha_p} \cup \cdots \cup A_{\alpha_{r}}.
$$
We also have a partitioning into $p$ block rows, 
\begin{equation}\label{row}
\bA^*={\mathbf{C}}^*_1\cup\cdots\cup {\mathbf{C}}^*_p,
\end{equation}
and multiplying \eqref{row} and \eqref{col} we then obtain a block matrix for $\bA^*\bA=|\bA|^2\in\bM_{d'}$,
$$
|\bA|^2= \left( {\mathbf{C}}^*_i{\mathbf{C}}_j\right)_{1\le i,j\le p}.
$$
By the decomposition \eqref{key2} we have
$$
|\bA|^2=\sum_{i=1}^p U_i{\mathbf{C}}^*_i{\mathbf{C}}_i U_i^*
$$
for some isometries $U_i\in\bM_{d' , n_i}$, where $n_i$ is the number of columns of ${\mathbf{C}}_i$. Hence, with the convention
$\alpha_{p+1}:=r+1$,
$$
|\bA|^2=\sum_{i=1}^p\sum_{k=\alpha_i}^{\alpha_{i+1}-1} U_i A_k^*A_k U_i^*
$$
establishing the theorem for a column compatible partitioning.

Now, we turn to the row compatible case. Thus we have a partitioning into $p$ block rows,
\begin{equation}\label{col1}
\bA={\mathbf{R}}_1\cup\cdots\cup {\mathbf{R}}_p,
\end{equation}
and each block $A_k$ belongs to one block row ${\mathbf{R}}_q$ and, as in the column compatible case, we may assume that we have $p$
integers $1=\alpha_1<\alpha_2<\cdots<\alpha_p<r$ such that
$$
{\mathbf{R}}_q= A_{\alpha_q} \cup \cdots \cup A_{\alpha_{q+1}-1}, \quad 1\le q< p, \quad{\text{and}} \quad {\mathbf{R}}_p=
A_{\alpha_p} \cup \cdots \cup A_{\alpha_{r}}.
$$
We also have a partitioning into $p$ block columns, 
\begin{equation}\label{row1}
\bA^*={\mathbf{R}}^*_1\cup\cdots\cup {\mathbf{R}}^*_p
\end{equation}
Mutiply \eqref{row1} and \eqref{col1} and note that
\begin{equation}\label{note1}
|\bA|^2=\sum_{l=1}^p {\mathbf{R}}^*_l{\mathbf{R}}_l.
\end{equation}
with  $p$  block matrices in $\bM_{d'}$, ($l=1,\ldots,p$),
\begin{equation}\label{note2}
{\mathbf{R}}^*_l{\mathbf{R}}_l = \left( A^*_iA_j\right)_{\alpha_{l}\le i,j<\alpha_{l+1}}
\end{equation}
where we still use $\alpha_{p+1}:=r+1$. Applying the decomposition \eqref{key2} to the block matrices 
\eqref{note2} yields
$$
{\mathbf{R}}^*_l{\mathbf{R}}_l=\sum_{i=\alpha_l}^{\alpha_{l+1}-1} U_i |A_i|^2U_i^*
$$
for some isometries $U_i$ of suitable sizes, 
and combining with \eqref{note1} completes the proof. \end{proof}

Denote by $\mu_1(S) \ge \mu_2(S)\ge \cdots$ the singular values of a matrix $S\in\bM_{n,m}$. This list is often limited to
$\min\{n,m\}$ elements, however we can naturally define $\mu_k(S)=0$ for any index $k$ larger than $\min\{n,m\}$. Given two matrices
of same size, a classical inequality of Weyl asserts that
$$
\mu_{j+k+1} (S+T) \le \mu_{j+1}(S) +\mu_{k+1}(T)
$$
for all nonnegative integers $j$ and $k$. This inequality and Theorem \ref{th-pyth} entail the next corollary.

\vskip 5pt
\begin{cor}\label{cor-sing} Let $\bA\in\bM_{d,d'}$ be partitioned into $r$ row or column compatible blocks $A_k\in\bM_{n_k, m_k}$.
Then, for all nonnegative integers $j_1,j_2,\ldots, j_r$,
$$
\mu^2_{j_1+j_2+\cdots+j_r+1}(\bA)\le \sum_{k=1}^{r} \mu^2_{j_k+1}(A_k).
$$
\end{cor}

\vskip 5pt
A special case of this inequality is given in the abstract for three blocks and $j_1=2$, $j_2=1$, $j_3=0$.

Since any partitioning into three blocks is  row or column compatible we have the next corollary.

\vskip 5pt
\begin{cor}\label{cor-three} Let $\bA\in\bM_{d,d'}$ be partitioned into three blocks $A, B, C$. Then, there exist some isometries
$U,V,W$ of suitable sizes such that
$$
|\bA|^2=U|A|^2U^* + V|B|^2V^* + W|C|^2W^*.
$$
\end{cor}

By using the triangle inequality for the Schatten $p$-norms we have the trace inequality
\begin{equation}\label{trace}
\left\{{\mathrm{ Tr}}\,|\bA|^{2p}\right\}^{1/p}\le \left\{{\mathrm{ Tr}}\,|A|^{2p} \right\}^{1/p}+ \left\{{\mathrm{
Tr}}\,|B|^{2p}\right\}^{1/p} + \left\{{\mathrm{ Tr}}\,|C|^{2p}\right\}^{1/p}, \quad p\ge 1,
\end{equation}
equivalently
\begin{equation}\label{schatt}
\|\bA\|^2_q \le \| A\|^2_q + \| B\|^2_q +\| C\|^2_q
\end{equation}
for all Schatten $q$-norms, $q\ge 2$.

Theorem \ref{th-pyth} entails another interesting relation between the blocks of a partitioned matrix and the full matrix.

\begin{cor}\label{cor-direct} Let $\bA\in\bM_{d,d'}$ be partitioned into $r$ row or column compatible blocks $A_k\in\bM_{n_k, m_k}$.
Then, for some isometries $V_j\in\bM_{m,d'}$, with $m=\sum_{k=1}^r m_k$,
\begin{equation*}
\bigoplus_{k=1}^r |A_k|^2 =\frac{1}{r}\sum_{j=1}^r V_j|\bA|^2  V_j^*.
\end{equation*}
\end{cor}

\vskip 5pt
\begin{proof} From Theorem \ref{th-pyth} and the main result of \cite{BL-direct} we have
$$
\bigoplus_{k=1}^r U_k|A_k|^2U_k =\frac{1}{r}\sum_{j=1}^r W_j|\bA|^2  W_j^*
$$
for some isometries $U_k\in\bM_{d',m_k}$ and some isometries $W_j\in \bM_{rd',d'}$. Since
$$
\bigoplus_{k=1}^r |A_k|^2= C\left\{ \bigoplus_{k=1}^r U_k|A_k|^2U_k\right\}C^*
$$
for some contraction $C\in\bM_{m, rd'}$, we infer
$$
\bigoplus_{k=1}^r |A_k|^2=\frac{1}{r}\sum_{j=1}^r CW_j|\bA|^2  W_j^*C^*.
$$
If $|\bA|$ is invertible, then, taking trace, the above equality ensures that contractions $CW_j$ satisfy
$W_j^*C^*CW_j={\mathbf{1}}_{d'}$ for all $j$. Hence the result is proved with $V_j=CW_j$. The general case follows by a limit
argument.
\end{proof}

We are in a position to estimate the Schatten norms of the blocks with the full matrix.
The following corollary was first obtained by Bhatia and Kittaneh \cite{Bh-K} in case of a matrix partitioned into $n\times n$
blocks.

\vskip 5pt
\begin{cor}\label{corBK}  Let $\bA\in\bM_{d,d'}$ be partitioned into  $r$ row or column compatible blocks $A_k\in\bM_{n_k, m_k}$. 
Then, for all Schatten $q$-norms, $q\ge 2$,
\begin{equation*}r^{{\frac{2}{q}}-1}\sum_{k=1}^r \| A_k \|_q^2 \le \| \bA\|_q^2 \le \sum_{k=1}^r \| A_k \|_q^2
\end{equation*}
These two inequalities are reversed for $2>q>0$.
\end{cor}

\vskip 5pt
\begin{proof} For $p:=q/2\ge 1$, the second inequality contains \eqref{schatt} and immediately follows from Theorem \ref{th-pyth} and
the triangle inequality for the Schatten $p$-norms. Corollary \ref{cor-direct} gives
$$
\| |\bA|^2\|_p\ge \left\| |A_1|^2\oplus \cdots\oplus |A_r|^2 \right\|_p 
$$
and since the concavity of $t\mapsto t^{1/p}$ entails
$$
\left\| |A_1|^2\oplus \cdots\oplus |A_r|^2 \right\|_p = \left(\| |A_1|^2\|_p^p+\cdots \| |A_r|^2\|_p^p\right)^{1/p}\ge
r^{\frac{1}{p}-1}\left(\| |A_1|^2\|_p+\cdots+ \| |A_r|^2\|_p\right)
$$
we get the first inequality. These inequalities are reversed for $0<p<1$.
\end{proof}

\vskip 5pt
Corollary \ref{cor-direct} is relevant to Majorisation Theory. We take  this opportunity to point out an interesting
fact about majorisation in the next proposition. Though this result might be well-known to some experts, it does not seem to be in
the literature. Let $A,B\in\bM_n^+$, the positive semi-definite cone of $\bM_n$. The majorisation $A\prec B$ means that
$$
\sum_{j=1}^k \mu_j(A) \le \sum_{j=1}^k \mu_j(B)
$$
for all $k=1,2,\ldots n$, with equality for $k=n$.
The majorisation $A\prec B$ is equivalent to
 $$A=\sum_{i=1}^{n} \alpha_iU_iBU_i^*$$
for some unitary matrices $U_i\in\bM_n$ and weights $\alpha_i\ge 0$ with $\sum_{i=1}^{n} \alpha_i=1$. This can be easily derived from
Caratheodory's theorem \cite{Zhan}. A more accurate statement holds.

\vskip 5pt
\begin{prop}\label{prop-maj}  Let $A,B\in\bM_n^+$, $A\prec B$.Then, for some unitary matrices $U_i\in\bM_n$,
$$
A=\frac{1}{n}\sum_{i=1}^n U_i BU_i^*.
$$
\end{prop}

\vskip 5pt
\begin{proof} By the Schur-Horn Theorem, we may assume that $A$ is the diagonal part of $B$. Then we use the simple idea of Equation
(2) in the nice paper of Bhatia \cite{Bh-monthly}.
\end{proof}

\vskip 5pt
Note that Corollary \ref{cor-direct} may be restated as 
\begin{equation*}
\bigoplus_{k=1}^r |A_k|^2 =\frac{1}{r}\sum_{j=1}^r U_j(|\bA|^2\oplus O)  U_j^*.
\end{equation*}
for some unitary matrices $U_j$ and some fixed zero matrix $O$. Hence we have an average of $r$ matrices in the unitary orbit of
$|\bA|^2\oplus O$, this number $r$ being (much) smaller than the one given by Proposition \ref{prop-maj}, $d=m_1+\cdots +m_r$.

\section{Compression onto a hyperplane}

By a hyperplane of $\bC^d$ we mean a vector subspace of dimension $d-1$. The next corollary is a singular value version of Cauchy's
Interlacing Theorem \cite[p.\ 59]{Bh}.

\vskip 5pt
\begin{cor}\label{hyper}  Let $A\in\bM_d$ and let  ${\mathcal{S}}$ be a hyperplane of  $\bC^d$ orthogonal to the unit vector $h$.
Set $\beta= \min\{\| Ah\|^2, \| A^*h\|^2 \} -|\langle h, Ah\rangle|^2$. Then, for all  $j=1,\ldots, d-1$,
$$
\mu_{j}^2(A) \ge \mu_{j}^2(A_{\mathcal{S}}) \ge \mu_{j+1}^2(A) -\beta.
$$
\end{cor}

\vskip 5pt
This double inequality is stronger than $\mu_{j}(A)\ge \mu_{j}(A_{\mathcal{S}}) \ge \mu_{j+1}(A) -\sqrt{\beta}$. If $A$ is a normal
matrix, then $\| Ah\|=\| A^*h\|$ and we have Corollary \ref{normal}. If $A=V$ is a unitary matrix, $\mu_j(V)=1$ for all $j$ and $\|
Vh\|=\| V^*h\|=1$ for all unit vectors, so we deduce from Corollary \ref{hyper} that
$\mu_{j}(V_{\mathcal{S}}) \ge |\langle h, Vh\rangle|$. In fact one can easily check that $\mu_j(V_{\mathcal{S}})=1$ for $j\le d-2$
and $\mu_{d-1}(V_{\mathcal{S}})= |\langle h, Vh\rangle|$. Hence Corollary \ref{hyper} is sharp.

\vskip 5pt
\begin{proof} (Corollary \ref{hyper}) The inequality $\mu_j(A)\ge \mu_j(A_{\mathcal{S}})$ is trivial. To deal with the other
inequality
we may assume that $h$ is the last vector of the canonical basis and that $A_{\mathcal{S}}$ is the submatrix of $A$ obtained by
deleting the last column and the last line. We partition $A$ as
$$
A=A_{\mathcal{S}}\cup B\cup C
$$
where $B$ contains  the $d-1$ entries below $A_{\mathcal{S}}$ and $C$ is the last column of $A$. We then apply to this partitioning 
Corollary \ref{cor-sing} with  $j_1=j-1$, $j_2=0$, and $j_3=1$ to get
$$
\mu_{j+1}^2(A) \le \mu_{j}^2(A_{\mathcal{S}}) +  \mu_1^2(B) +\mu_2^2(C).
$$
Since $\mu_2(C)=0$, we have
\begin{equation*}\
\mu_{j+1}^2(A) -\mu_1^2(B) \le \mu_{j}^2(A_{\mathcal{S}})
\end{equation*}
Observe that $\mu_1^2(B)= \| A^*h\|^2  -|\langle h,Ah\rangle|^2$, hence
\begin{equation}\label{svi1}
\mu_{j}^2(A_{\mathcal{S}}) \ge \mu_{j+1}^2(A) -\| A^*h\|^2  +|\langle h, Ah\rangle|^2.
\end{equation}

We may also partition $A$ as
$$
A=A_{\mathcal{S}}\cup R\cup L
$$
where $R$ contains the $d-1$ entries on the right of $A_{\mathcal{S}}$ and $L$ stands for the last line of $A$.
Arguing as above with $R$ and $L$ in place of $B$ and $C$ yields
\begin{equation}\label{svi2}
\mu_{j+1}^2(A) -\mu_{j}^2(A_{\mathcal{S}}) \le \mu_1^2(R)= \| Ah\|^2  -|\langle h, Ah\rangle|^2
\end{equation}
Combining \eqref{svi1} and \eqref{svi2} completes the proof.
\end{proof}

\vskip 5pt
\begin{cor}\label{cor-var} Let $\bA\in\bM_{d,d'}$ be partitioned into $r$ row or column compatible blocks $A_k\in\bM_{n_k, m_k}$.
Then, for each block $A_k$ and all $j\ge 1$,
$$
\mu^2_j(\bA)-\mu^2_j(A_k)\le \sum_{l\neq k} \mu^2_1(A_l).
$$
\end{cor}

\vskip 5pt
\begin{proof} Apply Corollary \ref{cor-sing} with $j_k=j-1$ and $j_l=0$ for all $l\neq k$. \end{proof}

\vskip 5pt
\begin{cor}\label{cor-comp} Let $A\in\bM_d$ and let ${\mathcal{S}}$ be a hyperplane of $\bC^d$ orthogonal to the unit vector $h$.
Then for all $j=1,\ldots, d-1$,
$$
\mu^2_j(A)-\mu^2_j(A_{\mathcal{S}})\le \| Ah\|^2 + \| A^*h\|^2-|\langle h,Ah\rangle|^2.
$$
\end{cor}

\vskip 5pt
\begin{proof}
We may suppose that $h$ is the last vector of the canonical basis and we partition $A$ into three blocks : $A_{\mathcal{S}}$, the
last column of $A$, and the $d-1$ entries below $A_{\mathcal{S}}$. We then apply the previous corollary.
 \end{proof}

\section{Four and five blocks}

Partitionings into four blocks are not necessarily row or column compatible. However, for such partitionings, Theorem \ref{th-pyth}
still holds.

\vskip 5pt
\begin{cor}\label{cor-four} Let $\bA\in\bM_{d,d'}$ be partitioned into four blocks $A, B, C, D$. Then, there exist some isometries
$U,V,W,X$ of suitable sizes such that
$$
|\bA|^2=U|A|^2U^* + V|B|^2V^* + W|C|^2W^* + X|D|^2X^*.
$$
\end{cor}

\vskip 5pt
\begin{proof} We assume that $A$ is the block in the upper left corner and we distinguish three cases.

(1) 
 $A$ has the same number  $d$ of  lines as $\bA$.
In such a case, letting $A'=B\cup C\cup D$, the partitioning $\bA=A\cup A'$ is column compatible, and we have two isometry matrices
$U, U'$ such that
\begin{equation}\label{f1}
|\bA|^2=U|A|^2U^* + U'|A'|^2U'^*.
\end{equation}
Since $A'$ is partitioned into three blocks, necessarily a row or column partitioning, we can apply the theorem to obtain the
decomposition
\begin{equation}\label{f2}
|A'|^2=V'|B|^2V'^* + W'|B|^2W'^* +X'|B|^2X'^* 
\end{equation} 
for some isometry matrices $V',W',X'$ of suitable sizes. Combining \eqref{f1} and \eqref{f2} we get the conclusion of the corollary
with the isometry matrices $V=U'V'$, $W=U'W'$, and $X=U'X'$.

(2) 
 $A$ has the same number  $d'$ of  columns as $\bA$.
Letting again $A'=B\cup C\cup D$, the partitioning $\bA=A\cup A'$ is  row compatible, and we may argue as in case (1).

(3)
$A$ has $l<d$ lines and $c<d'$ columns. There exist then a block, say $B$, on the top position, and just on the right of $A$, and
another block, say $C$ just below $A$ and on the left side. We consider three subcases (a), (b), (c).

(a) $B$ has fewer than $l$ lines. Then, the last block $D$ is necessarily below $B$ with the same number of columns as $B$, and so
$C$ has either the same number of columns as $A$ or $C$ has $d'$ columns as $\bA$. In the first case, $\bA=A\cup C \cup B \cup D$ is
a column compatible partitioning and we can apply the theorem. In the second case, the situation is the same as in (2).

(b) $B$ has exactly $l$ lines, like $A$. We denote by $\gamma$  the number of columns of $B$ and we consider three situations.

(i) $C$ has more than $c+\gamma$ columns. Then necessarily $C$ has $d'$ columns and $D$ is the upper right block with $l$ lines,
hence $\bA=A\cup B \cup D \cup C $ is a line compatible partitioning and we may apply the theorem.

(ii) $C$ has exactly $c+\gamma$ columns. Then, letting $\bA''=A\cup B\cup C$ with have a partitioning into three blocks, and $\bA =
\bA''\cup D$. Thus applying the theorem twice as in case (1) yields the conclusion.

(iii) $C$ has fewer than $c+\gamma$ columns. Then $D$ is the lower right block, with the same number of lines as $B$, and $\bA$ is
partitioned into line compatible blocks. Thus the theorem can be applied.

(c)   $B$ has more lines than $A$.  Let $\lambda$ be the number of line of $B$. Hence $\lambda>l$. There exist two situations

(iv) $\lambda<d$. Then $D$ is the lower right block, with the same number of columns as $B$, and $\bA$ is partitioned into line
compatible blocks. Thus we may apply the theorem.

(v) $\lambda=d$. Then $A'''= A\cup C\cup D$ is a partitioning into three blocks and $\bA= A'''\cup B$, thus applying twice the
theorem completes the proof.
\end{proof}

We do not know whether Corollary \ref{cor-four} can be extended or not to any partitioning in five blocks. For instance we are not
able to prove or disprove a version of Corollary \ref{cor-four} for the matrices
$$
\bA= \begin{pmatrix}  a_1&a_2&b_1 \\ d_1 & x& b_2 \\ d_2 & c_1 & c_2\\
\end{pmatrix}
\quad {\mathrm{or}} \quad
\bA=\begin{pmatrix} a_1&a_2 &a_3 &b_1 &b_2 \\ a_4&a_5 &a_6 &b_3 &b_4 \\ d_1&d_2 &x &b_5 &b_6 \\ d_3&d_4 &c_1&c_2 &c_3 \\ d_5&d_6
&c_4&c_5 &c_6
\end{pmatrix}
$$
partitioned into five obvious blocks $A,B,C,D,X$. Hence, that Theorem \ref{th-pyth} holds or not for any partitioning into five
blocks is an open problem. More generally, we may consider the following two questions.

\begin{question} For which partitionings does Theorem \ref{th-pyth} hold ? For which partitionings does Corollary \ref{cor-sing} hold
?
\end{question}

\vskip 5pt
Matrices partitioned into four blocks (usually of same size) are comon examples of partitionings. A nontrivial inequality follows
from the previous corollary.

\vskip 5pt
\begin{cor}\label{cor-four2} Let $\bA\in\bM_{d,d'}$ be partitioned into four blocks $A, B, C, D$, and let $p>2$. Then, there exist
some isometries $U,V,W,X$ of suitable sizes such that
$$
2^{2-p}|\bA|^{p}\le U|A|^pU^* + V|B|^pV^* + W|C|^pW^* + X|D|^pX^*.
$$
The inequality reverses for  $2>p>0$.
\end{cor}

\vskip 5pt
Letting $p=1$ we have Corollary \ref{revtriangle} with the constant 2 which is sharp, even for a positive block matrix, as shown by
the simple example
$$
\bA=\begin{bmatrix} A&A \\ A&A\end{bmatrix}.
$$

\vskip 5pt
\begin{proof}
For any monotone convex function $f(t)$ on the nonnegative axis, we have thanks to \cite[Corollary 2.4]{BL-London} and Corollary
\ref{cor-four},
\begin{align*}
f\left(\frac{|\bA|^2}{4}\right) &=f\left(\frac{U|A|^2U^* + V|B|^2V^* + W|C|^2W^* + X|D|^2X^*}{4}\right)  \\
&\le \Lambda\frac{f(U|A|^2U^*)+ V(f|B|^2V^*) + f(W|C|^2W^*) + f(X|D|^2X^*)}{4}\Lambda^*
\end{align*}
for some unitary matrix $\Lambda\in\bM_d'$. Picking $f(t)=t^{p/2}$ with $p>2$ yields the result. The reverse inequalities hold for
monotone concave functions $f(t)$ and $0<p<2$.
\end{proof}

\vskip 5pt
\begin{remark} The version of Corollary \ref{cor-four2} for three blocks $A,B,C$, and $p=1$ reads as the  inequality of the abstract,$$
\sqrt{3} |\bA| \ge  U|A|U^* + V|B|V^*+ W|C|W^*.
$$
The constant $\sqrt{3}$ is the best one: we cannot take a smaller constant for
 $$\bA=\begin{bmatrix} x&y&z \\ x&y&z\\ x&y&z\end{bmatrix}$$
 partitioned into its three lines. For two blocks, a similar sharp inequality holds with the constant $\sqrt{2}$.
\end{remark}

\section{Concave or convex functions}

For sake of simplicity we state our results for a square matrix $\bA$ partitioned into blocks. By adding some zero rows or zero
columns to a rectangular matrix, we could obtain statements for rectangular matrices (Remark \ref{last}).

Suppose that $\bA\in\bM_d$ is partitioned into blocks $A_k\in\bM_{n_k,m_k}$, $k=1,\ldots, r$. From Thompson's triangle inequality
(\cite{T} or \cite[p.\ 74]{Bh} we have
\begin{equation}\label{thomp}
|\bA| \le \sum_{k=1}^r U_k|A_k|U_k^*
\end{equation}
for some isometry matrices $U_k\in\bM_{d,m_k}$. The equality of Theorem \ref{th-pyth} and \eqref{thomp} suggest several other
inequalities, in particular, if $\bA$ is partioned in row or column compatible blocks,
\begin{equation}\label{spe}
|\bA|^3 \ge \sum_{k=1}^r V_k|A_k|^3V_k^*
\end{equation}
for some isometries $V_k\in\bM_{d,m_k}$. This is indeed true as shown in the following theorem. We do not know if \eqref{spe} can be
extended to any partitioning. Corollary \ref{cor-four} and the proof of Theorem \ref{th-convex} show that \eqref{spe} holds for four
blocks. The case of five blocks is open.

\vskip 5pt
\begin{theorem}\label{th-convex} Let $\bA\in\bM_{d}$ be partitioned into $r$ row or column compatible blocks $A_k\in\bM_{n_k, m_k}$,
and let $\psi(t)$ be a monotone function on $[0,\infty)$ such that $\psi(\sqrt{t})$ is convex and $\psi(0)= 0$. Then there exist some
isometries $V_k\in\bM_{d,m_k}$ such that
$$
\psi(|\bA|)\ge \sum_{k=1}^r V_k  \psi(|A_k|) V_k^*.
$$
\end{theorem}

\vskip 5pt
Theorem \ref{th-convex} considerably improves \eqref{trace}. A special case with $\psi(t)=t^3$ is given in the abstract.

\vskip 5pt
\begin{proof} Let $g(t)$ be a monotone convex function on $[0,\infty)$ such that $g(0)\le 0$, and let $A,B\in\bM_n$ be positive
(semidefinite). By \cite{AB} or \cite[Corollary 3.2]{BL-London} we have
\begin{equation*}
g(A+B) \ge Ug(A)U^* + Vg(B)V^*
\end{equation*}
for some unitary matrices $U,V\in\bM_n$. Using this inequality and Theorem \ref{th-pyth} we infer
\begin{equation*}\label{bien}
g(|\bA|^2)\ge \sum_{k=1}^r W_k g(U_k|A_k|^2U_k^*) W_k^*
\end{equation*}
for some unitary matrices $W_k$ and some isometry matrices $U_k\in\bM_{d,m_k}$. If $g(0)=0$, we have
$g(U_k|A_k|^2U_k^*)=U_kg(|A_k|^2)U_k^*$. Hence
\begin{equation*}
g(|\bA|^2)\ge \sum_{k=1}^r V_k g(|A_k|^2) V_k^*
\end{equation*}
with the isometry matrices $V_k=W_kU_k$. Applying this to $g(t)=\psi(\sqrt{t})$ completes the proof.
\end{proof}

\vskip 5pt
\begin{cor}\label{cor-concave} Let $\bA\in\bM_{d}$ be partitioned into $r$ row or column compatible blocks $A_k\in\bM_{n_k, m_k}$,
and let $\varphi(t)$ be a nonnegative function on $[0,\infty)$ such that $\varphi(\sqrt{t})$ is concave. Then there exist some
isometries $U_k\in\bM_{d,n_k}$ such that
$$
\varphi(|\bA|)\le \sum_{k=1}^{r} U_k  \varphi(|A_k|) U_k^*.
$$
\end{cor}

\vskip 5pt
\begin{proof} Since $\varphi(\sqrt{t})$ is nonnegative and concave, it is necessarily a monotone function (nondecreasing), hence
continuous on $(0,\infty)$. Since we are dealing with matrices we may further suppose that $\varphi(t)$ is also continuous at $t=0$.

(1) Assume that $\varphi(0)=0$. Theorem \ref{th-convex} applied to $\psi(t)=-\varphi(t)$ proves the corollary.

(2) Assume that $\varphi(0)>0$. Since the continuous functional calculus is continuous on the positive semidefinite cone of any
$\bM_m$, by a limit argument, we may assume that $|\bA|$ is invertible. So, suppose that the spectrum of $|\bA|^2$ lies in an
interval $[r^2, s^2]$ with $r > 0$.
Define a convex function $\phi(\sqrt{t})$ by $\phi(\sqrt{t})=\varphi(\sqrt{t})$ for $t\ge r^2$, $\phi(0)=0$, and the graph of
$\phi(\sqrt{t})$ on $[0,r^2]$ is a line segment. Hence $\phi(t)\le \varphi(t)$ and $\phi(|\bA|)=\varphi(|\bA|)$. Applying case (1) to
$\phi$ yields
$$
\varphi(|\bA|)=\phi(|\bA|)\le \sum_{k=1}^{r} U_k  \phi(|A_k|) U_k^* \le \sum_{k=1}^{r} U_k  \varphi(|A_k|) U_k^*
$$
for some isometry matrices $U_k$.
\end{proof}

\vskip 5pt
The next three corollaries follow from Corollary \ref{cor-concave}.

\vskip 5pt
\begin{cor}\label{cor-1} Let $\bA\in\bM_{mn}$ be partitioned into an $m\times m$ family of blocks $A_{i,j}\in\bM_{n}$, and let $0<
q\le 2$. Then there exist some isometries $U_{i,j}\in\bM_{mn, n}$ such that
$$
|\bA|^q\le \sum_{i,j=1}^m U_{i,j}  |A_{i,j}|^q U_{i,j}^*.
$$
\end{cor}

\vskip 5pt
\begin{cor}\label{cor-2} Let $\bA\in\bM_{mn}$ be partitioned into an $m\times m$ family of blocks $A_{i,j}\in\bM_{n}$, let $s\ge 1$
and $0< q\le 2$. Then,
$$
\left\{{\mathrm{Tr}\,}|\bA|^{qs}\right\}^{1/s}\le \sum_{i,j=1}^m  \left\{{\mathrm{Tr}\,}|A_{i,j}|^{qs}\right\}^{1/s}.
$$
\end{cor}

\vskip 5pt
\begin{cor}\label{cor-3} Let $A\in\bM_{n}$, let $c_k$ be the norm of the $k$-th column of $A$ and let $0< q\le 2$. Then there exist
some rank one projections $E_k\in\bM_{n}$ such that
$$
|\bA|^q\le \sum_{k=1}^n c_k^q E_k
$$
\end{cor}

\vskip 5pt
The last corollaries follow from Theorem \ref{th-convex}.

\vskip 5pt
\begin{cor}\label{cor-4} Let $\bA\in\bM_{mn}$ be partitioned into an $m\times m$ family of blocks $A_{i,j}\in\bM_{n}$, and let $p\ge
2$. Then there exist some isometries $U_{i,j}\in\bM_{mn, n}$ such that
$$
|\bA|^p\ge \sum_{i,j=1}^m U_{i,j}  |A_{i,j}|^p U_{i,j}^*.
$$
\end{cor}

\vskip 5pt
\begin{cor}\label{cor-5} Let $\bA\in\bM_{mn}$ be partitioned into an $m\times m$ family of blocks $A_{i,j}\in\bM_{n}$, let $0\le s\le
1$ and $p\ge 2$. Then,
$$
\left\{{\mathrm{Tr}\,}|\bA|^{ps}\right\}^{1/s}\ge \sum_{i,j=1}^m  \left\{{\mathrm{Tr}\,}|A_{i,j}|^{ps}\right\}^{1/s}.
$$
\end{cor}

\vskip 5pt
\begin{cor}\label{cor-6} Let $A\in\bM_{n}$, let $r_k$ be the norm of the $k$-th row of $A$ and let $p\ge 2$. Then there exist some
rank one projections $E_k\in\bM_{n}$ such that
$$
|\bA|^p\ge \sum_{k=1}^n c_k^p E_k.
$$
\end{cor}

\vskip 5pt
\begin{remark}\label{last} The proof of Theorem \ref{th-convex} is the same for a $d\times d'$ matrix $\bA$. So Corollary
\ref{cor-concave} also holds for $\bA\in\bM_{d,d'}$ if $\varphi(0)=0$. In case of $d\ge d'$, we may again use a limit argument and
assume that $|\bA|$ is invertible. In case of $d'>d$ we may argue as follows. Add some zero lines to $\bA$ in order to obtain
a square matrix $\bA_0\in\bM_d'$. Let $B_1\ldots,B_p$ be the blocks at the bottom of $\bA$, and $R_1,\ldots,R_q$ be the remaining
blocks of $\bA$. Add some zeros to the blocks $B_i$ in order to obtain blocks $B_i^0$ of $\bA_0$ in such a way that
$$
\bA_0 =\left(\bigcup_{i} B_i^0\right)\cup \left(\bigcup_{j} R_j\right)
$$
is a row or column compatible partitioning of $\bA_0$. Since it is a square matrix, we may apply Corollary \ref{cor-concave} and
since
$|\bA_0|=|\bA|$ and  $|B_i^0|=|B^i|$, we see that Corollary \ref{cor-concave} holds for $d\times d'$ matrices.
\end{remark}

\section{Around this article}

For a positive block-diagonal matrix, Theorem \ref{th-pyth} is trivial. The following statement \cite{BL-direct} is more interesting.

\vskip 10pt
\begin{theorem} \label{th-directsum} Let $A_i\in\bM_n^+$, $i\in\bI_m$. Then, for some isometries $V_k\in\bM_{mn,n}$, $k\in\bI_m$, 
\begin{equation*}
\bigoplus_{i=1}^m A_i =\frac{1}{m}\sum_{k=1}^m V_k\left\{\sum_{i=1}^m A_i \right\} V_k^*.
\end{equation*}
\end{theorem}

\vskip 10pt 
This says that direct sums are averages of usual sums, up to isometric congruences. It is a genuine non-commutative fact with no analogous statement for positive vectors and permutations of their components. 

\section{References of Chapter 8}

{\small
\begin{itemize}

\item[[12\!\!\!]]  J.S.\ Aujla and J.-C.\ Bourin, Eigenvalue inequalities for convex and log-convex functions,
\textit{Linear Algebra Appl.} 424 (2007), 25--35.

\item[[13\!\!\!]]  R.\ Bhatia, Matrix Analysis, Gradutate Texts in Mathematics, Springer, New-York, 1996.

\item[[14\!\!\!]] R.\ Bhatia, Pinching, trimming, truncating, and averaging of matrices. {\it Amer.\ Math.\ Monthly} 107 (2000),

\item[[17\!\!\!]] 
 R.\ Bhatia and F.\ Kittaneh, Norm inequalities for partitioned operators and an application. {\it Math. Ann}.\ 287
(1990), no.\ 4, 719--726.
no. 7, 602--608.

\item[[31\!\!\!]]  J.-C.\ Bourin and E.-Y.\ Lee, Unitary orbits of Hermitian operators with convex or concave functions, {\it Bull.\ Lond.\
Math.\ Soc.}\ 44 (2012), no.\ 6, 1085--1102.

\item[[35\!\!\!]] J.-C.\ Bourin and E.-Y.\ Lee, Direct sums of positive semi-definite matrices. {\it Linear Algebra Appl}.\ 463
(2014), 273--281.

\item[[42\!\!\!]]  J.-C.\ Bourin and E.-Y.\ Lee,  A Pythagorean Theorem for partitioned matrices, Proc. Amer. Math. Soc., in press.

\item[[92\!\!\!]]   R.-C.\ Thompson, Convex and concave functions of singular values of matrix sums, {\it Pacific J.\ Math.}\ 66 (1976),
285--290.

\item[[94\!\!\!]]  X.\ Zhan, The sharp Rado theorem for majorizations, {\it Amer.\ Math.\ Monthly} 110
(2003) 152--153.

\end{itemize}


\chapter{The spectral Theorem}

\section{Introduction}

Hermitian (or symmetric) matrices acting on ${\mathcal{H}}_n=\bC^n$ (or $\bR^n$) form a nice finite dimensional real vector space,
let us denote it by $\bS({\mathcal{H}}_n)$, with a order structure, $A\le B$, whenever $B-A$ is positive semidefinte.

The fundamental property of $\bS({\mathcal{H}}_n$) is the {\bf matrix spectral theorem} asserting that a Hermitian matrix can be
diagonalized in a suitable orthonormal basis. It is a very important theorem and, important too, it is an easy result well understood
by students. Combined with the order structure, this theorem shows that Hermitian matrices can be regarded as a generalization of
$\bR^n$ and opens the way of a full theory of Matrix Analysis, where facts for real finite sequences of numbers search for
counterparts in the Hermitian matrix world.

But, the extension of this diagonalization theorem for operators on a infinite dimensional separable Hilbert space ${\mathcal{H}}$ -
a key stone result of Functional Analysis - is more delicate: the standard literature invoke abstract constructions based on the
Gelfand isomorphism between some abstract algebras. It might be a too abstract and unnatural approach; the link with matrices
unexpectedly disappears though operators obviously have matrix representations.

However, and fortunately as it is conceptually desirable, it is possible, to give a simple proof deriving from the matrix case. Hence
operators are not more complicated than matrices any longer ! We do this rather easy and definitely pleasant job in this note.

\section{The matrix case, and two natural definitions for operators}

Let us first recall a sketch of the proof
of the spectral theorem for Hermitian matrices.
\begin{itemize}
\item if a subspace ${\mathcal{S}}\subset {\mathcal{H}}_n$ is invariant by $A$, then so is its orthocomplement
${\mathcal{S}}^{\perp}$.
\item $\lambda_1^{\uparrow}(A):=\min\langle h, Ah\rangle$ must be an eigenvalue,
\end{itemize}
so that repeating the process with the restriction of $A$ to the orthocomplement .... This is the variational method. and we arrive
to the following spectral decomposition of $A$:
\begin{equation}\label{matrix1}
A=\sum_{j=1}^k  \lambda_j^{\uparrow}(A)  P_j
\end{equation}
A very important feature of this decompostion is to allow a matrix functional calculus, given $f(t)\in C^0$, the matrix $f(A)$ makes
sense. This is the starting point of a branch of Matrix analysis. Note that the matrix functional calculus behave well with the order
structure of $C^0$,: $f(t)\ge g(t))$ implies $f(A) \ge g(A)$ for all Hermitian matrices $A$. this will be crucial in our approach.

Now, let us consider the typical situation in the infinite dimensional case, the multiplication operator on $L_2([0,1],$
$$
Z: f(t)\mapsto tf(t), \qquad f(t)\in L_2([0,1], {\mathrm d}t)
$$
This basic example shows that it is problematic to use eigenvalues and projection $P_j$ as in \eqref{matrix1}. However it seems
possible to have a corresponding notion for the sum
\begin{equation}\label{mspp}
E_j=P_1 +P_{2}+\cdots P_{j},\qquad j=1,2\cdots
\end{equation}
 of the first $j$ spectral projections of $A$ by using the operators
$$
E_w: f(t)\mapsto {\mathbb{I}}_{[0, w]}(t) f(t), \qquad w\in [0,1].
$$
This is a very good point, because we may write the matrix spectral theorem \eqref{matrix1} as:
\begin{equation}\label{matrix2}
A= \lambda_1^{\uparrow}(A) E_1 + \lambda_2^{\uparrow}(A) (E_2-E_1)+ \lambda_3^{\uparrow}(A)(E_3-E_2)+....+
\lambda_n^{\uparrow}(A)(E_n-E_{n-1})
\end{equation}
We will obtain a similar statement for Hilbert space operators.
To adapt the matrix proof to the infinite dimensional setting, we need two simple definitions. Let $\bS({\mathcal{H}})$ be the set of
Hermitian operators on a infinite dimensional, separable Hilbert space ${\mathcal{H}}$.

\vskip 5pt
\begin{definition} (sot-convergence, or strong convergence) We say that the sequence $\{X_n\}$ in $\bS({\mathcal{H}})$ is sot
convergent, or strongly converges, if it is norm bounded and, for all vectors $h\in{\mathcal{H}}$ (equivalently for all elements of
an orthonormal basis), $\{X_nh\}$ converges in ${\mathcal{H}})$. If $\{X_n\}$ is sot convergent, then there exists (a unic) $X\in
\bS({\mathcal{H}})$ such that $Xh=\lim_n X_nh$ for all vectors $h\in{\mathcal{H}}$ and we say that $X$ is the strong limit of
${X_n}$, written as $X={\mathrm{sot}}\lim_{n\to \infty} X_n$.
\end{definition}

Hence the strong convergence is merely the pointwise convergence in $\bS({\mathcal{H}})$. Of course, it is a classical notion and
"sot" refers to strong operator topology.

Given $A\in\bS({\mathcal{H}})$, and fixing a (orthonormal) basis, we may write $A$ as an infinite matrix $A=(a_{i,j})$, and so
extracting the $n$-by-$n$ left upper corners, $n=1,2,3,....$ we obtain a basic sequence of Hermitian matrices for which the matrix
spectral theorem and the functional calculus are available. Since we wish to extend it to $A$, it seems natural to use this sequence
of matrices, and the following terminology is convenient.

\vskip 5pt
\begin{definition} The sequence $\{A_n\}$ in $\bS({\mathcal{H}})$ is a {\bf basic sequence} for $A$ if there exists an increasing
sequence of finite rank projections $\{E_n\}$, sot-convergent to the identity $I$, such that $A_n =E_nAE_n$.
\end{definition}

\section{Continuous functional calculus}

In the introduction we have noted that the matrix spectral theorem allows to define the functional calculus of Hermitian matrices.
Our strategy in the Hilbert space operator case is opposite: we first built up a continuous functional calculus (thanks to the the
two above definition, and -of course- to the matrix spectral theorem)

\vskip 5pt
\begin{lemma} If $\{A_n\}_{n\in\bN}$ is a basic sequence for $A\in\bS({\mathcal{H}})$ and $p(t)$ is a polynomial, then
$${\mathrm{sot}}\lim_{n\to \infty} p(A_n) = p(A).$$
\end{lemma}

\vskip 5pt
\begin{proof} \end{proof}

\vskip 5pt
\begin{lemma}\label{lem-1} Let $f(t)\in C^0$ and let $A\in \bS({\mathcal{H}})$. Then, there exists an operator in
$\bS({\mathcal{H}})$, that we denote $f(A)$, such that
\begin{equation}
f(A)= {\mathrm{sot}}\lim_{n\to \infty} f(A_n)
\end{equation}
for any basic sequence $\{A_n\}_{n\in\bN}$ for $A$. Moreover $f(A)A=Af(A)$.
\end{lemma}

\vskip 5pt
Thus, we have natural way to define operators $f(A)$ from continuous function $f(t)$ and Hermitian arguments $A$. This is called the
(operator) continous functional calculus. Lemma 2.1 shows that for a polynomial, the continous functional calculus yields, of course,
the same operator as its algebraic definition.

\vskip 5pt
\begin{proof}
Fix a unit vector $h\in{\mathcal{H}}$  and a scalar $\varepsilon$, and pick a polynomial $p(t)$ such that
$
|f(t)-p(t)| <\varepsilon
$
whenever $|t|\le \|A\|$. Since $\|A_n\|\le \|A\|$, for any $n$, we infer
\begin{equation}\label{equa-1}
\| f(A_n)-p(A_n)\|\le \varepsilon.
\end{equation}
Let  $n_0$ be an integer such that
\begin{equation}\label{equa-2}
\| p(A_n)h-p(A)h \| \le \varepsilon
\end{equation}
for all $n\ge n_0$. Combining \eqref{equa-1} and \eqref{equa-2} we get
$
\| f(A_n)h-p(A)h \| \le 2\varepsilon
$
for all $n\ge n_0$. As $\varepsilon$ can be arbitrarily small, we obtain that $\{f(A_n)h\}$ is a Cauchy sequence in ${\mathcal{H}}$
and thus has a limit that we call $f(A)h$. It obviously defines a linear operator as $f(A_n)$ are linear operators. And since $\|
f(A_n)\|\le \max\{ |f(t)| : -\|A\| \le t\le \|A\|\}$,
$f(A)$ is a bounded linear operator. Thus we have the strong limit in $\bS$,
$
f(A)= {\mathrm{sot}}\lim_{n} f(A_n).
$
To check that $f(A)$ does not depend on the basic sequence of $A$, suppose that $\{A'_n\}$ is another basic sequence of $A$ and note
that we also have
$
\| f(A'_n)h-p(A)h \| \le 2\varepsilon
$
for $n$ large enough and thus $\{f'(A_n)\}$ also converges to $f(A)h$. Finally, as $f(A_n)A_n=A_nf(A_n)$, taking the strong limits,
we infer $f(A)A=Af(A)$.
 \end{proof}

\vskip 5pt
For a finite rank operator $Z\in \bS({\mathcal{H}})$, ${\mathrm{rank}}\, Z=m$, with spectral decomposition
$$
Z=\sum_{j=1}^m z_j^{\uparrow} P_j,
$$
where $z_j$ are the nonzero eigenvalues of $Z$ counted with their multiplicities and $P_j$ are corresponding rank one spectral
projections,
the functional calculus with $f\in C^0$ is obtained by the simple formulae
\begin{equation}\label{calculus-fd}
f(Z)= f(0)P_0+\sum_{j=1}^m f(z_j^{\uparrow}) P_j,
\end{equation}
 $P_0$ standing for the projection onto the nullspace of $Z$.
Therefore, given two continuous functions such that $f(t)\ge g(t)$, we have $f(Z)\ge g(Z)$. Applying this to a basic sequence we
obtain the following remark.

\vskip 5pt
\begin{remark}\label{remark}  If  $f(t), g(t)\in C^0$ satysfy $f(t)\ge g(t)$, then $f(A) \ge g(A)$ for all $A\in \bS({\mathcal{H}})$.\end{remark}

\vskip 5pt
The functional calculus for a finite rank operator \eqref{calculus-fd} also shows that
$$
\langle h, f(Z)f \rangle \le \sup_{t\in \bR} f(t)
$$
for all unit vectors $h$. Applying this to a basic sequence provides the next simple observation.

\vskip 5pt
\begin{remark}\label{remark2}  If  $f(t)\in C^0$ and $A\in \bS({\mathcal{H}})$, then, for all unit vectors $h$,
$$
\langle h, f(A)h \rangle \le \sup_{t\in \bR} f(t).
$$

\end{remark}

\section{Spectral projections}

The aim of this section is to define the operator version of the matrix spectral projections \eqref{mspp}. We first give a useful
fact, known as Vigier's theorem.

\vskip 5pt
\begin{lemma} Let $\{Z_n\}_{n\in\bN}$ be a decreasing sequence of positive operators on ${\mathcal{H}}$, that is $Z_n\ge Z_{n+1}$ for
all $n$. Then there exists a positive operator $Z$ such that $$Z ={\mathrm{sot}}\lim_{n\to\infty} Z_n$$
\end{lemma}

\begin{proof} The wot convergence follows from polarization. Then, we may infer the sot convergence. Indeed, given
$h\in{\mathcal{H}}$, one has
\begin{align*}
\| Z_{n+p}h-Z_nh \|^2&= \langle h, (Z_n-Z_{n+p})^2h\rangle \\
&= \langle h, (Z_n-Z_{n+p})^{1/2}  (Z_n-Z_{n+p})  (Z_n-Z_{n+p})^{1/2} h\rangle \\
&\le  \| Z_n-Z_{n+p}\|\langle h, (Z_n-Z_{n+p}) h\rangle.
\end{align*}
\end{proof}

\vskip 5pt
\begin{lemma} Let $A\in\bS({\mathcal{H}})$. For $\lambda\in(-\infty,\infty)$ and $n\in\bN^*$, define the three piece affine
continuous function $f_{\lambda,n}(t)$ such that
$$f_{\lambda,n}(t):=\begin{cases} 1 & \text{if \ $ t\le \lambda$},\\
1+n(\lambda-t) & \text{if \ $\lambda\le t\le \lambda + n^{-1}$}, \\
0 &  \text{if \ $ \lambda + n^{-1}\le t$}.\\
\end{cases}
$$
Then, there exists a projection $E(\lambda)\in\bS({\mathcal{H}})$ commuting with $A$ such that
$$ E(\lambda) = {\mathrm{sot}}\lim_{n\to\infty} f_{\lambda,n}(A)$$
and $\langle h, AE(\lambda)h\rangle \le \lambda$ for all unit vectors $h$.
\end{lemma} 

\vskip 5pt
\begin{proof} Note that $f_{\lambda, n}(t)\ge f_{\lambda, n+1}(t) \ge 0$, $n\in\bN$. Remark \ref{remark} shows that $\{f_{\lambda,
n}(A)\}_{n\in\bN}$ is a decreasing sequence of positive operators. Thanks to Lemma \ref{lem-1} we thus have a positive operator
$E(\lambda)$ as the strong limit
$$
 E(\lambda) = {\mathrm{sot}}\lim_{n\to\infty} f_{\lambda,n}(A).
$$
Since $ f_{\lambda,n}(A)A=A f_{\lambda,n}(A)$, we also have the commutativity of $A$ with this strong limit,
$E(\lambda)A=AE(\lambda)$.
Further, by Remark \ref{remark2}, given a unit vectors $h$,
$$
\langle h, A f_{\lambda,n}(A)h\rangle \le \sup_{t\in\bR} t f_{\lambda,n}(t) 
$$
As the supremum tends to $\lambda$ as $n\to\infty$, we get
 $\langle h, AE(\lambda)h\rangle \le \lambda$ for all unit vectors $h$.

Since $1\ge f_{\lambda,n}(t)$, Remark \ref{remark} shows that $I\ge E(\lambda)$ and hence $E(\lambda)\ge E^2(\lambda)$. We also have
$E^2(\lambda)\ge E(\lambda)$, indeed,
\begin{align*}
 E^2(\lambda)&= \left({\mathrm{sot}}\lim_{n\to\infty}  f_{\lambda,n}(A) \right)^2 \\
&={\mathrm{sot}}\lim_{n\to\infty} f^2_{\lambda,n}(A) \\
&\ge {\mathrm{sot}}\lim_{m\to\infty} f_{\lambda,m}(A)=E(\lambda)
\end{align*}
as for each integer $n$ there exists an integer $m$ such that $f^2_{\lambda,n}(t) \ge f_{\lambda,m}(t)$.
Thus $E^2(\lambda)=E(\lambda)$, i.e.,  $E(\lambda)$ is a projection.
\end{proof}

\begin{remark}\label{remarkcomm} The projection valued map $\lambda\to E(\lambda)$ is increasing and sot right continuous, we call
this map (as well as the family $\{E(\lambda)\}$ with $\lambda\in\bR)$ {\bf the spectral measure} of $A$. The previous lemma shows
that given two self-adjoint operators $A$ and $B$ with respectives spectral measures $\{E(\lambda)\}$ and $\{F(\nu)\}$, then the
commutativity assumption
$AB=BA$ ensures the commutativity of the spectral measures: $E(\lambda)F(\nu)= F(\nu)E(\lambda)$.
\end{remark}

\section{The spectral theorem}

In case of $A\in \bS({\mathcal{H}})$, the matrix decomposition \eqref{matrix2} (involving the matrix spectral measure....) remains
valid in a continuous form:

\begin{theorem}\label{thm1}
Let $A\in \bS({\mathcal{H}})$ with its spectral measure $E(\lambda)$. Then
$$
A=\int \lambda \,{\mathrm{d}} E(\lambda)
$$
\end{theorem}

Now we explain the Stieltjes-integral notation employed, and next give the proof which is a straightforward consequence of the existence
of the spectral measure built up in the previous Lemma.

\vskip 5pt
We may detail a little bit more the theorem, The spectral measure vanishes on an open set $\omega(A)$ of $\bR$:
$$
x\in \omega(A) \iff\exists \varepsilon >0 \  {\mathrm{s.t.}}\ E(x+\varepsilon)-E(x-\varepsilon)=0.
$$
The complementary set of $\omega$ is the compact set $\sigma(A)$, the spectrum of $A$. Then we may write the spectral theorem as

\vskip 5pt
\begin{theorem}\label{thm2}
Let $A\in \bS({\mathcal{H}})$ with its spectral measure $E(\lambda)$. Then
$$
A=\int_{\sigma(A)} \lambda \,{\mathrm{d}} E(\lambda).
$$
\end{theorem}

\vskip 5pt
Denote by $C^*(A)$ the unital $C^*$-algebra spanned by $A$. We have the following concrete version of the Gelfand isomorphism
mentionned in the introduction.

\vskip 5pt
\begin{cor} Let $A\in \bS({\mathcal{H}})$. Then the functional calculus from $C^0(\sigma(A))$ to  $C^*(A)$,
$$
f(t)\mapsto f(A):= \int_{\sigma(A)} f(\lambda) \,{\mathrm{d}} E(\lambda).
$$
is an isometric  $*$-isomorphism.
\end{cor}

\vskip 5pt
From Theorem \ref{thm1} and Remark \ref{remarkcomm} we get:

\vskip 5pt
\begin{cor} Let $A,B\in \bS({\mathcal{H}})$ be commuting. Then $AB\in \bS({\mathcal{H}})$.
\end{cor}

\vskip 5pt
Taking advantage that a normal operator $N$ has a decomposition $N=A+iB$ in which $A$ and $B$ are two commuting self-adjoint
operators, the above corollary and Remark \ref{remarkcomm} show that $A$ and $B$ have
two commuting respective spectral measures, say $E(x)$ and $F(y)$, $x,y\in\bR$. Thus  we obtain a spectral measure $G(z)$, $z\in\bC$, for $N$ and get the following spectral theorem for normal operators:

\vskip 5pt
\begin{theorem}
Let $N$ be a normal operator with its spectral measure $G(\lambda)$. Then
$$
N=\int_{\bC} \lambda \,{\mathrm{d}} G(\lambda)=\int_{\sigma(N)} \lambda \,{\mathrm{d}} G(\lambda)
$$
\end{theorem}

\end{document}